\documentclass[11pt]{amsart}
\date{October 18, 2023}
\usepackage{latexsym,amsmath,amsthm,amsfonts,amscd,amssymb,mathdots}
\usepackage{stmaryrd}
\usepackage[mathscr]{eucal} 
\usepackage{mathrsfs}
\usepackage[utf8]{inputenc}
\usepackage[all,cmtip]{xy}
\usepackage{hyperref}
\usepackage{graphics}
\usepackage{lscape}
\usepackage{enumerate}
\usepackage{array}
\usepackage{framed}
\usepackage{mathtools}
\usepackage{epsfig,color}
\usepackage{hyperref}
\usepackage{tikz}
\hypersetup{colorlinks}
\usepackage{tikz-cd}

\usepackage{color} 

\definecolor{darkred}{rgb}{1,0,0} 
\definecolor{darkgreen}{rgb}{0,1,0}
\definecolor{darkblue}{rgb}{0,0,1}

\hypersetup{colorlinks,
linkcolor=darkblue,
filecolor=darkblue,
urlcolor=black,
citecolor=darkred}

\theoremstyle{plain}  
\newtheorem{theorem}{Theorem}[section]
\newtheorem*{theoremA}{Theorem A}
\newtheorem*{theoremB}{Theorem B}
\newtheorem*{theoremC}{Theorem C}
\newtheorem*{theoremD}{Theorem D}
\newtheorem*{theoremE}{Theorem E}

\newtheorem*{theorem*}{Theorem}
\newtheorem*{conjecture*}{Conjecture}

\newtheorem{corollary}[theorem]{Corollary}
\newtheorem{lemma}[theorem]{Lemma}
\newtheorem{proposition}[theorem]{Proposition}

\newtheorem{tech-lemma}[theorem]{Technical Lemma}
\newtheorem{definition}[theorem]{Definition}
\theoremstyle{remark}
\newtheorem{example}[theorem]{Example}

\newtheorem{remark}[theorem]{Remark}
\newtheorem*{remark*}{Remark}

\newtheorem*{claim*}{Claim}

\newtheoremstyle{TheoremForIntro} 
        {.6em}{.6em}              
        {\itshape}                      
        {}                              
        {\bfseries}                     
        {.}                             
        { }                             
        {\thmname{#1}\thmnote{ \bfseries #3}}
    \theoremstyle{TheoremForIntro}

\numberwithin{equation}{section}
\renewcommand{\leq}{\leqslant}

\renewcommand{\geq}{\geqslant}

\newcommand{\R}{\mathbb{R}}

\newcommand{\Z}{\mathbb{Z}}
\newcommand{\C}{\mathbb{C}}
\newcommand{\HH}{\mathbb{H}}

\newcommand{\cB}{{\mathcal B}}
\newcommand{\cC}{{\mathcal C}}

\newcommand{\cE}{{\mathcal E}}

\newcommand{\cG}{{\mathcal G}}
\newcommand{\cH}{{\mathcal H}}

\newcommand{\cK}{{\mathcal K}}
\newcommand{\cL}{{\mathcal L}}
\newcommand{\cM}{{\mathcal M}}

\newcommand{\cO}{{\mathcal O}}
\newcommand{\cP}{{\mathcal P}}

\newcommand{\cT}{{\mathcal T}}

\newcommand{\cV}{{\mathcal V}}
\newcommand{\cW}{{\mathcal W}}
\newcommand{\cX}{{\mathcal X}}

\newcommand{\Fuch}{\mathrm{Fuch}}

\newcommand{\rPGL}{\mathrm{PGL}}
\newcommand{\rPSL}{\mathrm{PSL}}

\newcommand{\rPSO}{\mathrm{PSO}}
\newcommand{\rSU}{\mathrm{SU}}

\newcommand{\rA}{\mathrm{A}}
\newcommand{\rB}{\mathrm{B}}
\newcommand{\rD}{\mathrm{D}}
\newcommand{\rS}{\mathrm{S}}
\newcommand{\rK}{\mathrm{K}}
\newcommand{\rF}{\mathrm{F}}
\newcommand{\rE}{\mathrm{E}}
\newcommand{\rT}{\mathrm{T}}
\newcommand{\SU}{\mathrm{SU}}

\newcommand{\rP}{\mathrm{P}}
\newcommand{\rG}{\mathrm{G}}
\newcommand{\rU}{\mathrm{U}}
\newcommand{\rC}{\mathrm{C}}
\newcommand{\rGL}{\mathrm{GL}}
\newcommand{\rL}{\mathrm{L}}
\newcommand{\rZ}{\mathrm{Z}}
\newcommand{\rSL}{\mathrm{SL}}
\newcommand{\rSO}{\mathrm{SO}}
\newcommand{\rO}{\mathrm{O}}
\newcommand{\rSp}{\mathrm{Sp}}
\newcommand{\rSpin}{\mathrm{Spin}}

\newcommand{\lieg}{\mathfrak{g}}

\newcommand{\fgl}{\mathfrak{gl}}
\newcommand{\fsl}{\mathfrak{sl}}
\newcommand{\fsu}{\mathfrak{su}}
\newcommand{\fso}{\mathfrak{so}}
\newcommand{\fsp}{\mathfrak{sp}}
\newcommand{\fa}{\mathfrak{a}}
\newcommand{\fb}{\mathfrak{b}}
\newcommand{\fc}{\mathfrak{c}}
\newcommand{\fe}{\mathfrak{e}}
\newcommand{\ff}{\mathfrak{f}}
\newcommand{\fg}{\mathfrak{g}}
\newcommand{\fh}{\mathfrak{h}}
\newcommand{\fk}{\mathfrak{k}}
\newcommand{\fl}{\mathfrak{l}}
\newcommand{\fm}{\mathfrak{m}}

\newcommand{\fp}{\mathfrak{p}}

\newcommand{\fs}{\mathfrak{s}}

\newcommand{\fu}{\mathfrak{u}}

\newcommand{\fz}{\mathfrak{z}}

\DeclareMathOperator{\ad}{ad}
\DeclareMathOperator{\Ad}{Ad}
\DeclareMathOperator{\Aut}{Aut}
\DeclareMathOperator{\aut}{aut}

\DeclareMathOperator{\rk}{rk}

\DeclareMathOperator{\Hom}{Hom}

\DeclareMathOperator{\Id}{Id}

\newcommand{\G}{\mathrm{G}}
\newcommand{\rH}{\mathrm{H}}

\newcommand{\smtrx}[1]{\left (\begin{smallmatrix}#1\end{smallmatrix}\right)}

\let\oldmarginpar\marginpar
\renewcommand\marginpar[1]{\oldmarginpar{\tiny\bf\begin{flushleft} #1
\end{flushleft}}}

\newcommand{\rkfge}{r(e)} 
\newcommand{\tilfg}{\tilde\fg} 
\newcommand{\tilrG}{\tilde\rG} 
\newcommand{\tilfm}{\tilde\fm} 

\begin{document}

%
%

\title[Cayley correspondences and Higher Teichmüller spaces]
{A general Cayley correspondence and higher rank Teichmüller spaces}
%
%


\author[S. Bradlow]{Steven Bradlow}
\address{S. Bradlow,
\newline\indent
Department of Mathematics, University of Illinois at Urbana-Champaign\newline\indent
Urbana, IL 61801, USA}
\email{bradlow@math.uiuc.edu}

\author[B. Collier]{Brian Collier}
\address{B. Collier,
\newline\indent Department of Mathematics, University of California Riverside\newline\indent Riverside, CA 92521, USA}
\email{brian.collier@ucr.edu}

\author[O. Garc\'{i}a-Prada]{Oscar Garc\'{i}a-Prada}
\address{O. Garc\'{i}a-Prada,
\newline\indent Instituto de Ciencias Matem\'aticas,  CSIC-UAM-UC3M-UCM, \newline\indent 
Nicol\'as Cabrera, 13--15, 28049 Madrid, Spain}
\email{oscar.garcia-prada@icmat.es}

\author[P.~B.\ Gothen]{Peter B.\ Gothen}
\address{P.~B. Gothen,
\newline\indent Centro de Matemática da Universidade do Porto, 
    \newline\indent Faculdade de Ci\^encias da Universidade do Porto, 
    \newline\indent Rua do Campo Alegre s/n, 4169-007 Porto, Portugal}
\email{pbgothen@fc.up.pt}

\author[A. Oliveira]{Andr\'e Oliveira}
\address{A. Oliveira, 
\newline\indent 
Centro de Matem\'atica da Universidade do Porto,
     \newline\indent Faculdade de Ci\^encias da Universidade do Porto, 
     \newline\indent
     Rua do Campo Alegre s/n, 4169-007 Porto, Portugal \newline\indent \textsl{On leave from:}\newline\indent Departamento de Matem\'atica, Universidade de Tr\'as-os-Montes e Alto Douro, UTAD,\newline\indent
Quinta dos Prados, 5000-911 Vila Real, Portugal}
\email{andre.oliveira@fc.up.pt\newline\indent agoliv@utad.pt}

\thanks{
}
\keywords{Semistable Higgs bundles, connected components of moduli spaces}
\subjclass[2010]{14D20, 14F45, 14H60}

\begin{abstract}
 We introduce a new class of $\fsl_2$-triples in a complex simple Lie algebra $\fg$, which we call magical. Such an $\fsl_2$-triple canonically defines a real form and various decompositions of $\fg$. Using this decomposition data, we explicitly parameterize special connected components of the moduli space of Higgs bundles on a compact Riemann surface $X$ for an associated real Lie group, hence also of the corresponding character variety of representations of $\pi_1X$ in the associated real Lie group. This recovers known components when the real group is split, Hermitian of tube type, or $\rSO_{p,q}$ with $1<p\leq q$, and also constructs previously unknown components for the quaternionic real forms of $\rE_6$, $\rE_7$, $\rE_8$ and $\rF_4$. The classification of magical $\fsl_2$-triples is shown to be in bijection with the set of $\Theta$-positive structures in the sense of Guichard--Wienhard, thus the mentioned parameterization conjecturally detects all examples of higher rank Teichm\"uller spaces. Indeed, we discuss properties of the surface group representations obtained from these Higgs bundle components and their relation to $\Theta$-positive Anosov representations, which indicate that this conjecture holds.
\end{abstract}

\maketitle

\markleft{{Bradlow, Collier, Garc\'{i}a-Prada, Gothen, Oliveira}}

\tableofcontents


\section{Introduction}

In this paper we introduce a new framework for special components in moduli spaces of Higgs bundles. Via the nonabelian Hodge correspondence these components are the analogs of higher rank Teichm\"uller spaces in character varieties of surface group representations. The framework unifies previously described constructions for various types of real Lie groups, namely split real groups, Hermitian groups of tube type, and $\rSO_{p,q}$, and establishes the existence of new Teichm\"uller-like spaces for quaternionic exceptional real Lie groups.

Fix a closed orientable surface $\Sigma$ with genus $g\geq 2$ and fundamental group $\pi_1\Sigma$. For any reductive Lie group $\rG$, the $\rG$-character variety $\cX(\rG)$ parameterizes conjugacy classes of reductive representations $\pi_1\Sigma\to \rG$.  Recall that the Teichm\"uller space $\cT$ of complex structures on $\Sigma$ is realized as the set of conjugacy classes of \emph{Fuchsian representations} $\pi_1\Sigma\to \rPSL_2\R$. Moreover, $\cT$ defines an open and closed subset of $\cX(\rP\rSL_2\R)$ consisting entirely of discrete and faithful representations.  In the general setting, where $\rPSL_2\R$ is replaced by a reductive group $\rG$, there is a class of representations (introduced by Labourie \cite{AnosovFlowsLabourie} and since studied by many authors; see \cite{guichard_wienhard_2012,KLPDynamicsProperCocompact,AnosovAndProperGGKW}) called \emph{Anosov representations} which generalize many features of Fuchsian representations.  
These representations define \emph{open} subsets of the character variety consisting entirely of discrete and faithful representations with many interesting geometric and dynamical properties. Unlike $\cT\subset \cX(\rPSL_2\R)$, the Anosov loci are not necessarily closed, so do not automatically define connected components.  In cases where they do constitute such components, they define subsets of $\cX(\rG)$ which are open, closed and consist entirely of discrete faithful representations. Such spaces are called \emph{higher rank Teichm\"uller spaces} \cite{AnnaICM,BeatriceBourbaki}. 

One way of constructing Anosov representations is to post-compose a lift of a representation in $\cT$ with a homomorphism $\iota_e:\rSL_2\R\to\rG$. Up to conjugation, such homomorphisms are labeled by nilpotent elements $e$ in the Lie algebra of $\rG$.  When $\rG$ is a complex simple Lie group, there is a (unique, up to conjugation) special homomorphism $\iota_e:\rSL_2\C\to\rG$, called \emph{principal}, and the restriction of $\iota_e$ to $\rSL_2\R$ is contained in the split real form $\rG^\R\subset\rG$, \cite{ptds}.  In \cite{liegroupsteichmuller}, Hitchin used this to define connected components of $\cX(\rG^\R)$ containing $\iota_e(\cT)$ --- now called \emph{Hitchin components}. Representations in Hitchin components were shown to be Anosov by Labourie for $\rP\rSL_n\R$ \cite{AnosovFlowsLabourie} and, with different methods, by Fock--Goncharov \cite{fock_goncharov_2006} for general split groups. Other examples of components of Anosov representations arise from so-called \emph{maximal representations} into Hermitian Lie groups \cite{MaxRepsAnosov,CompteRenduBIW}. 

Recently, Guichard--Wienhard \cite{PosRepsGWPROCEEDINGS,GuichardWienhardPosFull} defined a generalization of Lusztig's theory of total positivity \cite{LusztigTotPos} called $\Theta$-positivity. 
Roughly, a parabolic subgroup $\rP_\Theta\subset\rG^\R$ of a real Lie group $\rG^\R$ has a \emph{$\Theta$-positive structure} if triples of pairwise disjoint transverse points in $\rG^\R/\rP_\Theta$ admit a cyclic order. For such pairs $(\rG^\R,\rP_\Theta)$, it is possible to define a set of \emph{$\Theta$-positive Anosov representations}. This set is open in $\cX(\rG^\R)$ and conjectured to be closed \cite{PosRepsGWPROCEEDINGS,GuichardWienhardPosFull}.  The $\Theta$-positive structures have been classified, leading to a list of possible higher rank Teichm\"uller spaces, which includes all the examples mentioned above as well as two other possible families.

The Hitchin components were discovered in \cite{liegroupsteichmuller} using the \emph{nonabelian Hodge correspondence}, which defines a homeomorphism between the character variety $\cX(\rG)$ and the moduli space $\cM(\rG)$ of polystable $\rG$-Higgs bundles on a Riemann surface $X$ with underlying surface $\Sigma$. In particular, using Higgs bundles, Hitchin parameterized the Hitchin component by a vector space of holomorphic differentials. The spirit of the current paper is similar, and Higgs bundles will be our main focus.  Due to the transcendental nature of this correspondence, it is very difficult to characterize the notions of Anosov representations and $\Theta$-positive structures in terms of Higgs bundles so we develop in this paper a new Lie theoretic notion, called \emph{magical $\fsl_2$-triple} in a complex Lie algebra $\fg$, which is adapted to the language of Higgs bundles.  

In one of our main results, we classify all such magical $\fsl_2$-triples, and confirm that this classification establishes a bijection between them and $\Theta$-positive structures. Furthermore we prove properties about the resulting Higgs bundles and 
find new connected components in moduli spaces $\cM(\rG^\R)$ where $\rG^\R$ is a real Lie group determined by a magical $\fsl_2$-triple.  We call these components {\it Cayley components} (see Definition \ref{cayleycomp}) because the construction generalizes a similarly named construction in the case where $\rG^\R$ is a Hermitian group of tube type.  Using the nonabelian Hodge correspondence to translate our results into statements about character varieties, we show that these components contain open sets of $\Theta$-positive Anosov representations and hence should describe new higher rank Teichm\"uller spaces. 

We now give slightly more detailed statements of our results, starting with a description of the magical $\fsl_2$-triples.

Let $\fg$ be a complex simple Lie algebra and $e\in\fg$ be a nonzero nilpotent element. By the Jacobson--Morozov theorem, $e$ can be completed to a triple $\{f,h,e\}$ which generates a subalgebra of $\fg$ isomorphic to $\fsl_2\C$. This defines a bijective correspondence between conjugacy classes of nonzero nilpotents and conjugacy classes of $\fsl_2\C$-subalgebras. Using the decomposition of $\fg$ as an $\fsl_2\C$-module, we define a \emph{vector space} involution 
\[\sigma_e:\fg\longrightarrow\fg,\]
which is $+\Id$ on the trivial $\fsl_2\C$-representation, $-\Id$ on the nonzero highest weight spaces and $-\Id$ on $f$ (see \S \ref{subsec: nilpotents and sl2 triples} for details). We call the $\fsl_2$-triple $\{f,h,e\}\subset\fg$ \emph{magical} if $\sigma_e$ is a \emph{Lie algebra involution}. 

The involution $\sigma_e$ was first defined by Hitchin for principal $\fsl_2$-triples. A key point in his work was showing that the involution $\sigma_e$ is a Lie algebra homomorphism \cite[Proposition 6.1]{liegroupsteichmuller}.  
Generalizing the main results of \cite{liegroupsteichmuller}, we show that magical $\fsl_2$-triples determine components of character varieties which conjecturally describe all higher rank Teichm\"uller components.  The character varieties in which these occur are determined by canonical real forms $\fg^\R\subset\fg$ associated to magical triples $\{f,h,e\}$ (see Definition \ref{def: canonical real form}). For principal triples, the canonical real form is the split real form of $\fg$ \cite[Proposition 6.1]{liegroupsteichmuller}.

\begin{theoremA}[Theorem \ref{thm character variety components}]
 Let $\rG$ be a complex simple Lie group with Lie algebra $\fg$ and $\{f,h,e\}\subset\fg$ be a magical $\fsl_2$-triple with canonical real form $\rG^\R\subset\rG$. Let $\Sigma$ be a closed orientable surface of genus $g\geq2$ and $\cX(\rG^\R)$ be the $\rG^\R$-character variety of $\Sigma$. Then, there exists a nonempty open and closed subset
    \[\cP_e(\rG^\R)\subset\cX(\rG^\R),\]
 which contains $\iota_e(\cT)$ and does not contain representations which factor through compact subgroups. 
 Moreover, the centralizer of any representation $\rho\in\cP_e(\rG^\R)$ is compact. In particular, there is no proper parabolic subgroup $\rP^\R\subset \rG^\R$ such that $\rho:\pi_1\Sigma\to\rP^\R\hookrightarrow\rG^\R$.
\end{theoremA}

As mentioned above, the sets $\cP_e(\rG^\R)$ are constructed by applying the nonabelian Hodge correspondence to Cayley components in the moduli space $\cM(\rG^\R)$ of $\rG^\R$-Higgs bundles. Briefly, a $\rG^\R$-Higgs bundle on a compact Riemann surface $X$ is a pair $(\cE,\varphi)$, where $\cE$ is a holomorphic principal bundle on $X$ and $\varphi$ (the Higgs field) is a holomorphic section of an associated vector bundle twisted by the holomorphic cotangent bundle $K$ of $X$ (see \S \ref{subsec: Higgs bundles} for more details). We will also consider the moduli space $\cM_L(\rG^\R)$ of $L$-twisted Higgs bundles, where the twisting line bundle $K$ is replaced by a line bundle $L$. 

The Cayley components  in $\cM(\rG^\R)$ are constructed from the Lie theoretic data of a magical $\fsl_2$-triple.  In addition to the real form $\fg^\R$, each magical $\fsl_2$-triple $\{f,h,e\}\subset\fg$ defines a real form $\fg^\R_\cC$ of the centralizer $\fg_0$ of the semisimple element $h$ (see Definition \ref{def: Cayley real form}). We call $\fg^\R_\cC$ the \emph{Cayley real form}. We also show that a magical $\fsl_2$-triple $\{f,h,e\}$ is principal (see Proposition \ref{prop subalge g(e)}) in a simple subalgebra $\fg(e)\subset\fg$ defined as the semisimple part of the double centralizer of $\{f,h,e\}$, i.e., the centralizer of the centralizer of $\{f,h,e\}$.  This defines a decomposition of the Cayley real form (see Proposition \ref{prop cayley real form class}) as
\[\fg_\cC^\R=\tilfg^\R\oplus\R^{\rkfge},\]
where $\tilfg^\R$ is either zero or a simple real Lie algebra and
$\rkfge=\rk(\fg(e))$ is the rank of $\fg(e)$. Hence we have a real Lie group 
\begin{equation}\label{cayleygroup intro}
\rG^\R_\cC=\tilrG^\R\times (\R^+)^{\rkfge},
\end{equation}
 which we call the \emph{Cayley group}.   
This additional structure imposed by the existence of a magical $\fsl_2$-triple leads to a concrete description of these new connected components in terms of moduli spaces associated to the Cayley group.

\begin{theoremB}[Theorem \ref{thm Cayley map open closed}]
Let $\rG$ be a complex simple Lie group with Lie algebra $\fg$, and $\{f,h,e\}\subset\fg$ be a magical $\fsl_2$-triple with canonical real form $\rG^\R$. Let $\fg(e)\subset\fg$ be the semisimple part of the double centralizer of $\{f,h,e\}$ and  $\rG_\cC^\R=\tilrG^\R\times (\R^+)^{\rkfge}$ be the Cayley group. Let $X$ be a compact Riemann surface of genus $g\geq 2$ with canonical bundle $K$, and let $\cM(\rG^\R)$ be the moduli space of $\rG^\R$-Higgs bundles over $X$.  Then there is a positive integer $m_c$ and a well-defined injective, open and closed map 
\begin{equation}\label{eq Cayley map intro}
    \Psi_e:\cM_{K^{m_c+1}}(\tilrG^\R)\times\bigoplus\limits_{j=1}^{\rkfge}H^0(K^{l_j+1})\longrightarrow \cM(\rG^\R),
\end{equation}
where $\{l_j\}$ are the exponents of $\fg(e)$ and $\cM_{K^{m_c+1}}(\tilrG^\R)$ is the moduli space of $K^{m_c+1}$-twisted $\tilrG^\R$-Higgs bundles. Furthermore, every Higgs bundle in the image of $\Psi_e$ has nowhere vanishing Higgs field. 
\end{theoremB}

\begin{remark}
 The connected components in the image of $ \Psi_e$ are the Cayley components.  The integer $m_c$ and the exponents of $\fg(e)$ come from the decomposition of $\fg$ as an $\fsl_2\C$-module. Namely, as an $\fsl_2\C$-module, $\fg=W_0\oplus W_{2m_c}\oplus \bigoplus_{j=1}^{\rkfge}W_{2l_j}$, where $W_{2k}$ is a direct sum of a certain number of copies of the unique irreducible $\fsl_2\C$-representation of dimension $2k+1$. See Lemma \ref{lem identification of Slodowy domain with Cayley group} for more details.
\end{remark}

The map $\Psi_e$ is a moduli space version of the global Slodowy slice map for Higgs bundles constructed in \cite{ColSandGlobalSlodowy}. However, it is nontrivial to show that when $\{f,h,e\}$ is magical the Slodowy map descends to an injective map on moduli spaces.  Our proof relies on our third main result, namely the classification of magical $\fsl_2$-triples given in Theorem \ref{thm: classification weighted dynkin}. 

\begin{theoremC}[Theorem \ref{thm: classification weighted dynkin} and Proposition \ref{prop canonical real form}] Let $\fg$ be a simple complex Lie algebra and let $\fg^\R\subset\fg$ be a real form. Then $\fg^\R$ is the canonical real form associated to a magical $\fsl_2$-triple if and only if it is one of the following:
\begin{enumerate}
\item $\fg$ is any type and $\fg^\R$ is its split real form;
\item $\fg$ has type $\rA_{2n-1},~ \rB_n,~\rC_n,~\rD_n,~\rD_{2n}$, or $\rE_7$ and $\fg^\R$ is Hermitian of tube type, i.e.~$\fg^\R$ is one of the following:
\begin{enumerate}
\item $\fsu_{n,n}$,
\item $\fso_{2,p}\ (\mathrm{with}\ 2+p=2n+1)$,
\item $\fsp_{2n}\R$,
\item $\fso_{2,p}\ (\mathrm{with}\ 2+p=n)$,
\item $\fso_{4n}^*$, or 
\item the real form of $\rE_7$ of Hermitian type;
\end{enumerate}
\item $\fg$ has type $\rB_n$ or $\rD_n$ and $\fg^\R$ is $\fso_{p,q}$ with $1<p<q$ and $p+q=2n+1$ or $p+q=2n$ respectively;
\item $\fg$ has type $\rE_6,~\rE_7,~\rE_8$ or $\rF_4$ and $\fg^\R$ is its quaternionic real form.
\end{enumerate}
\end{theoremC}

\begin{remark}
    A real form $\fg^\R$ is called \emph{quaternionic} if its associated Riemannian symmetric space is quaternionic K\"ahler, equivalently, the maximal compact subalgebra has a simple factor isomorphic to $\fsu_2$. Up to isomorphism, there is a unique quaternionic real form of type $\rE_6,~\rE_7,~\rE_8$ and $\rF_4$, see \cite[Appendix C]{knappbeyondintro}.
\end{remark}
The proof of Theorem C uses the correspondence between nilpotents in classical Lie algebras and partitions, and classification data of Dokovi\'c \cite{ExceptionalNilpotentsInner,ExceptionalNilpotentsOuter} for exceptional Lie algebras. 

While a magical $\fsl_2$-triple $\{f,h,e\}\subset\fg$ defines a canonical real form $\fg^\R\subset\fg,$ the generators $\{f,h,e\}$ are not real, i.e., $f,h,e\notin\fg^\R.$ 
We can obtain real $\fsl_2$-triples using the so called Cayley transform (see \S\ref{subsec: Real nilpotents and the Sekiguchi correspondence}).  The Cayley transform of a magical $\fsl_2$-triple, denoted by $\{\hat f,\hat h,\hat e\}$, has each of its generators in the canonical real form $\fg^\R.$  
This allows us to relate magical triples to the Guichard--Wienhard notion of $\Theta$-positivity.  Recall that a nilpotent element $\hat e\in\fg^\R$ determines a parabolic subgroup $\rP_{\hat e}^\R\subset\rG^\R$. The four families of pairs $(\rG^\R,\rP_{\hat e}^\R)$, which arise from magical $\fsl_2$-triples are the following:
\begin{enumerate}
        \item $\rG^\R$-split and $\rP^\R_{\hat e}$ is the Borel subgroup.
        \item $\rG^\R$ is a Hermitian group of tube type and $\rP^\R_{\hat e}$ is the maximal parabolic associated to the Shilov boundary.
        \item $\rG^\R$ is locally isomorphic to $\rSO_{p,q}$ and $\rP^\R_{\hat e}$ stabilizes an isotropic flag of the form  
        \[\R\subset\R^2\subset\cdots\subset \R^{p-1}\subset\R^{q+1}\subset\cdots\subset\R^{p+q-1}\subset\R^{p+q}.\]
        \item $\rG^\R$ is a quaternionic real form of $\rE_6$, $\rE_7$, $\rE_8$ or $\rF_4$, so that its restricted root system is that of $\rF_4$, and $\rP_{\hat e}^\R$ is determined by the simple roots $\{\alpha_1,\alpha_2\}$, where 
       \begin{center}
          \begin{tikzpicture}[scale=.4]
    \draw (-1,0) node[anchor=east]  {$\rF_4:$};
  \foreach \x in {0,...,3}
    \draw[xshift=\x cm,thick] (\x cm,0) circle (.25cm);
      \draw[thick] (0.25 cm,0) -- +(1.5 cm,0);
    \draw[thick] (2.25 cm, .1 cm) -- +(1.5 cm,0);
    \draw[thick] (2.8 cm, 0) -- +(.3cm, .3cm);
    \draw[thick] (2.8 cm, 0) -- +(.3cm, -.3cm);
    \draw[thick] (2.25 cm, -.1 cm) -- +(1.5 cm,0);
    \draw[thick] (4.25 cm, 0) -- +(1.5cm,0);
    \node at (0,0) [below = 1 mm ] {\scriptsize{$\alpha_1$}};
    \node at (2,0) [below = 1 mm ] {\scriptsize{$\alpha_2$}};
    \node at (4,0) [below = 1 mm ] {\scriptsize{$\alpha_3$}};
    \node at (6,0) [below = 1 mm ] {\scriptsize{$\alpha_4$}};
    \node at (7,0)  {.};
  \end{tikzpicture}
        \end{center}
    \end{enumerate}
Comparing this list with Guichard--Wienhard's classification of $\Theta$-positive structures gives the following theorem.

\begin{theoremD} [Theorem \ref{thm pos and mag}]   
Let $\rG$ be a complex simple Lie group and $\rG^\R\subset\rG$ be a real form. A pair $(\rG^\R,\rP^\R_\Theta)$ admits a $\Theta$-positive structure if and only if $\rP_\Theta^\R=\rP_{\hat e}^\R$, where $\{\hat f,\hat h,\hat e\}\subset\fg^\R$ is the Cayley transform of a magical $\fsl_2$-triple with canonical real from $\rG^\R.$
\end{theoremD} 
\begin{remark}
 Even though Theorem D results from observing that the two classifications agree, a posteriori, more can be said about the link between positivity and magical $\fsl_2$-triples. Namely, the collection of invariant cones arising from a positive structure is exactly the orbit of the nilpotent $\hat e$ by the identity component of the Levi factor $\rL_\Theta.$ It would be interesting to develop the link between these two perspectives further. 
\end{remark}
To further relate the open and closed sets $\cP_e(\rG^\R)$ from Theorem A  with $\Theta$-positivity, we prove that each of the sets $\cP_e(\rG^\R)$ contains an open set of $\Theta$-positive Anosov representations. As above, let $\fg(e)\subset\fg$ be the semisimple part of the double centralizer of a magical $\fsl_2$-triple $\{f,h,e\}$. The Lie algebra $\fg(e)$ defines a split subalgebra $\fg(e)^\R$ of the canonical real form $\fg^\R$. Let $\rG(e)^\R\subset\rG^\R$ be the connected subgroup with Lie algebra $\fg(e)^\R$. 
One special property of magical $\fsl_2$-triples is that their $\rG^\R$-centralizer $\rC^\R$ is compact. By construction, the groups $\rG(e)^\R$ and $\rC^\R$ commute, so we can form a $\rG^\R$-representation by multiplying a $\rG(e)^\R$-representation with a $\rC^\R$-representation.
In Proposition \ref{prop hitchin rep in Pe} we show that the sets $\cP_e(\rG^\R)$ contain representations of the form  
\begin{equation}
    \label{eq reps intro positive}\rho=\rho_{Hit}*\rho_{\rC^\R}:\pi_1\Sigma\longrightarrow\rG^\R, 
\end{equation}
where $\rho_{Hit}:\pi_1\Sigma\to\rG(e)^\R$ is a $\rG(e)^\R$-Hitchin representation and $\rho_{\rC^\R}:\pi_1\Sigma\to\rC^\R$ is any representation.  This allows us to prove the following theorem.

\begin{theoremE}[Theorem \ref{thm positive reps in Pe comps}]
Let $\rG$ be a simple complex Lie group with Lie algebra $\fg.$ Let $\{f,h,e\}\subset\fg$ be a magical $\fsl_2$-triple with canonical real form $\rG^\R\subset\rG.$ Then the set of representations $\rho_{Hit}*\rho_{\rC^\R}$ from \eqref{eq reps intro positive} are $\Theta$-positive Anosov representations. 
  In particular, each of the sets $\cP_e(\rG^\R)\subset\cX(\rG^\R)$ from Theorem $\rA$ contains a nonempty open set of $\Theta$-positive Anosov representations. 
\end{theoremE}

\begin{remark}\label{magic=theta} In the case that the magical $\fsl_2$-triple $\{f,h,e\}\subset\fg$ defines a principal $\fsl_2\C$-subalgebra this theorem is due to Labourie \cite{AnosovFlowsLabourie} in type $A$ and Fock--Goncharov \cite{fock_goncharov_2006} in general. For Hermitian groups of tube type it is due to \cite{MaxRepsAnosov}; see also \cite{CompteRenduBIW}. It is expected that the sets $\cP_e(\rG^\R)$ correspond exactly to the sets of $\Theta$-positive Anosov representations in all cases. 
\end{remark}

Our results follow on from several projects focused on the enumeration or understanding of distinguished components in the moduli spaces of Higgs bundles, and hence in character varieties for surface groups. 

Given that Higgs bundles consist of an underlying principal bundle together with a Higgs field, it is clear that the topological type of the principal bundle is an invariant of connected components in the moduli space. Similarly, for the character varieties, the topological type of the associated flat bundle is also an invariant of the components. For complex reductive Lie groups \cite{JunLiConnectedness,Oliveira_GarciaPrada_2016}, and also for compact groups \cite{ramanathan_1975}, the components are fully classified by these topological invariants. The count is more complicated for $\rG^\R$-Higgs bundles, where $\rG^\R$ is a non-compact real form.  Indeed, this was already evident in Goldman's component count for $\rG^\R$-character varieties where $\rG^\R$ is a finite cover of $\rPSL_2\R$ \cite{TopologicalComponents}, and in Hitchin's Higgs-bundle version for $\rSL_2\R$ \cite{selfduality}.  Since then the enumeration and study of connected components in Higgs bundle moduli space for other non-compact real forms has been extensively pursued; for example see \cite{liegroupsteichmuller,sp4GothenConnComp,UpqHiggs,AndrePGLnR,AndreOscarSUstar,HiggsbundlesSP2nR,Sp(2p2q)modulispaceconnected,SO2n*connected,CollierSOnn+1components,so(pq)BCGGO,TopInvariantsAnosov,Baraglia18,baragliaschaposnikmonodromyrank2,DavidLauraCayleyLanglands}.

From one perspective, the starting point for the present work is the description of the Hitchin components in \cite{liegroupsteichmuller}. In particular, Hitchin proved Theorems A and B above for the case in which the magical $\fsl_2$-triple $\{f,h,e\}\subset\fg$ defines a principal $\fsl_2\C$-subalgebra, i.e.~Case (1) of the Theorems C and D. Indeed, in this case, the first factor in \eqref{cayleygroup intro} equals the center of the group (which is finite),  and Hitchin's description is recovered exactly by the map \eqref{eq Cayley map intro}.  The opposite extreme, where the second factor in \eqref{cayleygroup intro} is $H^0(K^2)$, occurs in Case (2) of the classification theorems. In this case Theorems A and B recover results for $\rG^\R$-Higgs bundles when $\rG^{\R}$ is of Hermitian tube type (see \cite{HermitianTypeHiggsBGG} and \cite{BGRmaximalToledo}). In particular, the moduli space $\cM_{K^{m_c+1}}(\tilrG^\R)\times H^0(K^2)$ has $m_c=1$, and is then exactly the moduli space of $K^2$-twisted Higgs bundles for the Cayley partner to $\rG^{\R}$, i.e.~the space which describes components with maximal Toledo invariant. The third case in Theorems C and D includes the case investigated in \cite{so(pq)BCGGO} for $\rG^\R=\rSO_{p,q}$, in which case the map \eqref{eq Cayley map intro} recovers the description of the `exotic' components identified in \cite{so(pq)BCGGO}, but now adds the remaining locally isomorphic groups.

From a slightly different perspective, our results relate to a program initiated by Hitchin to count connected components by a Morse-theoretic method  \cite{selfduality,liegroupsteichmuller}.  Described more fully in \S\ref{section components}, the method is based on a proper function $F:\cM(\rG^\R)\to \R$ defined by the $L^2$-norm of the Higgs field, and exploits the fact that proper functions attain their minima on closed sets. The locus of local minima thus has at least as many components as the full moduli space. Obvious minima of $F$, where the Higgs field is identically zero, lie on components detected by the topological invariants of principal bundles. The existence of other components --- including the ones we study in this paper --- is detected by more subtle local minima.  In \S\ref{section components} we identify such minima coming from the components in the image of \eqref{eq Cayley map intro} and use this to enumerate the components.

We end this introduction with some open questions, organized in a series of conjectures, and a short discussion on what remains to be proven. 
\begin{conjecture*}
    \begin{enumerate}
        \item A representation $\rho\in\cX(\rG^\R)$ is $\Theta$-positive if and only if $\rho$ is in one of the spaces $\cP_e(\rG^\R)$ from Theorem $\rA$.
        \item A connected component of $\cX(\rG^\R)$ is a higher rank Teichm\"uller space if and only if it is a connected component of one of the spaces $\cP_e(\rG^\R)$ from Theorem $\rA$ or $\rG^\R$ is a Hermitian group of nontube type and the Toledo invariant is maximal. 
        \item All components of $\cX(\rG^\R)$ which are not higher rank Teichm\"uller spaces are uniquely labeled by  invariants which depend only on the topological type of $\rG^\R$-bundles over $\Sigma$.
        \item Other than the components in the image of the Cayley map \eqref{eq Cayley map intro} and the components with maximal Toledo invariant for Hermitian groups of nontube type, the components of $\cM(\rG^\R)$ are uniquely labeled by topological invariants of $\rG^\R$-bundles over $\Sigma$. 
    \end{enumerate}
\end{conjecture*}
We note that the first two conjectures are consistent with Guichard--Wienhard's conjecture that positivity provides the correct unifying framework for higher rank Teichm\"uller spaces \cite{PosRepsGWPROCEEDINGS}. 
Since the first version of this paper was released, one direction of the first conjecture has been settled. Namely, all representations in the spaces $\cP_e(\rG^\R)$ are $\Theta$-positive. This had already been established for split groups \cite{fock_goncharov_2006,AnosovFlowsLabourie} and for Hermitian groups of tube type \cite{BIWmaximalToledoAnnals,CompteRenduBIW}. 
For groups locally isomorphic to $\rSO_{p,q}$, Beyrer--Pozzetti recently proved that the space of positive representations is closed  \cite{BeyrerPozzettiSOpq}, and hence all representations in $\cP_e(\rSO_{p,q})$ are positive. 
Separately, Guichard--Labourie--Wienhard proved that positive representations are closed in the space of representations which do not factor through proper parabolic subgroups \cite{GLWPosRepsarXiv}. Hence, by Theorem A, positive representations define components of the character variety, namely the components $\cP_e(\rG^\R)$. 
Very recently, this has been shown with a proof independent from the results of this paper in \cite{BGLPWPosClosed}. 

The Hermitian groups of nontube type are locally isomorphic to $\rSU_{p,q}$ with $p\neq q$, $\rSO^*_{2n+2}$ and $\rE_6^{-14}$. For such groups there is not a notion of positivity; however, representations with maximal Toledo invariant always factor through a maximal tube type subgroup where they are positive. Hence, maximal representations into such groups define higher rank Teichm\"uller spaces \cite{BIWmaximalToledoAnnals}; see also \cite{UpqHiggs,BGRmaximalToledo}. 
Note that if (2) holds, then (3) and (4) are equivalent. The simple groups for which all conjectures have been established are $\rP\rSL_n\R$ by \cite{liegroupsteichmuller} and \cite{AnosovFlowsLabourie}, the Hermitian real forms locally isomorphic to $\rSU_{p,q}$ and $\rSO_{2,3}$ by \cite{UpqHiggs,chains-2018,nonmaxSp4,AndreQuadraticPairs} and \cite{BIWmaximalToledoAnnals}, the groups locally isomorphic to $\rSO_{p,q}$ for $2<p<q$ by \cite{so(pq)BCGGO} and \cite{BeyrerPozzettiSOpq}, and groups locally isomorphic to $\rSU^*_{2n}$ and $\rSp_{2p,2q}$ by \cite{AndreOscarSUstar,Sp(2p2q)modulispaceconnected}. The noncompact real forms of simple groups which are missing are $\rSp_{2n}\R$ for $n>2$, $\rSO_{2n}^*$, $\rSO_{2,n}$ for $n>3$, and all real  forms of exceptional type. 

\subsection*{Acknowledgments:} We would like to especially thank Jeff
Adams for his help with many of the Lie theoretic aspects of this
paper. We also thank Mark Burger, Carlos Florentino, François
Labourie, Ana Peón-Nieto, Beatrice Pozzetti, Andy Sanders and Anna
Wienhard for enlightening conversations. Finally, we also thank the anonymous referees
for a number of remarks which lead to relevant improvements in the paper. This material is based
partly upon work supported by the National Science Foundation under
Grant No. 1440140, while some of the authors were in residence at the
Mathematical Sciences Research Institute in Berkeley, California,
during the Fall semester of 2019. The authors acknowledge support from
U.S. National Science Foundation Grants DMS 1107452, 1107263, 1107367
``RNMS: GEometric structures And Representation varieties" (the GEAR
Network). B.C. was partially supported by the NSF under
Award No.1604263. O.G.-P. was partially supported by the Spanish
MINECO under the ICMAT Severo Ochoa grant No.SEV-2015-0554, and grant
No.MTM2016-81048-P.  P.G. and A.O. were supported by CMUP under the project with reference
UIDB/00144/2020 and by the project EXPL/MAT-PUR/1162/2021, both financed
by national funds through FCT -- Funda\c{c}\~ao para a Ci\^encia e a
Tecnologia, I.P. 
A.O. also wishes to thank the Department of
Mathematics of the University of Maryland and the Instituto de
Ciencias Matematicas (ICMAT), that he visited in the course of
preparation of this paper.



\section{Nilpotents and magical $\fsl_2$-triples}\label{sec: Magical nilpotents and sl_2-subalgebras}
Let $\fg$\label{frakg} be a finite-dimensional complex simple Lie algebra and $\rG$
be a connected complex Lie group with Lie algebra $\fg$. For
background on nilpotents we mostly follow \cite{CollMcGovNilpotents}.

\subsection{Nilpotents and $\fsl_2$-triples }\label{subsec: nilpotents and sl2 triples}
An element $e\in\fg$ is called \emph{nilpotent} if the corresponding adjoint map 
\[\ad_e:\fg\longrightarrow\fg\]
is a nilpotent endomorphism. The nilpotent elements of $\fg$ form a $\rG$-invariant cone consisting of finitely many $\rG$-orbits. In fact, there is a unique nilpotent orbit which is open and dense in the nilpotent cone, and elements in this orbit are called \emph{principal nilpotents}. 
For example, when $\rG=\rSL_n\C$, nilpotent orbits are in bijection with partitions of $n$ by the Jordan decomposition theorem. In this case, a principal nilpotent is conjugate to a full Jordan block.  

By the Jacobson--Morozov theorem, every nonzero nilpotent element $e\in\fg$ can be completed to a triple of nonzero elements $\{f,h,e\}\subset\fg$ satisfying 
    \begin{equation}\label{eq bracket relations sl2}
    \xymatrix{[h,e]=2e,&[h,f]=-2f&\text{and}&[e,f]=h.}
    \end{equation}
Moreover, if $\{f,h,e\}$ and $\{f',h,e\}$ are two such triples, then $f=f'$. 
A triple $\{f,h,e\}$ of nonzero elements verifying the bracket relations \eqref{eq bracket relations sl2} will be called an \emph{$\fsl_2$-triple} and the subalgebra $\langle f,h,e\rangle\subset\fg$ will be called the \emph{$\fsl_2\C$-subalgebra associated to $\{f,h,e\}$}. This defines a bijection between conjugacy classes of nilpotents and conjugacy classes of $\fsl_2\C$-subalgebras
\begin{equation}\label{eq:nilp<->sl2triple}
\xymatrix{\{e\in\fg \text{ nonzero nilpotent}\}/\rG \ar@{<->}[r]^-{1-1} & \{\phi:\fsl_2\C\to\fg\}/\rG.}
\end{equation}

An $\fsl_2$-triple $\{f,h,e\}$ defines two decompositions of $\fg$. One as an $\fsl_2\C$-module, namely
\begin{equation}
    \label{eq sl2 module decomp}\fg=\bigoplus\limits_{j=0}^MW_{j},
\end{equation}
where $W_j$ is isomorphic to a direct sum of $n_j$ copies (with
$n_j\geq0$) of the unique irreducible $(j+1)$-dimensional $\fsl_2\C$-representation. By \emph{$\fsl_2$-data} of  $\{f,h,e\}$ we will mean the collection of pairs of non-negative integers $(j,n_j)$ such that, for each $j>0$, the multiplicity $n_j$ of $W_j$ is positive (so we consider the pair $(0,n_0)$ part of the $\fsl_2$-data even if $n_0=0$).
Another decomposition of $\fg$ determined by $\{f,h,e\}$ is given by $\ad_h$-weight spaces,
\begin{equation}
    \label{eq adh weight decomp}\fg=\bigoplus_{j=-l}^l\fg_j,
\end{equation}
where $\fg_j=\{x\in\fg~|~\ad_h(x)=jx\}$. Note that $\ad_e:\fg_j\to\fg_{j+2}$ and $\ad_f:\fg_{j}\to\fg_{j-2}$. The subalgebra $\bigoplus_{j\geq 0}\fg_j$ is a parabolic subalgebra determined by the nilpotent $e.$

\begin{remark}\label{rmk: even nilpotents}
    A nilpotent $e\in\fg$ is called \emph{even} if $\ad_h$ only has even eigenvalues, i.e., if $\fg_j=0$ for all $j$ odd. The $\fsl_2\C$-subalgebra $\langle f,h,e\rangle\subset\fg$, for an even nilpotent $e$, defines a subgroup of the adjoint group of $\rG$ which is isomorphic to $\rPSL_2\C$.
\end{remark}

The centralizer $\ker(\ad_e)=V(e)=V\subset\fg$\label{p:centralizer-e} of $e$ decomposes into a direct sum of highest weight spaces of each $W_j$
\begin{equation}\label{eq decomp highest weight spaces}
    V=\bigoplus\limits_{j\geq0}V_j,
\end{equation}
where $V_j=W_j\cap\fg_j$. We have the following proposition (see Lemmas 3.4.5 and 3.7.3 of  \cite{CollMcGovNilpotents}).

\begin{proposition}\label{prop W_0 preserves graded intersections}
The subspace $V\subset\fg$ is a subalgebra such that $V_0=W_0$ is a reductive subalgebra and $\bigoplus_{j>0}V_j$ is a nilpotent subalgebra. In addition, for each $j,k$, the subspace $W_j\cap \fg_k \subset\fg$  is preserved by bracketing with $W_0$.
\end{proposition}

\begin{remark}\label{rem slodowy slice}\mbox{}
\begin{enumerate}
\item  Note that $V_0=W_0\subset\fg$ is the Lie subalgebra which centralizes the $\fsl_2\C$-subalgebra $\langle f,h,e\rangle$. We will often denote this subalgebra by $\fc=W_0\subset\fg$.
\item The affine space 
    \[f+V\subset\fg\]
    is a slice of the adjoint action of $\rG$ on $\fg$ through the nilpotent $f$, which is usually called a \emph{Slodowy slice} \cite{Slodowy-Simplesingularities-book}. Note that $\fc$ preserves the Slodowy slice. 
   \end{enumerate}
\end{remark}

\subsection{Magical $\fsl_2$-triples}\label{subsec: Magical nilpotents}
Let $\{ f,h,e\}\subset \fg$ be an $\fsl_2$-triple. Note that 
\[\fg=\bigoplus_{j=0}^M\bigoplus_{k=0}^jW_j\cap\fg_{j-2k}\] and that $W_j\cap\fg_{j-2k}=\ad_f^k(V_j)$. Consider the map $\sigma_e:\fg\to\fg$ defined by the linear extension of 
\begin{equation}
    \label{eq magical involution}\sigma_e(x)=\begin{dcases}
    x&\text{if } x\in V_0\\
    (-1)^{k+1}x&\text{if }x\in \ad_f^k(V_j) \text{ for some $0\leq k\leq j$ and $j>0$}.
\end{dcases}
\end{equation}
 This defines a vector space involution of $\fg$ with $\sigma_e|_{V_j}=-\Id$ for  $j>0$. On the given $\fsl_2$-triple, we have $\sigma_e(f)=-f$, $\sigma_e(h)=h$ and $\sigma_e(e)=-e$.
 
\begin{definition}\label{def: magical nilpotent}
    An $\fsl_2$-triple $\{f,h,e\}$ will be called \emph{magical} if the involution $\sigma_e:\fg\to\fg$ defined by \eqref{eq magical involution} is a Lie algebra involution.
    We will also refer to a nilpotent element $e\in\fg$ as \emph{magical} if it belongs to a magical $\fsl_2$-triple. 
\end{definition} 
\begin{remark}
    Although the terminology was not used, Hitchin showed in \cite[Proposition 6.1]{liegroupsteichmuller} that a principal $\fsl_2$-triple is magical. 
\end{remark}
\begin{remark}\label{remark subalgebra mag}
    Note that if $\{f,h,e\}\subset\fg$ is magical and contained in a reductive subalgebra $\fg'\subset\fg,$ then $\{f,h,e\}$ is magical in the subalgebra $\fg'.$
    \end{remark}
We will classify magical nilpotents in \S \ref{sec: mag class} and by \eqref{eq:nilp<->sl2triple} this will be equivalent to classifying magical $\fsl_2$-triples.  
A key feature of principal $\fsl_2$-triples is that the subalgebra $\fg_0$ is a Cartan subalgebra. We now generalize this to magical triples.
For an $\fsl_2$-triple $\{f,h,e\}$, let $Z_{2m_j}=W_{2m_j}\cap\fg_0$. Thus, we have a decomposition of $\fg_0$ as a $\fc$-module    
\begin{equation}\label{eq:g0Sl2module}
\fg_0=\fc\oplus \bigoplus_{j=1}^M Z_{2m_j}.
\end{equation}

\begin{proposition}\label{prop: bracket relation on g0}
If $\{f,h,e\}$ is a magical $\fsl_2$-triple, then $[Z_{2m_i},Z_{2m_j}]\subset \fc$ for all $m_i,m_j$, and $[Z_{2m_i},Z_{2m_j}]=0$ if $m_i\neq m_j$. 
\end{proposition}


Before giving the proof we recall some facts about $\fsl_2\C$-representation theory. Consider the decomposition \eqref{eq sl2 module decomp} of $\fg$.
The Lie bracket defines a morphism of $\fsl_2\C$-representations:
    \[[~,~]:W_{2m_i}\otimes W_{2m_j}\longrightarrow W_0\oplus\bigoplus_{k=1}^MW_{2m_k.}\]
According to the Clebsch--Gordan formula, the tensor product $W_{2m_i}\otimes W_{2m_j}$ decomposes as a direct sum of irreducible representations
    \begin{equation}\label{eq:C-G-decomp}
    W_{2m_i}\otimes W_{2m_j}\cong \bigoplus_{l=0}^{2\min(m_i,m_j)}\Big(S^{2m_i+2m_j-2l}\Big)^{\oplus n_in_j},
    \end{equation}
where $S^{d}$ is the $d^{th}$-symmetric product of the standard $\fsl_2\C$-representation $W_1$. The projection onto the summand $(S^{2m_i+2m_j-2l})^{\oplus n_in_j}$ is given by contracting $l$-times with the volume form on $\C^2$. 
If we represent $S^{2d}$ as homogeneous polynomials in $z_1,z_2$ of degree $2d$, the elements $x\in Z_{2m_i}$ and $y\in Z_{2m_j}$ are multiples of $z_1^{m_i}z_2^{m_i}$ and $z_1^{m_j}z_2^{m_j}$, respectively. Moreover, since the volume form is skew-symmetric, contracting $(2l+1)$-times $z_1^{m_i}z_2^{m_i}$ with $z_1^{m_j}z_2^{m_j}$ gives zero. Thus, the projection of the bracket $[x,y]$ to $Z_{2m_k}$ is zero when $m_i+m_j=m_k+1 \mod 2$.

\begin{proof}[Proof of Proposition \ref{prop: bracket relation on g0}]
Suppose $\{f,h,e\}$ is magical. Let $x\in Z_{2m_i}$, $y\in Z_{2m_j}$ and write $[x,y]=z_0+\sum z_k$, where $z_0\in \fc$ and $z_k\in Z_{2m_k}$. Note that $\sigma_e(z_0)=z_0$ and $\sigma_e(z_k)=(-1)^{m_k+1}z_k$. By assumption, we have $\sigma_e([x,y])=[\sigma_e(x),\sigma_e(y)]$, thus
    \[z_0+\sum (-1)^{m_k+1}z_k=(-1)^{m_i+m_j}\big(z_0+\sum z_k\big).\]
In particular, if $m_i+m_j=m_k \mod 2$ then $z_k=0$. It follows, by the above discussion, that $z_k=0$ for all $k>0$.  Thus, 
 \[[Z_{2m_i},Z_{2m_j}]\subset\fc\]
for all $m_i,m_j$. Moreover, by Schur's Lemma, the projection of the bracket $[x,y]$ to $W_0$ is zero unless the decomposition of $W_{2m_i}\otimes W_{2m_j}$ has the trivial representation $W_0$ as a summand. But by \eqref{eq:C-G-decomp} this only happens if $m_i=m_j$, completing the proof.
\end{proof}

By Proposition \ref{prop: bracket relation on g0}, a magical $\fsl_2$-triple $\{f,h,e\}$ defines a Lie algebra involution $\theta_e:\fg_0\to\fg_0$
\begin{equation}
    \label{eq g0 involution}
    \theta_e(x)=\begin{dcases}
    x&\text{if } x\in W_0\\
    -x& \text{if } x\in \bigoplus_{j=1}^M Z_{2m_j}.
\end{dcases}
\end{equation}
\begin{remark}
    Note that $\theta_e$ and $\sigma_e|_{\fg_0}$ are different since $\theta_e(h)=-h$ and $\sigma_e(h)=h$. 
\end{remark}

\subsection{The canonical real form associated to a magical nilpotent}\label{subsec: involutions and real forms}
In this section we mainly follow \cite[\S 3]{CartanGaloisCohomAdams}. A \emph{real form} of the complex Lie group $\rG$ is defined to be the fixed point set $\rG^\tau$ of an anti-holomorphic involution 
\[\tau:\rG\longrightarrow\rG.\]
We will sometimes refer to the involution $\tau$ itself as a real form.
Note that even though $\rG$ is connected, the real form $\rG^\tau$ may not be connected. For example, $\rSO_{p,q}\subset\rSO_{p+q}\C$ is a real form which has two components whenever $p$ or $q$ is nonzero. 
If the fixed point set $\rG^\tau\subset\rG$ is compact, the real form is said to be \emph{compact}. Such real forms exist and are unique up to conjugation. 

A holomorphic involution $\sigma:\rG\to\rG$ is called a \emph{Cartan involution for a real form $\tau$} if $\sigma\tau=\tau\sigma$ and, in addition, $\sigma\tau$ is a compact real form of $\rG$. Given a real form $\tau$, a Cartan involution $\sigma$ for $\tau$ exists and is unique up to conjugation by the identity component $(\rG^\tau)^0\subset\rG^\tau$. Conversely, given a holomorphic involution $\sigma$, there exists a real form $\tau$, unique up to conjugation by $(\rG^\sigma)^0$, such that $\sigma$ is a Cartan involution for $\tau$.

  The following proposition will be useful (cf.~\cite[Theorem 3.13]{CartanGaloisCohomAdams}).
\begin{proposition}\label{prop: compatible involution}
Let $\rG'\subset\rG$ be a reductive subgroup. If $\sigma:\rG\to\rG$ is a holomorphic involution of $\rG$ with $\sigma(\rG')=\rG'$ and $\tau_{\rG'}$ is a real form of $\rG'$ such that $\sigma|_{\rG'}$ a Cartan involution for $\tau_{\rG'}$, then there exists a real form $\tau:\rG\to\rG$ with $\sigma$ a Cartan involution for $\tau$ such that $\tau|_{\rG'}=\tau_{\rG'}$. Conversely, if $\tau:\rG\to\rG$ is a real form of $\rG$ with $\tau(\rG')=\rG'$ and $\sigma_{\rG'}$ is a Cartan involution for $\tau|_{\rG'}$, then there exists a Cartan involution $\sigma:\rG\to\rG$ for $\tau$ such that $\sigma|_{\rG'}=\sigma_{\rG'}$.
 \end{proposition}

An involution $\alpha:\rG\to\rG$ induces an involution $\alpha:\fg\to\fg$, and the Lie algebra of the fixed-point group $\rG^\alpha$ is the fixed-point subalgebra $\fg^\alpha$. Moreover, if $\alpha:\rG\to\rG$ is holomorphic or anti-holomorphic then $\alpha:\fg\to\fg$ is complex linear or conjugate-linear respectively. In the latter case, $\fg^\alpha$ is a real form of $\fg$, i.e., $\fg^\alpha\otimes\C\cong\fg$.
\begin{remark}
    An involution of the Lie algebra $\fg$ does \emph{not} always integrate to an involution of the group $\rG$. However, every inner involution of $\fg$ integrates to $\rG$. Also, when $\rG$ is an adjoint group or simply connected, every Lie algebra involution integrates to $\rG$. Whenever we are dealing with Lie algebra involutions, \emph{we will always assume $\rG$ is a Lie group for which the involution integrates}. 
 \end{remark}
Now fix a real form $\tau$ of $\rG$ and let $\sigma$ be a Cartan
involution for $\tau$. Denote the fixed-point groups by
$\rG^\R=\rG^\tau$ and $\rH=\rG^\sigma$.\label{p:maximal-compact}
Then 
\begin{displaymath}
  \rH^\R=\rH\cap\rG^\R
\end{displaymath}
 is a maximal compact subgroup of both $\rG^\R$ and $\rH$. Furthermore, the associated Lie algebra involution $\sigma:\fg^\R\to\fg^\R$ defines an $\rH^\R$-invariant decomposition of $\fg^\R$ into $\pm 1$-eigenspaces
\begin{displaymath}
\fg^\R=\fh^\R\oplus\fm^\R,
\end{displaymath}
called a \emph{Cartan decomposition}. The associated $\rH$-invariant decomposition $\fg=\fh\oplus\fm$ will be referred to as the \emph{complexified Cartan decomposition}.

Now we go back to our setting. Since the definition of a magical $\fsl_2$-triple involves a complex linear involution of $\fg$, there is a canonical real form of $\fg$ associated to each such triple.

\begin{definition}\label{def: canonical real form}
Let $\{f,h,e\}\subset\fg$ be a magical $\fsl_2$-triple and $\sigma_e:\fg\to\fg$ be the associated Lie algebra involution. Let $\tau_e:\fg\to\fg$ be a real form such that $\sigma_e$ is a Cartan involution \eqref{eq magical involution}. The Lie algebra $\fg^\R=\fg^{\tau_e}$ will be called \emph{the canonical real form of $\fg$ associated to $\{f,h,e\}$.} 
 \end{definition}

\begin{remark}\label{remark compatible involutions}
The $\fsl_2\C$-subalgebra $\fs=\langle f,h,e\rangle$ spanned by the magical $\fsl_2$-triple is $\sigma_e$-stable. Moreover, $\sigma_e|_\fs$ is a Cartan involution for the conjugate linear involution $\tau_\fs:\fs\to\fs$ defined by 
\begin{equation}\label{eq real form of mag sl2}
    \xymatrix{\tau_\fs(h)=- h&\tau_\fs(e)= f&\text{and}&\tau_\fs(f)= e.}
\end{equation}
Since $\fs^{\tau_\fs}$ is isomorphic to $\fsl_2\R$, we can choose the canonical real form $\tau_e:\fg\to\fg$ such that the magical $\fsl_2\C$-subalgebra defines a subalgebra of $\fg^\R$ isomorphic to $\fsl_2\R$. 
\end{remark}

\begin{definition}\label{def Lie group canonical real form} 
Let $\{f,h,e\}\subset\fg$ be a magical $\fsl_2$-triple and $\sigma_e:\fg\to\fg$ be the associated Lie algebra involution \eqref{eq magical involution}. Let $\tau_e:\fg\to\fg$ be a real form such that $\sigma_e$ is a Cartan involution. Let $\rG$ be a connected complex Lie group with Lie algebra $\fg$ such that $\sigma_e$ integrates to an involution $\sigma_e:\rG\to\rG$ and let $\tau_e:\rG\to\rG$ be the anti-holomorphic involution integrating $\tau_e$. We define the \emph{canonical real form $\rG^\R$ of $\rG$ associated to $e$} to be the fixed-point group $\rG^{\tau_e}\subset\rG$.
\end{definition}

The Lie algebra of the canonical real form $\rG^\R$ is the canonical real form $\fg^\R$ of Definition \ref{def: canonical real form}.
The complex linear Lie algebra involution $\theta_e:\fg_0\to\fg_0$ defined in \eqref{eq g0 involution} also associates a real form to a magical $\fsl_2$-triple.
\begin{definition}\label{def: Cayley real form}
    Let $\{f,h,e\}$ be a magical $\fsl_2$-triple, $\fg_0$ be the centralizer of $h$ and $\theta_e:\fg_0\to\fg_0$ be the Lie algebra involution from \eqref{eq g0 involution}. Let $\tau_0:\fg_0\to\fg_0$ be a real form, such that $\theta_e$ is a Cartan involution for $\tau_0$. The Lie algebra $\fg_0^{\tau_0}\subset\fg_0$ will be called the \emph{Cayley real form of $\fg_0$ associated to $e$}, and denoted by $\fg_\cC^\R$.
\end{definition} 
\begin{remark}\label{remark compact centralizer}
    Note that $\theta_e|_{\fc}=\sigma_e|_{\fc}=\Id:\fc\to\fc$ is a Cartan involution for a compact real form $\tau_\fc$ of $\fc$. Thus, by Proposition \ref{prop: compatible involution}, we can assume that the canonical real form $\tau_e:\fg\to\fg$ and the Cayley real form $\tau_0:\fg_0\to\fg_0$ are such that 
    $\tau_e|_\fc=\tau_\fc=\tau_0|_\fc$. In particular, the centralizer $\fc^{\tau_\fc}$ of the $\fsl_2\R$-subalgebra $\fs^{\tau_e}\subset\fg^{\tau_e}$ is compact (where  $\fs=\langle f,h,e\rangle$).
\end{remark}

\subsection{Real nilpotents and the Sekiguchi correspondence}\label{subsec: Real nilpotents and the Sekiguchi correspondence}

The classification of magical $\fsl_2$-triples will use the classification of nilpotent elements in real Lie algebras and the Sekiguchi correspondence. 
Fix a real form $\tau:\rG\to\rG$, a Cartan involution $\sigma:\rG\to\rG$ for $\tau$ and write $\rG^\R=\rG^\tau$, $\rH=\rG^\sigma$ and $\fg=\fh\oplus\fm$ for the complexified Cartan decomposition. In this section, we will refer to $\fsl_2$-triples in $\fg$ as $\fsl_2$-triples, to distinguish them from $\fsl_2\R$-triples in $\rG^\R$, which will also appear.

The \emph{Sekiguchi correspondence} gives a one-to-one correspondence between $\rG^\R$-conjugacy classes of nilpotents in $\fg^\R$ and $\rH$-conjugacy classes of nilpotents in $\fm$:
\begin{equation}
\label{eq sekiguchi correspondence}
\xymatrix{\{\hat e\in\fg^\R \text{ nonzero nilpotent}\}/\rG^\R \ar@{<->}[r]^-{1-1} &  \{e\in\fm\text{ nonzero nilpotent}\}/\rH.}
\end{equation}
It was proven independently in \cite{Sekiguchi} and \cite{DokovicKostConj}. 

We now describe the correspondence in more detail and refer the reader to \cite[Chapter 9]{CollMcGovNilpotents} and \cite[\S6.1]{CartanGaloisCohomAdams} for further details.
The Jacobson--Morozov theorem also holds over $\R$. Namely, every nonzero nilpotent $\hat e\in\fg^\R$ can be completed to an $\fsl_2\R$-triple $\{\hat f,\hat h,\hat e\}$, such that $\hat f,\hat h,\hat e\in\fg^\R\setminus\{0\}$ satisfy the bracket relations \eqref{eq bracket relations sl2}. Moreover, this defines a bijection on conjugacy classes 
\begin{equation*}
\xymatrix{\{ \hat e\in\fg^\R \text{ nonzero nilpotent}\}/\rG^\R \ar@{<->}[r]^-{1-1} &  \{\phi:\fsl_2\R\to\fg^\R\}/\rG^\R.}
\end{equation*}
Following \cite[Chapter 9.4]{CollMcGovNilpotents}, an $\fsl_2\R$-triple $\{\hat f,\hat h,\hat e\}\subset\fg^\R$ is called a \emph{Cayley triple} if $\sigma(\hat f)=-\hat e$, $\sigma(\hat e)=-\hat f$ and $\sigma(\hat h)=-\hat h$. Using Proposition \ref{prop: compatible involution}, one can show that every $\fsl_2\R$-triple is $(\rG^\R)^0$-conjugate to a Cayley triple. On the other hand, an $\fsl_2$-triple $\{f,h,e\}$ is called a \emph{normal triple} if $\sigma(f)=-f$, $\sigma(h)=h$ and $\sigma(e)=-e$. Note that every magical $\fsl_2$-triple is a normal triple with respect to the Cartan involution \eqref{eq magical involution}.

The \emph{Cayley transform} defines a bijection between Cayley triples in $\fg^\R$ and normal triples in $\fg$ by
\[\gamma:\xymatrix@R=0em{\text{Cayley\ triples}\ar[r]&\text{Normal triples}\\\{\hat f,\hat h,\hat e\}\ar@{|->}[r]& \{\frac{1}{2}(\hat f+\hat e-i\hat h),i(\hat e-\hat f),\frac{1}{2}(\hat f+\hat e+i\hat h)\},}
\]
with inverse given by 
\begin{equation}
    \label{eq inverse cayley transform}\gamma^{-1}:\xymatrix@R=0em{\text{Normal\ triples}\ar[r]&\text{Cayley triples}\\\{f, h, e\}\ar@{|->}[r]& \{\frac{1}{2}(f- e+i h),i(f+ e),\frac{1}{2}( f-e-i h)\}.}
\end{equation}
\begin{remark}
    We will refer to both $\gamma$ and $\gamma^{-1}$ as the Cayley transform. Note that $\gamma$ takes the standard generators of $\fsl_2\R$ to those of $\fsu_{1,1},$ and hence is defined by conjugating by the M\"obius transformation identifying the upper half space with the Poincar\'e disk. 
\end{remark}

For the proof of the following, see for instance \cite[Theorem 9.5.1]{CollMcGovNilpotents}.
\begin{proposition}
     The Cayley transform provides the bijective correspondence of the Sekiguchi correspondence \eqref{eq sekiguchi correspondence}.
 \end{proposition}

\begin{definition}\label{def cayley mag}
    Let $\fg^\R$ be a real form of $\fg$ with Cartan involution $\sigma$. A Cayley triple $\{\hat f,\hat h, \hat e\}\subset\fg^\R$ is \emph{magical} if its Cayley transform $\gamma(\{\hat f,\hat h, \hat e\})\subset\fg$ is magical and, moreover, $\fg^\R$ is the canonical real form of $\gamma(\{\hat f,\hat h, \hat e\})$. A nilpotent $\hat e\in\fg^\R$ will be called \emph{magical} if it belongs to a magical Cayley triple.
\end{definition}

Let $\{\hat f,\hat h,\hat e\}\subset\fg^\R$ be a Cayley triple and $\fc^\R\subset\fg^\R$ be its centralizer. Similarly, let $\fc\subset\fg$ be the centralizer of its Cayley transform $\{\gamma(\hat f),\gamma(\hat h),\gamma(\hat e)\}\subset\fg$. It is straightforward to check that $\fc^\R\otimes \C=\fc$.

Recall that $V(\gamma(\hat e))=\ker(\ad_{\gamma(\hat e)})\subset\fg$ denotes the centralizer of the nilpotent $\gamma(\hat e)\in\fg$.

\begin{proposition} \label{prop: magical real nilpotent properties}
   Let $\{\hat f,\hat h,\hat e\}\subset\fg^\R$ be a Cayley triple. Then $\{\hat f,\hat h,\hat e\}$ is magical if and only if $\fc^\R\subset\fh^\R$ and $\dim(\fh\cap V(\gamma(\hat e)))=\dim(\fc)$. 
\end{proposition}
\begin{proof}
If $\{\hat f,\hat h,\hat e\}\subset\fg^\R$ is magical, then $\fc^\R\otimes\C=\fc\subset\fh$ and $V(\gamma(\hat e))\cap\fh=\fc$ by Definition \ref{def: magical nilpotent}. Conversely, if $\fc^\R\subset\fh^\R$ and $\dim(\fh\cap V(\gamma(\hat e)))=\dim(\fc)$, then the Cartan involution $\sigma$ satisfies \eqref{eq magical involution}. Indeed, $\sigma$ is a Lie algebra involution which preserves $V(\gamma(\hat e))$. Moreover, $\sigma$ equals $\Id$ on $\fc$, equals $-\Id$ on the nontrivial highest weight spaces, and also $\sigma(\gamma(\hat f))=-\gamma(\hat f)$.
\end{proof}
The first point of Proposition \ref{prop: magical real nilpotent properties} says that the centralizer of a magical Cayley triple is compact. For the dimension of  $\fh\cap V(\gamma(\hat e))$ we will use the following result.

\begin{proposition}\label{prop dim of centralizer formula}
    \cite[Proposition 5]{KostantRallis} The dimension of $\fh\cap V(\gamma(\hat e))$ is given by 
    \[\dim(\fh\cap V(\gamma(\hat e)))=\frac{1}{2}\Big(\dim(V(\gamma(\hat e)))+\dim (\fh)-\dim(\fm)\Big).\]
\end{proposition}

\section{Classification of magical $\fsl_2$-triples}\label{sec: mag class}
In this section we classify (conjugacy classes of) magical $\fsl_2$-triples in complex simple Lie algebras $\fg$. For classical Lie algebras, we use a classification of nilpotents using signed Young diagrams. For exceptional Lie algebras, we use results of Dokovi\'c in \cite{ExceptionalNilpotentsInner,ExceptionalNilpotentsOuter}. 

\subsection{The classification theorem}\label{subsec: mag class}
There is a complete invariant of conjugacy classes of nilpotent elements of $\fg$ (and hence of $\fsl_2$-triples) called the weighted Dynkin diagram. We briefly recall how this works and refer the reader to \cite[\S3.5]{CollMcGovNilpotents} for more details. Recall that the Dynkin diagram of $\fg$ is a diagram associated to a Cartan subalgebra $\fa\subset\fg$ and a choice of simple roots $\Pi=\{\alpha_1,\ldots,\alpha_{\rk\fg}\}\subset \fa^*$. Its nodes are labeled by the simple roots $\alpha_i$.

Consider an $\fsl_2$-triple $\{f,h,e\}\subset\fg$. Since $h$ is semisimple, there exists a Cartan subalgebra $\fa\subset \fg$ containing $h$. Furthermore, we may choose a set of simple roots $\Pi=\{\alpha_1,\ldots,\alpha_{\rk\fg}\}\subset\fa^*$ so that $\alpha_i(h)\geq0$ for all $i$. In fact, the properties of $\fsl_2$-representation theory imply that $\alpha_i(h)\in\{0,1,2\}$. The \emph{weighted Dynkin diagram} associated to the $\fsl_2$-triple $\{f,h,e\}\subset\fg$ is defined to be the Dynkin diagram of $(\fg,\fa,\Pi)$, where the node associated to the simple root $\alpha_i$ is labeled by the integer $\alpha_i(h)$. 
Note that an $\fsl_2$-triple is even (see Remark \ref{rmk: even nilpotents}) if and only if every node is labeled with either a $0$ or a $2$.
It turns out that if two $\fsl_2$-triples in $\fg$ have the same weighted Dynkin diagram, then they are conjugate. 
However, not every Dynkin diagram whose nodes have labels in $\{0,1,2\}$ is the weighted Dynkin diagram of an $\fsl_2$-triple.

Here is one of the cornerstones of this paper: the classification of magical $\fsl_2$-triples. 
\begin{theorem}\label{thm: classification weighted dynkin}
    Let $\fg$ be a simple complex Lie algebra. Then an $\fsl_2$-triple  $\{f,h,e\}\subset\fg$ is magical if and only if the associated weighted Dynkin diagram is one of the following:
\begin{enumerate}
        \item $\fg$ is any type and every node is labeled with a  $2$;
        \item $\fg$ has type $\rA_{2n-1},~ \rB_n,~\rC_n,~\rD_n,~\rD_{2n}$, or $\rE_7$ with weighted Dynkin diagrams
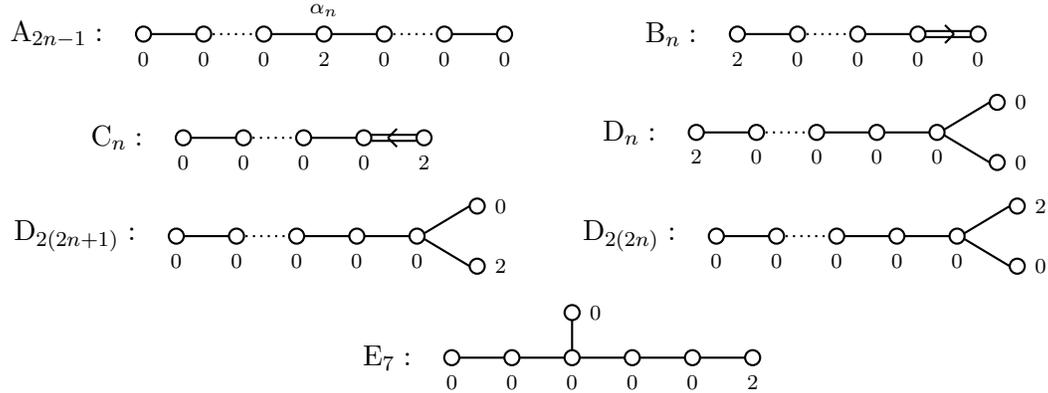
\begin{center}
   \hspace{-1cm} \begin{tabular}
        {ccc}
        \begin{tikzpicture}[scale=.4]
    \draw (-1,0) node[anchor=east]  {$\rA_{2n-1}:$};
    \foreach \x in {0,...,6}
    \draw[xshift=\x cm,thick] (\x cm,0) circle (.25cm);
    \node at (0,0) [below = 1mm ] {${\scriptstyle 0}$};
    \node at (2,0) [below = 1mm ] {${\scriptstyle 0}$};
    \node at (4,0) [below = 1mm ] {${\scriptstyle 0}$};
    \node at (6,0) [below = 1mm ] {${\scriptstyle 2}$};
    \node at (6,0) [above = 1mm ] {${\scriptstyle \alpha_n}$};
    \node at (8,0) [below = 1mm ] {${\scriptstyle 0}$};
    \node at (10,0) [below = 1mm ] {${\scriptstyle 0}$};
      \node at (12,0) [below = 1mm ] {${\scriptstyle 0}$};
    \draw[thick] (0.25 cm,0) -- +(1.5 cm,0);
    \draw[dotted,thick] (2.3 cm,0) -- +(1.4 cm,0);
    \draw[thick] (4.25 cm,0) -- +(1.5 cm,0);
    \draw[thick] (6.25 cm,0) -- +(1.5 cm,0);
    \draw[dotted,thick] (8.3 cm, 0) -- +(1.4 cm, 0);
    \draw[thick] (10.25 cm,0) -- +(1.5 cm,0);
    \end{tikzpicture}&~&
    \begin{tikzpicture}[scale=.4]
        \draw (-1,0) node[anchor=east]  {$\rB_{n}:$};
    \foreach \x in {0,...,4}
    \draw[xshift=\x cm,thick] (\x cm,0) circle (.25cm);
    \node at (0,0) [below = 1mm ] {${\scriptstyle 2}$};
    \node at (2,0) [below = 1mm ] {${\scriptstyle 0}$};
    \node at (4,0) [below = 1mm ] {${\scriptstyle 0}$};
    \node at (6,0) [below = 1mm ] {${\scriptstyle 0}$};
    \node at (8,0) [below = 1mm ] {${\scriptstyle 0}$};
   \draw[thick] (0.25 cm,0) -- +(1.5 cm,0);
    \draw[dotted,thick] (2.3 cm,0) -- +(1.4 cm,0);
    \draw[thick] (4.25 cm,0) -- +(1.5 cm,0);
    \draw[thick] (6.25 cm, -.1 cm) -- +(1.5 cm,0);
    \draw[thick] (6.9 cm, -.3cm) -- +(.3cm, .3cm);
    \draw[thick] (6.9 cm, .3cm) -- +(.3cm, -.3cm);
    \draw[thick] (6.25 cm, .1 cm) -- +(1.5 cm,0);
    \end{tikzpicture}
    \\
    \begin{tikzpicture}[scale=.4]
    \draw (-1,0) node[anchor=east]  {$\rC_{n}:$};
    \foreach \x in {0,...,4}
    \draw[xshift=\x cm,thick] (\x cm,0) circle (.25cm);
    \node at (0,0) [below = 1mm ] {${\scriptstyle 0}$};
    \node at (2,0) [below = 1mm ] {${\scriptstyle 0}$};
    \node at (4,0) [below = 1mm ] {${\scriptstyle 0}$};
    \node at (6,0) [below = 1mm ] {${\scriptstyle 0}$};
    \node at (8,0) [below = 1mm ] {${\scriptstyle 2}$};
   \draw[thick] (0.25 cm,0) -- +(1.5 cm,0);
    \draw[dotted,thick] (2.3 cm,0) -- +(1.4 cm,0);
    \draw[thick] (4.25 cm,0) -- +(1.5 cm,0);
    \draw[thick] (6.25 cm, .1 cm) -- +(1.5 cm,0);
    \draw[thick] (6.8 cm, 0) -- +(.3cm, .3cm);
    \draw[thick] (6.8 cm, 0) -- +(.3cm, -.3cm);
    \draw[thick] (6.25 cm, -.1 cm) -- +(1.5 cm,0);
    \end{tikzpicture}
    &~&
    \begin{tikzpicture}[scale=.4]
        \draw (-1,0) node[anchor=east]  {$\rD_n:$};
    \foreach \x in {0,...,4}
    \draw[xshift=\x cm,thick] (\x cm,0) circle (.25cm);
    \draw[thick] (10 cm, 1 cm) circle (.25cm);
    \draw[thick] (10 cm, -1 cm) circle (.25cm);
    \draw[thick] (0.25 cm,0) -- +(1.5 cm,0);
    \draw[dotted,thick] (2.3 cm,0) -- +(1.4 cm,0);
    \draw[thick] (4.25 cm,0) -- +(1.5 cm,0);
    \draw[thick] (6.25 cm, 0 cm) -- +(1.5 cm,0);
    \draw[thick] (8.25 cm, -.1 cm) -- +(1.5 cm, -.9 cm);
    \draw[thick] (8.25 cm, .1 cm) -- +(1.5 cm, .9 cm);
    \node at (0,0) [below = 1 mm ] {${\scriptstyle 2}$};
    \node at (2,0) [below = 1 mm ] {${\scriptstyle 0}$};
    \node at (4,0) [below = 1 mm ] {${\scriptstyle 0}$};
    \node at (6,0) [below = 1 mm ] {${\scriptstyle 0}$};
    \node at (8,0) [below = 1 mm ] {${\scriptstyle 0}$};
    \node at (10,1) [right = 1 mm ] {${\scriptstyle 0}$};
    \node at (10,-1) [right = 1 mm ] {${\scriptstyle 0}$};
    \end{tikzpicture}
    \\
    \begin{tikzpicture}[scale=.4]
    \draw (-1,0) node[anchor=east]  {$\rD_{2(2n+1)}:$};
    \foreach \x in {0,...,4}
    \draw[xshift=\x cm,thick] (\x cm,0) circle (.25cm);
    \draw[thick] (10 cm, 1 cm) circle (.25cm);
    \draw[thick] (10 cm, -1 cm) circle (.25cm);
    \draw[thick] (0.25 cm,0) -- +(1.5 cm,0);
    \draw[dotted,thick] (2.3 cm,0) -- +(1.4 cm,0);
    \draw[thick] (4.25 cm,0) -- +(1.5 cm,0);
    \draw[thick] (6.25 cm, 0 cm) -- +(1.5 cm,0);
    \draw[thick] (8.25 cm, -.1 cm) -- +(1.5 cm, -.9 cm);
    \draw[thick] (8.25 cm, .1 cm) -- +(1.5 cm, .9 cm);
    \node at (0,0) [below = 1 mm ] {${\scriptstyle 0}$};
    \node at (2,0) [below = 1 mm ] {${\scriptstyle 0}$};
    \node at (4,0) [below = 1 mm ] {${\scriptstyle 0}$};
    \node at (6,0) [below = 1 mm ] {${\scriptstyle 0}$};
    \node at (8,0) [below = 1 mm ] {${\scriptstyle 0}$};
    \node at (10,-1) [right = 1 mm ] {${\scriptstyle 2}$};
    \node at (10,1) [right = 1 mm ] {${\scriptstyle 0}$};
    \end{tikzpicture}
&&
\begin{tikzpicture}[scale=.4]
    \draw (-1,0) node[anchor=east]  {$\rD_{2(2n)}:$};
    \foreach \x in {0,...,4}
    \draw[xshift=\x cm,thick] (\x cm,0) circle (.25cm);
    \draw[thick] (10 cm, 1 cm) circle (.25cm);
    \draw[thick] (10 cm, -1 cm) circle (.25cm);
    \draw[thick] (0.25 cm,0) -- +(1.5 cm,0);
    \draw[dotted,thick] (2.3 cm,0) -- +(1.4 cm,0);
    \draw[thick] (4.25 cm,0) -- +(1.5 cm,0);
    \draw[thick] (6.25 cm, 0 cm) -- +(1.5 cm,0);
    \draw[thick] (8.25 cm, -.1 cm) -- +(1.5 cm, -.9 cm);
    \draw[thick] (8.25 cm, .1 cm) -- +(1.5 cm, .9 cm);
    \node at (0,0) [below = 1 mm ] {${\scriptstyle 0}$};
    \node at (2,0) [below = 1 mm ] {${\scriptstyle 0}$};
    \node at (4,0) [below = 1 mm ] {${\scriptstyle 0}$};
    \node at (6,0) [below = 1 mm ] {${\scriptstyle 0}$};
    \node at (8,0) [below = 1 mm ] {${\scriptstyle 0}$};
    \node at (10,-1) [right = 1 mm ] {${\scriptstyle 0}$};
    \node at (10,1) [right = 1 mm ] {${\scriptstyle 2}$};
    \end{tikzpicture}
    \end{tabular}
    \begin{tikzpicture}[scale=.4]
        \draw (-1,0) node[anchor=east]  {$\rE_7:$};
    \foreach \x in {0,...,5}
    \draw[thick,xshift=\x cm] (\x cm,0) circle (.25cm);
    \foreach \y in {0,...,4}
    \draw[thick,xshift=\y cm] (\y cm,0) ++(.25 cm, 0) -- +(15 mm,0);
    \draw[thick] (4 cm, 1.5 cm) circle (.25cm);
    \draw[thick] (4 cm, .25 cm) -- +(0, 1 cm);
    \node at (0,0) [below = 1 mm ] {${\scriptstyle 0}$};
    \node at (2,0) [below = 1 mm ] {${\scriptstyle 0}$};
    \node at (4,0) [below = 1 mm ] {${\scriptstyle 0}$};
    \node at (6,0) [below = 1 mm ] {${\scriptstyle 0}$};
    \node at (8,0) [below = 1 mm ] {${\scriptstyle 0}$};
    \node at (10,0) [below = 1 mm ] {${\scriptstyle 2}$};
    \node at (4,1.5) [right = 1 mm ] {${\scriptstyle 0}$};
    \end{tikzpicture}
\end{center}
        \item  $\fg$ has type $\rB_n$ or $\rD_n$ with weighted Dynkin diagrams
        \begin{center}
          \hspace{-1cm}  \begin{tabular}
                {ccc}
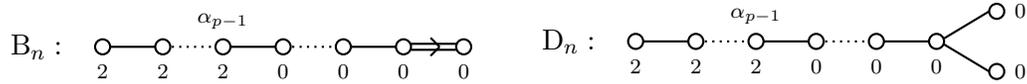
                \begin{tikzpicture}[scale=.4]
        \draw (-1,0) node[anchor=east]  {$\rB_{n}:$};
    \foreach \x in {0,...,6}
    \draw[xshift=\x cm,thick] (\x cm,0) circle (.25cm);
    \node at (0,0) [below = 1mm ] {${\scriptstyle 2}$};
    \node at (2,0) [below = 1mm ] {${\scriptstyle 2}$};
    \node at (4,0) [below = 1mm ] {${\scriptstyle 2}$};
    \node at (4,0) [above = 1mm ] {${\scriptstyle \alpha_{p-1}}$};
    \node at (6,0) [below = 1mm ] {${\scriptstyle 0}$};
    \node at (8,0) [below = 1mm ] {${\scriptstyle 0}$};
    \node at (10,0) [below = 1mm ] {${\scriptstyle 0}$};
    \node at (12,0) [below = 1mm ] {${\scriptstyle 0}$};
   \draw[thick] (0.25 cm,0) -- +(1.5 cm,0);
    \draw[dotted,thick] (2.3 cm,0) -- +(1.4 cm,0);
    \draw[thick] (8.25 cm,0) -- +(1.5 cm,0);
    \draw[dotted,thick] (6.3 cm,0) -- +(1.4 cm,0);
    \draw[thick] (4.25 cm,0) -- +(1.5 cm,0);
    \draw[thick] (10.25 cm, -.1 cm) -- +(1.5 cm,0);
    \draw[thick] (10.9 cm, -.3cm) -- +(.3cm, .3cm);
    \draw[thick] (10.9 cm, .3cm) -- +(.3cm, -.3cm);
    \draw[thick] (10.25 cm, .1 cm) -- +(1.5 cm,0);
    \end{tikzpicture}&~&\begin{tikzpicture}[scale=.4]
        \draw (-1,0) node[anchor=east]  {$\rD_n:$};
    \foreach \x in {0,...,5}
    \draw[xshift=\x cm,thick] (\x cm,0) circle (.25cm);
    \draw[thick] (12 cm, 1 cm) circle (.25cm);
    \draw[thick] (12 cm, -1 cm) circle (.25cm);
    \draw[thick] (0.25 cm,0) -- +(1.5 cm,0);
    \draw[dotted,thick] (2.3 cm,0) -- +(1.4 cm,0);
    \draw[thick] (4.25 cm,0) -- +(1.5 cm,0);
    \draw[dotted,thick] (6.3 cm, 0 cm) -- +(1.4 cm,0);
    \draw[thick] (8.25 cm, 0 cm) -- +(1.5 cm,0);
    \draw[thick] (10.25 cm, -.1 cm) -- +(1.5 cm, -.9 cm);
    \draw[thick] (10.25 cm, .1 cm) -- +(1.5 cm, .9 cm);
    \node at (0,0) [below = 1 mm ] {${\scriptstyle 2}$};
    \node at (2,0) [below = 1 mm ] {${\scriptstyle 2}$};
    \node at (4,0) [below = 1 mm ] {${\scriptstyle 2}$};
    \node at (4,0) [above = 1 mm ] {${\scriptstyle \alpha_{p-1}}$};
    \node at (6,0) [below = 1 mm ] {${\scriptstyle 0}$};
    \node at (8,0) [below = 1 mm ] {${\scriptstyle 0}$};
    \node at (10,0) [below = 1 mm ] {${\scriptstyle 0}$};
    \node at (12,1) [right = 1 mm ] {${\scriptstyle 0}$};
    \node at (12,-1) [right = 1 mm ] {${\scriptstyle 0}$};
    \end{tikzpicture}
            \end{tabular}
        \end{center}
where $1<p< n-1$ for $\rB_n$ and $1<p< n-2$ for $\rD_n;$
        \item $\fg$ has type $\rE_6,~\rE_7,~\rE_8$ or $\rF_4$ with weighted Dynkin diagrams 
        \begin{center}\hspace{-1cm}
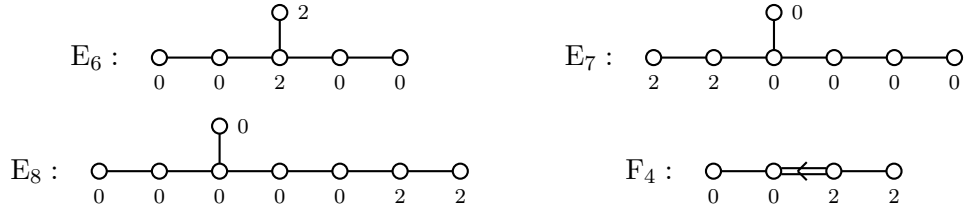
            \begin{tabular}
                {cccc}
                \begin{tikzpicture}[scale=.4]
        \draw (-1,0) node[anchor=east]  {$\rE_6:$};
    \foreach \x in {0,...,4}
    \draw[thick,xshift=\x cm] (\x cm,0) circle (.25cm);
    \foreach \y in {0,...,3}
    \draw[thick,xshift=\y cm] (\y cm,0) ++(.25 cm, 0) -- +(15 mm,0);
    \draw[thick] (4 cm, 1.5 cm) circle (.25cm);
    \draw[thick] (4 cm, .25 cm) -- +(0, 1 cm);
    \node at (0,0) [below = 1 mm ] {${\scriptstyle 0}$};
    \node at (2,0) [below = 1 mm ] {${\scriptstyle 0}$};
    \node at (4,0) [below = 1 mm ] {${\scriptstyle 2}$};
    \node at (6,0) [below = 1 mm ] {${\scriptstyle 0}$};
    \node at (8,0) [below = 1 mm ] {${\scriptstyle 0}$};
    \node at (4,1.5) [right = 1 mm ] {${\scriptstyle 2}$};
      \end{tikzpicture}
    &~&~&
\begin{tikzpicture}[scale=.4]
        \draw (-1,0) node[anchor=east]  {$\rE_7:$};
    \foreach \x in {0,...,5}
    \draw[thick,xshift=\x cm] (\x cm,0) circle (.25cm);
    \foreach \y in {0,...,4}
    \draw[thick,xshift=\y cm] (\y cm,0) ++(.25 cm, 0) -- +(15 mm,0);
    \draw[thick] (4 cm, 1.5 cm) circle (.25cm);
    \draw[thick] (4 cm, .25 cm) -- +(0, 1 cm);
    \node at (0,0) [below = 1 mm ] {${\scriptstyle 2}$};
    \node at (2,0) [below = 1 mm ] {${\scriptstyle 2}$};
    \node at (4,0) [below = 1 mm ] {${\scriptstyle 0}$};
    \node at (6,0) [below = 1 mm ] {${\scriptstyle 0}$};
    \node at (8,0) [below = 1 mm ] {${\scriptstyle 0}$};
    \node at (10,0) [below = 1 mm ] {${\scriptstyle 0}$};
    \node at (4,1.5) [right = 1 mm ] {${\scriptstyle 0}$};
    \end{tikzpicture}
    \\\begin{tikzpicture}[scale=.4]
        \draw (-1,0) node[anchor=east]  {$\rE_8:$};
    \foreach \x in {0,...,6}
    \draw[thick,xshift=\x cm] (\x cm,0) circle (.25cm);
    \foreach \y in {0,...,5}
    \draw[thick,xshift=\y cm] (\y cm,0) ++(.25 cm, 0) -- +(15 mm,0);
    \draw[thick] (4 cm, 1.5 cm) circle (.25cm);
    \draw[thick] (4 cm, .25 cm) -- +(0, 1 cm);
    \node at (0,0) [below = 1 mm ] {${\scriptstyle 0}$};
    \node at (2,0) [below = 1 mm ] {${\scriptstyle 0}$};
    \node at (4,0) [below = 1 mm ] {${\scriptstyle 0}$};
    \node at (6,0) [below = 1 mm ] {${\scriptstyle 0}$};
    \node at (8,0) [below = 1 mm ] {${\scriptstyle 0}$};
    \node at (10,0) [below = 1 mm ] {${\scriptstyle 2}$};
    \node at (12,0) [below = 1 mm ] {${\scriptstyle 2}$};
    \node at (4,1.5) [right = 1 mm ] {${\scriptstyle 0}$};
    \end{tikzpicture}
&~&~&
    \begin{tikzpicture}[scale=.4]
    \draw (-1,0) node[anchor=east]  {$\rF_4:$};
  \foreach \x in {0,...,3}
    \draw[xshift=\x cm,thick] (\x cm,0) circle (.25cm);
    \draw[thick] (0.25 cm,0) -- +(1.5 cm,0);
    \draw[thick] (2.25 cm, .1 cm) -- +(1.5 cm,0);
    \draw[thick] (2.8 cm, 0) -- +(.3cm, .3cm);
    \draw[thick] (2.8 cm, 0) -- +(.3cm, -.3cm);
    \draw[thick] (2.25 cm, -.1 cm) -- +(1.5 cm,0);
    \draw[thick] (4.25 cm, 0) -- +(1.5cm,0);
    \node at (0,0) [below = 1 mm ] {${\scriptstyle 0}$};
    \node at (2,0) [below = 1 mm ] {${\scriptstyle 0}$};
    \node at (4,0) [below = 1 mm ] {${\scriptstyle 2}$};
    \node at (6,0) [below = 1 mm ] {${\scriptstyle 2}$};
  \end{tikzpicture}
            \end{tabular}
        \end{center}
    \end{enumerate}
\end{theorem}

The following is an immediate corollary. 
\begin{corollary}\label{cor mag implies even and injmap h->m}
  Every magical $\fsl_2$-triple $\{f,h,e\}\subset\fg$ is even. In particular, if $\fg=\fh\oplus\fm$ is the $\pm1$-eigenspace of the Lie algebra involution \eqref{eq magical involution}, then  $\fc=\ker(\fh\xrightarrow{\ad_f}\fm)$ and moreover $\ad_f(\fm)\xrightarrow{\ad_f}\ad_f^2(\fm)$ is an isomorphism. 
\end{corollary}

\subsection{The proof}

We now prove Theorem \ref{thm: classification weighted dynkin}. Let $\fg^\R\subset\fg$ be a real form of a complex simple Lie algebra and $\fg=\fh\oplus\fm$ be a complexified Cartan decomposition. Let $\{f,h,e\}\subset\fg$ be a normal $\fsl_2$-triple and $\fc\subset\fg$ its $\fg$-centralizer. We will classify (conjugacy classes of) magical $\fsl_2$-triples of $\fg$ among the normal ones. This will be done via the corresponding real notions of Definition \ref{def cayley mag} by the Sekiguchi correspondence and using Propositions \ref{prop: magical real nilpotent properties} and \ref{prop dim of centralizer formula}.  We will actually prove the theorem by classifying (conjugacy classes of) magical nilpotents in $\fg$ (see \eqref{eq:nilp<->sl2triple}).

We start with the exceptional case. In \cite{ExceptionalNilpotentsInner} and \cite{ExceptionalNilpotentsOuter}, Dokovi\'c computes the dimensions $\dim(\fh\cap\fc)$ and $\dim(\fh\cap V(e))$ for all real forms $\fg^\R$ of simple exceptional Lie algebras. By Proposition \ref{prop: magical real nilpotent properties}, the normal triple $\{f,h,e\}$ is magical if and only if these dimensions are both equal to the dimension of $\fc^\R\subset\fh^\R$.

\begin{proof}[Proof of Theorem \ref{thm: classification weighted dynkin} for exceptional Lie algebras]
For $\fg^\R\subset\fg$ a real form of inner type, the nilpotent orbits (thus the conjugacy classes of nilpotents) are listed in tables VI-XV of \cite{ExceptionalNilpotentsInner}. The first column of each table lists the associated weighted Dynkin diagram of $\fg$, the fourth column lists the dimension of $\fh\cap V(e)$, the fifth column lists the dimension of $\fh\cap\fc$, and the last column lists the isomorphism class of $\fc^\R$.  
For the two outer real forms of $\fe_6$, the weighted Dynkin diagram is column 1 of Tables VI and VII of \cite{ExceptionalNilpotentsOuter}, while the dimensions of $\fh\cap V(e)$ and $\fh\cap \fc$ are columns 9 and 10 of Table VI and columns 12 and 13 of Table VII. 

Table \ref{Table of exceptional inner real forms} of Section \ref{sec diagrams and tables} summarizes this information for inner real forms of $\fg$; note that the real forms $\ff_{4}^{-20}$ and $\fe_6^{-14}$ do not admit magical nilpotents.
For the two outer real forms of $\fe_6$, there is only one magical nilpotent. Namely, the real form $\fe_6^{-26}$ has no magical nilpotents and there is one magical nilpotent in the split real form $\fe_6^6$ (Table VII row 20 of \cite{ExceptionalNilpotentsOuter}). In this case, the weighted Dynkin diagram is that of Case (1) of Theorem \ref{thm: classification weighted dynkin} and $\fc^\R=0$.
  \end{proof}  

We now move to the case of real forms of classical Lie algebras. 
Conjugacy classes of nilpotent endomorphisms of $\C^n$ are in bijective correspondence with \emph{partitions} of $n$. Namely, if $n=\sum_{i=1}^nr_i\cdot i$ is a partition of $n$, with $r_i\geq 0$ the multiplicity of $i$, then the nilpotent endomorphism associated to this partition is 
\begin{equation}
    \label{eq jordan block of nilp}e=\smtrx{
    J_1^{\oplus r_1}\\&\ddots\\&&J_n^{\oplus r_n}},
\end{equation}
where $J_i$ is the standard $i\times i$ Jordan block. Note that $n=n\cdot 1=1+1+\cdots+1$ corresponds to the zero nilpotent whereas $n=1\cdot n=n$ corresponds to the principal nilpotent.

The following proposition classifies conjugacy classes of nilpotents in $\fsl_n\C$, $\fso_n\C$ and $\fsp_{2m}\C$. For a proof, see \cite[Chapter 5.1]{CollMcGovNilpotents}.

\begin{proposition}\label{prop: class nilp}
Let $\rG$ be a connected complex simple Lie group with Lie algebra $\fg$.
\begin{itemize}
    \item For $\fg=\fsl_n\C$, $\rG$-conjugacy classes of nilpotents are in bijective correspondence with partitions of $n$.
    \item For $\fg=\fso_{2n+1}\C$, $\rG$-conjugacy classes of nilpotents are in bijective correspondence with partitions of $2n+1=\sum_{i=1}^{2n+1}r_i\cdot i$, where $r_i$ is even whenever $i$ is even.
    \item  For $\fg=\fsp_{2n}\C$, $\rG$-conjugacy classes of nilpotents are in bijective correspondence with partitions of $2n=\sum_{i=1}^{2n}r_i\cdot i$, where $r_i$ is even whenever $i$ is odd.
    \item For $\fg=\fso_{2n}\C$, $\rG$-conjugacy classes of nilpotents are in bijective correspondence with partitions of $2n=\sum_{i=1}^{2n}r_i\cdot i$, where $r_i$ is even whenever $i$ is even, except that there are two classes associated to partitions which have $r_i=0$ for all $i$ odd.
\end{itemize}
\end{proposition}

Note that the above proposition is independent of the choice of $\rG$ under the given conditions, since the any two choices are related by a quotient by central elements.

Given a partition $n=\sum_{i=1}^n r_i\cdot i$, define the \emph{dual partition} by $n=\sum_{j=1}^ns_j$, where $s_j=\sum_{i=j}^nr_i$. The following proposition describes the centralizer of a nilpotent and the centralizing subalgebra of an associated $\fsl_2\C$-subalgebra; see \cite[Chapter 6.1]{CollMcGovNilpotents}. 

\begin{proposition}\label{prop: class nilp centralizers}
    Let $\fg$ be $\fsl_n\C$, $\fso_n\C$ or $\fsp_{2m}\C$. Let $e\in\fg$ be a nilpotent element with corresponding partition $n=\sum_{i=1}^nr_i\cdot i$ and dual partition $n=\sum_{j=1}^ns_j$, with $2m=n$ for $\fsp_{2m}\C$. Finally, let $V(e)=\ker(\ad_e)\subset\fg$ be the centralizer of $e$ and $\fc$ be the centralizer of an associated $\fsl_2\C$-subalgebra. Then $\dim(V(e))$ and $\fc$ are characterized as follows:
\begin{center}
    \begin{tabular}
        {|c|c|c|c|}\hline
        $\fg$&$\fsl_n\C$&$\fso_n\C$&$\fsp_{2m}\C$\\\hline
        $\dim(V(e))$&$\sum_{j=1}^ns_j^2-1$&$\frac{1}{2}(\sum_{j=1}^ns_j^2-\sum_{i-\text{odd}}r_i)$& 
    $\frac{1}{2}(\sum_{j=1}^ns_j^2+\sum_{i-\text{odd}}r_i)$\\\hline
        $\fc$&$\fs(\bigoplus_{i=1}^n \fgl_{r_i}\C)$&$\bigoplus_{i-\text{even}}\fsp_{r_i}\C\oplus \bigoplus_{i-\text{odd}}\fso_{r_i}\C$&$ \bigoplus_{i-\text{odd}}\fsp_{r_i}\C\oplus \bigoplus_{i-\text{even}}\fso_{r_i}\C$\\\hline
    \end{tabular}
\end{center}
\end{proposition}

The different noncompact real forms $\fg^\R\subset\fg$ of the Lie algebras $\fsl_n\C,~\fso_n\C,~\fsp_{2m}\C$ are described in Table \ref{Table of classical real forms} of Section  \ref{sec diagrams and tables}.
We follow \cite[Chapter 9.3]{CollMcGovNilpotents} for the classification of nilpotents in these real forms. In $\fsl_n\R$ and $\fsu^*_{2m}$, such classification can be phrased in terms of partitions. For the remaining real forms in the mentioned table, it can be phrased in terms of signed Young diagrams. Recall that partitions of $n$ are described by Young diagrams. We will use the convention that the Young diagram associated to a partition $n=\sum_{i=1}^nr_i\cdot i$ has $r_i$ rows of length $i$. A \emph{signed Young diagram} is a Young diagram in which each box is decorated with a $+$ or $-$ sign and these signs alternate along each row. The \emph{signature} of a signed Young diagram is $(p,q)$ if there are $p$ plus signs and $q$ minus signs. Given a signed Young diagram, for each sub-diagram of rows of length $i$, let $p_i$ denote the number of rows with leftmost box labeled $+$ and $q_i$ denote the number of rows with leftmost box labeled $-$. The following proposition collects a set of propositions proved in Section 9.3 of \cite{CollMcGovNilpotents}. 

\begin{proposition}\label{prop: nilp in gR}The classification of conjugacy classes of nilpotent elements in classical real Lie algebras reads as follows:
\begin{itemize}
\item $\rSL_n\R$-conjugacy classes of nilpotents in $\fsl_n\R$ are in 1-1 correspondence with partitions $n=\sum_{i=1}^nr_i\cdot i$, except that there are two orbits associated to partitions with $r_i=0$ for all $i$ odd. The centralizer of an associated $\fsl_2\R$-subalgebra is isomorphic to $\fs(\bigoplus_{i=1}^n\fgl_{r_i}\R)$.

\item $\rSU^*_{2m}$-conjugacy classes of nilpotents in $\fsu^*_{2m}$ are in 1-1 correspondence with partitions $m=\sum_{i=1}^m r_i\cdot i$. The centralizer of an associated $\fsl_2\R$-subalgebra is isomorphic to $\fs(\bigoplus_{i=1}^m\fu^*_{2r_i})$.

\item  $\rSU_{p,q}$-conjugacy classes of nilpotents in $\fsu_{p,q}$ are in 1-1 correspondence with signed Young diagrams of signature $(p,q)$. The centralizer of an associated $\fsl_2\R$-subalgebra is isomorphic to $\fs(\bigoplus_{i=1}^n\fu_{p_i,q_i})$.

\item $\rSO_{p,q}$-conjugacy classes of nilpotents in $\fso_{p,q}$are in 1-1 correspondences with signed Young diagrams of signature $(p,q)$ where even rows occur with even multiplicity and have their leftmost boxes labeled with $+$, except that there are two orbits for diagrams in which all rows have even length. 
    The centralizer of an associated $\fsl_2\R$-subalgebra is isomorphic to $\bigoplus_{i-even}\fsp_{p_i+q_i}\R\oplus\bigoplus_{i-odd}\fso_{p_i,q_i}$.

\item   $\rSO_{2m}^*$-conjugacy classes of nilpotents in $\fso_{2m}^*$ are in 1-1 correspondence with signed Young diagrams of size $m$ and any signature in which rows with odd length have their leftmost boxes labeled with a $+$. 
    The centralizer of an associated $\fsl_2\R$-subalgebra is isomorphic to $\bigoplus_{i-even}\fsp_{2p_i,2q_i}\oplus\bigoplus_{i-odd}\fso_{2(p_i+q_i)}^*$.

\item $\rSp_{2m}\R$-conjugacy classes of nilpotents in $\fsp_{2m}\R$ are in 1-1 correspondence with signed Young diagrams of size $2m$ of any signature where odd rows occur with even multiplicity and have their leftmost boxes labeled with $+$. 
    The centralizer of an associated $\fsl_2\R$-subalgebra is isomorphic to $\bigoplus_{i-odd}\fsp_{p_i+q_i}\R\oplus\bigoplus_{i-even}\fso_{p_i,q_i}$. 

\item $\rSp_{2p,2q}$-conjugacy classes of nilpotents in $\fsp_{2p,2q}$ are in 1-1 correspondence with signed Young diagrams of signature $(p,q)$ in which even rows have their leftmost boxes labeled $+$. 
    The centralizer of an associated $\fsl_2\R$-subalgebra is isomorphic to $\bigoplus_{i-odd}\fsp_{2p_i,2q_i}\oplus\bigoplus_{i-even}\fso^*_{2(p_i+q_i)}$.
    \end{itemize}
\end{proposition}

\begin{remark}\label{remark: partition of g from gR data}
For the classical Lie algebras other than $\fsu^*_{2m},~\fso_{2m}^*,~\fsp_{2p,2q}$, the partition of the associated nilpotent orbit in the complexification $\fg$ corresponds to the Young diagram obtained by forgetting the signs. For $\fsu^*_{2m},~\fso_{2m}^*,~\fsp_{2p,2q}$, the partition of the associated nilpotent orbit in $\fg$ corresponds to the Young diagram obtained doubling every row and forgetting the signs.
 \end{remark}

 We now classify magical nilpotent elements for classical real forms in terms of signed Young diagrams and partitions.
\begin{theorem}\label{thm: partition classification}
    Let $\fg^\R$ be a real form of a classical complex simple Lie algebra $\fg$. A nilpotent $\hat e\in\fg^\R$ is magical if and only if it is one of the following cases:
    \begin{enumerate}
        \item $\fg^\R\cong\fsl_n\R$ and the associated Young diagram has one row of length $n$.
        \item $\fg^\R\cong\fso_{p,p+1}$ or $\fg^\R\cong\fso_{p+1,p}$ and the signed Young diagram has one row of length $2p+1$.
         \item $\fg^\R\cong\fso_{p,p}$ and the signed Young diagram has one row of length $2p-1$ and one row of length $1$.
        \item $\fg^\R\cong\fsp_{2m}\R$ and the signed Young diagram has one row of length $2m$.
        \item $\fg^\R\cong\fsu_{m,m},~\fso^*_{4m}~,\fsp_{2m}\R$ and the signed Young diagram has $m$-rows of length $2$ and the leftmost boxes are either all labeled $+$ or all labeled $-$.
        \item $\fg^\R\cong\fso_{p,q}$ and the signed Young diagram has one row of length $2\min\{p,q\}-1$ and $(|q-p|+1)$-rows of length $1$, where the labels of the length $1$ row are the same and opposite the label of the leftmost box of the row of length $2\min\{p,q\}-1$.
    \end{enumerate}
\end{theorem}
\begin{remark}
    In the first four cases, $\fg^\R$ is split and we have the principal nilpotent. Case (5) corresponds to Lie algebras which are Hermitian of tube type, and the same holds in (6) if $p=2$ or $q=2$. 
\end{remark}

\begin{proof}
Let $\fg^\R$ be a real form of a classical complex simple Lie algebra $\fg$ and $\fg=\fh\oplus\fm$ be a Cartan decomposition. By Propositions \ref{prop: magical real nilpotent properties} and \ref{prop dim of centralizer formula} a nilpotent $\hat e\in\fg^\R$ is magical if and only if the centralizer of an associated $\fsl_2\R$-subalgebra $\fc^\R$ is compact and 
   \begin{equation}\label{eq: cond-magical dim-dim=0}
   2\dim(\fc^\R\otimes \C)-\dim(V(\gamma(\hat e)))-\dim(\fh)+\dim(\fm)
      \end{equation}
      vanishes.
    Now we use Proposition \ref{prop: nilp in gR} together with this criterion to detect magical nilpotents in $\fg^\R$.

For $\fg^\R=\fsl_n\R$, $\fc^\R=\fs(\bigoplus_{i=1}^n\fgl_{r_i}\R)$. So $\fc^\R$ is compact if and only if the partition is $n=1\cdot n$, i.e.~the corresponding Young diagram has just one row of length $n$. So we are left with this corresponding nilpotent $\hat e$ (namely $\gamma(\hat e)$ is the principal nilpotent). In this case, $\fc^\R=0$. Moreover, the dual partition is $n=n\cdot 1$ so Proposition \ref{prop: class nilp centralizers}, together with Table \ref{Table of classical real forms} of Section  \ref{sec diagrams and tables}, show that $-\dim(V(\gamma(\hat e)))-\dim(\fh)+\dim(\fm)=-n+1+n-1=0$. Hence \eqref{eq: cond-magical dim-dim=0} equals to zero, so $\hat e\in\fsl_n\R$ is magical, proving (1).

The remaining cases will be dealt with by a similar argument, where in each case we use Proposition \ref{prop: nilp in gR} to identify $\fc^\R$ and then Proposition \ref{prop: class nilp centralizers} and Table \ref{Table of classical real forms} of Section \ref{sec diagrams and tables} to compute \eqref{eq: cond-magical dim-dim=0}.

For $\fg^\R=\fsu^*_{2m}$, $\fc^\R=\fs(\bigoplus_{i=1}^m\fu^*_{2r_i})$ is compact if and only if $r_m=1$ and $r_i=0$ for $i\neq m$, so that $\fc^\R=\fsu^*_2=\fsu_2$. We are then left with the nilpotent in $\hat e\in\fsu^*_{2m}$ whose corresponding nilpotent (under the Cayley transform) in $\fg=\fsl_{2m}\C$ is given by the partition $2m=2\cdot m$. Its dual partition is $2m=\sum_{j=1}^{2m}s_j$, with $s_j=2$ for $1\leq j\leq m$ and $s_j=0$ otherwise. Then \eqref{eq: cond-magical dim-dim=0} equals $6-6m$.
Hence the nilpotent $\hat e\in \fsu^*_{2m}$ can only be magical if $m=1$. But $\fsu^*_{2}$ is compact and thus has no nonzero nilpotent elements (recall that magical nilpotents are nonzero by definition), so $\hat e$ is not magical. We conclude that $\fsu^*_{2m}$ does not admit magical nilpotents.

Consider now the case of $\fg^\R=\fsu_{p,q}$. Then  $\fc^\R=\fs(\bigoplus_{i=1}^{p+q}\fu_{p_i,q_i})$, which is compact if and only if $p_i=0$ or $q_i=0$ for each $i$.  The associated nilpotent in $\fsl_{p+q}\C$ corresponds to the partition $p+q=\sum_{i=1}^{p+q}r_i\cdot i$, where $r_i=p_i+q_i$. We see that \eqref{eq: cond-magical dim-dim=0} is given by
\begin{equation}
    \label{eq supq form}2\sum_{i=1}^{p+q}r_i^2-\sum_{i=1}^{p+q}s_i^2-(q-p)^2,
\end{equation}
with $p+q=\sum_{i=1}^{p+q}s_i$ the corresponding dual partition. We want to understand when \eqref{eq supq form} vanishes.

First assume $r_1=0$. Using $s_i=r_i+s_{i+1}$ twice, \eqref{eq supq form} can be rewritten as
\[-4\sum_{i=2}^{p+q-1}r_is_{i+1}-\sum_{i=3}^{p+q}s_i^2-(q-p)^2.\]
If $r_i\neq 0$ for some $i>2$, then this expression is strictly negative, therefore the corresponding partition does not correspond to a magical nilpotent in $\fsu_{p,q}$. If $r_2$ is the only nonzero $r_i$, the previous expression equals $-(q-p)^2$, hence \eqref{eq supq form} vanishes if and only if $p=q$. So the nonzero nilpotent  determined by that partition is magical, and corresponds to Case (5) for $\fg^\R=\fsu_{m,m}$.

Now suppose $r_1\neq0$. Since the Jordan block $J_1$ is a $1\times 1$ zero matrix, a nilpotent $\hat e\in\fsu_{p,q}$ with $r_1\neq0$ is contained in a subalgebra isomorphic to $\fsu_{p-r_1,q}$ (in case $r_1=p_1$) or $\fsu_{p,q-r_1}$ (in case $r_1=q_1$).  In this subalgebra, $\hat e$ has no $r_1$-term. If it is magical, then by the above argument we must have $r_i=0$ for $i>2$ and $q-p=\pm r_1$. Thus, \eqref{eq supq form} is given by  
\[2r_1^2+2r_2^2-r_1^2-2r_1r_2-r_2^2-r_2^2-r_1^2=-2r_1r_2.\]
This is zero if and only if $r_1=0$ or $r_2=0$, but we are assuming $r_1\neq 0$ and if $r_2=0$ then $\hat e$ is the zero nilpotent. So there are no magical nilpotents in $\fsu_{p,q}$ other than the one detected in the previous paragraph.

Consider now $\fg^\R=\fso_{p,q}$. Then $\fc^\R=\bigoplus_{i-even}\fsp_{p_i+q_i}\R\oplus\bigoplus_{i-odd}\fso_{p_i,q_i}$. This is compact if and only if $p_i+q_i=0$ for $i$ even and either $p_i=0$ or $q_i=0$ for $i$ odd. The partition of the associated nilpotent in $\fso_{p+q}\C$ is $p+q=\sum_{i=1}^{p+q}r_i\cdot i$, where $r_i=p_i+q_i$ with $p_i,q_i$ under the stated conditions. Then twice the quantity \eqref{eq: cond-magical dim-dim=0} equals to 
\begin{equation}
    \label{eq sopq form} 2\sum_{i=1}^{p+q}r_i^2-\sum_{i=1}^{p+q}r_i-\sum_{i=1}^{p+q}s_i^2+p+q-(q-p)^2.
\end{equation}

First assume that $r_1=0$.
If only one $r_k$ is nonzero, then $p+q=r_k\cdot k$ and $(q-p)^2=r_k^2$, because in this case the odd number $r_k$ equals the signature $\pm(p-q)$. Therefore \eqref{eq sopq form} simplifies to $r_k(k-1)(r_k-1)$. Since $r_1=0$, this is zero if and only if $r_k=1$ and thus $k=p+q$ and $k$ is odd. This proves that Case (2) of the theorem is a magical nilpotent for $\fso_{p,p+1}$ if the left most box is labeled $-$ and $\fso_{p+1,p}$ if the left most box is labeled $+$. 
Still assuming $r_1=0$, and using again that $s_i=r_i+s_{i+1}$, \eqref{eq sopq form} can be rewritten as
\[-4\sum_{i=2}^{p+q-1}r_is_{i+1}-\sum_{i=3}^{p+q}s_i^2-\sum_{i=1}^{p+q}r_i+p+q-(q-p)^2.\]
If at least two $r_i$ are nonzero, then $p+q=\sum_{i=1}^{p+q-1}s_i$ (because $r_{p+q}=0$) and $-4\sum_{i=2}^{p+q-1}r_is_{i+1}+2s_2<0$. Such a nilpotent is not a magical one because
\[-4\sum_{i=2}^{p+q-1}r_is_{i+1}-\sum_{i=3}^{p+q}s_i^2-\sum_{i=1}^{p+q}r_i+p+q-(q-p)^2\leq-4\sum_{i=2}^{p+q-1}r_is_{i+1}+2s_2-\sum_{i=1}^{p+q}r_i-(q-p)^2<0.\]

Now assume $r_1\neq 0$. As in the $\fsu_{p,q}$-case, the nilpotent $\hat e$ is contained in a subalgebra isomorphic to $\fso_{p-r_1,q}$ or $\fso_{p,q-r_1}$ and has no $r_1$-term. If $\hat e$ is magical, by the above argument, the partition must be of the form $p+q=r_1\cdot 1+1\cdot (2\min\{p,q\}-1)$. Since the signature of the signed Young diagram is $(p,q)$, if the left most box of the row of length $2\min\{p,q\}-1$ is labeled $+$, then each row of length $1$ is labeled $-$ and vice versa.  This means that $(q-p)^2=(1-r_1)^2$. In this case \eqref{eq sopq form} is given by 
\[2r_1^2+2-r_1-1-(r_1+1)^2-(2\min\{p,q\}-2)+r_1+2\min\{p,q\}-1-(1-r_1)^2.\]
This expression always vanishes, proving Case (6) of the theorem. 

For $\fg^\R=\fso_{2m}^*$, we have that $\fc^\R=\bigoplus_{i-even}\fsp_{2p_i,2q_i}\oplus\bigoplus_{i-odd}\fso_{2(p_i+q_i)}^*$. This is compact if and only if $p_i+q_i=0$ for all $i$ odd and either $p_i=0$ or $q_i=0$ for all $i$ even. So we are left with nilpotents $\hat e\in \fso_{2m}^*$ whose partition of the corresponding nilpotent in $\fso_{2m}\C$ is (cf.~Remark \ref{remark: partition of g from gR data}) $2m=\sum_{i=1}^m(2 r_i)\cdot i$, with $r_i=p_i+q_i$ verifying these conditions. Then twice the quantity \eqref{eq: cond-magical dim-dim=0} is given by
\begin{equation}
    \label{eq: so2m* form}2\sum_{i=1}^{m}(2r_i)^2+2\sum_{i=1}^{m} 2r_i-\sum_{i=1}^{m}s_i^2-2m,
\end{equation}
where $s_i=\sum_{j=i}^m2r_j$. Since $r_1=0$ and $s_i=2r_i+s_{i+1}$, \eqref{eq: so2m* form} is given by
\[-4\sum_{i=2}^{m-1}2r_is_{i+1}+2\sum_{i=1}^m2r_i-\sum_{i=3}^{m}s_i^2-2m.\]
If $r_i$ is nonzero for $i>2$, then the above expression is negative and the nilpotent is not magical. If $r_2$ is the only nonzero $r_i$, then \eqref{eq: so2m* form} equals $4r_2-2m$.
This is zero if and only if $r_2=\frac{m}{2}$ with $\frac{m}{2}$ an integer. In such case the nilpotent $\hat e$ is magical, proving the part of Case (5) regarding $\fso_{2m}^*$.

Consider now $\fg^\R=\fsp_{2m}\R$. We have that $\fc^\R=\bigoplus_{i-odd}\fsp_{p_i+q_i}\R\oplus\bigoplus_{i-even}\fso_{p_i,q_i}$, so it is  compact if and only if $p_i+q_i=0$ for $i$ odd and either $p_i=0$ or $q_i=0$ for all $i$ even, so we are left with nilpotents under these conditions. 
Write $2m=\sum_{i=1}^{2m}r_i\cdot i$ for the partition of the associated nilpotents (by the Cayley transform) in $\fsp_{2m}\C$, where $r_i=p_i+q_i$ satisfy the previous constrains. Again, twice the quantity \eqref{eq: cond-magical dim-dim=0} equals
\begin{equation}
    \label{eq: sp2mR form}2\sum_{i=1}^{2m}r_i^2-2\sum_{i=1}^{2m}r_i-\sum_{i=1}^{2m}s_i^2+2m.
\end{equation}
with where $s_i=\sum_{j=i}^{2m}r_j$.
Since $r_1=0$ and $s_i=r_i+s_{i+1}$, \eqref{eq: sp2mR form} can be rewritten as
\[-4\sum_{i=2}^{2m-1}r_is_{i+1}-\sum_{i=3}^{2m}s_i^2-2\sum_{i=1}^{2m}r_i+2m.\]
Similarly to the previous cases, if at least two $r_i$ are nonzero, then this expression is negative so the corresponding nilpotents are not magical. If $2m=r_k\cdot k$, then \eqref{eq: sp2mR form} is given by $(2-k)r_k^2-2r_k+2m$, which is zero if and only if $k=2$ or $k=2m$. This proves Case (4) and completes the proof of Case (5).

Finally, let us consider $\fg^\R=\fsp_{2p,2q}$, in which case we know that $\fc^\R=\bigoplus_{i-odd}\fsp_{2p_i,2q_i}\oplus\bigoplus_{i-even}\fso^*_{2(p_i+q_i)}$, which is compact if and only if $p_i+q_i=0$ for every $i$ even and either $p_i=0$ or $q_i=0$ for all $i$ odd.
Let $2p+2q=\sum_{i=1}^{p+q}2r_i\cdot i$ be the partition of the associated nilpotents in $\fsp_{2p+2q}\C$ (see Remark \ref{remark: partition of g from gR data}), where each $r_i=p_i+q_i$ verifies the previous conditions. Then we have that twice the number \eqref{eq: cond-magical dim-dim=0} is given by
\begin{equation}
    \label{eq: sp2p2q form}2\sum_{i=1}^{p+q}(2r_i)^2+\sum_{i=1}^{p+q}2r_i-\sum_{i=1}^{p+q}s_i^2-2p-2q-4(q-p)^2,
\end{equation}
where $s_i=\sum_{j=i}^{p+q}2r_j$. If $r_1=0$, then \eqref{eq: sp2p2q form} can be rewritten as
\[\sum_{i=1}^{p+q}2r_i-4\sum_{i=2}^{p+q-1}2r_is_{i+1}-\sum_{i=3}^{p+q}s_i^2-2p-2q-4(q-p)^2.\]
This expression is always negative, hence no magical nilpotents arise with $r_1=0$. As in previous cases, if $r_1\neq0$, then the nilpotent must lie in a subalgebra isomorphic to $\fsp_{2p-2r_1,2q}$ or $\fsp_{2p,2q-2r_1}$ and  have $r_1=0$ in that subalgebra. Moreover, if $\{f,h,e\}$ is magical in $\fsp_{2p,2q},$ then it is magical in the subalgebra (see Remark \ref{remark subalgebra mag}).  But the previous argument says that there is no magical nilpotent in such setting. Hence $\fsp_{2p,2q}$ does not admit any magical nilpotents, and this completes the proof.
\end{proof}

It remains to translate the partition classification of Theorem \ref{thm: partition classification} into the weighted Dynkin diagram classification of Theorem \ref{thm: classification weighted dynkin}. The algorithm for doing this goes as follows. Let $\fa\subset\fg$ be the Cartan subalgebra of diagonal matrices and choose the simple roots such that the positive root spaces correspond to upper triangular matrices. Given a (signed) Young diagram from Theorem \ref{thm: partition classification}, let $e\in\fg$ be the associated nilpotent given by \eqref{eq jordan block of nilp}. Complete $e$ to an $\fsl_2$-triple $\{f,h,e\}$ such that $h$ is a diagonal matrix. Finally, conjugate $\{f,h,e\}$ so that the eigenvalues of $h$ are decreasing and compute $\alpha_i(h)$. We will sketch this process for one case; see \cite[Chapters 3.6, 5.3]{CollMcGovNilpotents} for more details.

Completing a nilpotent  $e\in\fsl_n\C$ to an $\fsl_2$-triple $\{f,h,e\}$ with $h$ diagonal requires doing this for each Jordan block $J_k$. Such an $\fsl_2$-triple $\{f,h,e\}$ is given by 
\[\left\{\smtrx{0\\\mu_1&0\\&\ddots&\ddots\\&&\mu_{k-1}&0},
\smtrx{k-1\\&k-3\\&&\ddots\\&&& 1-k},
\smtrx{0&1\\&\ddots&\ddots\\&&0&1\\&&&0}\right\},\]
where $\mu_j=j(k-j)$. Consider now Case (6) of Theorem \ref{thm: partition classification} with $p\leq q$. Then the resulting semisimple element $h$ is diagonal with entries $\{2p-2,2p-4,\ldots,2-2p,0,\ldots,0\}$. Rearranging the matrix so that the eigenvalues are decreasing, yields 
\[\{2p-2,2p-4,\ldots,2,0,\ldots,0,-2,-4,\ldots,2-2p\}.\]
Using \cite[Lemma 5.3.3]{CollMcGovNilpotents} when $p+q$ is odd and \cite[Lemma 5.3.4]{CollMcGovNilpotents} when $p+q$ is even, we conclude that the associated weighted Dynkin diagram is given by Case (3) of Theorem \ref{thm: classification weighted dynkin}. We leave the remaining cases to the reader. 

\section{Explicit data and real forms for magical $\fsl_2$-triples}\label{sec sl2 data and Cayley realform}

Associated to a given magical $\fsl_2$-triple $\{f,h,e\}\subset\fg$, we explicitly exhibit, in this section, the $\fsl_2\C$-module and $\ad_h$-weight space decompositions \eqref{eq sl2 module decomp} and \eqref{eq adh weight decomp}, the associated canonical real form of $\lieg$ of Definition \ref{def: canonical real form} and the Cayley real form of $\lieg_0$ of Definition \ref{def: Cayley real form}. We also show that the magical $\fsl_2$-triple $\{f,h,e\}\subset\fg$ arises from a principal $\fsl_2$-triple in a simple subalgebra $\fg(e)\subset\fg$, defined as the semisimple part of the double centralizer of $\{f,h,e\}$. 
Finally, we deduce the detailed Lie theoretic information for magical $\fsl_2$-triples in Case (4) of Theorem \ref{thm: classification weighted dynkin}.

\subsection{$\fsl_2$-data}

Recall from Definition \ref{def: canonical real form} that a magical triple $\{f,h,e\}$ of $\fg$ determines a canonical real form $\fg^\R$ via the involution $\sigma_e:\fg\to\fg$. 
Note that the canonical real form of a magical $\fsl_2$-triple follows from Theorem \ref{thm: partition classification} for the classical cases and \cite{ExceptionalNilpotentsInner} for the exceptional cases.

\begin{proposition}\label{prop canonical real form}
   The canonical real forms $\fg^\R\subset\fg$ associated to magical $\fsl_2$-triples are given as follows:
   \begin{enumerate}
       \item In Case (1) of Theorem \ref{thm: classification weighted dynkin}, $\fg^\R$ is the split real form
       \begin{center} \hspace{-1.6 cm}
           \begin{tabular}
               {|c|c|c|c|c|c|c|c|c|c|}\hline
              $\fg$& $\rA_n$&$\rB_n$&$\rC_n$&$\rD_n$&$\rE_6$&$\rE_7$&$\rE_8$&$\rF_4$&$\rG_2$\\\hline
              $\fg^\R$& $\fsl_{n+1}\R$&$\fso_{n,n+1}$&$\fsp_{2n}\R$&$\fso_{n,n}$& $\fe_6^6$& $\fe_7^7$&$\fe_8^8$&$\ff_4^4$&$\fg_2^2$
               \\\hline
           \end{tabular}
       \end{center}
       \item In Case (2) of Theorem \ref{thm: classification weighted dynkin}, $\fg^\R$ is the Hermitian Lie algebra of tube type given by
       \begin{center}
           \begin{tabular}
           {|c|c|c|c|c|c|c|}\hline
           $\fg$&$\rA_{2n-1}$&$\rB_{n}$&$\rC_n$&$\rD_n$&$\rD_{2n}$&$\rE_7$\\\hline
           $\fg^\R$&$\fsu_{n,n}$&$\fso_{2,2n-1}$&$\fsp_{2n}\R$&$\fso_{2,2n-2}$&$\fso_{4n}^*$&$\fe_7^{-25}$\\\hline
       \end{tabular}
       \end{center}
       \item In Case (3) of Theorem \ref{thm: classification weighted dynkin} with $\fg=\fso_N\C$, $\fg^\R\cong\fso_{p,N-p}$.
       \item In Case (4) of Theorem \ref{thm: classification weighted dynkin}, $\fg^\R$ is the quaternionic real form of $\fg$
       \begin{center}
    \begin{tabular}
        {|c|c|c|c|c|}\hline
        $\fg$&$\rE_6$&$\rE_7$&$\rE_8$&$\rF_4$\\\hline
       $\fg^\R$& $\fe_6^{2}$&$\fe_7^{-5}$&$\fe_8^{-24}$&$\ff_4^{4}$\\\hline
    \end{tabular}
\end{center}
   \end{enumerate}
\end{proposition}

Let $\fa\subset\fg$ be a Cartan subalgebra and denote the root space decomposition by
\[\fg=\fa\oplus\bigoplus_{\alpha\in \Delta}\fg_\alpha,\]
where $\Delta\subset\fa^*\backslash \{0\}$ is the set of roots and $\fg_\alpha=\{y\in\fg~|~ \ad_x(y)=\alpha(x) y, \ \forall\ x\in\fa\}$ is the root space of $\alpha\in \Delta$. Choosing a set of simple roots $\Pi=\{\alpha_1,\ldots,\alpha_{\rk\fg}\}\subset\Delta$ splits the roots into positive and negative roots $\Delta=\Delta^+\sqcup\Delta^-$, where $\Delta^+$ (resp.\ $\Delta^-$) consists of roots $\alpha=\sum_{i=1}^{\rk\fg}a_i\alpha_i$ with $a_i\in\Z_{\geq0}$ (resp.\ $a_i\in\Z_{\leq0}$) for all $i$.

Let $\{f,h,e\}$ be an $\fsl_2$-triple, with $h\in\fa$ and $\alpha_i(h)\geq 0$ for all $\alpha_i\in\Pi$. The element $h$ acts on a root space $\fg_\alpha$ with weight $\sum_{i=1}^{\rk\fg}a_i\alpha_i(h)$, where $\alpha=\sum_{i=1}^{\rk\fg}a_i\alpha_i$. Thus, the $\ad_h$-weight space decomposition \eqref{eq adh weight decomp} of $\fg$ is given by 
\[\fg=\bigoplus_{j\in\Z}\fg_j,\]
where $\fg_j$ is a direct sum of root spaces $\fg_\alpha$ with $\alpha=\sum_{i=1}^{\rk\fg}a_i\alpha_i$ and $j=\sum_{i=1}^{\rk\fg}a_i\alpha_i(h)$  if $j\neq0$, and $\fg_0$ is the direct sum of $\fa$ and the set of root spaces $\fg_\alpha$ with $\alpha=\sum_{i=1}^{\rk\fg}a_i\alpha_i$ such that $0=\sum_{i=1}^{\rk\fg}a_i\alpha_i(h)$, i.e.~
\begin{equation}\label{eq:weight-spdirectsumroot-sp}
\fg_0=\fa\oplus\bigoplus_{\alpha(h)=0}\fg_\alpha\ \text{ and }\ \ \fg_j=\bigoplus_{\alpha(h)=j}\fg_\alpha,\ j\neq 0.
\end{equation}

We record the Lie algebra $\fg_0$ of a magical nilpotent. This follows immediately from \eqref{eq:weight-spdirectsumroot-sp} and from the weighted Dynkin diagram classification of Theorem \ref{thm: classification weighted dynkin}.
\begin{proposition}\label{prop g0 isomorphism type}
The subalgebra $\fg_0\subset\fg$ associated to a magical $\fsl_2$-triple in $\fg$ is described as follows:
   \begin{enumerate}
       \item In Case (1) of Theorem \ref{thm: classification weighted dynkin}, $\fg_0\cong\C^{\oplus \rk\fg}$.
       \item In Case (2) of Theorem \ref{thm: classification weighted dynkin}, 
       \begin{center}\hspace{-1.5cm}
        \begin{tabular}
        {|c|c|c|c|c|c|c|}\hline
            $\fg$&$\rA_{2n-1}$&$\rB_n$&$\rC_n$&$\rD_n$&$\rD_{2n}$&$\rE_7$\\

            \hline $\fg_0$&$\fsl_{n}\C\oplus\fsl_n\C\oplus\C$&$\fso_{2n-1}\C\oplus\C$&$\fsl_n\C\oplus\C$&$ \fso_{2n-2}\C\oplus\C$&$\fsl_{2n}\C\oplus\C$&$\fe_6\oplus\C$\\
            \hline
        \end{tabular}
    \end{center}
    \item In Case (3) of Theorem \ref{thm: classification weighted dynkin} with $\fg=\fso_N\C$, then $\fg_0=\C^{p-1}\oplus \fso_{N-2p+2}\C$.
    \item In Case (4) of Theorem \ref{thm: classification weighted dynkin},
     \begin{center}\hspace{-2cm}
        \begin{tabular}
        {|c|c|c|c|c|}\hline
            $\fg$&$\rE_{6}$&$\rE_7$&$\rE_8$&$\rF_4$\\

            \hline $\fg_0$&$\fsl_{3}\C\oplus\fsl_3\C\oplus\C^2$&$ \fsl_{6}\C\oplus\C^2$&$\fe_6\oplus\C^2$&$\fsl_{3}\C\oplus\C^2$\\
            \hline
        \end{tabular}
    \end{center}
\end{enumerate}
\end{proposition} 

The $\fsl_2\C$-module decomposition $\fg=\bigoplus_{j}W_j$ from \eqref{eq sl2 module decomp} can be deduced from the $\ad_h$-weight space decomposition. Namely, 
\begin{equation}
    \label{eq nj expression}n_{j}=\dim(\fg_{j})-\dim(\fg_{j+2}).
\end{equation} 
Recall that the $\fsl_2$-data of a magical nilpotent is determined by the collection of pairs of non-negative integers $\{(m_j,n_j)\}_{j=0}^M$ such that, for each $j\geq 1$, the multiplicity $n_{2m_j}$ of $W_{2m_j}$ is positive. Thus, the $\fsl_2\C$-module decomposition of a magical nilpotent is given by
   \[\fg=\fc\oplus\bigoplus\limits_{j=1}^MW_{2m_j.}\]

\begin{proposition}\label{prop: magical sl2 data}
The $\fsl_2$-data of a magical $\fsl_2$-triple $\{f,h,e\}$ is given as follows:
\begin{enumerate}
    \item In Case (1) of Theorem \ref{thm: classification weighted dynkin} the set of $\{m_j\}$ is given by 
    \begin{center}\hspace{-1cm}
        \begin{tabular}
            {|c|c|c|}\hline
            $\rA_{n}:$ $\{0,1,2,\ldots,n\}$&$\rB_n:$ $\{0,1,3,\ldots,2n-1\}$&$\rC_n:$ $\{0,1,3,\ldots,2n-1\}$\\\hline
            $\rD_n:$ $\{0,1,3,\ldots,2n-3,n-1\}$&$\rE_6:$ $\{0,1,4,5,7,8,11\}$&$\rE_7:$ $\{0,1,5,7,9,11,13,17\}$\\\hline
            $\rE_8:$ $\{0,1,7,11,13,17,19,23,29\}$&$\rF_4:$ $\{0,1,5,7,11\}$ &$\rG_2:$ $\{0,1,5\}$\\\hline
        \end{tabular}
    \end{center}
     For all cases, $n_0=0$ and $n_{2m_j}=1$, with the exception that $n_{4n-2}=2$ for $\rD_{2n}$.

    \item In Case (2) of Theorem \ref{thm: classification weighted dynkin}, $\{m_j\}=\{0,1\}$ and $n_0$ and $n_2$ are given as follows:
    \begin{center}
        \begin{tabular}
        {|c|c|c|c|c|c|c|}\hline
            $\fg$&$\rA_{2n-1}$&$\rB_n$&$\rC_n$&$\rD_n$&$\rD_{2n}$&$\rE_7$\\

            \hline $n_0$&$n^2-1$&$2n^2-5n+3$&$\dfrac{n(n-1)}{2}$&$2n^2-7n+6$&$n(2n+1)$&$52$\\

            \hline $n_2$&$n^2$&$2n-1$&$\dfrac{n(n+1)}{2}$&$2n-2$&$n(2n-1)$&$27$\\
            \hline
        \end{tabular}
    \end{center}

    \item In Case (3) of Theorem \ref{thm: classification weighted dynkin}, we have $\{m_j\}=\{0,1,3,\ldots,2p-3,p-1\}$ and
    \[n_0=\dfrac{(N-2p)(N-2p+1)}{2}\ \ \ \text{and}\ \ \ n_{2m_j}=\begin{dcases}
        N-2p+2&p\text{ even and }m_j=p-1\\
        N-2p+1&p\text{ odd and }m_j=p-1\\
        1&\text{otherwise},
    \end{dcases}\] 
    where $N=2n+1$ in type $\rB_n$ and $N=2n$ in type $\rD_n$.

    \item In Case (4) of Theorem \ref{thm: classification weighted dynkin}, $\{m_j\}=\{0,1,3,5\}$, $n_2=1$, $n_{10}=1$ and $n_0$ and $n_6$ are given as follows:
\begin{center} 
    \begin{tabular}
        {|c|c|c|c|c|}\hline
        $\fg$&$\rE_6$&$\rE_7$&$\rE_8$&$\rF_4$\\\hline
        $n_0$&$8$&$21$&$52$&$3$\\\hline
        $n_6$&$8$&$14$&$26$&$5$\\\hline
    \end{tabular}
\end{center}
\end{enumerate}
\end{proposition}
\begin{proof}
  For Case (1), all the nodes of the Dynkin diagram have label $2$, and hence the nilpotent $e\in\fg$ is principal. We have that $n_0=0$ since the $\fg$-centralizer of a principal $\fsl_2$-triple is trivial. The integers $m_j$ with $n_{2m_j}>0$ are the exponents of $\fg$ (see \cite[Chapter 4]{CollMcGovNilpotents}). 

  For Case (2), there is one root $\alpha_{M}$ with label $2$ and all other roots are labeled $0$. Moreover, if $\sum\limits a_i\alpha_i$ is the expression of the highest root, then $a_M=1$. Thus, the $\ad_h$-weight space decomposition is $\fg=\fg_{-2}\oplus\fg_0\oplus\fg_2$ and the $\fsl_2\C$-module  decomposition is $\fg=W_0\oplus W_2$.  We have $\dim(\fg_0)=n_0+n_2$ and $\dim(\fg)=3n_2+n_0$, hence $n_2=\frac{\dim(\fg)-\dim(\fg_0)}{2}$. 

  We compute the cases of $\rA_{2n-1}$ and leave the rest to the reader. For the $\rA_{2n-1}$ weighted Dynkin diagram we have $\fg_0=\fsl_{n-1}\C\oplus\C\oplus\fsl_{n-1}\C$. Hence,
    \[\xymatrix{n_2=\frac{(4n^2-1)-(2n^2-2)-1}{2}=n^2&\text{and}&n_0=\dim(\fg)-3n_2=n^2-1.}\]

For Case (3), $\rB_n$ and $\rD_n$ are similar. We will prove the $\rB_n$-case and leave $\rD_n$ to the reader. The proof is by induction, showing that 
\newsavebox{\boxA}
\sbox{\boxA}{
    \begin{tikzpicture}[scale=.4]
        \draw (-1,0) node[anchor=east]  {$\rB_{n-1}$:};
    \foreach \x in {0,...,6}
    \draw[xshift=\x cm,thick] (\x cm,0) circle (.25cm);
    \node at (0,0) [below = 1mm ] {${\scriptstyle 2}$};
    \node at (2,0) [below = 1mm ] {${\scriptstyle 2}$};
    \node at (4,0) [below = 1mm ] {${\scriptstyle 2}$};
    \node at (4,0) [above = 1mm ] {${\scriptstyle \alpha_{p-2}}$};
    \node at (6,0) [below = 1mm ] {${\scriptstyle 0}$};
    \node at (8,0) [below = 1mm ] {${\scriptstyle 0}$};
    \node at (10,0) [below = 1mm ] {${\scriptstyle 0}$};
    \node at (12,0) [below = 1mm ] {${\scriptstyle 0}$};
   \draw[thick] (0.25 cm,0) -- +(1.5 cm,0);
    \draw[dotted,thick] (2.3 cm,0) -- +(1.4 cm,0);
    \draw[thick] (8.25 cm,0) -- +(1.5 cm,0);
    \draw[dotted,thick] (6.3 cm,0) -- +(1.4 cm,0);
    \draw[thick] (4.25 cm,0) -- +(1.5 cm,0);
    \draw[thick] (10.25 cm, -.1 cm) -- +(1.5 cm,0);
    \draw[thick] (10.9 cm, -.3cm) -- +(.3cm, .3cm);
    \draw[thick] (10.9 cm, .3cm) -- +(.3cm, -.3cm);
    \draw[thick] (10.25 cm, .1 cm) -- +(1.5 cm,0);
    \end{tikzpicture}
}
\newsavebox{\boxB}
\sbox{\boxB}{
\begin{tikzpicture}[scale=.4]
        \draw (-1,0) node[anchor=east]  {$\rB_{n}$:};
    \foreach \x in {0,...,6}
    \draw[xshift=\x cm,thick] (\x cm,0) circle (.25cm);
    \node at (0,0) [below = 1mm ] {${\scriptstyle 2}$};
    \node at (2,0) [below = 1mm ] {${\scriptstyle 2}$};
    \node at (4,0) [below = 1mm ] {${\scriptstyle 2}$};
    \node at (4,0) [above = 1mm ] {${\scriptstyle \alpha_{p-1}}$};
    \node at (6,0) [below = 1mm ] {${\scriptstyle 0}$};
    \node at (8,0) [below = 1mm ] {${\scriptstyle 0}$};
    \node at (10,0) [below = 1mm ] {${\scriptstyle 0}$};
    \node at (12,0) [below = 1mm ] {${\scriptstyle 0}$};
   \draw[thick] (0.25 cm,0) -- +(1.5 cm,0);
    \draw[dotted,thick] (2.3 cm,0) -- +(1.4 cm,0);
    \draw[thick] (8.25 cm,0) -- +(1.5 cm,0);
    \draw[dotted,thick] (6.3 cm,0) -- +(1.4 cm,0);
    \draw[thick] (4.25 cm,0) -- +(1.5 cm,0);
    \draw[thick] (10.25 cm, -.1 cm) -- +(1.5 cm,0);
    \draw[thick] (10.9 cm, -.3cm) -- +(.3cm, .3cm);
    \draw[thick] (10.9 cm, .3cm) -- +(.3cm, -.3cm);
    \draw[thick] (10.25 cm, .1 cm) -- +(1.5 cm,0);
    \end{tikzpicture}
}
\begin{center}
 \begin{tikzpicture}
                \node at (0,0){\usebox{\boxA}};
                \node at  (8,0){\usebox{\boxB}};
                \node at (4,0){$\Longrightarrow$};
    \end{tikzpicture}
            \end{center}
The base case was proven in Case (2). 
Let $\alpha=\sum_{j=1}^na_j\alpha_j$ be a positive root in $\rB_n$. The root space of $\alpha$ is in $\fg_{2\sum_{j=1}^{p-1}a_j}$, the $2\sum_{j=1}^{p-1}a_j$-eigenspace of $\ad_h$. 

The set of roots with $a_1=0$ defines a subsystem of type $\rB_{n-1}$, with corresponding subalgebra $\fso_{2n-1}\C\subset\fg$. 
On the other hand, there are $2n-1$ positive roots in $\rB_{n}$ with $a_1\neq0$, namely
\[\Bigg\{\beta_i=\sum_{j=1}^i\alpha_j\Bigg\}_{i\in\{1,\ldots,n\}}\cup \ \ \ \ \Bigg\{\gamma_i=\beta_n+\sum_{k=i}^n\alpha_{k}\Bigg\}_{i\in\{n,\ldots,2\}}.\]
We have 
\[\fg_{\beta_i}\subset \begin{dcases}
    \fg_{2i}&i\leq p-2\\
    \fg_{2p-2}&p-1\leq i\leq n
\end{dcases}\ \ \ \ \text{and}\ \ \ \ \fg_{\gamma_i}\subset \begin{dcases}
    \fg_{2p-2}& p\leq i\leq n\\
    \fg_{2p-2+2(p-i)}&2\leq i\leq p-1
\end{dcases}.\]
In particular, for $j\geq 0$ we have 
\[\dim(\fg_{2j})=\begin{dcases}
    \dim(\fg_{2j}\cap\fso_{2n-1}\C)+3+2n-2p&j=p-1\\
    \dim(\fg_{2j}\cap\fso_{2n-1}\C)+1&\text{otherwise}
\end{dcases}.\]
Set $\theta_{2m_j}=\dim(\fg_{2m_j}\cap\fso_{2n-1}\C)-\dim(\fg_{2m_j+2}\cap\fso_{2n-1}\C)$. Using \eqref{eq nj expression} we have
\[
  n_{2m_j}=\theta_{2m_j}+\begin{dcases}
2n-2p+2   &m_j=p-1\\
2p-2n-2&m_j=p-2\\
1&m_j=2p-3\\
0& \text{otherwise}
\end{dcases}.\]
The result for $\rB_{n-1}$ gives the values of $\theta_{2m_j}$. Thus, we have $n_{0}=(2n+1-2p)(n-p+1)$, $n_{2m_j}=1$ for $m_j\in\{1,3,\ldots,2p-3\}\setminus \{p-1\}$ and
\[n_{2p-2}=\begin{dcases}
        2n-2p+3&p\text{ even}\\
        2n-2p+2&p\text{ odd}
    \end{dcases}.\]

Finally, we refer to the diagrams in section \ref{subsec root posets exceptional} to prove Case (4). In these diagrams, the circles denote the positive roots and the integer labels correspond the $\ad_h$-eigenvalue on the root space. For $\rE_8$, we have $\dim(\fg_0)=2+\dim(\fe_6)=80$, and
  \[\xymatrix@=1em{\dim(\fg_{10})=1,&\dim(\fg_{8})=1,&\dim(\fg_6)=27,&\dim(\fg_4)=27,&\dim(\fg_2)=28.}\]
  Thus, the nonzero $n_{2m_j}$'s  are $n_{10}=1$, $n_6=26$, $n_2=1$ and $n_0=52$. This settles the $\rE_8$ case, and the other ones, $\rE_6$, $\rE_7$ and $\rF_4$, are left to the reader. 
 \end{proof}
 
\subsection{The centralizer $\fc$ and its centralizer}


The next description of the centralizer $\fc$ of a magical $\fsl_2$-triple $\{f,h,e\}\subset\fg$ follows, for classical Lie algebras, from the partition classification of magical nilpotents of Theorem \ref{thm: partition classification} and \cite[Theorem 6.1.3]{CollMcGovNilpotents}. For the exceptional Lie algebras, $\fc$ is the complexification of the last column in the tables of \cite{ExceptionalNilpotentsInner}, see Table \ref{Table of exceptional inner real forms} of Section \ref{sec diagrams and tables}. 

\begin{proposition}\label{prop c subalgebra}
    The centralizer $\fc\subset\fg$ of a magical $\fsl_2$-triple is given as follows:
    \begin{enumerate}
        \item In Case (1) of Theorem \ref{thm: classification weighted dynkin}, $\fc=0$;
        \item In Case (2) of Theorem \ref{thm: classification weighted dynkin}, 
         \begin{center}
           \begin{tabular}
           {|c|c|c|c|c|c|c|}\hline
           $\fg$&$\rA_{2n-1}$&$\rB_{n}$&$\rC_n$&$\rD_n$&$\rD_{2n}$&$\rE_7$\\\hline
           $\fc$&$\fsl_n\C$&$\fso_{2n-2}\C$&$\fso_{n}\C$&$\fso_{2n-3}\C$&$\fsp_{2n}\C$&$\ff_4$\\\hline
       \end{tabular}
       \end{center}
        \item In Case (3) of Theorem \ref{thm: classification weighted dynkin} with $\fg=\fso_N\C$, $\fc\cong \fso_{N-2p+1}\C$;
        \item In Case (4) of Theorem \ref{thm: classification weighted dynkin},
\begin{center}
    \begin{tabular}
        {|c|c|c|c|c|}\hline
        $\fg$&$\rE_6$&$\rE_7$&$\rE_8$&$\rF_4$\\\hline
       $\fc$& $\fsl_3\C$&$\fsp_6\C$&$\ff_4$&$\fso_3\C$\\\hline
    \end{tabular}
\end{center}
    \end{enumerate}
\end{proposition}

We now show that a magical $\fsl_2$-triple $\{f,h,e\}\subset\fg$ arises from a principal $\fsl_2$-triple in a simple subalgebra $\fg(e)\subset\fg$.
\begin{proposition}\label{prop subalge g(e)}
    Let $\{f,h,e\}\subset\fg$ be a magical $\fsl_2$-triple and $\fc\subset\fg$ be the centralizer of $\{f,h,e\}$. Then the centralizer of $\fc$ is the direct sum $\fz(\fc)\oplus\fg(e)$, where $\fz(\fc)$ is the center of $\fc$ and $\fg(e)\subset\fg$ is a simple subalgebra such that $\{f,h,e\}$ is a principal $\fsl_2$-triple of $\fg(e)$. The subalgebra $\fz(\fc)\oplus\fg(e)$ is described as follows:
    \begin{itemize}
        \item For Case (1) of Theorem \ref{thm: classification weighted dynkin},  $\fg(e)=\fg$ and $\fz(\fc)=0$.
        \item For Case (2) of Theorem \ref{thm: classification weighted dynkin}, $\fg(e)=\langle f,h,e\rangle\cong\fsl_2\C$ and $\fz(\fc)=\{0\}$, unless $\fg\cong\fso_5\C$ in which case $\fz(\fc)=\fc\cong\C$.
        \item For Case (3) of Theorem \ref{thm: classification weighted dynkin}, $\fg(e)\cong\fso_{2p-1}\C\subset\fso_N\C=\fg$ and $\fz(\fc)=0$, unless $\fg\cong\fso_{2p+1}\C$ in which case $\fz(\fc)=\fc\cong\C$. 
        \item For Case (4) of Theorem \ref{thm: classification weighted dynkin}, $\fg(e)\cong Lie(\rG_2)$\footnote{We use the notation $Lie(\rG_2)$ for the Lie algebra of the exceptional group $\rG_2$ since $\fg_2$ denotes the weight $2$ space of $\ad_h$.} and $\fz(\fc)=0.$
    \end{itemize}
\end{proposition}

\begin{proof}
    We first identify the listed subalgebras $\fg(e)$ and show they centralize $\fc,$ then establish $\fz(\fc)\oplus\fg(e)$ is the centralizer of $\fc.$ 
   The first part is obvious in Cases (1) and (2).

    For Case (3) of Theorem \ref{thm: classification weighted dynkin}, the magical nilpotent $e\in\fso_N\C$ corresponds to the Young diagram with one row of length $2p-1$ and $N-2p+1$-rows of length $1$, by Case (6) of Theorem \ref{thm: partition classification}. This corresponds to principally embedding $e$ in $\fso_{2p-1}\C$ followed by the embedding $\fso_{2p-1}\C\subset\fso_N\C$. In this case, the centralizer $\fc$ of $\{f,h,e\}$ is isomorphic to $\fso_{N-2p+1}\C$. The centralizer of $\fg(e)=\fso_{2p-1}\C$ is also isomorphic to $\fso_{N-2p+1}\C$ and contains the centralizer of $\{f,h,e\}$.  Thus $\fc$ centralizes $\fg(e)$.

    For Case (4) of Theorem \ref{thm: classification weighted dynkin}, we use the classification of nilpotents by Bala--Carter's theory (see \cite[\S8]{CollMcGovNilpotents}). Very briefly, $\rG$-conjugacy classes of nilpotents in $\fg$ are in bijective correspondence with $\rG$-conjugacy classes of pairs $(\fl,\fp_\fl)$. Here, $\fl\subset\fg$ is the Levi factor of a parabolic subalgebra of $\fg$ and $\fp_\fl\subset\fl$ is the parabolic subalgebra of $\fl$ associated 
    to a so-called \emph{distinguished} nilpotent of the semisimple part $[\fl,\fl]$ of $\fl$, i.e.~one which does not belong to any proper Levi subalgebra of $[\fl,\fl]$. In particular, the principal nilpotent in $[\fl,\fl]$ is distinguished and corresponds to the Borel subalgebra of $\fl$.  

    In the tables of \cite[\S8.4]{CollMcGovNilpotents}, the label of the nilpotent has the form $X_N(a_i)$, where $X_N$ is the type of the associated Levi $\fl$ and $a_i$ is the number of simple roots in a Levi of $\fp_\fl$. The notation $X_N(a_0)=X_N$ is used and, in this case, the associated distinguished nilpotent of $\fl$ is principal. The labels of the weighted Dynkin diagrams of the magical nilpotents from Case (4) of Theorem \ref{thm: classification weighted dynkin} are $\rB_3$ for $\fg=\ff_4$ and $\rD_4$ for $\fg=\fe_6,~\fe_7,~\fe_8$. 
    Thus, the magical nilpotent in $\rF_4$ arises from the principal nilpotent in $\fso_7\C\subset\ff_4$ and the magical nilpotents in type $\rE_i$ arise from a principal nilpotent in $\fso_8\C\subset\fe_i$ for $i=6,7,8$. 

 Now, a principal nilpotent in $\fso_7\C$ or $\fso_8\C$ is induced by a principal nilpotent in a subalgebra isomorphic of type $\rG_2$, $Lie(\rG_2)\subset\fso_7\C\subset\fso_8\C$. More precisely, for a principal $\fsl_2$-triple $\{f,h,e\}\subset\fso_7\C\subset\fso_8\C$, the $\fsl_2\C$-module decomposition is 
    \[W_2\oplus W_6\oplus W_{10,}\]
    where the multiplicity $n_6$ of $W_6$ is $1$ for $\fso_7\C$ and $2$ for $\fso_8\C$, and
    \[Lie(\rG_2)\cong W_2\oplus W_{10.}\]
    Recall from Proposition \ref{prop: magical sl2 data} that the magical $\fsl_2$-triple in $\fg=\ff_4,~\fe_6,~\fe_7,~\fe_8$ of Case (4) of Theorem \ref{thm: classification weighted dynkin}, induces the $\fsl_2\C$-module decomposition 
    \[\fg=W_0\oplus W_2\oplus W_6\oplus W_{10},\] and we have $\fg(e)=W_2\oplus W_{10}\cong Lie(\rG_2)$. 

To complete the proof we claim that $\fc$ centralizes $W_2\oplus W_{10}$. We have $W_2=\langle f,h,e\rangle$ and hence $\fc$ commutes with $W_2$. The multiplicity of $n_{10}$ is 1. Hence $Z_{10}=W_{10}\cap\fg_0$ is 1-dimensional and $\fc$ acts by a character on $Z_{10}$. But $\fc$ has no nontrivial characters by Proposition \ref{prop c subalgebra}. The space $W_{10}$ is generated by the action of $W_2$ on $Z_{10}$, so $\fc$ centralizes 
$\fg(e)=W_2\oplus W_{10}$. 

Finally we argue that $\fz(\fc)\oplus\fg(e)$ is equal to the
centralizer of $\fc.$ By Proposition \ref{prop c subalgebra},
\begin{equation}
  \label{eq:def-re-fgss}
  \fg_0=\C^{\rkfge}\oplus\tilfg,
\end{equation}
where $\rkfge = \rk(\fg(e))$ and $\tilfg=\fg_{0,ss}\subset \fg_0$
is the semisimple part of $\fg_0$.
Moreover, $\fc\subset\tilfg$ is the complexification of the maximal compact subalgebra of $\tilfg.$ By construction $\fg(e)\cap\fg_0=\C^{\rkfge}$.  Since $\tilfg$ has a trivial center, the only elements in $\tilfg$ which centralize $\fc$ is the center of $\fc.$ 
From Proposition \ref{prop c subalgebra}, $\fz(\fc)=0$ except when $\fc=\fso_2\C=\C.$  
So the intersection of the centralizer of $\fc$ with $\fg_0$ is $\fg(e)\oplus\fz(\fc).$ 
Let $x$ be an arbitrary element of the centralizer of $\fc$ and write $x=\sum x_{2j}$ for $x_{2j}\in\fg_{2j}.$ Since $[\fc,\fg_j]\subset\fg_j$ we must have $[x_j,\fc]=0$ for all $j.$ 
For $j>0$, we have $[\fc,\ad_f^{j}x_j]=0$ and $\ad_f^{j}x_j\in\fg_0\cap(\fg(e))$, and, for $j<0$, we have $[\fc, \ad_e^j(x_j)]=0$ and $\ad_e^jx_j\in\fg_0\cap(\fg(e))$. 
Since $\{f,h,e\}\subset\fg(e),$ we conclude that $x_j\in\fg(e)$ for all $j\neq0.$ Hence, $\fz(\fc)\oplus\fg(e)$ is the centralizer of $\fc.$
\end{proof}
The following proposition is immediate from Propositions \ref{prop: magical sl2 data} and \ref{prop subalge g(e)} (and proof).
\begin{proposition}\label{prop sl2 decomp of g(e)}
  Let $\{f,h,e\}\subset\fg$ be a magical $\fsl_2$-triple and $\fg=\bigoplus_{j=0}^MW_{2m_j}$ be the $\fsl_2\C$-module decomposition. 
  \begin{itemize}
      \item For Case (3) of Theorem \ref{thm: classification weighted dynkin}, we have 
      \[\fg(e)\cong\fso_{2p-1}\C=\begin{dcases}
          \bigoplus_{j=1}^{p-1}W_{4j-2}&p\text{ odd}\\
          (W_{2p-2}\cap\fg(e))\oplus\bigoplus_{j=1, j\neq \frac{p}{2}}^{p-1}W_{4j-2}&p\text{ even}
      \end{dcases}\]
      \item For Case (4) of Theorem \ref{thm: classification weighted dynkin}, $\fg(e)\cong Lie(\rG_2)=W_2\oplus W_{10}$.
   \end{itemize}  
\end{proposition}

Finally, we prove the following lemma which will be useful in the next section.
\begin{lemma}\label{lemma:Ccentralizesg(e)}
    Let $\{f,h,e\}\subset\fg$ be a magical $\fsl_2$-triple and $\rG$ be a connected Lie group with Lie algebra $\fg$, such that the involution $\sigma_e$ in \eqref{eq magical involution} integrates to $\rG$. Let $\rC\subset\rG$ be the centralizer of $\{f,h,e\}$. Then $\rC$ centralizes the subalgebra $\fg(e)\subset\fg$ described in Proposition \ref{prop subalge g(e)}.
\end{lemma}
\begin{proof}
In cases (1) and (2) of Theorem \ref{thm: classification weighted dynkin}, this is immediate, since $\rC$ is the center of $\rG$ in Case (1) and $\fg(e)=\{f,h,e\}$ in Case (2). 
For cases (3) and (4), note that we have $[\fc,\fg(e)]=0$ by Proposition \ref{prop subalge g(e)}. Thus, we must understand how the group of components $\pi_0(\rC)$ acts on $\fg(e)$. Note that it suffices to show that $\rC$ acts trivially when $\rG$ is simply connected. 

For $\rG$ simply connected and $e\in\fg$ a nilpotent, the fundamental group of the $\rG$-orbit $\rG\cdot e\subset\fg$ is given by the components of $\rC$ (see \cite[Lemma 6.1.1]{CollMcGovNilpotents}),
\[\pi_1(\rG\cdot e)=\pi_0(\rC).\]
For Case (4) of Theorem \ref{thm: classification weighted dynkin}, $\pi_1(\rG\cdot e)$ is trivial  (see \cite[\S8.4]{CollMcGovNilpotents}). Thus, $\rC$ is connected and we conclude that $\rC$ acts trivially on $\fg(e)$. 

For Case (3), we have $\pi_1(\rG\cdot e)=\pi_0(\rC)=\Z/2\Z$ \cite[\S6.1]{CollMcGovNilpotents}. The $\rSO_N\C$-centralizer of $\{f,h,e\}$ also has two connected components since it is given by $\rS(\rO_1\C\times \rO_{N-2p+1}\C)\cong\rO_{N-2p+1}\C$ \cite[Theorem 6.1.3]{CollMcGovNilpotents}. Thus, it suffices to prove that the $\rSO_N\C$-centralizer of $\{f,h,e\}$ also centralizes $\fg(e)$. In this case, we have $\fg(e)\subset\fg$ is $\fso_{2p-1}\C\subset\fso_N\C$, and the $\rSO_N\C$-centralizer of $\fso_{2p-1}\C$ is $\rS(\rO_1\C\times\rO_{N-2p+1}\C)$. Thus, $\rC$ centralizes $\fg(e)$.
\end{proof}

\subsection{The Cayley real form}

By Proposition \ref{prop g0 isomorphism type} the subalgebra $\fg_0\subset\fg$ associated to a magical $\fsl_2$-triple has the form 
\[\fg_0=\C^{\rkfge}\oplus\tilfg=\fg_0\cap\fg(e)\oplus\tilfg.\]
Recall that it has a special real form -- the Cayley real form -- denoted by $\fg_\cC^\R$ and defined in Definition \ref{def: Cayley real form}.

\begin{proposition}\label{prop cayley real form class}
Let $\fg_\cC^\R\subset\fg_0$ be the Cayley real form associated a magical $\fsl_2$-triple $\{f,h,e\}\subset\fg$. Then 
\[\fg_\cC^\R\cong\R^{\rkfge}\oplus\tilfg^\R,\]
where $\tilfg^\R\subset\tilfg$ is the real form with complexified maximal compact subalgebra $\fc\subset\fg$. Thus,
    \begin{enumerate}
        \item for Case (1) of Theorem \ref{thm: classification weighted dynkin}, $\fg_\cC^\R\cong\R^{ \rk\fg}$.
        \item for Case (2) of Theorem \ref{thm: classification weighted dynkin},
          \begin{center}
           \begin{tabular}
           {|c|c|c|c|c|c|c|}\hline
           $\fg$&$\rA_{2n-1}$&$\rB_{n}$&$\rC_n$&$\rD_n$&$\rD_{2n}$&$\rE_7$\\\hline
           $\fg_\cC^\R$&$\R\oplus\fsl_n\C$&$\R\oplus\fso_{1,2n-2}$&$\R\oplus\fsl_{n}\R$&$\R\oplus\fso_{1,2n-3}$&$\R\oplus\fsu_{2n}^*$&$\R\oplus\fe_6^{-26}$\\\hline
       \end{tabular}
       \end{center}
       \item for Case (3) of Theorem \ref{thm: classification weighted dynkin}, $\fg_\cC^\R\cong \R^{ p-1}\oplus\fso_{1,N-2p+1}$.
       \item for Case (4) of Theorem \ref{thm: classification weighted dynkin}, 
       \begin{center}
    \begin{tabular}
        {|c|c|c|c|c|}\hline
        $\fg$&$\rE_6$&$\rE_7$&$\rE_8$&$\rF_4$\\\hline
       $\fg_\cC^\R$& $\R^2\oplus\fsl_3\C$&$\R^2\oplus\fsu_6^*$&$\R^2\oplus\fe_6^{-26}$&$\R^2\oplus\fsl_3\R$\\\hline
    \end{tabular}
\end{center}
    \end{enumerate}
\end{proposition}
\begin{proof}
    The Cayley real form is the real form of $\fg_0$ with the property that the complexification of the maximal compact subalgebra is $\fc$. The classification follows from Proposition \ref{prop c subalgebra}.
\end{proof}

\begin{remark}\label{rem R factors of cayley}
    Note that, in all of the cases, each $Z_{2m_j}$ with $n_{2m_j}=1$ contributes with an $\R$-factor to $\fg_\cC^\R$. In Case (2), the $\R$-factor of $\fg_\cC^\R$ is given by $\langle h\rangle$, the real span of $h$, and in Case (3), with $p$ even, an additional $\R$-factor arises from $\fg(e)\cap Z_{2p-2}$.
\end{remark}

Let $\fg^\R\subset\fg$ be any real form of a complex reductive Lie algebra, with complexified Cartan decomposition $\fg=\fh\oplus\fm$. Recall that the \emph{real rank} of $\fg^\R$ is defined to be the maximal dimension of a subalgebra $\fa\subset\fm$ such that the direct sum of $\fa$ with its $\fh$-centralizer is a Cartan subalgebra of $\fg$. 

From Propositions \ref{prop canonical real form} and \ref{prop cayley real form class}, a simple comparison of the real ranks (see for instance Appendices C.3 and C.4 of \cite{knappbeyondintro}) proves the next result.
\begin{proposition}\label{prop real ranks the same}
     Let $\{f,h,e\}\subset\fg$ be a magical $\fsl_2$-triple. Then the real rank of the canonical real form $\fg^\R$ equals the real rank of the Cayley real form $\fg^\R_\cC$. 
 \end{proposition} 

We will also need the notion of the Cayley group $\rG^\R_\cC$.

\begin{definition}\label{def cayley group}
    Let $\{f,h,e\}\subset\fg$ be a magical $\fsl_2$-triple with Cayley real form $\fg_\cC^\R=(\R^+)^{\rkfge}\oplus\tilfg^\R$. Let $\rG$ be a connected Lie group with Lie algebra $\fg$ such that the involution $\sigma_e$ from \eqref{eq magical involution} integrates to an involution $\sigma_e:\rG\to\rG$. Let $\rG^\R\subset\rG$ be the canonical real form and $\rC\subset\rG$ be the centralizer of $\{f,h,e\}$. Then the \emph{Cayley group} of $\{f,h,e\}$ and $\rG$ is the group 
    \[\rG_\cC^\R=(\R^+)^{\rkfge}\times \tilrG^\R~,\] where $\tilrG^\R$ is the real Lie group with Lie algebra $\tilfg^\R$ and maximal compact $\rC\cap\rG^\R$.
\end{definition}
\begin{remark}
    In general, the complexification of the maximal compact of the Cayley group $\rG^\R_\cC$ is $\rG^{\sigma_e}\cap\rC=\rH\cap\rC$. For a principal $\fsl_2$-triple, $\rC=\rZ(\rG)\subset\rG$ is the center of $\rG$ and $\tilfg^\R=0$. Thus $\tilrG^\R=\rZ(\rG^\R)$ is the center of $\rG^\R$. In particular,  $\rC\neq \rC\cap\rH$ in general. For example, when $\rG=\rSL_n\C$ and $\{f,h,e\}$ is a principal $\fsl_2$-triple, $\rC=\Z/n\Z$ is the center of $\rSL_n\C$ but the center of the canonical real form $\rSL_n\R$ is either $\Z/2\Z$ or trivial.
\end{remark}


    

\subsection{Lie theory structure for magical nilpotents in exceptional Lie algebras} \label{sec lie theory for except}

 Let $\{f,h,e\}\subset\fg$ be a magical $\fsl_2$-triple from Case (4) of Theorem \ref{thm: classification weighted dynkin}. In this section we will study the structure of the magical $\fsl_2$-triple in more detail. The root poset diagrams in \S\ref{subsec root posets exceptional} will be important, so frequently referenced in this discussion. In these diagrams, the labeling of a line connecting a positive root $\beta$ to a higher positive root $\gamma$ corresponds to the simple root $\alpha_j$ for which $\gamma=\beta+\alpha_j$. The labeling of every line can be deduced from the labeling of the left most line of each row.

 Recall that the $\fsl_2\C$-module decomposition \eqref{eq sl2 module decomp} is $\fg=W_0\oplus W_2\oplus W_6\oplus W_{10}$, the $\Z$-grading \eqref{eq adh weight decomp} is $\fg=\bigoplus_{j=-5}^5\fg_{2j}$ and the subalgebra $\fg(e)\subset\fg$ described in Proposition \ref{prop subalge g(e)} is $\fg(e)\cong Lie(\rG_2)=W_2\oplus W_{10}$. 
The complexified Cartan decomposition of the involution $\sigma_e:\fg\to\fg$ is $\fg=\fh\oplus\fm$, where
\begin{equation}
    \label{eq hm decomp}\xymatrix@=1em{\fh=\fg_{-8}\oplus\fg_{-4}\oplus\fg_0\oplus\fg_4\oplus\fg_8&\text{and}&\fm=\fg_{-10}\oplus\fg_{-6}\oplus\fg_{-2}\oplus\fg_{2}\oplus\fg_{6}\oplus\fg_{10}}.
\end{equation}

Each of the weight spaces $\fg_{2j}$ with $j\neq0$ is a direct sum of root spaces as in \eqref{eq:weight-spdirectsumroot-sp}. From the diagrams in \S\ref{subsec root posets exceptional}, it is clear that the weight spaces $\fg_{\pm2}$ decompose as a direct sum of two $\fg_0$-representations,
\begin{equation}
   \label{eq G0 invariant g2 decomp} \fg_{\pm2}=\fg_{\pm2}^b\oplus\fg_{\pm\tilde\alpha},
\end{equation}
where $\fg_{\tilde\alpha}$ is the root space of the simple root $\tilde\alpha$ in the diagrams in \S\ref{subsec root posets exceptional}, and $\fg_{\pm2}^b$ is the direct sum of root spaces in $\fg_{\beta}\subset\fg_{\pm2}$ with $\beta\neq\pm\tilde\alpha$.
We can then decompose $f\in\fg_{-2}$ and $e\in\fg_2$ as 
\begin{equation}
    \label{eq decomp except}\xymatrix{f=f_b+\tilde f & \text{and}&e=e_b+\tilde e,}
\end{equation} where $e_b$, $f_b$ and $\tilde e$, $\tilde f$ are the projections of $e$ and $f$ onto $\fg_{\pm2}^b$ and $\fg_{\pm\tilde\alpha}$ respectively. Define further $\tilde h=[\tilde e,\tilde f].$

\begin{lemma}\label{lem decomp except}
  Each of the terms $\tilde f$, $f_b$, $\tilde e$ and $e_{b}$ in  \eqref{eq decomp except} is nonzero.
\end{lemma}
\begin{proof}
 The $\fsl_2\C$-module decomposition $\fg=W_0\oplus W_2\oplus W_6\oplus W_{10}$ implies that $\ad_f:\fg_{10}\to\fg_8$ and $\ad_f:\fg_6\to\fg_4$ are isomorphisms. The map $\ad_f:\fg_{10}\to\fg_8$ is equal to $\ad_{\tilde f}$ since $\fg_8$ and $\fg_{10}$ are root spaces which differ by the root $\tilde\alpha$. So $\tilde f$ cannot be zero. On the other hand, $\ad_f:\fg_6\to\fg_4$ is given by $\ad_{f_b}$ since $\fg_4$ and $\fg_6$ are both a direct sum of root spaces $\fg_\beta$, where $\beta$ has the form $\beta=\sum n_i\alpha_i$ and the coefficient of $\tilde\alpha$ is 1. So again $f_b\neq 0$. Similar arguments imply $\tilde e\neq0$ and $e_b\neq0$.
\end{proof}

Note that $\tilde\alpha$ is a red root labeled with a $2$ in \S\ref{subsec root posets exceptional}. Denote by $\{\beta_1,\beta_2,\beta_3\}$ the other red root spaces which are still labeled with a $2$. We claim that $\{\tilde\alpha,\beta_1,\beta_2,\beta_3\}$ are a $\rD_4$-system with 
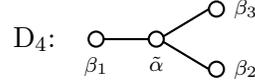
\begin{center}
  \begin{tikzpicture}[scale=.4]
    \draw (-1,0) node[anchor=east]  {$\rD_{4}$:};
    \draw[thick] (0.25 cm,0) -- +(1.5 cm,0);
    \draw[thick] (2.25 cm, -.1 cm) -- +(1.5 cm, -.9 cm);
    \draw[thick] (2.25 cm, .1 cm) -- +(1.5 cm, .9 cm);
    \node at (0,0) [below = 1 mm ] {${\scriptstyle \beta_1}$};
    \draw[thick] (0 cm,0 cm) circle (.25cm);
    \node at (2,0) [below = 1 mm ] {${\scriptstyle \tilde\alpha}$};
    \draw[thick] (2 cm,0 cm) circle (.25cm);
    \node at (4,-1) [right = 1 mm ] {${\scriptstyle \beta_2}$};
    \draw[thick] (4 cm,-1cm) circle (.25cm);
    \node at (4,1) [right = 1 mm ] {${\scriptstyle \beta_3}$};
    \draw[thick] (4 cm,1 cm) circle (.25cm);
    \end{tikzpicture} 
\end{center}
Since there is an action of the symmetric group on three letters on the roots of $\rD_4$, the choice of which $\beta_i$ corresponds to which root space in $\fg_2^b$ is irrelevant. 
\begin{lemma}\label{lem red so8}
The root spaces associated to the red roots in the diagrams of Section \ref{subsec root posets exceptional} form a subalgebra isomorphic to $\fso_8\C$.
\end{lemma} 
\begin{proof}
The proof is by direct computation. The positive roots of $\rD_4$ are 
\begin{equation}
    \label{eq d4 pos roots}\{\tilde\alpha,\beta_1,\beta_2,\beta_3\}\cup\bigg\{\tilde\alpha+\sum_{n_i\in\{0,1\}}n_i\beta_i\bigg\}\cup\{2\tilde\alpha_1+\beta_1+\beta_2+\beta_3\}.
\end{equation}
Using the expression of $\beta_i$ in terms of the simple roots of $\fg$, one checks that $\{\tilde\alpha,\beta_1,\beta_2,\beta_3\}$ satisfy the relations of a $\rD_4$ root system and that no other linear combinations of $\{\tilde\alpha,\beta_1,\beta_2,\beta_3\}$ define roots in $\fg$.  
From the diagrams \S\ref{subsec root posets exceptional}, it is clear that $\tilde\alpha+\beta_i$ is a root, but none of $\beta_i+\beta_j$,  $\tilde\alpha-\beta_i$ or $\beta_i-\beta_j$ is a root. Any other linear combination of $\tilde\alpha,\beta_1,\beta_2,\beta_3$ will have the coefficient of $\tilde\alpha$ being nonzero and a coefficient $n_i\geq 2$. All such roots in $\fg$ are listed in the tables \cite[Appendix C.2]{knappbeyondintro} and one checks that the only expressions which are roots of $\fg$ are in \eqref{eq d4 pos roots}.\footnote{Note that in the notation of \ref{subsec root posets exceptional}, $\tilde\alpha=\alpha_1,\alpha_8$ for $\fg=\fe_7,\fe_8$ respectively, while in the notation of \cite[Appendix C.2]{knappbeyondintro}, $\tilde\alpha=\alpha_7,\alpha_8$ for $\fg=\fe_7,\fe_8$ respectively.}
\end{proof}
Recall that the coroot $h_\alpha$ associated to a root $\alpha\in\fa^*$ is defined by $h_\alpha=2\frac{\alpha^*}{\langle \alpha,\alpha\rangle}$ where $\alpha^*\in\fa$ satisfies $\langle \alpha^*,x\rangle=\alpha(x)$ for all $x\in\fa$. Let $\Delta_+\subset\fa^*$ denote a set of positive roots with simple roots $\{\alpha_1,\ldots,\alpha_{\rk(\fg)}\}$, and let $\{f_i,h_{\alpha_i},e_i\}$ be $\fsl_2$-triples with $e_{i}\in\fg_{\alpha_i}$ and $f_i\in\fg_{-\alpha_i}$. 
This data determines a principal $\fsl_2$-triple $\{f,h,e\}\subset\fg$  given by, 
\[h=\sum_{\alpha\in\Delta^+}h_{\alpha}=\sum_{i=1}^{\rk(\fg)}r_ih_{\alpha_i},\ \ \ \ \ e=\sum\limits_{i=1}^{\rk(\fg)}\sqrt{r_i}e_i,\ \ \ \ \ f=\sum_{i=1}^{\rk(\fg)}\sqrt{r_i}f_i.\] 
For the simple roots $\{\tilde\alpha,\beta_1,\beta_2,\beta_3\}$ of $\rD_4$, the above construction yields
\begin{equation}\label{eq principal so8 sl2}
    \{ f, h, e\}=\Bigg\{\sqrt{12}f_{\tilde\alpha}+\sqrt{6}\sum_{j=1}^3f_{\beta_j},12h_{\tilde\alpha}+6\sum_{j=1}^3h_{\beta_j},\sqrt{12}e_{\tilde\alpha}+\sqrt{6}\sum_{j=1}^3e_{\beta_j}\Bigg\}.
\end{equation}
\begin{lemma}
    A principal $\fsl_2$-triple $\{f,h,e\}\subset\fso_8\C\subset\fg$  in the $\fso_8\C$-subalgebra from Lemma \ref{lem red so8} is a magical $\fsl_2$-triple in $\fg$ from Case (4) of Theorem \ref{thm: classification weighted dynkin}.
\end{lemma}
\begin{proof}
Consider the simple roots $\{\tilde\alpha,\beta_1,\beta_2,\beta_3\}$ of $\rD_4$ and the principal $\fsl_2$ given above. 
 We must show that the numbers $\alpha_i(h)$ match the weighted Dynkin diagrams from Case (4) of Theorem \ref{thm: classification weighted dynkin}. 
Let $\alpha_i$ be a simple root of $\fg$ which is not in $\{\beta_1,\beta_2,\beta_3\}$. Then $\alpha_i$ is orthogonal to $\tilde\alpha$, and 
\[\alpha_i\Bigg(12h_{\tilde\alpha}+6\sum_{j=1}^3h_{\beta_j}\Bigg)=\alpha_i\Bigg(6\sum_{j=1}^3h_{\beta_j}\Bigg).\]
If $\alpha_i$ is orthogonal to each $\beta_j$, then $\alpha_i(\tilde h)=0$. For $\fg=\fe_7,~\fe_8$ respectively, the simple roots $\{\alpha_4,\alpha_5,\alpha_7\},~\{\alpha_2, \alpha_3,\alpha_4,\alpha_5\}$ are orthogonal to each $\beta_j$. For the remaining simple roots $\alpha_j\notin\{\tilde\alpha,\beta_1,\beta_2,\beta_3\}$, there is a unique $\beta_l$ such that $\alpha_i+\beta_l$ is a root and a unique $\beta_k\neq \beta_l$ such that $-\alpha_i+\beta_k$ is a root.
Hence
\[\alpha_i\Bigg(12h_{\tilde\alpha}+6\sum_{j=1}^3h_{\beta_j}\Bigg)=6\alpha_i(h_{\beta_l})+6\alpha_i(h_{\beta_k})=12\frac{\langle\alpha_i,\beta_l\rangle}{\langle \beta_l,\beta_l\rangle}+12\frac{\langle\alpha_i,\beta_k\rangle}{\langle \beta_k,\beta_k\rangle.}\]
Since the roots, $\beta_l$, $\beta_k$, $\alpha_i+\beta_l$ and $\alpha_i-\beta_k$ have the same length we have $\alpha_i(\tilde h)=0$.
Finally, if $\beta_1$ is the simple root which is also a simple root of $\fg$, then $\tilde\alpha(\tilde h)=2$ and $\beta_1(\tilde h)=2$. Thus, the weighted Dynkin diagram of $\{ \tilde f,\tilde h, \tilde e\}$ corresponds to a magical $\fsl_2$-triple from Case (4) of Theorem \ref{thm: classification weighted dynkin}.
\end{proof}

\begin{lemma}\label{lem fg(e) in red so8}
  Let $\{f,h,e\}\subset\fg$ be a magical $\fsl_2$-triple from Case (4) in Theorem \ref{thm: classification weighted dynkin}. Then $\fg_{\tilde\alpha}\subset\fg(e)$ is a simple root space for $\fg(e)$ associated to a long root and $\langle e_b\rangle\subset\fg(e)$ is a simple root space for $\fg(e)$ associated to a short root. 
\end{lemma}
\begin{proof}
This follows from the fact that $Lie(\rG_2)\subset\fso_8\C$ and a principal $\fsl_2$-triple in $Lie(\rG_2)$ is also principal in $\fso_8\C$. Using the decomposition \eqref{eq principal so8 sl2}, we have $e_b=\sqrt{12}(e_{\beta_1}+e_{\beta_2}+e_{\beta_3})$ and $\tilde e=\sqrt{6}e_{\tilde\alpha}$. A direct computation shows that $\{\tilde e,e_b\}$ generate the positive roots of $Lie(\rG_2)$ with highest root $[\tilde e,[e_b,[e_b,[e_b,\tilde e]]]]$.
\end{proof}
\begin{remark}
    For Case (4) of Theorem \ref{thm: classification weighted dynkin}, this gives a direct proof that the sum $W_2\oplus W_{10}$ in the $\fsl_2\C$-module decomposition of $\fg$ is a subalgebra isomorphic to $Lie(\rG_2)$.
\end{remark}

Recall that the canonical real form $\fg^\R$ associated to $\{f,h,e\}$ is the quaternionic real form of $\fg$. Thus, $\fh\cong\fsl_2\C\oplus\fh'$, where $\fh'$ is $\fsp_6\C$, $\fsl_6\C$, $\fso_{12}\C$ and $\fe_7$ when $\fg$ is $\ff_4$, $\fe_6$, $\fe_7$, $\fe_8$ respectively. 
\begin{lemma}
    The decomposition $\fh=\fsl_2\C\oplus\fh'$ is given by 
    \begin{equation}\label{eq h'+sl2 decomp}
        \xymatrix{\fsl_2\C=\fg_{-8}\oplus[\fg_{-8},\fg_8]\oplus\fg_8&\text{and}&\fh'=\fg_{-4}\oplus[\fg_{-4},\fg_4]\oplus\fg_4.}
    \end{equation}
\end{lemma}
\begin{proof}
   It is clear that $\fh'=\fg_{-4}\oplus[\fg_{-4},\fg_4]\oplus\fg_4$ is a subalgebra and $\fg_{-8}\oplus[\fg_{-8},\fg_8]\oplus\fg_8$ is a subalgebra isomorphic to $\fsl_2\C$. Note that $[\fg_{\pm8},\fg_{\pm4}]=0$ since there is not a weight-$12$ summand of the grading. Using the root poset diagrams in \S\ref{subsec root posets exceptional}, $\fg_{\pm 8}$ and $\fg_{\pm4}$ are direct sums of root spaces $\fg_\alpha$ with $\alpha=\pm 1\tilde\alpha+\sum_{\alpha_i\neq \tilde\alpha}n_i\alpha_i$. This implies $[\fg_8,\fg_{-4}]=0$ and $[\fg_{-8},\fg_4]=0$ since $[\fg_8,\fg_{-4}]\subset\fg_4$ and $[\fg_{-8},\fg_4]\subset\fg_{-4}$. Now the Jacobi identity implies that $[\fh',\fg_{-8}\oplus[\fg_8,\fg_{-8}]\oplus\fg_8]=0$.
\end{proof}
\begin{lemma}\label{lem h' sl2 invariant decomp of m}
Consider the decomposition $\fh=\fsl_2\C\oplus\fh'$ from \eqref{eq h'+sl2 decomp} and the decomposition of $\fm$ from \eqref{eq hm decomp} and \eqref{eq G0 invariant g2 decomp}. Then $\fm$ decomposes as 
\[\fm=\vcenter{\xymatrix@=0em{\fg_{10}&\oplus&\fg_{6}&\oplus&\fg_2^b&\oplus&\fg_{-\tilde\alpha}\\
&&&\oplus&&&\\
\fg_{+\tilde\alpha}&\oplus&\fg_{-2}^b&\oplus&\fg_{-6}&\oplus&\fg_{-10},}}\]
where the rows are $\fh'$-invariant and the columns are $\fsl_2\C$-invariant.
\end{lemma}
\begin{proof}
Observe that $\fg_{\pm4}$, $\fg_{\pm6}$, $\fg_{\pm 8}$, $\fg_{-\tilde\alpha}$ are direct sums root spaces $\fg_{\alpha}$ with $\alpha=\sum_{i}n_i\alpha$, where the coefficient of the simple root $\tilde\alpha$  is $\pm1$ and $\fg_{\pm10}$ is the root space for $\pm\tilde\alpha\pm\gamma$ where $\fg_8$ is the root space for the root $\gamma$. 
Thus, the rows are preserved by bracketing with $\fg_{\pm4}$ and the columns are preserved by bracketing with $\fg_{\pm 8}$. 
\end{proof}

Finally, we deduce some bracket relations which will be useful later. 
\begin{lemma}\label{lem exceptional bracket relations}
  Let $\{f,h,e\}\subset\fg$ be a magical $\fsl_2$-triple from Case (4) of Theorem \ref{thm: classification weighted dynkin} and let $\fg=W_0\oplus W_2\oplus W_6\oplus W_{10}$ be the $\fsl_2\C$-module decomposition. Let $f=f_b+\tilde f$ and $e=e_b+\tilde e$ be the decompositions \eqref{eq decomp except} and $V_6=\ker(\ad_e|_{W_6})$. Then, for any $\phi\in V_6$,
\begin{equation}
     \label{eq adfb^3 f nonzer0}\ad_{f_b}^3(\tilde f)\neq 0\in\fg_{-8},
 \end{equation} 
 \begin{equation}
     \label{eq ad^3phi}\ad_f^3(\phi)=[f_b,[\tilde f,[f_b,\phi]]]=[[f_b,\tilde f],[f_b,\phi]],
 \end{equation}
  \begin{equation}
     \label{eq ad f_b+phi}\ad_{f_b+\phi}\tilde f= [f_b,\tilde f],
 \end{equation}
 \begin{equation}
    \label{eq ad^3 f_b+phi} \ad_{f_b+\phi}^3\tilde f= \ad_{f_b}^3(\tilde f)+3\ad_{f}^3(\phi)+\ad^2_{\phi}\circ \ad_{f_b}(\tilde f).
 \end{equation}
\end{lemma}
\begin{proof}
Equation \eqref{eq adfb^3 f nonzer0} follows directly from Lemma \ref{lem fg(e) in red so8} and the bracket relations in $Lie(\rG_2)$. For equation \eqref{eq ad^3phi}, we have 
\[\ad_f^3(\phi)=[f_b+\tilde f,[f_b+\tilde f,[f_b+\tilde f,\phi]]].\]
Since $\ad_{\tilde f}\fg_6=0\subset \fg_4$, $[\tilde f,\phi]=0$ and so
\[\ad_f^3(\phi)=[f_b,[f_b,[f_b,\phi]]]+[\tilde f,[\tilde f,[f_b,\phi]]]+[\tilde f,[f_b,[f_b,\phi]]]+[f_b,[\tilde f,[f_b,\phi]]].\]
We will show that the first three terms are zero. Recall that $\fg_6$ is a direct sum of root spaces $\fg_\alpha$ where the coefficient of $\tilde\alpha$ is $1$. 
Thus, 
\[\xymatrix{[f_b,[f_b,\phi]]\subset\fg_{\tilde\alpha}&\text{and}[\tilde f,[f_b,\phi]]\subset\fg_2^b.}\]
Since $[\fg_{\pm2}^b,\fg_{\mp\tilde\alpha}]=0$, the first two terms are zero. 
For the third term, note that $\ad_{f_b}^2(V_6)\subset\fg_{\tilde\alpha}$ is the projection of $\ad_f^2(V_6)$ onto $\fg_{\tilde\alpha}$. But $\fg_{\tilde\alpha}\subset W_2\oplus W_{10}$ by Lemma \ref{lem fg(e) in red so8} and $\ad_f^2(V_6)\cap W_2\oplus W_{10}=0$. Hence $\ad_{f_b}^2(\phi)=0$ for $\phi\in V_6$, and 
\[\ad_f^3(\phi)=[f_b,[\tilde f,[f_b,\phi]]].\]
The Jacobi identity and $\ad_{f_b}^2(\phi)=0$ imply $[f_b,[\tilde f,[f_b,\phi]]]=[[f_b,\tilde f],[f_b,\phi]]$.

Equation \eqref{eq ad f_b+phi} follows since $[\fg_{-\tilde\alpha},\fg_6]=0$. For \eqref{eq ad^3 f_b+phi}, we have 
\[\ad^3_{f_b+\phi}(\tilde f)=[f_b+\phi,[f_b+\phi,[f_b,\tilde f]]]\]
since $[\tilde f,\phi]=0$. Thus, 
\[\ad^3_{f_b+\phi}=\ad^3_{f_b}(\tilde f)+[f_b,[\phi,[f_b,\tilde f]]]+[\phi,[f_b,[f_b,\tilde f]]+\ad^2_\phi([f_b,\tilde f]).\]
The middle two terms are in $\fg_0$. Using the Jacobi identity and $[\tilde f,\phi]=0$, we have 
\begin{equation*}
\begin{split}
[f_b,[\phi,[f_b,\tilde f]]]+[\phi,[f_b,[f_b,\tilde f]]]&=-[f_b,[\tilde f,[\phi,f_b]]]-[[f_b,\tilde f],[\phi,f_b]]-[f_b,[[f_b,\tilde f],\phi]]\\
&=[f_b,[\tilde f,[f_b,\phi]]]+[[f_b,\tilde f],[f_b,\phi]]+[f_b,[\tilde f,[f_b,\phi]]]\\&
=3\ad_f^3(\phi),
\end{split}
\end{equation*}
by \eqref{eq ad^3phi}.
\end{proof}
As a result of the above discussion, we have the following proposition. Recall that a nonzero nilpotent is magical if it belongs to a magical $\fsl_2$-triple.
\begin{proposition}
The nilpotent $[f_b,\tilde f]\subset\fg_{-4}$ is a magical nilpotent in $\fh'$ of the type of Case (2) of Theorem \ref{thm: classification weighted dynkin} and $[f_b,[f_b,[f_b,\tilde f]]]$ is a magical (i.e.~nonzero) nilpotent in $\fsl_2\C$.
\end{proposition}
\begin{remark}\label{rem fb magical from case 2}
    Note that $[\fg_{-2}^b,\fg_2^b]\subset\fg_0$ is isomorphic to $[\fg_4,\fg_{-4}]$, thus $\{f_b,[e_b,f_b],e_b\}\subset\fg_{-2}^b\oplus[\fg_{2}^b,\fg_{-2}^b]\oplus\fg_2^b$ is a magical nilpotent from Case (2) of Theorem \ref{thm: classification weighted dynkin}.
\end{remark}

We also need to understand the group $\rH$ and its action on $\fm$. Let $\rG$ and $\rH'$ be the simply connected groups with Lie algebras $\fg$ and $\fh'$ respectively. From the description  of the Lie algebras $\fh'$ above, $\rH'$ is $\rSp_6\C$, $\rSL_6\C$, $\rSpin_{12}\C$, $\rE_7$ when $\fg$ is $\ff_4$, $\fe_6$, $\fe_7$, $\fe_8$, respectively. The group $\rH\subset\rG$ is a quotient
\[\rH=(\rH'\times \rSL_2\C)/\Z_2,\]
where $\Z_2$ has generator $(\mu',\mu_2)$ for $\mu'\in\rH'$ and $\mu_2\in\rSL_2\C$ the unique order two elements of the center. 

As an $\rH$-representation, $\fm$ is the tensor product $\fm=V'\otimes V_2$, where $V_2$ is the standard representation of $\rSL_2\C$ and $V'$ is an irreducible $\rH'$-representation known as a \emph{minuscule representation}. The decomposition $\fh'=\fg_{-4}\oplus[\fg_{-4},\fg_4]\oplus\fg_4$ defines a maximal parabolic subgroup $\rP'\subset\rH'$ with Lie algebra $[\fg_{-4},\fg_4]\oplus\fg_4$. In fact, $V'$ is the irreducible representation associated to the Pl\"ucker embedding of $\rH'/\rP'\to \mathbb P(V')$, that is, $\rH'/\rP'$ is isomorphic to the unique closed $\rH'$-orbit in $\mathbb P(V')$.
For example, when $\rH'=\rSL_6\C$, $V'=\Lambda^3\C^6$ is the third exterior product of the standard representation of $\rSL_6\C$ and $\rSL_6\C/\rP'$ is the Grassmannian of three planes in $\C^6$. When $\fh'=\fe_7$, then $V'$ is the unique irreducible $\rE_7$-representation of dimension $56$.

The following result describes the $\rH'$-orbit structure of $\mathbb P(V')$. We refer the reader to the work of Landsberg--Manivel, specifically \cite[\S5.3]{LandsbergManivelorbits}. For the case $\rH'=\rSL_6\C$ this orbit structure was described in \cite{DonagiGrassmannians}. For $\rSp_6\C$ and $\rSpin_{12}\C$ some aspects of the orbit structure are described in \cite{Igusa}, and for $\rE_7$ in \cite{HarisE7}.
\begin{proposition}\label{prop H' orbits on P(V')}
    Consider the action of $\rH'$ on $\mathbb P(V')$ described above. There are four $\rH'$-orbits, $\cO_1,\cO_2,\cO_3,\cO_4$. Moreover, the following facts completely characterize $\cO_1,\cO_3,\cO_4$ 
    \begin{enumerate}
        \item $\cO_1$ is closed and isomorphic to $\rH'/\rP'$;
        \item $\cO_3$ has codimension 1 and $\overline \cO_3$ is the tangent variety of $\rH'/\rP'$;
        \item $p\in\cO_3$ if and only if $p$ is contained in a unique tangent line of $\rH'/\rP'$; 
        \item $\cO_4$ is open.
    \end{enumerate}
\end{proposition}


In the decomposition of $\fm$ given by Lemma \ref{lem h' sl2 invariant decomp of m}, the subspace $\fg_{\tilde\alpha}\oplus\fg_{-2}^b\oplus\fg_{-6}\oplus\fg_{-10}$ is $\rH'$-invariant and hence isomorphic to the representation $V'$. The following proposition will be used in the next section.
\begin{proposition}\label{prop H' stabilizer in parabolic}
    Consider the $\rH'$-invariant subspace of $\fm$ given by $\fg_{\tilde\alpha}\oplus\fg_{-2}^b\oplus\fg_{-6}\oplus\fg_{-10}$. \begin{enumerate}
        \item The point $(\tilde e,0,0,0)\in \fg_{\tilde\alpha}\oplus\fg_{-2}^b\oplus\fg_{-6}\oplus\fg_{-10}$ defines a point in the closed orbit in $\mathbb P(\fg_{\tilde\alpha}\oplus\fg_{-2}^b\oplus\fg_{-6}\oplus\fg_{-10})$ whose stabilizer is the parabolic subgroup $\rP'$ of $\rH'$ with Lie algebra $[\fg_{-4},\fg_4]\oplus\fg_4$. 
        \item For all $\mu\in\C$, a point $(\mu \tilde e,f_{b},0,0)$ defines a point in the codimension one orbit of $\mathbb P(\fg_{\tilde\alpha}\oplus\fg_{-2}^b\oplus\fg_{-6}\oplus\fg_{-10})$ whose stabilizer is contained in the parabolic $\rP'$.
    \end{enumerate}
\end{proposition}
\begin{proof}
Write a point in $\fh'$ as $(x,y,z)\in\fg_{-4}\oplus [\fg_{-4},\fg_4]\oplus\fg_4$ and consider $(\tilde e,0,0,0)\in V'=\fg_{\tilde\alpha}\oplus\fg_{-2}^b\oplus\fg_{-6}\oplus\fg_{-10}$. The bracket is given by
\[[(x,y,z),(\tilde e,0,0,0)]=(\lambda(y) \tilde e,[x,\tilde e],0,0),\]
where $\lambda(y)\in\C$ and where $[x,\tilde e]\in\fg_{-2}^b$ is zero if and only if $x=0$. Thus, the $\rH'$-stabilizer  $[\tilde e,0,0,0]\in\mathbb P(V')$ is the parabolic subgroup $\rP'\subset\rH'$ with Lie algebra $[\fg_{-4},\fg_4]\oplus\fg_4$.

For the second point, we first analyze the case $\mu=0$. Note that $\ad_{f_b}:\fg_{-4}\to\fg_{-6}$ is an isomorphism and $\dim(\fg_{-2}^b)=\dim(\fg_{-6})$. Thus, 
\[\dim(\mathbb P( V'))=\dim(\fg_{4}\oplus\fg_{-4})+1=\dim (\fh')-\dim([\fg_4,\fg_{-4}])+1.\]
So $[0,f_b,0,0]\in\mathbb P( V')$ is in the codimension one orbit $\cO_3$ if and only if 
\[\dim(\{w\in\fh'~|~[w,f_b]=\lambda f_b \text{ for some $\lambda\in\C$}\})=\dim([\fg_{-4},\fg_{4}]).\]
To show this, write $w=(x,y,z)\in\fg_{-4}\oplus [\fg_{-4},\fg_4]\oplus\fg_4$. Then the bracket $[w,f_b]$ is given by 
\[[(x,y,z),f_b]=([z,f_b],[y,f_b],[x,f_b],0)\in V'.\]
Since $\ad_{f_b}:\fg_{4}\to\fg_{\tilde\alpha}$ is surjective and $\ad_{f_b}:\fg_{-4}\to\fg_{-6}$ is an isomorphism, the space of $(x,0,z)\in\fh'$ with $\ad_{(x,0,z)}f_b=\lambda f_b$ has dimension $\dim(\fg_{4})-1$. 

Recall that $[\tilde f,f_b]\in\fg_{-4}$ is a magical nilpotent from Case (2) of Theorem \ref{thm: classification weighted dynkin}. For $y\in[\fg_{-4},\fg_4]$, we decompose $[\fg_{-4},\fg_4]=[[\tilde f,f_b],\fg_{4}]\oplus\fc$.
Then $\ad_{f_b}:[[\tilde f,f_b],\fg_{4}]\to\fg_{-2}^b$ is an isomorphism, so there is a one-dimensional subspace of $[[\tilde f,f_b],\fg_{4}]$ which acts on $f_b$ by scalar multiplication. Since $[\fc,f_b]=0$, we have 
\[\dim(\{w\in\fh'~|~[w,f_b]=\lambda f_b \text{ for some $\lambda\in\C$}\})=\dim(\fg_{4})-1+1+\dim(\fc)=\dim([\fg_{-4},\fg_4]).\]

By the above computation, the Lie algebra of the stabilizer of $[0,f_b,0,0]\in\mathbb P(V')$ is contained in $[\fg_{-4},\fg_4]\oplus\fg_4$, so in the Lie algebra of $\rP'$. To show that the stabilizer of $[0,f_b,0,0]$ is indeed contained in $\rP'$, we use the description of the codimension 1 orbit $\cO_3$ of Proposition \ref{prop H' orbits on P(V')}. Namely, there is a unique projective line $\ell\subset\mathbb P(V')$ which is tangent to the closed orbit $\rH'/\rP'$ and passes through $[0,f_b,0,0]$. This line is given by 
\[\ell(\lambda)=[\tilde e, \lambda f_b,0,0]\subset\mathbb P(V').\] 
Since the tangent line is unique, the action of the stabilizer of $[0,f_b,0,0]$ on $\ell$ must fix the intersection of $\ell$ with the closed orbit, which is given by $[\ell(0)]=[\tilde e,0,0,0]$. Since the stabilizer of $\tilde e$ is $\rP'$, we conclude that the stabilizer of $[0,f_b,0,0]$ is contained in $\rP'$.

Finally, since $\ad_{f_b}:\fg_{4}\to\fg_{\tilde\alpha}$ is surjective and $[f_{\tilde\alpha},\fg_4]=0$, for every $\mu\in\C$, there is $x\in\fg_4$ such that $\Ad_{\exp(x)}(0,f_b,0,0)=(\mu\tilde e,f_b,0,0)$. Thus, $[\mu \tilde e,f_b,0,0]$ and $[0,f_b,0,0]$ are in the same $\rH'$-orbit. Moreover since the stabilizers of $[\mu \tilde e,f_b,0,0]$ and $[0,f_b,0,0]$ are conjugate via $\exp(x)\in \rP'$, we conclude that the stabilizer of $[\mu \tilde e,f_b,0,0]$ is contained in $\rP'$, completing the proof.
 \end{proof}

\section{Higgs bundles and the Cayley map}\label{sec: Higgs bundles and the Cayley map}
From now on, $X$ will denote a fixed compact Riemann surface of genus $g\geq 2$, with canonical bundle $K$. All geometric objects we will consider are over $X$. Let $\rH$ be a complex reductive Lie group.
\subsection{Higgs bundles}\label{subsec: Higgs bundles}
Let $\cE_\rH\to X$ be a holomorphic principal $\rH$-bundle. Given a holomorphic action of $\rH$ on a space $Y$, we denote the associated fiber bundle by $\cE_\rH[Y]=(\cE_\rH\times Y)/\rH$,
where $(x ,y)\cdot g=(x\cdot g,g^{-1}\cdot y)$. When $V$ is a vector space, $\cE_\rH[V]$ is a holomorphic vector bundle, and when $\rH$ acts by group homomorphisms on a complex Lie group $\rG$, then $\cE_\rH[\rG]$ is a holomorphic principal $\rG$-bundle.

\begin{definition}\label{def:Higgs pair}
Let $\rG$ be a complex reductive Lie group, $V$ be a complex vector space with a holomorphic $\rG$-action and $L$ be a holomorphic line bundle on $X$. An \emph{$L$-twisted $(\rG,V)$-Higgs pair} is a pair $(\cE_\rG,\varphi)$ consisting of a holomorphic $\rG$-bundle $\cE_\rG\to X$ and a holomorphic section $\varphi\in H^0(\cE_\rG[V]\otimes L)$. The section $\varphi$ is  called the \emph{Higgs field}.
\end{definition}
There is a natural $\C^*$-action on the set of $L$-twisted $(\rG,V)$-Higgs pairs given by 
\begin{equation}\label{eq C* action def}
    \lambda\cdot(\cE,\varphi)=(\cE,\lambda\varphi).
\end{equation}
Our main objects of interest, Higgs bundles, are a particular class of Higgs pairs. 

\begin{definition}\label{def:Higgs bundle}
Let $\rG^\R\subset\rG$ be a real form of a complex semisimple Lie group $\rG$. Let $\rH^\R\subset\rG^\R$ be a maximal compact subgroup, $\rH\subset\rG$ be its complexification and $\fg=\fh\oplus\fm$ be a complexified Cartan decomposition.
An \emph{$L$-twisted $\rG^\R$-Higgs bundle} is an $L$-twisted $(\rH,\fm)$-Higgs pair $(\cE_\rH,\varphi)$.
\end{definition}

We will denote the set of $L$-twisted $\rG^\R$-Higgs bundles by $\cH_L(\rG^\R)$. When the twisting line bundle $L$ is the canonical bundle $K$, we will refer to a $K$-twisted $\rG^\R$-Higgs bundle simply as a \emph{$\rG^\R$-Higgs bundle} and write $\cH_K(\rG^\R)=\cH(\rG^\R)$. 

Let $E_\rH$ be the smooth underlying bundle of a holomorphic bundle $\cE_\rH$. The gauge group $\cG_\rH$ of smooth bundle automorphisms of $E_\rH$ acts on $\cH_L(\rG^\R)$ by pulling back the holomorphic structure and pulling back the Higgs field. In particular, if $(\cE_\rH,\varphi)$ is an $L$-twisted Higgs bundle and $g\in\cG_\rH(E_\rH)$, then 
\[g\cdot \varphi=\Ad_{g}(\varphi).\] 
The automorphism group of an $L$-twisted Higgs bundle $(\cE_\rH,\varphi)$ is the group of holomorphic gauge transformations $g$ of $\cE_\rH$ such that $\Ad_{g}(\varphi)=\varphi$.

\begin{example}\label{ex uniformizing Higgs}
    Here are some relevant examples of Higgs bundles: 
\begin{itemize}
    \item The complex group $\rG$ can be regarded as a real form of $\rG\times \rG$. In this situation $H=G$, $\fm=\fg$, and an $L$-twisted 
 $\rG$-Higgs bundle is thus a  pair $(\cE_\rG,\varphi)$, where $\cE_\rG$ is a holomorphic principal $\rG$-bundle and
$\varphi\in H^0(\cE_\rG[\fg]\otimes L)$. 
    \item For $\rG^\R=\R^+$, an $L$-twisted $\rG^\R$-Higgs bundle is just a holomorphic section $\varphi$ of $L$.
    \item For $\rG^\R=\rPSL_2\R$, we have $\rH\cong\C^*$ and $\fm=\fm^-\oplus\fm^+=\langle f\rangle\oplus\langle e\rangle\cong\C\oplus\C$. To be consistent with later notation, we set $\rH=\rT$ for $\rG^\R=\rPSL_2\R$. The adjoint action of $\rT$ on $\fm$ is given by 
\begin{equation}\label{eq:Taction(elf)-PSL2}
\lambda\cdot(f,e)=(\lambda^{-1}f,\lambda e),
\end{equation} 
where $\lambda\in\rT$.
\item For $\rG^\R=\rSL_2\R$, we have $\rH\cong\C^*$, $\fm=\langle f\rangle\oplus\langle e\rangle$ and the action of $\rH$ is $\lambda\cdot(f,e)=(\lambda^{-2}f,\lambda^{2}e)$.
\end{itemize} 
\end{example}

\begin{definition}
  \label{def:uniformizing-HB}
The \emph{uniformizing Higgs bundle} for the compact Riemann surface $X$ is the $\rPSL_2\R$-Higgs bundle $(\cE_\rT,f)$, where $\cE_\rT$ is the frame bundle of the canonical bundle $K\to X$ and $f\in H^0(\cE_\rT[\langle f\rangle]\otimes K)\cong H^0(\cO)$ is a constant nonzero section.
\end{definition}

\begin{remark}
Since $\deg(K)=2g-2$ is even, the uniformizing $\rPSL_2\R$-Higgs bundle $(\cE_\rT,f)$ lifts to an $\rSL_2\R$-Higgs bundle $(\cE_{\rT'},f)$, where $\cE_{\rT'}$ is the frame bundle of one of the $2^{2g}$ square roots $K^\frac{1}{2}$ of the canonical bundle. We will call such a Higgs bundle a \emph{lift of the uniformizing Higgs bundle} of $X$. Using the standard representation of $\rSL_2\C$ on $\C^2$, an $\rSL_2\C$-Higgs bundle is a holomorphic rank $2$ bundle $V$ with trivial determinant and a holomorphic bundle map $\Phi:V\to V\otimes K$. For a lift of the uniformizing Higgs bundle, we have 
\[(V,\Phi)=\left(K^\frac{1}{2}\oplus K^{-\frac{1}{2}}, \smtrx{0&0\\1&0}:K^\frac{1}{2}\oplus K^{-\frac{1}{2}}\to K^\frac{3}{2}\oplus K^\frac{1}{2}\right).\]
\end{remark}

Given two Lie groups $\rH_1,\rH_2$ and holomorphic principal $\rH_1,\rH_2$-bundles $\cE_{\rH_1}$, $\cE_{\rH_2}$ respectively, the fiber product $\cE_{\rH_1}\times_X\cE_{\rH_2}$ is a holomorphic principal $(\rH_1\times\rH_2)$-bundle. When $\rH_1,\rH_2\subset\rH$ are commuting subgroups, the multiplication map $m:\rH_1\times\rH_2\to\rH$ is a group homomorphism and $(\cE_{\rH_1}\times_X\cE_{\rH_2})[\rH]$ is a holomorphic principal $\rH$-bundle. This is analogous to twisting a vector bundle by a line bundle. 
We will use the notation
\begin{equation}
  \label{eq: star notation}
  (\cE_{\rH_1}\star\cE_{\rH_2})[\rH]=(\cE_{\rH_1}\times_X\cE_{\rH_2})[\rH].
\end{equation}
\subsection{The Cayley map}\label{subsec: Cayley map}
We first describe the global Slodowy slice construction of \cite{ColSandGlobalSlodowy} for an arbitrary even nilpotent $e\in\fg$. When $e\in\fg$ is a magical nilpotent (recall from Corollary~\ref{cor mag implies even and injmap h->m} that every magical nilpotent is even) this leads to $\rG^\R$-Higgs bundles, where $\rG^\R$ is the canonical real form associated to the corresponding magical $\fsl_2$-triple.

Let $e\in\fg$ be an even nilpotent, $\{f,h,e\}\subset\fg$ be an associated $\fsl_2$-triple and $\rG$ be a connected Lie group with Lie algebra $\fg$. Let $\rS\subset\rG$ be the connected subgroup with Lie algebra the $\fsl_2\C$-subalgebra $\fs=\langle f,h,e\rangle$ and $\rC\subset\rG$ be the centralizer of $\{f,h,e\}$.\label{p:def-S-group}
When $\rS\cong\rPSL_2\C$ let $(\cE_\rT,f)$ be the uniformizing Higgs bundle of $X$, and when $\rS\cong\rSL_2\C$ let $(\cE_\rT,f)$ be a lift of the uniformizing Higgs bundle of $X$ to $\rSL_2\R$. The embedding $\rT\hookrightarrow\rS\hookrightarrow\rG$ defines a holomorphic $\rG$-bundle $\cE_\rG=\cE_\rT[\rG]$ by extension of structure group. 

Given a holomorphic $\rC$-bundle $\cE_\rC\to X$, consider the holomorphic $\rG$-bundle 
\[\cE_\rG=(\cE_\rC\star\cE_\rT)[\rG]\]
with the notation \eqref{eq: star notation}.
Since $\rC$ and $\rT$ preserve the subspaces $\fg_j\cap W_{i}\subset\fg$ (in particular the highest weight subspaces $V_j$; cf.~\eqref{eq decomp highest weight spaces}) and also $\langle f\rangle\subset \fg$, the adjoint bundle $\cE_\rG[\fg]$ decomposes as
\[\cE_\rG[\fg]=(\cE_\rC\star\cE_\rT)[\fg]=\bigoplus_{j\in\Z}(\cE_\rC\star\cE_\rT)[\fg_j]\]
and $(\cE_\rC\star\cE_\rT)[V_{j}]\subset(\cE_\rC\star\cE_\rT)[\fg_j]$ and $(\cE_\rC\star\cE_\rT)[\langle f\rangle]\subset(\cE_\rC\star\cE_\rT)[\fg_{-2}]$ define holomorphic subbundles. Moreover, 
since $\rC$ acts trivially on $\langle f\rangle$,
\[(\cE_\rC\star\cE_\rT)[\langle f\rangle]\cong \cE_\rT[\langle f\rangle]\cong K^{-1},\]
by \eqref{eq:Taction(elf)-PSL2}. So, from a holomorphic $\rC$-bundle $\cE_\rC$ and from sections $\phi_j\in H^0((\cE_\rC\star\cE_\rT)[V_j]\otimes K)$, we define the $\rG$-Higgs bundle
\begin{equation}
    \label{eq add high}(\cE_\rG,\varphi)=((\cE_\rC\star\cE_\rT)[\rG],f+\phi_0+\phi_1+\cdots +\phi_N).  
\end{equation}

Recall that $Z_{2m_j}=W_{2m_j}\cap\fg_0$. We have that $\fg_0=W_0\oplus\bigoplus_{j=1}^MZ_{2m_j}$ and, since $e$ is even, $\ad_f^{m_j}:V_{2m_j}\to Z_{2m_j}$ is an isomorphism. Thus, viewing $f$ as a holomorphic section of $(\cE_\rC\star\cE_\rT)[\fg]\otimes K$, we have an isomorphism of holomorphic vector bundles
\[\ad_f^{m_j}:(\cE_\rC\star\cE_\rT)[V_{2m_j}]\otimes K\xrightarrow{\ \ \cong\ \ } \cE_\rC[Z_{2m_j}]\otimes K^{m_j+1}, \]
where we have used the fact that $\rT$ acts trivially on $Z_{2m_j}$ to identify $\cE_\rC[Z_{2m_j}]\otimes K^{m_j+1}$ with $(\cE_\rC\star\cE_\rT)[Z_{2m_j}]\otimes K^{m_j+1}$. 

Let now $\cB_e(\rG)$ denote the set of tuples $((\cE_\rC,\phi_0),\psi_{m_1},\ldots,\psi_{m_N})$, where $(\cE_\rC,\phi_0)$ is a holomorphic $\rC$-Higgs bundle and $\psi_{m_j}\in H^0(\cE_\rC[Z_{2m_j}]\otimes K^{m_j+1})$. By the above discussion, the Higgs bundles of the form \eqref{eq add high} can be described by the map 
\begin{equation}
    \label{eq Slodowy map}
\widehat\Psi_e:\xymatrix@R=0em{\cB_e(\rG)\ar[r]&\cH(\rG)\\ (\cE_\rC,\phi_0,\psi_{m_1},\ldots,\psi_{m_N})\ar@{|->}[r]&((\cE_\rC\star\cE_\rT)[\rG],f+\phi_0+\phi_{m_1}+\cdots +\phi_{m_N})}
\end{equation}
where $\phi_{m_j}\in H^0((\cE_\rC\star\cE_\rT)[V_{2m_j}]\otimes K)$ and $\psi_{m_j}=\ad_f^{m_j}(\phi_{m_j})$. We will refer to this map as the \emph{Slodowy map}; see also \cite{ColSandGlobalSlodowy}.

Note that the map $\widehat\Psi_e$ is equivariant for the action of the $\rC$-gauge group $\cG_\rC$. More precisely, if $g\in\cG_{\rC}$, then $g\star\Id_\rT \in\cG_\rG$ is a $\rG$-gauge transformation of $(\cE_\rC\star\cE_\rT)[\rG]$, and 
\[\widehat\Psi_e(g\cdot(\cE_\rC,\phi_0,\psi_{m_1},\ldots,\psi_{m_N}))=g\star\Id_\rT \cdot \widehat\Psi_e(\cE_\rC,\phi_0,\psi_{m_1},\ldots,\psi_{m_N}).\]

\begin{lemma}\label{lem: G^R Higgs for Slodowy slice}
Let $e\in\fg$ be a magical nilpotent and $\rG^\R\subset\rG$ be the canonical real form. Then the Higgs bundle $\widehat\Psi_e(\cE_\rC,\phi_0,\psi_{m_1},\ldots,\psi_{m_N})$ from \eqref{eq Slodowy map} is contained in $\cH(\rG^\R)$ if and only if $\phi_0=0$ and the bundle $\cE_\rC$ reduces to $\rC\cap \rH$.  
\end{lemma}
\begin{proof}
  Let $\fg=\fh\oplus\fm$ be the complexified Cartan decomposition of the real form $\fg^\R$, hence given by $\sigma_e$. By the definition of a magical nilpotent, $h\in\fh$, $\fc=W_0\subset\fh$, $V_{2m_j}\subset\fm$ and $f\in\fm$. Thus, 
  \[(\cE_\rC\star\cE_\rT)[\rG]\cong(\cE_\rC\star\cE_\rT)[\rH][\rG]\]
  if and only if $\cE_\rC\cong\cE_{\rC\cap\rH}[\rC]$, and $f+\phi_0+\phi_{m_1}+\cdots+\phi_{m_N}\in H^0((\cE_\rC\star\cE_\rT)[\fm]\otimes K)$ if and only if $\phi_0=0$.
\end{proof}

Given a magical $\fsl_2$-triple $\{f,h,e\}\subset\fg$, recall the subalgebra $\fg(e)\subset\fg$ from Proposition \ref{prop subalge g(e)} and the Cayley real form $\fg_\cC^\R=\R^{\rkfge}\oplus\tilfg^\R$ from Proposition \ref{prop cayley real form class}. The Cayley group is defined to be the real Lie group $\rG_\cC^\R=(\R^+)^{\rkfge}\times \tilrG^\R~,$ where $\tilrG^\R$ is the real Lie group with Lie algebra $\tilfg^\R$ and maximal compact $\rC\cap\rG^\R$ (see Definition \ref{def cayley group}).
Recall from Proposition \ref{prop: magical sl2 data}, that the $\fsl_2$-data of a magical $\fsl_2$-triple has at most one $m_j>0$ with $\dim(Z_{2m_j})>1$.

\begin{lemma}\label{lem identification of Slodowy domain with Cayley group}
    Let $\{f,h,e\}\subset\fg$ be a magical $\fsl_2$-triple with $\fsl_2$-data $\{m_j\}_{j=1}^M$ and let $\fg(e)\subset\fg$ be the subalgebra from Proposition \ref{prop subalge g(e)}. Then there is a natural identification
    \[\{x\in\cB_e(\rG)~|~\widehat\Psi_e(x)\in\cH(\rG^\R)\}\longleftrightarrow\cH_{K^{m_c+1}}(\tilrG^\R)\times \prod_{j=1}^{\rkfge}\cH_{K^{l_j+1}}(\R^+).\]
    Here $m_c$ is zero in Case (1) of Theorem \ref{thm: classification weighted dynkin} and is the unique positive $m_j$ with $\dim(Z_{2m_j})>1$ otherwise. The integers $\{l_j\}$ are the exponents of $\fg(e)$, which are
    \[\{l_j\}=\begin{dcases}
        \{m_j\}_{j=1}^M &\text{Cases (1), (2) and (3) with p-even of Theorem \ref{thm: classification weighted dynkin}}\\
        \{m_j\}_{j=1}^M\setminus\{p-1\}&\text{Case (3) p-odd of Theorem \ref{thm: classification weighted dynkin}}\\
       \{m_j\}_{j=1}^M\setminus\{3\}&\text{Case (4) of Theorem \ref{thm: classification weighted dynkin}}.
    \end{dcases}\]
\end{lemma}
\begin{remark}\label{rem R+ differentials cayley}
Recall that $\cH_{L}(\R^+)\cong H^0(L)$, so 
\[\cH_{K^{m_c+1}}(\tilrG^\R)\times \prod_{j=1}^{\rkfge}\cH_{K^{l_j+1}}(\R^+)\cong \cH_{K^{m_c+1}}(\tilrG^\R)\times \prod_{j=1}^{\rkfge}H^0(K^{l_j+1}).\]
 Let $\rZ(\rG^\R)$ be the center of $\rG^\R$. In Case (1) of Theorem \ref{thm: classification weighted dynkin}, $\cH_{K^{m_c+1}}(\tilrG^\R)$ is the finite set of $\rZ(\rG^\R)$-bundles on $X$ so the value of $m_c$ is unimportant.
\end{remark}
\begin{proof}
    By Lemma \ref{lemma:Ccentralizesg(e)}, $\rC$ acts trivially on $\fg(e)\cap\fg_0$.  When $n_{2m_j}=1$, we have $Z_{2m_j}\subset\fg(e)$ and thus $\psi_{m_j}\in H^0(\cE_\rC(Z_{2m_j})\otimes K^{m_j+1})= H^0( K^{m_j+1})$. This proves Case (1).

From Proposition \ref{prop sl2 decomp of g(e)}, we see that for Case (3) with $p$ odd and Case (4), we have $\fg(e)\cap Z_{2m_c}=\{0\}$ and $\tilfg=\fc\oplus Z_{2m_c}$. Thus, $(\cE_{\rC},\psi_{m_c})$ is a $K^{m_c}$-twisted $\tilrG^\R$-Higgs bundle whenever $\cE_\rC$ reduces to $\rC\cap\rH$. Thus, for Case (3) with $p$-odd and Case (4), the result follows.

For Case (2) and Case (3) with $p$-even, we have
$Z_{2m_c}\cap\fg(e)\cong\C$, by Propositions \ref{prop subalge g(e)}
and \ref{prop sl2 decomp of g(e)}. Hence $Z_{2m_c}$ decomposes
$\rC$-invariantly as $Z_{2m_c}=\C\oplus \tilfm$, where the $\C$-factor
is $\fg(e)\cap Z_{2m_c}$ and $\tilfg=\fc\oplus\tilfm$ is the Cartan
decomposition giving the real form $\tilfg^\R$.\label{p:cartan-tilfg}
Hence 
    \[\cE_{\rC}[Z_{2m_c}]\otimes K^{m_c+1}\cong K^{m_c+1}\oplus \cE_{\rC}[\tilfm]\otimes K^{m_c+1.}\]
    Thus, $(\cE_{\rC},\psi_{m_c})=(\cE_{\rC},q_{m_c+1}\oplus\tilde\psi_{m_c})$, where $q_{m_c+1}\in H^0(K^{m_c+1})$ and $(\cE_{\rC},\tilde\psi_{m_c})$ is a $\tilrG^\R$-Higgs bundle whenever $\cE_\rC$ reduces to $\rC\cap\rH$.
 \end{proof}

To summarize, from a magical $\fsl_2$-triple $\{f,h,e\}\subset\fg$, the Slodowy map \eqref{eq Slodowy map} defines a map
\[\widehat\Psi_e:\cH_{K^{m_c+1}}(\tilrG^\R)\times \prod_{j=1}^{\rkfge}\cH_{K^{l_j+1}}(\R^+)\longrightarrow\cH(\rG^\R)\]
given by
\begin{equation}\label{eq Cayley map config}
\widehat\Psi_e((\cE_{\rC},\tilde\psi_{m_c}),q_1,\ldots,q_{\rkfge})=\bigg(\cE_{\rC}\star\cE_{\rT}[\rH],f+\tilde\phi_{m_c}+\sum_{j=1}^{\rkfge}q_{j}\bigg),
 \end{equation}
where $\rG^\R$ is the canonical real form of $e$; here
$\tilde\phi_{m_c}=\ad_f^{-m_c}(\tilde\psi_{m_c})$ and
$q_j\in H^0(K^{l_j+1})$. Note that, by a slight
abuse of notation, we have left the isomorphism of line bundles
$\ad_f^{-l_j}$ implicit and denoted the image of $q_j$ by the same symbol.

We will refer to the map \eqref{eq Cayley map config} as the \emph{Cayley map} since it generalizes the Cayley correspondence of \cite{BGRmaximalToledo} which concerns Case (2) of Theorem \ref{thm: classification weighted dynkin}. In the subsequent sections we will show that the Cayley map  actually preserves the polystability conditions, hence descends to a map on moduli spaces, which will be injective, with open and closed image.

\begin{remark}\label{rem cayley map for other L}
 Note that everything we just described also holds when the line bundle $K$ is replaced by another twisting line bundle $L\to X$. So there is a similarly defined Cayley map, for the $L$-twisted version, one takes $\cE_\rT$ to be the holomorphic frame bundle of $L$ when $\rS\cong\rPSL_2\C$ and $\cE_\rT$ the holomorphic frame bundle of a square root of $L$ when $\rS\cong\rSL_2\C$. In particular, when $\rS\cong\rSL_2\C$ the degree of $L$ must be even. 
\end{remark}

\subsection{The Cayley map is injective on gauge orbits}\label{sec injectivity}
In this section we prove the Cayley map is injective on gauge orbits. We will use the following lemma. 
\begin{lemma}\label{lem in parabolic}
    Let $\{f,h,e\}\subset\fg$ be an $\fsl_2$-triple and $\fg=\bigoplus_{j\in\Z}\fg_{j}$ be the associated $\Z$-grading. Let $\rP\subset\rG$ be the parabolic subgroup with Lie algebra $\fp=\bigoplus_{j\geq 0}\fg_j$ and let $x,x'\in V=\ker(\ad_e)$. If an element $g\in\rP$ satisfies $\Ad_g(f+x)=f+x'$, then $g\in\rC$, with $\rC\subset\rG$ the centralizer of  $\{f,h,e\}$.
\end{lemma}
\begin{proof}
Since $x\in\fp$, we have $\Ad_g(x)\in\fp$. Thus, $\Ad_g(f+x)=f+x'$ implies $\Ad_g(f)=f$. The intersection of the centralizer of $f$ with $\rP$ is $\rC$. So $g\in\rC$.
\end{proof}
\begin{proposition}
    \label{prop injective Cayley} Let $\{f,h,e\}\subset\fg$ be a magical $\fsl_2$-triple and  
    \[\widehat\Psi_e:\xymatrix{\cH_{K^{m_c+1}}(\tilrG^\R)\times \prod_{j=1}^{\rkfge}\cH_{K^{l_j+1}}(\R^+)\ar[r]&\cH(\rG^\R),}\]
    be the Cayley map from \eqref{eq Cayley map config}. Then two points 
    \[\xymatrix@=.5em{\widehat\Psi_e((\cE_{\rC\cap\rH},\tilde\psi_{m_c}),q_1,\ldots,q_{\rkfge})&\text{and}&\widehat\Psi_e((\cE_{\rC\cap\rH}',\tilde\psi_{m_c}'),q_1',\ldots,q_{\rkfge}')}\]
    are in the same $\rH$-gauge orbit if and only if $(\cE_{\rC\cap\rH},\tilde\psi_{m_c})$ and $(\cE_{\rC\cap\rH}',\tilde\psi_{m_c}')$ are in the same $\rC\cap\rH$-gauge orbit and moreover $q_{j}=q_{j}'$ for all $j$.
\end{proposition}
\begin{proof}
We will prove Proposition \ref{prop injective Cayley} for each case of Theorem \ref{thm: classification weighted dynkin}. Note that it suffices to prove the result for the adjoint group $\rG_{\Ad}$. Indeed, consider a general $\rG$ and let $\pi:\rG\to\rG_{\Ad}$ be the covering. An $\rH$-gauge transformation $g:\cE_\rC\star\cE_{\rT}[\rH]\to\cE_\rC'\star\cE_{\rT}[\rH]$ induces a gauge transformation between the associated bundles for the adjoint group, and if the induced gauge transformation is valued in $\pi(\rC\cap\rH)$ then $g$ must be valued in $\rC\cap\rH$. The $\rC\cap\rH$-gauge group acts trivially on the differentials $q_{j}$, so if $g$ is valued in $\rC\cap\rH$, $q_{j}=q_{j}'$ for all $j$.

Case (1) was proven in \cite{liegroupsteichmuller} using the Hitchin section and moduli spaces. 
Alternatively, suppose $g:\cE_{\rC\cap\rH}\star\cE_\rT[\rH]\to\cE_\rC'\star\cE_\rT[\rH]$ is a holomorphic gauge transformation such that 
\[\Ad_g\bigg(f+\sum_{j=1}^{\rk(\fg)}q_{j}\bigg)=f+\sum_{j=1}^{\rk(\fg)}q_{j}'.\]
The Lie algebra bundle decomposes as 
$\cE_{\rC\cap\rH}\star\cE_\rT[\fg]\otimes K \cong \bigoplus\cE_{\rC\cap\rH}\star\cE_\rT[\fg_j\cap W_{2m_i}]\otimes K$ with each summand $\cE_{\rC\cap\rH}\star\cE_\rT[\fg_j\cap W_{2m_i}]\otimes K\cong K^{j+1}$.
Since $g$ is holomorphic, we have 
\[\Ad_g\bigg(\bigoplus_{j\geq 0}\cE_{\rC\cap\rH}\star\cE_\rT[\fg_j]\otimes K\bigg)\subset\bigoplus_{j\geq 0}\cE_{\rC\cap\rH}\star\cE_\rT[\fg_j]\otimes K.\]
Hence $g$ is valued in the intersection of $\rH$ with the parabolic subgroup associated to $\bigoplus_{j>0}\fg_j$. Thus, $g$ is valued in $\rC\cap\rH$ by Lemma \ref{lem in parabolic}. 

For Case (2) of Theorem \ref{thm: classification weighted dynkin}, the $\Z$-grading is $\fg=\fg_{-2}\oplus\fg_0\oplus\fg_2$ with $\fh=\fg_0$. Hence, any gauge transformation $g:\cE_\rC\star\cE_{\rT}[\rH]\to\cE_{\rC}'\star\cE_{\rT}[\rH]$ is valued in the intersection of $\rH$ with parabolic subgroup associated to $\fg_0\oplus\fg_2$. By Lemma \ref{lem in parabolic}, $g$ is valued in $\rC\cap\rH$.

For Case (3), Proposition \ref{prop injective Cayley} was proven in \cite[Lemma 4.6]{so(pq)BCGGO} when $\rG=\rSO_N\C$, i.e., for $\rG^\R\cong\rSO_{p,q}$. 
As a result, we focus on $\rG=\rPSO_N\C$. For $N$-odd, $\rSO_N\C=\rP\rSO_N\C$ and we are done. 
For $N$-even the centralizer $\rC$ of the magical $\fsl_2$-triple is $\rO_{N-2p+1}\C$ for $\rG=\rSO_N\C$ and $\rO_{N-2p+1}\C/\pm\Id$ for $\rG=\rPSO_N\C$ (see \cite[Theorem 6.1.3]{CollMcGovNilpotents}). But $N$ even implies $\rO_{N-2p+1}\C/\pm\Id\cong\rSO_{N-2p+1}\C$. Since every $\rSO_{N-2p+1}\C$-bundle lifts to a $\rO_{N-2p+1}\C$-bundle, every $\rPSO_N\C$-Higgs bundle in the image of $\hat\Psi_e$ lifts to an $\rSO_N\C$-Higgs bundle in the image of $\hat\Psi_e$.

For Case (4), we use holomorphicity and Proposition \ref{prop H' stabilizer in parabolic} to apply Lemma \ref{lem in parabolic}. Recall that the space $\fm$ decomposes as in Lemma \ref{lem h' sl2 invariant decomp of m}. Write the Higgs field as 
\begin{equation}
 \label{eq phi except decomp}   f+q_2+\phi_3+q_6=\begin{pmatrix}
    q_6&\phi_3&q_2^b&\tilde f\\\tilde q_2& f_b&0&0
\end{pmatrix},
\end{equation}
where the rows are sections of $\cE_{\rC}\star\cE_\rT[\fg_{10}\oplus\fg_6\oplus\fg_{2}^b\oplus\fg_{-\tilde\alpha}]\otimes K$ and $\cE_{\rC}\star\cE_\rT[\fg_{\tilde\alpha}\oplus\fg_{-2}^b\oplus\fg_{-6}\oplus\fg_{10}]\otimes K$, respectively. 
Recall also that $\fg=\fsl_2\C\oplus\fh'=\fg_{-8}\oplus\fg_{-4}\oplus\fg_0\oplus\fg_4\oplus\fg_8$. 

Consider a holomorphic gauge transformation $g:\cE_{\rC}\star\cE_\rT[\rH]\to\cE_{\rC}'\star\cE_\rT[\rH]$. We have $\cE_{\rC}\star\cE_{\rT}[\fg_{-8}]\cong K^{-4}$, thus holomorphicity implies 
\[\Ad_g(\cE_\rC\star\cE_{\rT}[\fg_{-4}\oplus\fg_0\oplus\fg_4\oplus\fg_8])\subset\cE_\rC'\star\cE_{\rT}[\fg_{-4}\oplus\fg_0\oplus\fg_4\oplus\fg_8].\]
Hence $g$ is valued in the parabolic of $\rP\subset\rH$ with Lie algebra $\fg_{-4}\oplus\fg_0\oplus\fg_4\oplus\fg_8$. The action of $\rP$ on $\fm$ preserves the top row of \eqref{eq phi except decomp}. If it preserves the image of $\hat\Psi_e$, we have the gauge transformation
\[\Ad_g\begin{pmatrix}
    q_6&\phi_3&q_2^b&\tilde f\\\tilde q_2& f_b&0&0
\end{pmatrix}=\begin{pmatrix}
    q_6'&\phi_3'&(q_2^b)'&\tilde f\\\tilde q_2'& f_b&0&0
\end{pmatrix}.\]
By Proposition \ref{prop H' stabilizer in parabolic}, the gauge transformation $g$ is valued in the parabolic of $\rH$ with Lie algebra $\fg_0\oplus\fg_4\oplus\fg_8$. Thus, Lemma \ref{lem in parabolic} implies $g$ is valued in $\rC\cap \rH$.
\end{proof}
We have the following immediate corollary.
\begin{corollary}\label{Cor same automorphism groups}
 Let $((\cE_{\rC\cap\rH},\tilde\psi_{m_c}),q_1,\ldots,q_{\rkfge})$ be in the domain of the Cayley map \eqref{eq Cayley map config}. Then the automorphism group of  $((\cE_{\rC\cap\rH},\tilde\psi_{m_c}),q_1,\ldots,q_{\rkfge})$ is equal to the automorphism group of $\widehat\Psi_e((\cE_{\rC\cap\rH},\tilde\psi_{m_c}),q_1,\ldots,q_{\rkfge})$.
\end{corollary}
\section{Moduli spaces of Higgs bundles}\label{sec:Moduli spaces}
\subsection{Stability conditions and moduli spaces}
In this section we introduce the moduli space of $L$-twisted Higgs bundles, recall some of its features and discuss related objects. For this section we fix a compact Riemann surface $X$ with genus $g\geq 2$.
%
We start by recalling the notions of (semi,poly)stability and the moduli spaces for Higgs pairs. See \cite{HiggsPairsSTABILITY} for more details.
Let $\rG$ be a complex reductive Lie group with Lie algebra $\fg$, equipped with a non-degenerate $\rG$-invariant $\C$-bilinear pairing $\langle\cdot,\cdot\rangle$. 
Let $\rK^\R\subset \rG$ be a maximal compact subgroup with Lie algebra $\fk^\R$.

An element $s\in i\fk^\R$ defines a parabolic subgroup $\rP_s$ and a Levi subgroup $\rL_s$ of $\rG$ by taking
\begin{align*}
\rP_s & =\{g\in \rG \mid e^{t s}ge^{-t s} \text{ is bounded as
$t\to\infty$} \}\subset \rG,\\
\rL_s & =\{g\in \rG\mid e^{t s}ge^{-t s}=g \text{ for all }t\}\subset \rP_s.
\end{align*}
Also, given a holomorphic representation $\rho:\rG\to\rGL(V)$, we have the subspaces 
\begin{equation}\label{s-subspaces}
\begin{split}
V_s& =\{v\in V \mid \rho(e^{t s})v 
\text{ is bounded as }t\to\infty\},\\
V_s^0& =\{v\in V \mid \rho(e^{t s})v=v \text{ for all $t$}
\}\subset V_s.
\end{split}
\end{equation}
Here, $V_s\subset V$ is $\rP_s$-invariant and $V_s^0\subset V_s$ is $\rL_s$-invariant. For the adjoint representation $\Ad:\rG\to\rGL(\fg)$, we have that $\fg_s^0\subset\fg_s$ are the Lie algebras $\fl_s\subset\fp_s$ of $\rL_s\subset\rP_s$.
Since $\langle s, [\fp_s,\fp_s]\rangle=0$, the element $s\in i\fk^\R$ defines the character of $\fp_s$
\[\chi_s := \langle s, - \rangle:\fp_s\longrightarrow\C.\]

Given a holomorphic $\rG$-bundle $\cE_\rG$, we  define the degree of a structure group reduction from $\rG$ to $\rP_s$ using Chern--Weil theory and the character $\chi_s$. Let $\rL_s^\R=\rK^\R\cap \rL_s$ be a maximal compact subgroup of $\rL_s$; the inclusion $\rL_s^\R\subset \rL_s$ is a homotopy equivalence. 
Now suppose $\cE_{\rP_s}\subset \cE_{\rG}$ is a reduction of $ \cE_\rG$ to $\rP_s$. There is a further reduction
$\cE_{\rL_s^\R}\subset \cE_{\rP_s}$ which is unique up to homotopy. Consider a connection $A$ on $\cE_{\rL_s^\R}$ with curvature $F_A\in \Omega^2(\cE_{\rL_s^\R}[\fl_s^\R])$. Then $\chi_s(F_A)$ is a 2-form on $X$ with values in $i\R$. Define the degree of the reduction $\cE_{P_s}\subset\cE_{\rG}$ to be the real number
\[\deg(\cE_{P_s})=\frac{i}{2\pi}\int_X\chi_s(F_A).\]

Let $d\rho:\fg\to\fgl(V)$ be the differential of $\rho$ and $\fz^\R$ be the center of $\fk^\R$. Consider the orthogonal decomposition
$\fz^\R= \ker(d\rho_{|\fz^\R})\oplus  \ker(d\rho_{|\fz^\R})^{\perp}$, and  define
\[ \fk^\R_{\rho} = \fk^\R_{ss} + \ker(d\rho_{|\fz^\R})^{\perp},\]
where $\fk^\R_{ss}$ is the semisimple part of  $\fk^\R$.
Thus $\fk^\R = \fk^\R_{\rho} + \ker(d\rho_{|\fz^\R})$.
We are now ready to define $\alpha$-stability notions, for $\alpha\in i\fz^\R$.

\begin{definition}\label{def:L-twisted-pairs-stability}
Let $\alpha\in i\fz^\R$. An $L$-twisted $(\rG,V)$-Higgs pair $(\cE_\rG,\varphi)$ is:
\begin{itemize}
  \item \emph{$\alpha$-semistable} if for any $s\in i\fk^\R$ 
   and any holomorphic reduction $\cE_{\rP_s}\subset\cE_{\rG}$ such that
    $\varphi\in H^0(\cE_{\rP_s}[V_s]\otimes L)$, we have $\deg(\cE_{\rP_s})\geq \langle\alpha,s\rangle$.

\item \emph{$\alpha$-stable} if for any $s\in i\fk^\R_\rho$ 
   and any holomorphic reduction $\cE_{\rP_s}\subset\cE_{\rG}$ such that
    $\varphi\in H^0(\cE_{\rP_s}[V_s]\otimes L)$, we have $\deg(\cE_{\rP_s})> \langle\alpha,s\rangle$.
 
\item \emph{$\alpha$-polystable} if it is $\alpha$-semistable and whenever $s\in i\fk^\R$ and $\cE_{\rP_s}\subset\cE_\rG$ is a holomorphic reduction with $\deg (\cE_{\rP_s})=\langle\alpha,s\rangle$,
 there is a further holomorphic reduction $\cE_{\rL_s}\subset\cE_{\rP_s}$ such that $\varphi\in H^0(\cE_{\rL_s}[V_s^0]\otimes L)$.
\end{itemize}
\end{definition}

\begin{remark}\label{rem 0-stability}
In this paper, the case $\alpha\neq 0$ will only appear in very specific situations, therefore we will refer to $0$-(semi,poly)stability simply as (semi,poly)stability. It is clear that the (semi,poly)stability of a Higgs pair is preserved by the action of gauge group and the $\C^*$-action from \eqref{eq C* action def}.
\end{remark}
\begin{remark}\label{rem subbundles}
    Consider an $L$-twisted $\rG$-Higgs bundle $(\cE_\rG,\varphi)$ for a semisimple Lie group $\rG$. Using the adjoint representation, we can form the Higgs vector bundle $(\cE_\rG[\fg],\ad_\varphi)$. In this case, $0$-polystability of $(\cE_\rG,\varphi)$ is equivalent to the polystability criterion involving degrees of invariant subbundles. Namely, $(\cE_\rG,\varphi)$ is 0-polystable if and only if for any holomorphic subbundle $\cV\subset\cE_\rG[\fg]$ with $\ad_\varphi(\cV)\subset\cV\otimes L$, we have $\deg(\cV)\leq 0$ and furthermore, if $\deg(\cV)=0$, then $(\cE_\rG[\fg],\ad_\varphi)$ splits as a direct sum of stable Higgs vector bundles of degree $0$. This follows from the Hitchin--Kobayashi correspondence (see \S\ref{sec hitchin kob}).
\end{remark}
\begin{remark}\label{rem stability and coverings}
 Let $\rG_1\to\rG_2$ be a covering and $(\cE_{\rG_2},\varphi)$ be a $\rG_2$-Higgs bundle which lifts to an $\rG_1$-Higgs bundle $(\cE_{\rG_1},\varphi)$, i.e.~$\cE_{\rG_1}(\rG_2)=\cE_{\rG_2}$. Then $(\cE_{\rG_2},\varphi)$ is polystable if and only if $(\cE_{\rG_1},\varphi)$ is polystable. Indeed, any holomorphic parabolic reduction $\cE_{\rP_s}\subset\cE_{\rG_1}$ induces a holomorphic parabolic reduction $\cE_{\rP_s}(\rG_2)\subset\cE_{\rG_2}$ and any holomorphic parabolic reduction $\cE_{\rP_s}\subset\cE_{\rG_2}$ lifts to a reduction $\cE_{\rP_s'}\subset\cE_{\rG_1}$. 
\end{remark}

The following result will be useful. For a proof, see \cite[\S2.10]{HiggsPairsSTABILITY}.
\begin{proposition}\label{prop JordanHolder}
Suppose $(\cE_\rG,\varphi)$ is a strictly polystable $L$-twisted $(\rG,V)$-Higgs pair. Then there exists an $s\in i\fk^\R$, a holomorphic reduction $\cE_{\rL_s}\subset\cE_\rG$ with $\deg(\cE_{\rL_s})=0$ and $\varphi\in H^0(\cE_{\rL_s}(V_s^0)\otimes L)$ such that $(E_{\rL_s},\varphi)$ is a stable as an $L$-twisted $(\rL_s,V_s^0)$-Higgs pair. 
\end{proposition}

We will only need to consider the moduli space for $L$-twisted $\G^\R$-Higgs bundles over $X$, where $\rG^\R$ is a real form of $\rG$. Denote it by $\cM_L(\rG^\R)$. We define it as the space of gauge orbits of polystable $L$-twisted $\rG^\R$-Higgs bundles
\[\cM_L(\rG^\R)=\cH_\rL^{ps}(\rG^\R)/\cG_\rH~,\]
where $\cH_\rL^{ps}(\rG^\R)\subset\cH_\rL(\rG^\R)$ is the subset of polystable $L$-twisted $\rG^\R$-Higgs bundles.

In order to endow $\cM_L(\rG^\R)$ with a topology, suitable Sobolev completions must be used in standard fashion; see \cite{fanHiggsmodulianalytic}, where a detailed adaptation to Higgs bundles is studied in the case $\G=\rGL_n\C$. 
Then the orbits of the $\cG_{\rH}$-action on $\cH_L(\rG^\R)^{ps}$ are closed in the space of semistable $\rG^\R$-Higgs bundles, thus the moduli space $\cM_L(\rG^\R)$ becomes a Hausdorff topological space. If $\cH^s_L(\rG^\R)\subset \cH^{ps}_L(\rG^\R)$ denotes the subset of stable Higgs bundles, then $\cH^s_L(\rG^\R)$ is open in $\cH^{ps}_L(\rG^\R)$. The stable objects thus define an open subset of  $\cM_L(\rG^\R)$. 

\begin{remark} 
A GIT construction of $\cM_L(\rG^\R)$ (actually in the more general setting of Higgs pairs) may be found in \cite{schmitt_2005}, from which is clear that $\cM_L(\rG^\R)$ parameterizes $S$-equivalence classes of semistable $L$-twisted $\rG^\R$-Higgs bundles. This construction generalizes the construction of the moduli space of $\rG^\R$-Higgs bundles by Ramanathan \cite{ramanathan_1975} when $\rG^\R$ is compact and Simpson \cite{SimpsonModuli1,SimpsonModuli2} when $\rG^\R$ is complex reductive (see also Nitsure \cite{NitsureHiggs} for $\rG^\R=\rGL_n\C$).
\end{remark}

\subsection{Local structure of the moduli spaces}\label{sec: local struct}
We now recall some deformation theory for Higgs bundles, for more details see \cite{biswas-ramanan} and \cite{HiggsPairsSTABILITY}. Fix a holomorphic line bundle $L$ on $X$ and let $(\cE_\rH,\varphi)$ be an $L$-twisted $\G^\R$-Higgs bundle. The double complex of sheaves 
\begin{equation}\label{eq complex of sheaves}
  C^\bullet(\cE_\rH,\varphi):\xymatrix{\cE_\rH[\fh]\ar[r]^-{\ad_\varphi}&\cE_\rH[\fm]\otimes L}
\end{equation} 
governs infinitesimal deformations of $(\cE_{\rH},\varphi)$. Thus, when $(\cE_\rH,\varphi)$ is polystable, \eqref{eq complex of sheaves} encodes the local structure of the moduli space $\cM_L(\rG^\R)$ near the point defined by $(\cE_\rH,\varphi)$.
The complex \eqref{eq complex of sheaves} defines a long exact sequence in hypercohomology:
\begin{equation}\label{EQ deformation complex DEF}
\xymatrix@R=1em@C=1.3em{0\ar[r]&\HH^0(C^\bullet(\cE_\rH,\varphi))\ar[r]&H^0(\cE_\rH[\fh])\ar[r]^{ \ad_\varphi\ \ \ \ }&H^0(\cE_\rH[\fm]\otimes L)\ar[r]&\HH^1(C^\bullet(\cE_\rH,\varphi))\ar@{->}`r/3pt [d] `/10pt[l] `^d[llll] `^r/3pt[d][dlll]\\&H^1(\cE_\rH[\fh])\ar[r]^{\ad_\varphi \ \ }&H^1(\cE_\rH[\fm]\otimes L)\ar[r]&\HH^2(C^\bullet(\cE_\rH,\varphi))\ar[r]&0.}
\end{equation}
We have the following proposition; see \cite[Lemma 2.25 and Proposition 3.8]{HiggsPairsSTABILITY}.

\begin{proposition}\label{prop automorphism}
If the $L$-twisted $\rG^\R$-Higgs bundle $(\cE_\rH,\varphi)$ is polystable, then its automorphism group $\Aut(\cE_\rH,\varphi)$ is a complex reductive group which is identified with a closed subgroup of the automorphisms of the fiber $(\cE_\rH(x),\varphi(x))$ for any $x\in X$. 
The zeroth hypercohomology group $\HH^0(C^\bullet(\cE_\rH,\varphi))$ is the Lie algebra of $\Aut(\cE_\rH,\varphi)$. 
\end{proposition}

Note that the automorphism group $\Aut(\cE_\rH,\varphi)$ acts on $\HH^1(C^\bullet(\cE_\rH,\varphi))$.
Using standard slice methods of Kuranishi (see \cite[Chapter 7.3]{DiffGeomCompVectBun} for details for the moduli space of holomorphic bundles), a neighborhood of the isomorphism class of a polystable Higgs bundle $(\cE_\rH,\varphi)$ in $\cM_L(\rG^\R)$ is given by 
\begin{equation*}\label{EQ: local kuranishi}
  \kappa^{-1}(0)\sslash\Aut(\cE_\rH,\varphi)
\end{equation*}
where $\kappa:\HH^1(C^\bullet(\cE_\rH,\varphi))\to \HH^2(C^\bullet(\cE_\rH,\varphi))$ is the so called Kuranishi map. 
When $\HH^2(\cE_\rH,\varphi)=0$, a neighborhood of the isomorphism class of $(\cE_\rH,\varphi)$ in $\cM_L(\rG^\R)$ is isomorphic to 
\begin{equation}\label{eq localmodel nhgb}
\HH^1(C^\bullet(\cE_\rH,\varphi))\sslash \Aut(\cE_\rH,\varphi).
\end{equation}
We will use the following result in \S\ref{sec: open and closed} to prove that for the Higgs bundles considered there, the corresponding $\HH^2$ vanishes. Therefore, we have no need to recall the construction of the Kuranishi map. 

\begin{proposition}\label{prop H2 vanishes}
Let $\rG^\R\subset\rG$ be a real form of a complex semisimple Lie group $\rG$ and let $L$ be a holomorphic line bundle with $\deg(L)>2g-2$. Then for any polystable $L$-twisted $\rG^\R$-Higgs bundle $(\cE_\rH,\varphi)$ we have $\HH^2(C^\bullet(\cE_\rH,\varphi))=0$.
\end{proposition}
\begin{proof}
It suffices to prove the statement for the $L$-twisted $\rG$-Higgs bundle $(\cE_\rG,\varphi)=(\cE_\rH[\rG],\varphi)$ since there is an inclusion $\HH^2(C^\bullet(\cE_\rH,\varphi))\subset \HH^2(C^\bullet(\cE_\rG,\varphi))$. Since $(\cE_\rG,\varphi)$ is semistable, any subbundle $\cV\subset\cE_\rG[\fg]$ with $\ad_\varphi(\cV)=0$ satisfies $\deg(\cV)\leq 0$ by Remark \ref{rem subbundles}. 

Suppose $0\neq\HH^2(C^\bullet(\cE_\rG,\varphi))$. By Serre duality $\HH^2(C^\bullet(\cE_\rG,\varphi))$ is isomorphic to the dual of $\HH^0$ of the complex 
\[C^{\bullet}(\cE_\rG,\varphi)^*\otimes K: ~\cE_\rG[\fg]^*\otimes L^{-1}K\xrightarrow{\ad_{\varphi}^*\otimes \Id_K}\cE_\rG[\fg]^*\otimes K.\]
The Killing form on $\fg$ identifies $\cE_\rG[\fg]^*$ with $\cE_\rG[\fg]$ and $\ad_\varphi^*$ with $-\ad_\varphi$, so the complex 
\[\cE_\rG[\fg]\otimes L^{-1}K\xrightarrow{-\ad_\varphi\otimes \Id_K}\cE_\rG[\fg]\otimes K\]
has nonzero $\HH^0$. 
Thus, there is a nonzero $s\in H^0(\cE_\rG[\fg]\otimes L^{-1}K)$ such that $-\ad_\varphi(s)=0$. 
Let $M\subset\cE_\rG[\fg]\otimes L^{-1}K$ be the holomorphic line bundle generated by $s$ and note that $\deg(M)\geq0$. 
However, $M\otimes LK^{-1}\subset\cE_\rG[\fg]$ satisfies $\ad_\varphi(M\otimes LK^{-1})=0$. So, by semistability of $(\cE_\rG,\varphi)$, 
\[0\leq \deg(M)<\deg(M\otimes LK^{-1})\leq0.\]
This contradiction implies $\HH^2(C^\bullet(\cE_\rG,\varphi))=0$.
\end{proof}

\subsection{The Hitchin map}\label{sec: Hitch fib}

A fundamental ingredient in the theory of Higgs bundles is the {\em Hitchin map} \cite{IntSystemFibration}. We briefly explain this in the setting of $L$-twisted $\rG^\R$-Higgs bundles, for a simple real Lie group $\rG^\R$; see \cite{GarciaPradaPeonRamanan-HKRsection,IntSystemFibration,DonagiGaitsgory} for more details. 

Consider the GIT quotient map $\chi:\fm\to\fm\sslash\rH$. Note that $\chi$ is $\C^*$-equivariant with respect to the standard scaling action of $\C^*$ on $\fm$ and the action of $\C^*$ on $\fm\sslash\rH$ induced by the action of $\C^*$ on the graded ring $\C[\fm]^\rH$ of $\rH$-invariant polynomial functions on $\fm$. 
Namely, if $p\in\C[\fm]^\rH$ is homogeneous, the $\C^*$-action on $\fm\sslash\rH$ is determined by $t\cdot p=t^{\deg(p)}p$. 
Let $\cL$ be the holomorphic  $\C^*$-bundle associated to  $L$ and consider the rank $r$ vector bundle $\cL[\fm\sslash\rH]$ associated to $\cL$ via the $\C^*$-action on $\fm\sslash\rH$. 
The quotient map $\chi:\fm\to\fm\sslash\rH$ defines an $\rH$-invariant map $\fm\otimes L\to\cL[\fm\sslash\rH]$. By $\rH$-invariance this defines the Hitchin map:
\begin{equation}
    \label{eq:h1}h:\cM_L(\rG^\R)\to \cB_L(\rG^\R)= H^0(\cL[\fm\sslash\rH]),\ \ \ \ h(\cE_\rH,\varphi)=\chi(\varphi),
\end{equation}
where the space $\cB_L(\rG^\R)$ is called the Hitchin base. 

Choosing a homogeneous basis $(\chi_1,\ldots,\chi_r)$ of the ring $\C[\fm]^{\rH}$ defines an isomorphism of $\fm\sslash\rH\xrightarrow{\ \cong\ }\C^r$ given by $x\mapsto (\chi_1(x),\ldots,\chi_r(x))$. If the degree $\chi_j$ is $m_j'+1$ with $m'_1<\ldots<m'_r$, then the non-negative integers $m_i'$ are the exponents of the real Lie algebra $\lieg^\R$ (see for example \cite[Proposition 4.4]{GarciaPradaPeonRamanan-HKRsection}). By definition, they are the exponents of the complex Lie algebra obtained by complexifying the maximal split subalgebra of $\fg^\R$ (if $\fg^\R$ is complex, these are its exponents appearing in Case $(1)$ of Proposition \ref{prop: magical sl2 data}). 

Any choice of such homogeneous basis $(\chi_1,\ldots,\chi_r)$ defines an isomorphism 
\begin{equation}\label{eq:isomHitchinbase}
H^0(\cL[\fm\sslash\rH])\xrightarrow{\ \cong\ } \bigoplus_{j=1}^rH^0(L^{m_j'+1}),\ \ \ \ x\mapsto(\chi_1(x),\ldots,\chi_r(x)).
\end{equation}
Using this basis, we obtain the more familiar description of the Hitchin map
\[h: \cM_L(\rG^\R)\longrightarrow\bigoplus_{j=1}^{r}H^0(L^{m'_j+1}),\ \ \ \ h(\cE_\rH,\varphi)=(\chi_1(\varphi),\ldots,\chi_r(\varphi)).\]

For complex Lie groups and $L=K$, the Hitchin map $h$ has many special features, most notably it is an algebraic completely integrable system \cite{IntSystemFibration}. The property  we will use to prove the Cayley map is closed, and which is true for arbitrary groups and twistings, is that the Hitchin map \eqref{eq:h1} is proper. 
This follows from \cite[Theorem 6.1]{NitsureHiggs} for $\rGL_n\C$ and from the fact that the moduli space $\cM_L(\rG^\R)$ admits a finite (and hence proper) map to $\cM_L(\rGL_n\C)$ for some $n$ in such a way that the Hitchin map of $\cM_L(\rG^\R)$ is the restriction of the Hitchin map in $\cM_L(\rGL_n\C)$.

\begin{proposition}
    The Hitchin map $h:\cM_L(\rG^\R)\to\cB_L(\rG^\R)$ from \eqref{eq:h1} is proper. 
 \end{proposition}

\subsection{The Hitchin--Kobayashi correspondence}\label{sec hitchin kob} Finally, we consider an equation for a special metric associated to general $L$-twisted polystable $(\rG,V)$-Higgs pairs.  
Let $\rG$ be a complex reductive Lie group and fix a maximal compact subgroup $\rK^\R\subset\rG$ and a $\rK^\R$-invariant Hermitian inner-product on $V$ so that $d\rho:\fk^\R\to\fu(V)$ is the associated unitary representation.
Let $(\cE_\rG,\varphi)$ be an $L$-twisted $(\rG,V)$-Higgs pair. Fix a metric $h_L$ on the line bundle $L$. 
A metric on $\cE_\rG$ is by definition a reduction of structure group $h$ of $\cE_\rG$ to $\rK^\R$. Fix a metric $h$ and let $E_h\subset \cE_\rG$ be the associated $\rK^\R$-bundle. 
The Hermitian inner-product on $V$ and the metric $h_L$ on $L$ induce a Hermitian metric $h\otimes h_L$ on the bundle $E_h[V]\otimes L$. 
For $\varphi\in H^0(\cE_\rG[V]\otimes L)$ we can make sense of the following expression:
\begin{equation}
    \label{eq moment map}
\mu(\varphi)=d\rho^*\Big(-\frac{i}{2}\varphi\otimes \varphi^*_{h\otimes h_L}\Big),
\end{equation}
where we identify $i\varphi\otimes \varphi^{*_{h\otimes h_L}}$ with a section of $E_h(\fu(V))^*$. Hence $\mu(\varphi)$ defines a section of $E_h(\fk^\R)^*$. Using the nondegenerate pairing, we view $\mu(\varphi)$ as a section of $E_h(\fk^\R)$.
\begin{remark}\label{rem moment map}
    The action of $\rK^\R$ on $V$ is Hamiltonian and the expression for $\mu$ in \eqref{eq moment map} is a bundle version of the moment map for the action.
\end{remark}
 Now fix a K\"ahler form $\omega$ on $X$. Given a metric $h$ on $\cE_\rG$ there is a unique connection (the Chern connection) which is compatible with the holomorphic structure and the metric reduction. The \emph{Hitchin--Kobayashi correspondence} states the following.
 \begin{theorem}\label{hk-twisted-higgs}\cite[Theorem 2.24]{HiggsPairsSTABILITY}
An $L$-twisted $(\rG,V)$-Higgs pair $(\cE_\rG,\varphi)$ is $\alpha$-polystable if and only if there is a metric $h$ on $\cE_\rG$ 
solving
\begin{equation} \label{eq:Hitchin-Kobayashi}
F_h+\mu(\varphi)\omega=-i \alpha\omega,
\end{equation}
where $F_h\in\Omega^2(E_h[\fh^\R])$ denotes the curvature of the
Chern connection of $h$.
\end{theorem}



\begin{remark}\label{rmk:Kahler form and parameter}
The existence of solutions $h$ of \eqref{eq:Hitchin-Kobayashi} is independent of the choice of $h_L$.
Also, equation \eqref{eq:Hitchin-Kobayashi} implies that $\alpha$ depends on the fixed K\"ahler form $\omega$. If one chooses a different K\"ahler form $\omega'$, then a solution of \eqref{eq:Hitchin-Kobayashi} will still be a solution for the corresponding equation with $\omega'$, for a different $\alpha'$. This means that, to check for the existence of solutions of \eqref{eq:Hitchin-Kobayashi}, we can fix any $\omega$, and always work with it. 
\end{remark}

When now specialize to the case of Higgs bundles and Higgs pairs arising from $\Z/n\Z$-gradings of $\fg$.
Let $\tau:\fg\to\fg$ be the compact real-form associated to $\rK^\R\subset\rG$ and let $\langle\cdot,\cdot \rangle$ be a nondegenerate $\rG$-invariant complex bilinear form. The form 
$\langle x,-\tau(y)\rangle$ is a $\rK^\R$-invariant positive definite Hermitian inner product on $\fg$. In this case, the moment map $\mu:\fg\to(\fk^\R)^*\to\fk^\R$ is given by $\mu(x)=[x,-\tau(x)]$. 

Given a metric $h$ on $\cE_\rG$ and a metric $h_L$ on $L$, $\tau$ defines an involution $\tau_h:E_h(\fg)\otimes L \to E_h(\fg)\otimes L$.
Thus, for $L$-twisted $\rG$-Higgs bundles, equation \eqref{eq:Hitchin-Kobayashi} is 
\[ F_h+[\varphi,-\tau_{h}(\varphi)]\omega=-i\alpha\omega.\]
When $L=K$ we can view the Higgs field as a $(1,0)$-form valued in $E_h(\fg)$. In this case, we can use $\tau$ and conjugation on $1$-forms to define the involution $\tau_h:\Omega^{1,0}(E_h(\fg))\to\Omega^{0,1}(E_h(\fg))$, and solving \eqref{eq:Hitchin-Kobayashi} is equivalent to solving 
\begin{equation}
    \label{eq Hitchin eq} F_h+[\varphi,-\tau_h(\varphi)]=-i\alpha.
\end{equation}

\begin{remark}
     When $\alpha=0$, equation \eqref{eq Hitchin eq} is usually referred to as the Hitchin equations or the self-duality equations. In this case, the 
Hitchin--Kobayashi correspondence was proven by Hitchin for $\rG=\rSL_2\C$ \cite{selfduality} and by Simpson in general \cite{SimpsonVHS}.
 \end{remark} 
 \begin{remark}
     The uniformizing $\rPSL_2\R$-Higgs bundle $(\cE_\rT,\varphi)$ from Example \ref{ex uniformizing Higgs} (and any lift of it to $\rSL_2\R$) is $0$-stable. Since $\cE_\rT$ is the frame bundle of $K^{-1}$, any metric on $\cE_\rT$ defines a metric on the surface; the metric solving \eqref{eq Hitchin eq} has constant curvature \cite{selfduality}.
 \end{remark}

Finally, suppose $\hat\rG\subset\rG$ is a $\tau$-invariant subgroup with maximal compact subgroup $\hat\rK^\R=\hat\rG\cap\rK^\R$ and $V\subset\fg$ is a $\hat\rG$-invariant vector subspace of $\fg$ such that $\langle \cdot, \cdot\rangle|_V$ is nondegenerate. 
In this case, the moment map equations for the action of $\hat\rK^\R$ on $V$ is given by orthogonally projecting $[x,-\tau(x)]$ onto the Lie algebra $\hat\fk^\R\subset\fk^\R$. 
For example, the quiver bundle equations of \cite{KHCquiversvortices} are an example of this. 
An important special case of this occurs when the orthogonal projection $\fk^\R\to\hat\fk^\R$ does not loose any information, i.e., when $[V,-\tau(V)]\subset\hat\fk^\R$. In this case, when $\alpha=0$ a solution to $(\hat\rG,V)$-Higgs pair equations  also solves the $\rG$-Higgs bundle equations. Thus, if $(\cE_{\hat\rG},\varphi)$ is an $0$-polystable $L$-twisted $(\hat\rG,V)$-Higgs pair, then the associated $\rG$-Higgs bundle $(\cE_{\hat\rG}(\rG),\varphi)$ obtained by extending the structure group is polystable as a $\rG$-Higgs bundle.

For example, consider a $\Z/n\Z$-grading $\fg=\bigoplus_{j\in\Z/n\Z}\hat\fg_j$, i.e., $[\hat\fg_j,\hat\fg_k]\subset\hat\fg_{j+k\mod n}$. 
The connected subgroup $\hat\rG_0\subset\rG$ with Lie algebra $\hat\fg_0$ acts on each summand $\hat\fg_j$. The compact involution $\tau:\fg\to\fg$ can be chosen so that $\tau(\hat\fg_j)=\hat\fg_{-j\mod n}$. By the above discussion, we have the following proposition which was first observed by Simpson \cite[Proposition 6.3]{KatzMiddleInvCyclicHiggs} in the context of vector bundles. 
\begin{proposition}
    \label{prop: fixedpoint} Let $(\cE_{\hat\rG_0},\varphi)$ be a $0$-polystable $L$-twisted $(\hat\rG_0,\hat\fg_1)$-Higgs pair. Then the $L$-twisted $\rG$-Higgs bundle $(\cE_{\hat\rG_0}[\rG],\varphi)$ is polystable as a Higgs bundle.
\end{proposition}

\section{The generalized Cayley correspondence}\label{sec: open and closed}

In this section we prove that the Cayley map $\Psi_e$ from \eqref{eq Cayley map config} descends to an injective map on moduli spaces which is open and closed, thus proving Theorem B from the introduction. 
For this section, $\{f,h,e\}$ will be a magical $\fsl_2$-triple,
$\rS\subset\rG$ will be the associated connected subgroup,
$\rC\subset\rG$ will be its centralizer and $\rG^\R\subset\rG$ will be
the associated canonical real form. Recall that $\rH\subset\rG$ is the
complexification of the maximal compact $\rH^\R\subset\rG^\R$.
To simplify notation, throughout this section we denote $\rC\cap\rH$
simply by $\rC$.\footnote{Note that in fact $\rC\cap\rH=\rC$ except
  for the split real forms $\rSL_n\R$ and $\rE^6_6$.}

\subsection{Generalized Cayley correspondence and direct consequences}

Recall from \eqref{eq Cayley map config} that the Cayley map is given by 
 \[\widehat\Psi_e:\xymatrix@R=0em{\cH_{K^{m_c+1}}(\tilrG^\R)\times \prod_{j=1}^{\rkfge}\cH_{K^{l_j+1}}(\R^+)\ar[r]&\cH(\rG^\R),\\
 ((\cE_{\rC},\tilde\psi_{m_c}),q_1,\ldots,q_{\rkfge})\ar@{|->}[r]&(\cE_{\rC}\star\cE_{\rT}[\rH],f+\tilde\phi_{m_c}+\sum_{j=1}^{\rkfge}q_{j})}\]
 where $(\cE_{\rT},f)$ is the uniformizing $\rP\rSL_2\R$ (resp.\ $\rSL_2\R$) Higgs bundle if $\rS\cong\rP\rSL_2\C$ (resp.\ $\rS\cong\rSL_2\C)$.
There is a natural notion of stability on the domain of the Cayley map since it is a product of Higgs bundle spaces. Moreover, every $q_j\in H^0(K^{l_j+1})=\cH_{K^{l_j+1}}(\R^+)$ is polystable. Hence a point
\[((\cE_{\rC},\tilde\psi_{m_c}),q_1,\ldots,q_{\rkfge})\in \cH_{K^{m_c+1}}(\tilrG^\R)\times \prod_{j=1}^{\rkfge}\cH_{K^{l_j+1}}(\R^+)\]
is polystable if and only if $(\cE_{\rC},\tilde\psi_{m_c})\in\cH_{K^{m_c+1}}(\tilrG^\R)$ is polystable.

\begin{theorem}\label{thm Cayley map open closed}
The Cayley map $\hat\Psi_e$ descends to an injective map on moduli spaces,
\begin{equation}\label{eq:Cayleymoduli}
\Psi_e:\cM_{K^{m_c+1}}(\tilrG^\R)\times \prod_{j=1}^{\rkfge}\cM_{K^{l_j+1}}(\R^+)\longrightarrow\cM(\rG^\R).
\end{equation}
which is open and closed. 
\end{theorem}
We also refer to $\Psi_e$ as the \emph{Cayley map}. 

\begin{corollary}
    The image of the Cayley map $\Psi_e$ is a union of connected components of $\cM(\rG^\R)$ isomorphic to $\cM_{K^{m_c+1}}(\tilrG^\R)\times \prod_{j=1}^{\rkfge}\cM_{K^{l_j+1}}(\R^+)$. Every $\rG^\R$-Higgs bundle $(\cE_\rH,\varphi)$ in the image of the Cayley map has nowhere vanishing Higgs field $\varphi$.
\end{corollary}

\begin{definition}\label{cayleycomp}We refer to the connected components in the image of the Cayley map as the \emph{Cayley components in $\cM(\rG^\R)$}.
\end{definition}

\begin{remark}
For Case (2) of Theorem \ref{thm: classification weighted dynkin}, the Cayley map generalizes the Cayley correspondence of \cite{BGRmaximalToledo,HiggsbundlesSP2nR,sp4GothenConnComp} for Higgs bundles for Hermitian groups of tube type with maximal Toledo invariant. As a result, we refer to the isomorphism defined by the Cayley map as the generalized Cayley correspondence.
For Case (1) of Theorem \ref{thm: classification weighted dynkin}, the Cayley map recovers the Hitchin section of \cite{liegroupsteichmuller} for split real groups. In fact, for all cases, when the $\tilrG^\R$-Higgs bundle $(\cE_{\rC},\tilde\psi_{m_c})$ is trivial, the Cayley map recovers the Hitchin section for the split subgroup $\rG(e)^\R\subset\rG^\R$ with Lie algebra $\fg(e)^\R.$
Finally, for $\rG^\R=\rSO_{p,q}$ with $2\leq p\leq q$, the Cayley map recovers the connected components of $\cM(\rSO_{p,q})$ parameterized in \cite{so(pq)BCGGO,CollierSOnn+1components}.
\end{remark}

\begin{remark}
  When $\rG^\R\subset\rG$ is a split real form with Lie algebra $\fsp_{2n}\R$, $\fso_{n,n+1}$ or the quaternionic real form of $\ff_4$, there are two magical $\fsl_2$-triples, one from Case (1) of Theorem \ref{thm: classification weighted dynkin} and one from Case (2), Case (3) or Case (4), respectively. Note that these are the only cases where the semisimple part $\tilrG^\R\subset\rG_\cC^\R$ of the Cayley group is split and contains a unique magical $\fsl_2$-triple. For these groups, the Cayley map for Case (1) of Theorem \ref{thm: classification weighted dynkin} is obtained by iterating the Cayley maps. For example, when $\rG^\R$ is the quaternionic real form of $\rF_4,$ we have the following diagram 
  \[\xymatrix@=1em{H^0(K^2)\oplus H^0(K^6)\oplus H^0(K^8)\oplus H^0(K^{12})\ar[dr]_{\ \ \  \ \Psi_{e,1}}\ar[rr]^{\Id\oplus \Psi_{e,1}^{K^4}}& &H^0(K^2)\oplus H^0(K^6)\oplus \cM_{K^4}(\rSL_3\R)\ar[dl]^{\Psi_{e,4}}\\&\cM(\rG^\R)&},\]
  where $\Psi_{e,1}$ is the Cayley map from Case (1) of Theorem \ref{thm: classification weighted dynkin}, $\Psi_{e,4}$  is the Cayley map from Case (4) of Theorem \ref{thm: classification weighted dynkin} and $\Psi_{e,1}^{K^4}$ is the $K^4$-twisted version of the Cayley map from Case (1) of Theorem \ref{thm: classification weighted dynkin} for $\rSL_3\R$.
\end{remark}


Even though the Hitchin components are all smooth and contractible, this is not a general feature for the connected components defined by the generalized Cayley correspondence. Nevertheless, in the process of proving Theorem \ref{thm Cayley map open closed}, we show in Proposition \ref{prop H2 cayley vanishes} that for Higgs bundles in the image of the Cayley map, the second hypercohomology group $\HH^2(C^\bullet(\cE_\rH,\varphi))$ vanishes. 
As a result, $\HH^1(C^\bullet(\cE_\rH,\varphi))\sslash\Aut(\cE_\rH,\varphi)$ is a local model for the moduli space $\cM(\rG^\R)$ around $(\cE_\rH,\varphi)$. It follows immediately that $\cM(\rG^\R)$ is locally irreducible around $(\cE_\rH,\varphi)$. Hence, we have the following:
\begin{corollary}
Every Cayley component in $\cM(\rG^\R)$ is locally irreducible and irreducible.
\end{corollary}


    

The proof of Theorem \ref{thm Cayley map open closed} is broken into three parts. In \S\ref{subsec:Cayley descends to moduli} we prove that the Cayley map is well-defined and injective, we then prove the Cayley map is open in \S\ref{sec open} and closed in \S\ref{subsec: closed}.
\subsection{The Cayley map descends to moduli spaces}\label{subsec:Cayley descends to moduli}
We first prove the Cayley map descends to an injective map of moduli spaces. 

\begin{theorem}\label{thm Cayley map moduli spaces}
If $((\cE_{\rC},\tilde\psi_{m_c}),q_1,\ldots,q_{\rkfge})\in \cH_{K^{m_c+1}}(\tilrG^\R)\times \prod_{j=1}^{\rkfge}\cH_{K^{l_j+1}}(\R^+)$ is stable (resp.\ polystable), then $\widehat\Psi_e((\cE_{\rC},\tilde\psi_{m_c}),q_1,\ldots,q_{\rkfge})$ is a stable (resp.\ polystable) $\rG^\R$-Higgs bundle. 
In particular, the Cayley map \eqref{eq:Cayleymoduli} is well defined.
\end{theorem}

\begin{remark}
By Remark \ref{rem cayley map for other L}, the Cayley map can be defined for $L$-twisted Higgs bundles. The proof of Theorem \ref{thm Cayley map moduli spaces} given below also applies to this setting when $\deg(L)>0.$
\end{remark}

The difficult step in the proof of Theorem \ref{thm Cayley map moduli spaces} is proving the following lemma.
\begin{lemma}\label{lem no differentials stability}
    If $(\cE_{\rC},\tilde\psi_{m_c})\in\cH_{K^{m_c+1}}(\tilrG^\R)$ is stable (resp.\ polystable), then the $\rG^\R$-Higgs bundle $\widehat\Psi_e((\cE_{\rC},\tilde\psi_{m_c}),0,\ldots,0)$ is stable (resp.\ polystable).
\end{lemma}

Before proving Lemma \ref{lem no differentials stability}, we will prove Theorem \ref{thm Cayley map moduli spaces} assuming Lemma \ref{lem no differentials stability}. 

\begin{proof}[Proof of Theorem \ref{thm Cayley map moduli spaces} assuming Lemma \ref{lem no differentials stability}]
First note that the map is injective by Proposition \ref{prop injective Cayley}. 
The idea of the proof that the map is well defined is similar to Hitchin's proof \cite{liegroupsteichmuller} that the image of the Hitchin section consists of stable Higgs bundles. First assume $(\cE_{\rC},\tilde\psi_{m_c})$ is stable. Since stability is preserved by the $\C^*$-action, 
    \[\widehat\Psi_e((\cE_{\rC},\lambda\tilde\psi_{m_c}),0,\ldots,0)=(\cE_{\rC}\star\cE_\rT[\rH],f+\lambda\tilde\phi_{m_c})\]
    is a stable $\rG^\R$-Higgs bundle for all $\lambda\in\C^*$ by Lemma \ref{lem no differentials stability}. 
Since stability is open, 
    \[\widehat\Psi_e((\cE_{\rC},\tilde\psi_{m_c}),t_{1} q_1,\ldots,t_{\rkfge} q_{\rkfge})=\Bigg(\cE_{\rC}\star\cE_\rT[\rH], f+\tilde\phi_{m_c}+ \sum_{j=1}^{\rkfge}t_{j} q_{j}\Bigg)\]
    is stable for sufficiently small $t_j\in\R$. Thus, $(\cE_{\rC}\star\cE_\rT[\rH], \lambda^2 (f+\tilde\phi_{m_c}+ \sum_{j=1}^{\rkfge}t_{j} q_{j}))$ is stable for all $\lambda\in\C^*$. 

    Let $g_\lambda:\cE_\rT\to\cE_\rT$ be the holomorphic gauge transformation which acts on $f$ by $g_\lambda\cdot f=\lambda^{-2} f$, then $\Id_{\cE_{\rC}}\star g_\lambda$ acts on $\cE_{\rC}\star \cE_\rT[\fg_{2j}]\otimes K$ with eigenvalue $\lambda^{2j}$. 
    Since stability is also preserved by the gauge group, 
    \begin{align*}
        (\Id_{\cE_{\rC}}\star g_\lambda)&\cdot\Bigg(\cE_{\rC}\star\cE_\rT[\rH], \lambda^2\Bigg(f+\tilde\phi_{m_c}+ \sum_{j=1}^{\rkfge}t_{j} q_{j}\Bigg)\Bigg)\\&=\Bigg(\cE_{\rC}\star\cE_\rT[\rH], f+\lambda^{2m_c+2}\tilde\phi_{m_c}+\sum_{j=1}^{\rkfge}\lambda^{2l_j+2}t_{j} q_{j}\Bigg)\\
        &=\widehat\Psi_e((\cE_{\rC},\lambda^{2m_c+2}\tilde\psi_{m_c}),\lambda^{2l_1+2}t_{1}q_1,\ldots, \lambda^{2l_{\rkfge}+2}t_{\rkfge}q_{\rkfge})
    \end{align*} 
is stable for all $\lambda\in\C^*$. Thus, $\Psi_e((\cE_{\rC},\tilde\psi_{m_c}),q_1,\ldots, q_{\rkfge})$ is stable.

If $(\cE_{\rC},\tilde\psi_{m_c})$ is strictly polystable, then $\widehat\Psi_e(\cE_{\rC},\tilde\psi_{m_c})=(\cE_{\rC}\star\cE_\rT[\rH],f+\tilde\phi_{m_c})$ is a strictly polystable $\rG^\R$-Higgs bundle by Lemma \ref{lem no differentials stability}. 
Suppose $s\in i\fh^\R$ and $\cE_{\rP_s}\subset\cE_{\rC}\star\cE_\rT[\rH]$ is a holomorphic reduction to the parabolic $\rP_s$ with $\deg(\cE_{\rP_s})=0$ and such that $f+\tilde\phi_{m_c}\in H^0(\cE_{\rP_s}[\fm_s]\otimes K)$. 
By the definition of polystability there is a further holomorphic reduction $\cE_{\rL_s}\subset \cE_{\rP_s}$ such that $f+\tilde\phi_{m_c}\in H^0(\cE_{L_s}[\fm_{s}^0]\otimes K)$.  
We claim that this implies $s\in\fc$. Indeed, write $s=\sum s_{2j}$, where $s_{2j}$ is the projection of $s$ onto the graded piece $\fg_{2j}$ and suppose $k$ is the smallest $j$ with $s_{2j}\neq0$. If $v\in\fg_{2m_c}$, then $2k-2$-graded piece of $[s,f+v]=[s_{2k},f]$. Since $\{f,h,e\}$ is magical, $\ker(\ad_f)\cap\fh=\fc$. Thus, $[s_{2k},f]=0$ implies $s\in\fc$. 

By Proposition \ref{prop JordanHolder}, there is $s\in i\fh^\R$ and a holomorphic reduction $\cE_{\rL_s}\subset \cE_{\rC}\star\cE_\rT[\rH]$ with $f+\tilde\phi_{m_c}\in H^0(\cE_{\rL_s}[\fm_s^0]\otimes K)$ such that $(\cE_{\rL_s},f+\tilde\phi_{m_c})$ is a stable $\rG^\R_s$-Higgs bundles. Here $\rG^\R_s$ is the real form of the $\rG$-centralizer of $s$ associated to the complexified Cartan decomposition $\fg_s^0=\fl_s\oplus\fm_s^0$. 
Since, $s\in\fc$ and $[\fc,\fg(e)]=0$, it follows that $\widehat\Psi_e((\cE_{\rC},\tilde\psi_{m_c}),q_1,\ldots, q_{\rkfge})$ is a $\rG^\R_s$-Higgs bundle. Openness and $\C^*$-invariance of stability implies $\widehat\Psi_e((\cE_{\rC},\tilde\psi_{m_c}),q_1,\ldots, q_{\rkfge})$ is a stable $\rG^\R_s$-Higgs bundle and hence a polystable $\rG^\R$-Higgs bundle.
\end{proof}

We will prove Lemma \ref{lem no differentials stability} in each of the four cases of magical nilpotents from Theorem \ref{thm: classification weighted dynkin}. The result is immediate for Case (1),  it was proven in \cite{BGRmaximalToledo} for Case (2), and for Case (3) the result was proven in \cite{so(pq)BCGGO} for $\rG=\rSO_N\C$.  Our proof in Case (4) relies on the details of the proof of Case (2) so we outline the proof of \cite{BGRmaximalToledo}.
\begin{proof}[Proof of Lemma \ref{lem no differentials stability} Case (1)]
    For Case (1) of Theorem \ref{thm: classification weighted dynkin}, $\rC$ is the center of $\rG^\R$ and $\tilde\phi_{m_c}=0$. Thus, $\widehat\Psi_{e}(\cE_{\rC},\tilde\phi_{m_c},0,\ldots,0)=(\cE_{\rC}\star\cE_\rT[\rH],f)$.
      This is a polystable Higgs bundle since the solution metric for $(\cE_{\rT},f)$ induces a solution to the $\rG^\R$-Higgs bundle equations. It is stable since a principal nilpotent is not contained in the Levi subalgebra of any proper parabolic subalgebra of $\fg$.
\end{proof}

\begin{proof}[Proof of Lemma \ref{lem no differentials stability} Case (3)]
For Case (3) of Theorem \ref{thm: classification weighted dynkin} with $\rG=\rSO_N\C$ (and hence $\rG^\R=\rSO_{p,N-p}$), Lemma \ref{lem no differentials stability} was proven in \cite[Lemma 4.5]{so(pq)BCGGO}. Roughly, $m_c+1=p$ and there is a $\Z/2p\Z$-grading $\fg=\bigoplus\hat\fg_j$ such that $(\cE_\rC\star\cE_\rT[\hat\rG_0],f+\tilde\phi_{p-1})$ is a $(\hat\rG_0,\hat\fg_1)$-Higgs pair. This pair is shown to be polystable and Proposition \ref{prop: fixedpoint} is applied.  
By Remark \ref{rem stability and coverings}, it suffices to show that every $\rPSO_N\C$-Higgs bundle in the image of $\hat\Psi_e$ lifts to a $\rSO_N\C$-Higgs bundle in the image of $\hat\Psi_e$. This was shown in \S\ref{sec injectivity}.
\end{proof}

\begin{proof}[Proof of Lemma \ref{lem no differentials stability} Case (2)]
The proof for Case (2) is the result of Lemmas 5.5, 5.6 and 5.7 of \cite{BGRmaximalToledo}. We outline the argument here in the notation of the current article. In this case, and $m_c=1$ and $\rH=\rG_0\subset\rG$ is the centralizer of $h\in\fg$. 

Let $(\cE_\rC,\psi_1)$ be a stable (resp.\ polystable) $K^2$-twisted $\tilrG^\R$-Higgs bundle. By \cite[Lemma 5.5]{BGRmaximalToledo}, $(\cE_\rC\star\cE_T[\rH],\psi_1)$ is an $\alpha$-stable (resp.\ $\alpha$-polystable) $K^2$-twisted $\rH$-Higgs bundle for $\alpha=\frac{h}{2}\in\fz(\fh)$. This is proven using equations. 
Next one proves a finite dimensional GIT result (\cite[Lemma 5.6]{BGRmaximalToledo}) for the magical nilpotent $f\in\fg_{-2}$. Namely, if $s\in i\fh$ and $f\in \fg_{-2,s}$, then $\langle h,s\rangle\geq 0$ and if equality holds, then $f\in\fg_{-2,s}^0$. 

Now consider $\Psi_e(\cE_\rC,\psi_1)=(\cE_{\rC}\star\cE_{\rT}[\rH],f+\phi_1)$, where $\ad_f(\phi_1)=\psi_1\in H^0(\cE_{\rC}\star \cE_\rT[\fg_0]\otimes K^2)=H^0(\cE_{\rC}[\fh]\otimes K^2)$. 
Let $s\in i\fh^\R$ and $\cE_{\rP_s}\subset\cE_{\rC}\star\cE_{\rT}[\rH]$ be a holomorphic reduction such that $f+\phi_1\in H^0(\cE_{\rP_s}[\fm_s]\otimes K)$. Since $\rP_s$ preserves the splitting $\fm=\fg_{-2}\oplus\fg_{2}$, we have $f \in H^0(\cE_{\rP_s}[\fg_{-2,s}]\otimes K)$ and $\phi_1\in H^0(\cE_{\rP_s}[\fg_{2,s}]\otimes K)$. Hence $\psi_1=[f,\phi_1]\in H^0(\cE_{\rP_s}[\fh_s]\otimes K^2)$. We have $\deg(\cE_{\rP_s})\geq \langle\frac{h}{2},s\rangle$ by Lemma 5.5 and $\langle\frac{h}{2},s\rangle\geq 0$ by Lemma 5.6. Thus, $\deg(\cE_{\rP_s})\geq \langle\frac{h}{2},s\rangle\geq 0$.

If $\deg(\cE_{\rP_s})=0$, then $f\in\fg_{-2,s}^0$ and there is a holomorphic reduction $\cE_{\rL_s}\subset\cE_{\rP_s}$ such that $\psi_1=[f,\phi_1]\in H^0(\cE_{\rL_s}[\fh_{s}^0]\otimes K^2)$. Note that $[s,\phi_1]=0$ since $\ad_f:\fg_{2}\to\fg_0$ is injective and
\[0=[s,[f,\phi_1]]=-[\phi_1,[s,f]]-[f,[\phi_1,s]]=[f,[s,\phi_1]].\]
Hence $f+\phi_1\in H^0(\cE_{\rL_s}[\fm_{s}^0]\otimes K)$ and $\widehat\Psi_e(\cE_\rC,\psi_1)$ is a polystable $\rG^\R$-Higgs bundle.
\end{proof}

Before proving Case (4) below, we recall some relevant notions from previous sections. Let $\{f,h,e\}\subset\fg$ be a magical $\fsl_2$-triple from Case (4) of Theorem \ref{thm: classification weighted dynkin}. Recall from \S\ref{sec lie theory for except} that $m_c=3$, $\tilde\phi_3=\phi_3$ and the $\Z$-grading is given by $\fg=\bigoplus_{j=-5}^5\fg_{2j}$. Moreover, $\fg_{-2}$ decomposes $\fg_0$-invariantly as $\fg_{-2}=\fg_{-\tilde\alpha}\oplus\fg_{-2}^b$, where $\tilde\alpha$ is the simple root in the diagrams in \S\ref{sec diagrams and tables}.
Consider the $\Z/4\Z$-grading given by $\fg=\bigoplus_{j\in\Z/4\Z}\hat\fg_j$, where
\[\xymatrix@R=0em{\hat\fg_0=\fg_{-8}\oplus\fg_0\oplus\fg_8~,&\hat\fg_{1}=\fg_{-10}\oplus \fg_{-2}\oplus\fg_6~\\\hat\fg_{2}=\fg_{-4}\oplus\fg_4,&\hat\fg_3=\fg_{-6}\oplus\fg_{2}\oplus\fg_{10}~.}\]
By \eqref{eq hm decomp}, the complexified Cartan decomposition $\fg=\fh\oplus\fm$ of the canonical real form satisfies $\fh=\hat\fg_0\oplus\hat\fg_2$ and $\fm=\hat\fg_1\oplus\hat\fg_3$. 
Recall from \eqref{eq h'+sl2 decomp} that $\fh=\fh'\oplus\fsl_2\C$, and note that $\hat\fg_0=\fh_0'\oplus\fsl_2\C$.
Let $\rG_0\subset\hat\rG_0\subset\rG$ be the connected subgroups with Lie algebras $\fg_0\subset\hat\fg_0$ respectively. The adjoint action of $\rG_0$ and $\hat\rG_0$ preserve the spaces $\fg_j$ and $\hat\fg_j$ respectively. Moreover, by Lemma \ref{lem h' sl2 invariant decomp of m}, $\hat\fg_1$ decomposes $\hat\rG_0$-invariantly as
\begin{equation}
    \label{eq hatG0 invar hatg1 decomp}\hat\fg_1=(\fg_{-\tilde\alpha}\oplus\fg_{-10})\oplus(\fg_{-2}^b\oplus\fg_6).
\end{equation}

Consider the $K^4$-twisted $\tilrG^\R$-Higgs bundle $(\cE_\rC,\psi_3)$, and recall 
\[\hat\Psi_e((\cE_{\rC},\psi_3),0,0)=(\cE_\rC\star\cE_\rT[\rH],f+\phi_3),\] 
where $\ad_f^3(\phi_3)=\psi_3$. Since $\rC\star\rT\subset\rG_0$ and $f+\phi_3\in H^0(\cE_\rC\star\cE_\rT[\hat\fg_1]\otimes K)$, 
\[(\cE_{\hat\rG_0},\Phi)=(\cE_\rC\star\cE_\rT[\hat\rG_0],f+\phi_3)\]
is  a $K$-twisted $(\hat\rG_0,\hat\fg_1)$-Higgs pair. Using the decomposition \eqref{eq hatG0 invar hatg1 decomp} we write
\[\Phi=(f_b+\phi_3)\oplus (\tilde f+0)\in H^0(\cE_{\hat\rG_0}[\fg_{-2}^b\oplus\fg_{6}]\otimes K)\oplus H^0(\cE_{\hat\rG_0}[\fg_{-\tilde\alpha}\oplus\fg_{-10}]\otimes K).\] 
This, implies that $\ad_{f_b+\phi_3}(\tilde f)\in H^0(\cE_{\hat\rG_0}[\hat\fg_{2}]\otimes K^2)$. Recall from \eqref{eq ad f_b+phi}, that $\ad_{f_b+\phi_3}(\tilde f)=[f_b,\tilde f]\in\fg_{-4}$ is a magical nilpotent in $\fh'$ from Case (2) of Theorem \ref{thm: classification weighted dynkin}. Since the splitting $\hat\fg_2=\fg_{-4}\oplus\fg_{4}$ is $\hat\rG_0$-invariant, 
\begin{equation}
    \label{eq ad fb+phi3 tilde f}\ad_{f_b+\phi_3}\tilde f=[f_b,\tilde f]\in H^0(\cE_{\hat\rG_0}[\fg_{-4}]\otimes K^2).
\end{equation}

Also, $\theta=\ad^3_{f_b+\phi_3}(\tilde f)\in H^0(\cE_{\hat\rG_0}[\hat\fg_0]\otimes K^4)$. Thus, $(\cE_{\hat\rG_0},\theta)$ is a $K^4$-twisted $\hat\rG_0$-Higgs bundle. Moreover, the decomposition $\hat\fg_0=\fh_0'\oplus\fsl_2\C$ gives
\[\theta=\theta'\oplus\theta_2\in H^0(\cE_{\hat\rG_0}[\fh_0']\otimes K^4)\oplus H^0(\cE_{\hat\rG_0}[\fsl_2\C]\otimes K^4)~.\]
The bracket relations of Lemma \ref{lem exceptional bracket relations} now imply 
\begin{equation}
    \label{eq theta decomp}\xymatrix{\theta'=3\psi_3&\text{and}&\theta_2=\ad_{f_b}^3(\tilde f)+\ad_{\phi_3}^2\circ\ad_{f_b}(\tilde f).}
\end{equation}
In particular, $\theta_2$ is in the $K^4$-twisted $\rSL_2\C$-Hitchin section.

\begin{proof}
    [Proof of Lemma \ref{lem no differentials stability} Case (4)]
    Suppose $(\cE_{\rC},\psi_3)$ is a polystable $K^4$-twisted $\tilrG^\R$-Higgs bundle. To show that $\hat\Psi_e((\cE_\rC,\psi_3),0,0)$ is a polystable $\rG^\R$-Higgs bundle, it suffices to show that $(\cE_{\rC}\star\cE_\rT(\hat\rG_0),f+\phi_3)$ is a polystable $(\hat\rG_0,\hat\fg_1)$-Higgs pair by Proposition \ref{prop: fixedpoint}. 

    Consider the $(\hat\rG_0,\hat\fg_1)$-Higgs pair $(\cE_{\hat\rG_0},\Phi)=(\cE_{\rC}\star\cE_\rT(\hat\rG_0),f+\phi_3)$. Let $\hat\rH_0^\R\subset\hat\rG_0$ be a compact real form with Lie algebra $\hat\fh_0^\R$. Fix $s\in i\hat\fh_0^\R$ and let $\rP_s\subset\hat\rG_0$ be the corresponding parabolic. 
    Since $\hat\fg_0=\fh_0'\oplus\fsl_2\C$ we can write $s=s'+s_2$, where $s'\in \fh_0'$ and $s_2\in\fsl_2\C$. Let $\cE_{\rP_s}\subset\cE_{\hat\rG_0}$ be a holomorphic reduction such that $\Phi\in H^0(\cE_{\rP_s}[\hat\fg_{1,s}]\otimes K)$. Note that the inclusions $\rP_{s}\subset\rP_{s'}$ and $\rP_s\subset\rP_{s_2}$ define holomorphic reductions $\cE_{\rP_s}\subset\cE_{\rP_{s'}}\subset\cE_{\hat\rG_0}$ and $\cE_{\rP_s}\subset\cE_{\rP_{s_2}}\subset\cE_{\hat\rG_0}$
    We are interested in showing
    \[\deg(\cE_{\rP_s})=\deg(\cE_{\rP_{s'}})+\deg(\cE_{\rP_{s_2}})\geq0.\]

Since $\hat\rG_0$ preserves the splitting \eqref{eq hatG0 invar hatg1 decomp}, we have 
  \begin{equation}
      \label{eq higgs field decom for hatG0}\xymatrix{f_b+\phi_3\in H^0(\cE_{\rP_s}[\hat\fg_{1,s}]\otimes K)&\text{and}&\tilde f\in H^0(\cE_{\rP_s}[\hat\fg_{1,s}]\otimes K)}.
  \end{equation}
    Thus, $\ad_{f_b+\phi_3}^3(\tilde f)=\theta\in H^0(\cE_{\hat\rG_0}[\hat\fg_{0,s}]\otimes K^4)$, and, using the decomposition \eqref{eq theta decomp},  
    \[\xymatrix{\theta'=3\psi_3\in H^0(\cE_{\hat\rG_0}[\fh_{0,s'}']\otimes K^4)&\text{and}&\theta_2\in H^0(\cE_{\hat\rG_0}[\fsl_2\C_{s_2}]\otimes K^4).}\]
Since $\theta_2$ is in the $K^4$-twisted Hitchin section, we have 
\[\deg(\cE_{\rP_{s_2}})\geq 0\]
with equality if and only if $s_2=0$.

To show that $\deg(\cE_{\rP_{s'}})\geq0$, we use an argument similar to the proof of Case (2) of Theorem \ref{thm: classification weighted dynkin}. 
Write $h=h'+h_2$, where $h'\in\fh_0'$ and $h_2\in\fsl_2\C$ are both nonzero, and let $\rT',\rT_2\subset\rH$ be the subgroups generated by $\exp(th')$ and $\exp(th_2)$. The $\hat\rG_0$-bundle $\cE_\rT(\hat\rG_0)$ is given by 
\[\cE_{\rT}(\hat\rG_0)=\cE_{\rT'}\star\cE_{\rT_2}(\hat\rG_0).\]
Fix the K\"ahler form $\omega$ associated to the hyperbolic metric uniformizing the Riemann surface $X$, so that $F_K=-i\omega$.

Since $\theta_2$ is in the $\rPSL_2\C$-Hitchin section, there is a metric $h_{\rT_2}$ on $\cE_{\rT_2}$ so that 
\[F_{h_{\rT_2}}+[\theta_2,-\tau(\theta_2)]\omega=0.\] 
Let $h_{\rT'}$ be the uniformizing metric on $\cE_{\rT'}$ and take $h_\rT=h_{\rT'}\star h_{\rT_2}$. Since $(\cE_\rC,3\psi_3)$ is polystable, there is a metric $h_\rC$ on $\cE_\rC$ such that 
\[F_{h_\rC}+[3\psi_3,-\tau(3\psi)]\omega=0.\]
Thus, $h_\rC\star h_{\rT_2}\star h_{\rT'}$ defines a metric on $\cE_{\hat\rG_0}$ which satisfies
\[F_{h_\rC\star h_{\rT_2}\star h_{\rT'}}+[3\psi_3,-\tau(3\psi_3)]\omega+[\theta_2,-\tau(\theta_2)]\omega=F_{h_{\rT'}}=-i\lambda\omega h'\]
for some positive constant $\lambda$. The exact value of $\lambda$ is not important. 
Thus, $(\cE_{\hat\rG_0},3\psi_3+\theta_2)$ is an $\alpha=\lambda h'$-polystable $K^4$-twisted $\hat\rG_0$-Higgs bundle, and hence
\[\deg(\cE_{\rP_s})=\deg(\cE_{\rP_{s'}})+\deg(\cE_{\rP_{s_2}})\geq\deg(\cE_{\rP_s'})\geq \langle \lambda h',s'\rangle.\] 

Note that $\ad_{f_b+\phi_3}(\tilde f)=[f_b,\tilde f]\in H^0(\cE_{\rP_{s'}}[\fg_{-4,s'}]\otimes K^2)$ by \eqref{eq higgs field decom for hatG0} and \eqref{eq ad fb+phi3 tilde f}. Since $[f_b,\tilde f]\subset\fg_{-4}$ is a magical nilpotent in $\fh'$ corresponding to Case (2) of Theorem \ref{thm: classification weighted dynkin}, the finite dimensional GIT result \cite[Lemma 5.5]{BGRmaximalToledo} applies and gives $\langle \lambda h',s'\rangle \geq 0$ with equality if and only if $[f_b,\tilde f]\in\fg_{-4,s'}^{0}$. Thus, $\deg(\cE_{\rP_s})\geq 0$.

So far we have shown that $(\cE_{\hat\rG_0},\Phi)$ is a semistable $(\hat\rG_0,\hat\fg_1)$-Higgs pair. Suppose $\deg(\cE_{\rP_s})=0$, then $\deg(\cE_{\rP_{s_2}})=0$ and $\deg(\cE_{\rP_{s'}})=0$, and hence $s_2=0$. The $\alpha$-polystable of the $K^4$-twisted $\hat\rG_0$-Higgs bundle $(\cE_{\hat\rG_0},\theta)$ implies there is a holomorphic reduction $\cE_{\rL_s}\subset\cE_{\rP_s}$ such that $\theta\in H^0(\cE_{\rL_s}[\hat\fg_{0,s}^0]\otimes K^4)$. In particular, $\psi_3\in H^0(\cE_{\rL_s}[\hat\fg_{0,s}^0]\otimes K^4)$.
Since the splitting $\fg_{-\tilde\alpha}\oplus\fg_{-2}^b\oplus\fg_{6}$ is $\fh'_0$-invariant and $s=s'\in\fh_0'$, we have
\[\xymatrix@=1em{\tilde f\in H^0(\cE_{\rP_s}[\hat\fg_{1,s}]\otimes K)~,&f_b\in H^0(\cE_{\rP_s}[\hat\fg_{1,s}]\otimes K)&\text{and}&\phi_3\in H^0(\cE_{\rP_s}[\hat\fg_{1,s}]\otimes K).}\]
Thus, 
\[\xymatrix{[f_b,\tilde f]\in H^0(\cE_{\rP_s}[\fg_{-4,s}]\otimes K^2)&\text{and}&[f_b,\phi_3]\in H^0(\cE_{\rP_s}[\fg_{4,s}]\otimes K^2).}\]  
We have $0=\deg(\cE_{\rP_s})\geq\langle s,h'\rangle\geq 0$, thus the finite dimensional GIT lemma implies $[f_b,\tilde f]\in H^0(\cE_{\rL_s}[\fg_{-4,s}^0]\otimes K^2)$.
 By \eqref{eq ad^3 f_b+phi}, $\psi_3=[[f_b,\tilde f],[f_b,\phi_3]]$, and hence $[f_b,\phi_3]\in H^0(\cE_{\rL_s}[\fg_{4,s}^0]\otimes K^2)$.  Finally, since $\tilde f$, $f_b$, and $\phi_3$ are each in $H^0(\cE_{\rP_s}[\hat\fg_{1,s}]\otimes K)$, we have 
\[\xymatrix@=1em{\tilde f\in H^0(\cE_{\rL_s}[\hat\fg_{1,s}^0]\otimes K)~,&f_b\in H^0(\cE_{\rL_s}[\hat\fg_{1,s}^0]\otimes K)&\text{and}&\phi_3\in H^0(\cE_{\rL_s}[\hat\fg_{1,s}^0]\otimes K).}\]
Hence $(\cE_{\rC}\star\cE_{\rT}[\hat\rG_0],f+\phi_3)$ is a 0-polystable $K$-twisted $(\hat\rG_0,\hat\fg_1)$-Higgs pair.
 \end{proof}

\subsection{The Cayley map is open}\label{sec open}

We now prove that the Cayley map is open. Recall the deformation complex and description of the local structure of the moduli space from \S\ref{sec: local struct}. By Corollary \ref{cor mag implies even and injmap h->m}, $\ker(\ad_f:\fh\to\fm)=\fc$ and $\ad_f:\ad_f(\fm)\to\ad_f^2(\fm)$ is an isomorphism. Hence we have $\rC\star\rT$-invariant splittings 
\begin{equation}
    \label{eq invariant splittings adf}\xymatrix{\fh=\fc\oplus\ad_f(\fm)&\text{and}&\fm= V_{\fm}\oplus\ad_f^2(\fm)},
\end{equation}
where $V_\fm=\bigoplus_{j>0}V_{2m_j}$ is the set of highest weight spaces in $\fm$. 

Recall that the Cayley map \eqref{eq:Cayleymoduli} is defined by 
\[\Psi_e(\cE_\rC,\psi)=((\cE_\rC\star\cE_\rT)[\rH],f+\varphi),\]
where $\psi\in \bigoplus_{j>0}H^0(\cE_\rC[Z_{2m_j}]\otimes K^{m_j+1})$ and $\varphi\in H^0(\cE_\rC\star\cE_\rT[V_\fm]\otimes K)$ is determined by $\psi$ using the isomorphisms $\ad_f^{m_j}:(\cE_\rC\star\cE_\rT)[V_{2m_j}]\otimes K\to \cE_\rC[Z_{2m_j}]\otimes K^{m_j+1}$.
The deformation complex for $(\cE_\rC,\psi)$ is 
\[C^\bullet_\cC: \cE_{\rC}[\fc]\xrightarrow{\ \ad_{\psi}\ }\bigoplus_{j>0}\cE_{\rC}[Z_{2m_j}]\otimes K^{m_j+1}.\]
On the other hand, since $[\fc,f+V_\fm]\subset V_\fm$, the deformation complex for $\Psi_e(\cE_\rC,\psi)$ is 
\[C^\bullet_\cH: \cE_{\rC}[\fc]\oplus (\cE_{\rT}\star\cE_\rC)[\ad_f(\fm)]\xrightarrow{\smtrx{
    \ad_\varphi&\alpha\\0&\beta}}(\cE_{\rT}\star\cE_\rC)[V_\fm]\otimes K\oplus(\cE_{\rT}\star\cE_\rC)[\ad^2_f(\fm)]\otimes K,\]
where we've used the fact that $\rT$ acts trivially on $\fc$ to identify $(\cE_\rC\star\cE_\rT)[\fc]\cong\cE_\rC[\fc]$, and $\alpha$ and $\beta$ are defined by post composing $\ad_{f+\varphi}:(\cE_\rC\star\cE_\rT)[\ad_f]\to (\cE_\rC\star\cE_\rT)[\fm]\otimes K$ with the projection onto the $(\cE_\rC\star\cE_\rT)[V_\fm]\otimes K$ and $(\cE_\rC\star\cE_\rT)[\ad_f^2(\fm)]\otimes K$, respectively.

The Cayley map induces a short exact sequence of complexes 
\[0\longrightarrow C^\bullet_\cC\longrightarrow C^\bullet_\cH\longrightarrow C^\bullet_\cH/ C^\bullet_\cC\longrightarrow 0,\]
such that the quotient complex is isomorphic to 
\[C^\bullet_\cH/ C^\bullet_\cC: (\cE_\rC\star\cE_\rT)[\ad_f(\fm)]\xrightarrow{\ \beta\ }(\cE_\rC\star\cE_\rT)[\ad^2_f(\fm)]\otimes K.\]

\begin{proposition}\label{prop:isom in hypercohomo}
The quotient complex $C^\bullet_\cH/ C^\bullet_\cC$ has trivial hypercohomology. In particular, 
\[\HH^\bullet(C^\bullet_{\cC})\cong\HH^\bullet(C^\bullet_{\cH})~.\]
\end{proposition}
\begin{proof}
It suffices to show that the map $\beta:(\cE_\rC\star\cE_\rT)[\ad_f(\fm)]\to(\cE_\rC\star\cE_\rT)[\ad_f^2(\fm)]\otimes K$ is an isomorphism. First, $\ad_f:(\cE_\rC\star\cE_\rT)[\ad_f(\fm)]\to(\cE_\rC\star\cE_\rT)[\ad_f^2(\fm)]\otimes K$ induces an isomorphism of holomorphic bundles. 
Since $v\in V_\fm\subset\bigoplus_{j>0}\fg_j$, for any $v\in V_\fm$, the composition of $\ad_{f+v}:\ad_f(\fm)\to\fm$ with projection onto $\ad_f^2(\fm)$ is injective and hence also defines an isomorphism $\ad_f(\fm)\to\ad_f^2(\fm)$. 
Thus, $\beta$ is an isomorphism and $C^\bullet_\cH/ C^\bullet_\cC$ has trivial hypercohomology.
\end{proof}

We can now prove the second hypercohomology of the complexes $C^\bullet_\cC$ and $C^\bullet_\cH$ vanishes.
\begin{proposition}\label{prop H2 cayley vanishes}
Suppose $(\cE_\rC,\psi)$ is a polystable object in the domain of $\Psi_e$. Then 
\[0=\HH^2(C^\bullet_\cC(\cE_\rC,\psi))=\HH^2(C^\bullet_\cH(\Psi_e(\cE_\rC,\psi)))=0.\]
\end{proposition}
\begin{proof}
    Since the domain of the Cayley map is identified with a product of moduli spaces of $L$-twisted Higgs bundles with $\deg(L)>2g-2$, Proposition \ref{prop H2 vanishes} implies $\HH^2(C^\bullet_\cC(\cE_\rC,\psi))=0$. Now, Proposition \ref{prop:isom in hypercohomo} implies $\HH^2(C^\bullet_\cH(\Psi_e(\cE_\rC,\psi)))=0$.
\end{proof}

\begin{remark}
Note that isomorphism of hypercohomology groups and vanishing of $\HH^2$ in this general context is much cleaner than the one in \cite[\S4.2]{so(pq)BCGGO} for $\rG^\R=\rSO_{p,q}$, which took several pages. This is a reflection of the power of the magical $\fsl_2$-triple perspective.
\end{remark}

We can now prove the Cayley map is open.

\begin{proposition}\label{prop:open}
The Cayley map $\Psi_e:\cM_{K^{m_c+1}}(\tilrG^\R)\times \prod_{j=1}^{\rkfge} H^0(K^{l_j+1})\to\cM(\rG^\R)$ is open. In particular, its image is open in $\cM(\rG^\R)$.
\end{proposition}
\begin{proof}
Let $(\cE_\rC,\psi)$ be a point in the domain of the $\Psi_e$. By Proposition \ref{prop H2 vanishes}, local neighborhoods of $(\cE_\rC,\psi)$ and $\Psi_e(\cE_\rC,\psi)$ are respectively isomorphic to 
\[\xymatrix{\HH^1(C^\bullet_\cC(\cE_\rC,\psi))\sslash\Aut(\cE_\rC,\psi)&\text{and}&\HH^1(C^\bullet_\cH(\Psi_e(\cE_\rC,\psi)))\sslash\Aut(\Psi_e(\cE_\rC,\psi)).}\]
By Proposition \ref{prop:isom in hypercohomo}, $\Psi_e$ induces an isomorphism $\HH^1(C^\bullet_\cC(\cE_\rC,\psi))\cong\HH^1(C^\bullet_\cH(\Psi_e(\cE_\rC,\psi)))$ which is $\Aut(\cE_\rC,\psi)$-equivariant. By Corollary \ref{Cor same automorphism groups} we have $\Aut(\cE_\rC,\psi)=\Aut(\Psi_e(\cE_\rC,\psi))$. 
Thus, the Cayley map induces a local isomorphism and hence is open.
\end{proof}

\subsection{The Cayley map is closed}\label{subsec: closed}

Recall from Remark \ref{rem slodowy slice}, that the Slodowy slice $f+\ker(\ad_e)=f+V\subset\fg$ is a slice for the adjoint action of $\rG$. 
We have an $\Ad_\rH$ invariant decomposition $V=\fc\oplus V_\fm$, and $f+V_\fm$ is a slice through $f$ for the $\rH$-action in $\fm$. Moreover, $f+V_\fm$ decomposes $\Ad_\rC$-invariantly as
\begin{equation}\label{eq:decomp-Sm}
f+V_\fm=f+\bigoplus_{j=1}^MV_{2m_j},
\end{equation}
where $\rC$-acts trivially on every summand except $V_{2m_c}$. Recall that the Cayley real form $\fg_\cC^\R$ is a real form of $\fg_0$ and has complexified Cartan decomposition $\fg_0=\fc\oplus Z_\fm$, where we define\footnote{Note that $Z_\fm$ is not a subset of $\fm$.} 
\[Z_\fm=\bigoplus_{j=1}^MZ_{2m_j}.\]
 There is a $\rC$-equivariant isomorphism $\psi_e:Z_\fm\to f+V_\fm$ induced by the $\rC$-equivariant isomorphisms $\ad_f^{m_j}:V_{2m_j}\to Z_{2m_j}$.

Let $\chi:\fm\to\fm\sslash\rH$ and $\chi_\cC:Z_\fm\to Z_\fm\sslash \rC$ be the adjoint quotient maps, and let $\chi_e:f+V_{\fm}\to\fm\sslash\rH$ be the restriction of $\chi$ to $f+V_{\fm}$. The composition $\chi_e\circ\psi_e:Z_\fm\to \fm\sslash\rH$ defines a map
\begin{equation}
    \label{eq gamma e}\gamma_e:Z_\fm\sslash\rC\to\fm\sslash\rH
\end{equation}
such that
\begin{equation}\label{eq:adjointmaps-commut}
\gamma_e\circ\chi_\cC=\chi_e\circ\psi_e.
\end{equation}
Recall that, by choosing a homogeneous basis of invariant polynomials, $\fm\sslash\rH$ and $Z_\fm\sslash \rC$ are identified with affine spaces of dimension the real rank of $\fg^\R$ and $\fg_\cC^\R$ respectively. Thus, by Proposition \ref{prop real ranks the same}, $\fm\sslash \rH$ and $Z_\fm\sslash \rC$ have the same dimension. 

\begin{proposition}\label{prop:faithflat,surj}
    Let $\{f,h,e\}\subset\fg$ be a magical $\fsl_2$-triple. Then $\chi_e:f+V_\fm\to\fm\sslash\rH$ and $\gamma_e:Z_\fm\sslash\rC\to\fm\sslash \rH$ are flat and surjective, thus faithfully flat. Moreover, $\gamma_e$ has finite fibers. 
\end{proposition}
\begin{proof}
By \cite[Theorem 9]{KostantRallis}, every fiber of the surjective morphism $\chi:\fm\to\fm\sslash\rH$ has pure dimension equal to $\dim(\fm)-\dim(\fm\sslash\rH)$. Since both $\fm$ and $\fm\sslash\rH$ are affine spaces, the so called ``miracle flatness theorem'' implies that $\chi_e$ is flat; see, for example \cite[Exercise III.10.9]{Hartshorne-AlgebraicGeometry} or \cite[pg. 158]{Fischer-CAG-book}.

On the other hand, the orbit map $\mu:\rH\times (f+V_\fm)\to\fm$ is smooth, and hence flat since $f+V_\fm$ is a slice for the $\rH$-action on $\fm$.
Thus, $\chi\circ\mu: \rH\times (f+V_\fm)\to\fm\sslash\rH$ is also flat. However, this morphism factors through $f+V_\fm$, so that we have a commutative diagram
\[\begin{tikzcd}
    \rH\times (f+V_\fm)\ar{dr}[swap]{\mathrm{pr}_2}\ar{r}{\mu}&\fm\ar{r}{\chi}&\fm\sslash\rH\\
    &f+V_\fm\ar[u,hook]\ar{ru}[swap]{\chi_e}&
    \end{tikzcd}\]
where $\mathrm{pr}_2$ is the canonical projection.
Since both $\chi\circ\mu$ and $\mathrm{pr}_2$ are flat, the morphism $\chi_e:f+V_\fm\to\fm\sslash\rH$ is flat too by \cite[Corollary 2.2.11]{EGAIV-tome2}. 

 As in \cite[\S7.4]{Slodowy-Simplesingularities-book}, to show that $\chi_e$ is surjective, we show that it is equivariant with respect to a $\C^*$-action with positive weights. 
Choose a basis $(p_1,\ldots,p_{r})$ of $\rH$-invariant polynomials on $\fm$ which are homogeneous of degree $m_{1}',\ldots, m_{r}'$. This identifies $\fm\sslash \rH$ with $\C^r$, via $[y]\mapsto (p_1(y),\ldots, p_r(y))$ for $y\in\fm$. We have 
\begin{equation}
    \label{eq base action}\chi_e(t^2y)=(t^{2m'_1}p_1(y),\ldots,t^{2m'_r}p_r(y)).
\end{equation}
Now consider the $\C^*$-action on $f+V_\fm$ by 
\begin{equation}
    \label{eq slice action}t \cdot \Bigg(f+\sum_{j=1}^Mv_{2m_j}\Bigg)=f+\sum_{j=1}^Mt^{2+2m_j}v_{2m_j},
\end{equation}
where $v_{2m_j}\in V_{2m_j}$. 
There is an element $g\in\rT\subset\rH$ so that  
\[\Ad_g\Bigg(t^2f+\sum_{j=1}^Mt^2v_{2m_j}\Bigg)=f+ \sum_{j=1}^Mt^{2m_j+2}v_{2m_j}=t\cdot\Bigg(f+ \sum_{j=1}^Mv_{2m_j}\Bigg).\]
Since the polynomials $p_j$ are $\rH$-invariant, the map $\chi_e:f+V_\fm\to\fm\sslash\rH$ is equivariant with respect to the $\C^*$-actions \eqref{eq slice action} and \eqref{eq base action}.

Now, flatness implies $\chi_e:f+V_\fm\to\fm\sslash\rH$ is open, so its image is an open set $U\subset\fm\sslash\rH$ containing $0=\chi_e(f)$. 
By $\C^*$-equivariance, it follows that $U$ must be $\C^*$-invariant. Since the weights of the $\C^*$-action are positive, we conclude that $U=\fm\sslash\rH$, and thus $\chi_e^\fm$ is surjective. 

For the map $\gamma_e$, surjectivity follows immediately from surjectivity of $\chi_e$. To prove flatness we use a similar argument as above. The argument for flatness of $\chi:\fm\to\fm\sslash\rH$ also applies to $\chi_\cC:Z_\fm\to Z_\fm\sslash\rC$, thus $\chi_\cC$ is flat. Hence, both $\chi_e\circ \psi_e=\gamma_e\circ \chi_\cC$ and $\chi_\cC$ are flat. Therefore, again by \cite[ Corollary 2.2.11]{EGAIV-tome2}, $\gamma_e$ is flat as well.
Finally, a faithfully flat morphism between affine spaces of the same dimension has finite fibers. So $\gamma_e$ has finite fibers since $\dim(Z_\fm\sslash\rC)=\dim(\fm\sslash\rH)$.
\end{proof}
\begin{remark}
    Note that the proof that $\chi_e:f+V_\fm\to\fm\sslash\rH$ is flat and surjective holds for general normal $\fsl_2$-triples $\{f,h,e\}\subset\fh\oplus\fm$.
\end{remark}

The global version of the above picture is given by taking the Hitchin maps from \S\ref{sec: Hitch fib} on the domain and target of the Cayley map $\Psi_e$ defined in \eqref{eq:Cayleymoduli}. Let $\cK$ be the holomorphic frame bundle of $K$. The Hitchin base on the domain is 
\[\cB_\cC=\cB_{K^{m_c+1}}(\tilrG^\R)\times \prod_{j=1}^{\rkfge}H^0(K^{l_j+1})\cong \bigoplus_{j>0} H^0(\cK^{m_j+1}[Z_{2m_j}\sslash \rC]),\]
because $\cB_{K^{m_c+1}}(\tilrG^\R)=H^0(\cK^{m_c+1}[Z_{2m_c}\sslash \rC])$ by the definition of the group $\tilrG^\R$ (see Definitions \ref{def cayley group} and \ref{def: Cayley real form}) and where we used the isomorphism \eqref{eq:isomHitchinbase} to identify $H^0(K^{l_j+1})$ with $H^0(\cK^{m_j+1}[Z_{2m_j}])$ for each $j\neq c$ (see also Lemma \ref{lem identification of Slodowy domain with Cayley group}), as well as the fact that $\rC$ acts trivially on $Z_{2m_j}$ precisely when $j\neq c$. The Hitchin base for $\cM(\rG^\R)$ is $\cB(\rG^\R)=H^0(\cK[\fm\sslash\rH])$. Let $h_\cC$ and $h$ be the respective Hitchin maps.
From the previous discussion, we conclude that the Cayley map $\Psi_e$ is compatible with the Hitchin maps $h_\cC$ and $h$ in the sense of the next proposition.

\begin{proposition}
There is a commutative diagram
\begin{equation}\label{eq comm diagram proper}
    \vcenter{\xymatrix@C=4em{\cM_{K^{m_c+1}}(\tilrG^\R)\times \prod_{j=1}^{\rkfge}H^0(K^{l_j+1})\ar[r]^{\ \ \ \ \ \ \ \ \ \ \ \ \ \ \ \ \ \ \ \ \Psi_e}\ar[d]_{h_\cC}&\cM(\rG^\R)\ar[d]_h\\\cB_\cC\ar[r]^{\Gamma_e}&\cB(\rG^\R)}}
\end{equation}
where $\Gamma_e$ is a proper map.
\end{proposition}
\begin{proof}
 By Proposition \ref{prop:faithflat,surj}, the map $\gamma_e:Z_\fm\sslash\rC\to \fm\sslash\rH$ defines a proper map 
\[\Gamma_e:\cB_\cC\longrightarrow\cB(\rG^\R),\] and the commutativity of the diagram follows from \eqref{eq:adjointmaps-commut}.
\end{proof}
\begin{remark}
    We expect that the map $\Gamma_e$ is an isomorphism, but, for our purposes, being proper is sufficient.
\end{remark}

We now complete the proof of Theorem \ref{thm Cayley map open closed} by showing the Cayley map is closed.

 \begin{proposition}\label{prop:close}
 The image $\mathrm{Im}(\Psi_e)$ of the Cayley map $\Psi_e$ is closed in $\cM(\rG^\R)$.
 \end{proposition}
 \begin{proof}
 Consider a sequence $x_n=\Psi_e(y_n)$ that diverges in $\mathrm{Im}(\Psi_e)$. In particular $y_n$ diverges in the domain of the Cayley map. Since the maps $h_\cC$ and $\Gamma_e$ in the diagram \eqref{eq comm diagram proper} are proper, we conclude that $h_\cC(y_n)$ diverges in $\cB_\cC$ and $\Gamma_e(h_\cC(y_n))$ diverges in the Hitchin base $\cB(\rG^\R)$ of $\cM(\rG^\R)$. 
 Since the diagram \eqref{eq comm diagram proper} commutes and the Hitchin map $h$ is proper, we conclude that $x_n$ diverges in $\cM(\rG^\R)$. Hence the image of $\Psi_e$ is closed in $\cM(\rG^\R)$.
 \end{proof}

 \subsection{Remarks on local minima of energy and components}\label{section components}
The connected components of the moduli spaces of $\rG^\R$-Higgs bundles have been subject to an extensive study through the last three decades (see for example, \cite{selfduality,liegroupsteichmuller,sp4GothenConnComp,UpqHiggs,HermitianTypeHiggsBGG,AndrePGLnR,Oliveira_GarciaPrada_2016,CollierSOnn+1components}). 
Most of the works dealt with $\rG^\R$ in a case-by-case basis, and the main tool, pioneered by Hitchin \cite{selfduality,liegroupsteichmuller}, to detect and count such components was the \emph{Hitchin function} defined by taking the $L^2$-norm of the Higgs field. 
Namely, the $L^2$-norm of the Higgs field with respect to the metric solving the Hitchin equations \eqref{eq Hitchin eq} defines a proper function on the moduli space
\begin{equation}\label{eq:energyfunction}
F:\cM(\rG^\R)\longrightarrow\R,\ \ \ \ (\cE_\rH,\varphi)\mapsto \int_X||\varphi||^2.
\end{equation}
Since proper maps attain their minimum on every closed set we have an inequality
\[|\pi_0(\cM(\rG^\R))|\leq|\pi_0(\text{local min of $F$})|.\]

\begin{remark}
    The strategy is then to classify local minima of $F$ and show that each component of the local minimum define a component of $\cM(\rG^\R)$. There is an obvious global minimum which occurs when the Higgs field $\varphi$ is identically zero. The component count of the global minimum is then given by the component count of the moduli space of polystable $\rH$-bundles. By \cite{ramanathan_1975}, the number of such components is determined by the number of different topological types of $\rH$-bundles.  
\end{remark}

We briefly recall the local minimum criterion for stable Higgs bundles whose second hypercohomology $\HH^2$ vanishes; see for example the Appendix of \cite{so(pq)BCGGO} for details. The local minima of $F$ are in particular fixed points of the $\C^*$-action on $\cM(\rG^\R)$. If $(\cE_\rH,\varphi)$ is a stable $\C^*$-fixed point with $\varphi\neq 0$, then there is a $\Z$-grading $\fg=\bigoplus_{j\in\Z}\fh_j\oplus\fm_j$ and a holomorphic $\rH_0$-bundle $\cE_{\rH_0}$, where $\rH_0\subset\rH$ is the connected with Lie algebra $\fh_0$, such that 
\[\xymatrix{\cE_{\rH_0}[\rH]\cong\cE_\rH&\text{and}&\varphi\in H^0(\cE_{\rH_0}[\fm_{-1}]\otimes K).}\]
As a result, for all $j$, the Higgs field $\varphi$ defines a map 
\begin{equation}\label{eq: subcomplex-K}
\ad_\varphi:\cE_{\rH_0}[\fh_j]\longrightarrow\cE_{\rH_0}[\fm_{j-1}]\otimes K.
\end{equation}
It turns out that if the stable $\rG^\R$-Higgs bundle $(\cE_\rH,\varphi)$ is such that $\HH^2(C^\bullet(\cE_\rH,\varphi))=0$, then it is a local minimum of $F$ if and only if $\ad_\varphi:\cE_{\rH_0}[\fh_j]\xrightarrow{ \   \cong \  }\cE_{\rH_0}[\fm_{j-1}]\otimes K$ is an isomorphism for all $j<0$; see \cite[\S3.4]{BGGHomotopyGroups} and  \cite[Remark 4.16]{UpqHiggs}. 

Recall from Corollary \ref{cor mag implies even and injmap h->m} that if $\{f,h,e\}\subset\fg$ is a magical $\fsl_2$-triple, then $\ad_f:\fh_{j}\to\fm_{j-1}$ is injective for all $j<0$. This implies the $\rG^\R$-Higgs bundle $\Psi_e(\cE_\rC)=(\cE_\rC\star\cE_\rT[\rH],f)$ defines a local minimum of the Hitchin function.
\begin{proposition}
    Let $\{f,h,e\}\subset\fg$ be a magical $\fsl_2$-triple and $\rC\subset\rH$ be its $\rH$-centralizer. Then the $\rG^\R$-Higgs bundle $(\cE_\rC\star\cE_\rT[\rH],f)$ is a local minimum of the Hitchin function $F$.
\end{proposition}

Since the image of the Cayley map $\Psi_e$ is a union of connected components of the moduli space $\cM(\rG^\R)$, it is natural to ask how many components are those. Of course that number equals the number of connected components of the moduli space $\cM_{K^{m_c+1}}(\tilrG^\R)$. This question has been studied whenever $\rG^\R$ is one of the classical groups corresponding to Cases (1), (2) and (3) of Theorem \ref{thm: classification weighted dynkin}. 

The classification of local minima of the Hitchin function \eqref{eq:energyfunction} also applies to $L$-twisted Higgs bundles when $\deg(L)>2g-2$, the only difference being that a metric on $L$ must be fixed to make sense of the $L^2$-norm. Moreover, all the results of \cite[Appendix 1]{so(pq)BCGGO} hold for $L$-twisted $\rG^\R$-Higgs bundles.
To count the components in the image of the Cayley map, one first classifies the stable local minima of the $L$-twisted Hitchin function $F_L:\cM_L(\rG^\R)\to\R$ and then  the polystable local minima. 
As in the $K$-twisted case, the crucial computation to detect the stable local minima among the $\C^*$-fixed points is \cite[Lemma 3.11]{BGGHomotopyGroups} (see also \cite[Remark 4.16]{UpqHiggs}). 
These results can be easily adapted the $L$-twisted setup. 
Consider the $L$-twisted version of \eqref{eq: subcomplex-K},
\begin{equation}\label{eq: subcomplex-L}
\ad_\varphi:\cE_{\rH_0}[\fh_j]\longrightarrow\cE_{\rH_0}[\fm_{j-1}]\otimes L.
\end{equation}
\begin{proposition}\label{prop l-twisted minima criterion}
If $\deg(L)>2g-2$, a stable $L$-twisted $\rG^\R$-Higgs bundle $(\cE_\rH,\varphi)$ with $\varphi\neq 0$ is a local minimum of the Hitchin function $F_L$ if and only if \eqref{eq: subcomplex-L} is an isomorphism for every $j<0$.
\end{proposition}

Recall from Proposition \ref{prop JordanHolder} that a strictly polystable $\rG^\R$-Higgs bundle admits a Jordan--H\"older reduction to a stable $\hat\rG^\R$-Higgs bundle, for a subgroup $\hat\rG^\R\subset\rG^\R$. Such subgroup $\hat\rG^\R$ is independent of the twisting line bundle \cite[\S2.10]{HiggsPairsSTABILITY}.
So the identification of strictly polystable local minima of $F_L$ is done by identifying stable local minima for $F_L$ in $\cM(\hat\rG^\R)$ and then checking if such minima still define local minima in $\cM(\rG^\R)$. 
Using Proposition \ref{prop l-twisted minima criterion} and the minima classification in the literature, we arrive at the following count of Cayley components, i.e.~of connected components in the image of the Cayley map, for Case (4) of Theorem \ref{thm: classification weighted dynkin}.

\begin{proposition}
Let $\rG$ be a complex simple Lie group of type $\rF_4,$ $\rE_6,$ or $\rE_7$ and $\rG^\R\subset\rG$ be the quaternionic real form. Let $\Psi_e$ be the Cayley map from Theorem \ref{thm Cayley map open closed}. Then,
\begin{itemize}
\item $|\pi_0(\mathrm{Im}(\Psi_e))|=3$ for $\rG$ of type $\rF_4$;
\item $|\pi_0(\mathrm{Im}(\Psi_e))|=1$ for $\rG$ the simply connected group of type $\rE_6$;
\item $|\pi_0(\mathrm{Im}(\Psi_e))|=3$ for $\rG$ the adjoint group of type $\rE_6$;
\item $|\pi_0(\mathrm{Im}(\Psi_e))|=1$ for $\rG$ the simply connected group of type $\rE_7$;
\item $|\pi_0(\mathrm{Im}(\Psi_e))|=2$ for $\rG$ the adjoint group of type $\rE_7.$
\end{itemize}
\end{proposition}
\begin{proof}
Suppose $\rG^\R$ is a quaternionic real form of the simply connected group of type $\rF_4$, $\rE_6$, $\rE_7$ or $\rE_8$. 
By Proposition \ref{prop cayley real form class}, the semisimple part $\tilrG^\R$ of the Cayley group $\rG^\R_\cC$ is $\rSL_3\R$, $\rSL_3\C$ and $\SU^*_6$ and $\rE_6^{-26}$, respectively. 
For $\rF_4$ and $\rE_8,$ the adjoint group is simply connected but for $\rE_6$ and $\rE_7$ the fundamental group of the adjoint group is $\Z/3\Z$ and $\Z/2\Z$, respectively, and $\tilrG^\R$ is $\rP\rSL_3\C$ and $\rP\rSU^*_6$, respectively. 
The number of connected components of the image of the Cayley map $\Psi_e$ is equal to the number of connected components of the moduli space $\cM_{K^4}(\tilrG^\R)$. 

For $\tilrG^\R=\rSL_3\R$, the number of connected components of $\cM_K(\rSL_3\R)$ is 3. This was computed in \cite{liegroupsteichmuller} by showing the only nonzero local minima of the Hitchin function arises from Case (1) of Theorem \ref{thm: classification weighted dynkin}. These methods can easily be adapted to the $K^4$-twisted situation and no extra local minima arise. Thus,  $|\pi_0(\mathrm{Im}(\Psi_e))|=3$ for $\rG$ of type $\rF_4$. 
Similarly, when $\tilrG^\R$ is $\rSL_3\C$ or $\rP\rSL_3\C$, there are no nonzero local minima of the Hitchin function by \cite{Oliveira_GarciaPrada_2016} and the number of components is 1 or 3, respectively. These methods also generalize directly to the $K^4$-twisted situation, and give the desired component count. 
Finally, for $\tilrG^\R=\rSU^*_6$ it is shown in \cite[Proposition 4.6]{AndreOscarSUstar} that there are no nonzero local minima of the Hitchin function. This computation also applies to $\tilrG^\R=\rP\rSU_6^*$. These techniques also generalize immediately to the $K^4$-twisted case and give the desired component counts.
\end{proof}

\begin{remark}\label{CayleyCompleteness}
  When $\rG$ has type $\rE_8$, we expect the image of the Cayley map $\Psi_e$ to be connected since the maximal compact of the Cayley group has type $\rF_4$ which is simply connected, and hence has only one topological type. 
  In general, it is expected that the Cayley map is the only source of connected components of the moduli space of $\rG^\R$-Higgs bundles which are not labeled by topological invariants of $\rG^\R$-bundles. This has been proven for the real groups $\rSL_n\R$ \cite{liegroupsteichmuller,TopologicalComponents}, $\rU_{p,q}$ \cite{UpqHiggs,chains-2018}, $\rPGL_n\R$ \cite{AndrePGLnR}, $\rSU_{2n}^*$ \cite{AndreOscarSUstar},  $\rSO_{p,q}$ with $p=1$ or $2<p\leq q$ \cite{so(pq)BCGGO}, $\rSO_{2,3}$ \cite{AndreQuadraticPairs,nonmaxSp4} and $\rSp_{2p,2q}$ \cite{Sp(2p2q)modulispaceconnected}. Moreover, when there is a Cayley map for these groups, the number of connected components in the image of the Cayley map is counted. 
\end{remark}

\section{Positive surface group representations}\label{sec: Positive surface group representations}
In this section we deduce properties of the surface group representations associated to Higgs bundles in the image of the Cayley map via the nonabelian Hodge correspondence. 

For this section $\rG$ is a complex simple Lie group and $\rG^\R\subset \rG$ is a real form. We fix a maximal compact subgroup $\rH^\R\subset \rG^\R$ with complexification $\rH$, and consider the Cartan decomposition $\fg^\R=\fh^\R\oplus\fm^\R$, and its complexification $\fg=\fh\oplus\fm$.

\subsection{Surface group representations}
Let $\Sigma$ be a compact smooth oriented surface, without boundary, and $\pi_1\Sigma$ be its fundamental group. Consider the space $\Hom(\pi_1\Sigma,\rG^\R)$ of all representations of $\pi_1\Sigma\to\rG^\R$.
The group $\rG^\R$ acts on $\Hom(\pi_1\Sigma,\rG^\R)$ by conjugation. 
Recall that a representation $\pi_1\Sigma \to \rG^\R$ is called 
\emph{reductive} if its composition with the adjoint representation 
of $\rG^\R$ in $\fg^\R$
decomposes as a direct sum of irreducible representations. Let $\Hom^+(\pi_1\Sigma,\rG^\R)$ be the $\rG^\R$-invariant the subspace consisting of reductive representations.
\begin{definition}
The \emph{$\rG^\R$-character variety} $\cX(\rG^\R)$ of 
$\pi_1\Sigma$ is defined as the orbit space  
\[\cX(\rG^\R) = \Hom^{+}(\pi_1\Sigma,\rG^\R) /\rG^\R.\]
\end{definition}

\begin{example}\label{EX: Fuchsian reps}
Let $\rS^\R$ be $\rPSL_2\R$ or $\rSL_2\R$. The space of {\em Fuchsian representations} $\Fuch(\rS^\R)\subset\cX(\rS^\R)$ is defined to be the subset of conjugacy classes of {\em faithful} representations $\rho_{\Fuch}:\pi_1\Sigma\to\rS^\R$ with {\em discrete image}.  
The space $\Fuch(\rPSL_2\R)$ defines two connected components of $\cX(\rPSL_2\R)$ \cite{TopologicalComponents} and is in one-to-one correspondence with the Teichm\"uller space of isotopy classes of marked Riemann surface structures on the surface $\Sigma$ with either of its orientations. 
Every Fuchsian representation $\rho\in\Fuch(\rPSL_2\R)$ lifts to a representation $\tilde\rho_{\Fuch}\in \Fuch(\rSL_2\R)$. There are $2^{2g}$ such lifts and each lift lies in a distinct connected component of $\cX(\rSL_2\R)$. 

If $\fg^\R$ is the Lie algebra of $\rG^\R$ and $e\in\fg^\R$ is a nonzero nilpotent, the inclusion of the associated $\fsl_2\R$-subalgebra in $\fg^\R$ induces an embedding $\iota_e:\rS^\R\to\rG^\R$, which in turn defines a map on character varieties
\begin{equation}\label{eq fuch locus}
    \iota_e:\Fuch(\rS^\R)\to\cX(\rG^\R).
\end{equation}
Such maps define ways to deform the Teichm\"uller space of $\Sigma$ inside the character variety $\cX(\rG^\R)$. We will call the set $\iota_e(\Fuch(\rS^\R))$ the {\em Fuchsian locus}.
\end{example}

The following theorem links the $\G^\R$-character variety and the $\G^\R$-Higgs bundle moduli space and is known as the \emph{nonabelian Hodge correspondence}. It was proven by Hitchin \cite{selfduality}, Donaldson \cite{harmoicmetric}, Corlette \cite{canonicalmetrics} and Simpson \cite{SimpsonVHS} in various generalities (see also \cite{HiggsPairsSTABILITY}). 

\begin{theorem}\label{Nonabelian Hodge Correspondence}
  Let $\Sigma$ be a closed oriented surface of genus $g\geq2$ and $\G^\R$ be a real semisimple Lie group. 
  For each Riemann surface structure $X$ on $\Sigma$, there is a homeomorphism between the moduli space $\cM(\G^\R)$ of $\G^\R$-Higgs bundles on $X$ and the $\rG^\R$-character variety $\cX(\G^\R)$. 
\end{theorem} 
One direction of the nonabelian Hodge correspondence is given by considering solutions to the Hitchin equations \eqref{eq Hitchin eq}. Namely, given a polystable $\rG^\R$-Higgs bundle $(\cE_\rH,\varphi)$, there is a metric $h$ on $\cE_\rH$ such that $F_h+[\varphi,-\tau_h(\varphi)]=0$, where $F_h$ is the curvature of the Chern connection $A_h$ associated to $h.$ 
If $E_h\subset\cE_\rH$ is the associated $\rH^\R$-bundle, then the connection $D=A_h+\varphi-\tau(\varphi)$ defines a flat connection on the smooth $\rG^\R$-bundle $E_h[\rG^\R]$. The flat connection $D$ defines the associated reductive representation $\rho:\pi_1\Sigma\to \rG^\R.$

For the other direction, let $\rho:\pi_1\Sigma\to\rG^\R$ be a reductive representation and consider the associated  $\rG^\R$-bundle with flat connection $D_\rho$,
\[E_\rho=\widetilde\Sigma\times_\rho\rG^\R,\]
where $\widetilde\Sigma$ is the universal cover of $\Sigma.$ Each metric $h$ on $E_\rho$ defines a decomposition of the flat connection $D_\rho=A_h+\Psi$, where $A_h$ preserves the metric. Fixing a Riemann surface structure $X$ on $\Sigma$ allows us to decompose $A_h$ and $\Psi$ into $(1,0)$ and $(0,1)$-parts. 
If $E_h\subset E_\rho$ is the $\rH^\R$-bundle associated to $h$, then the $(0,1)$-part of $A_h$ defines a holomorphic structure on the $\rH$-bundle $E_h[\rH]$ and the $(1,0)$-part of $\Psi$ defines a section of $E_h[\fm]\otimes K.$ By Corlette's Theorem \cite{canonicalmetrics}, there is a metric $h$ on $E_\rho$ (the \emph{harmonic metric}) which defines a polystable $\rG^\R$-Higgs bundle
\[(\cE_\rH,\varphi)=((E_h[\rH],A_h^{0,1}),\Psi^{1,0}).\]
Note that for complex groups $\rG$ we have $\rH=\rG$ and the underlying smooth bundle of $\cE_\rG$ is $E_\rho=E_h[\rG].$
\begin{definition}
Let $\rG$ be a complex reductive Lie group, $\rG^\R\subset\rG$ be a real form, and $\hat\rG^\R\subset\rG^\R$ be a reductive subgroup. Let $\hat\rH\subset\rH\subset\rG$ be the complexifications of maximal compact subgroups of $\hat\rG^\R\subset\rG^\R$ and $\fg=\fh\oplus\fm$ and $\hat\fg=\hat\fh\oplus\hat\fm$ be associated complexified Cartan decompositions with $\hat\fm\subset\fm$.
\begin{itemize}
    \item A representation $\rho:\pi_1\Sigma\to\rG^\R$ \emph{factors} through $\hat\rG^\R$ if $\rho=\iota\circ\hat\rho$, where $\hat\rho:\pi_1\Sigma\to\hat\rG^\R$ and $\iota:\hat\rG^\R\to\rG^\R$ is the inclusion. 
    \item A $\rG^\R$-Higgs bundle $(\cE_\rH,\varphi)$ \emph{reduces} to a $\hat\rG^\R$-Higgs bundle $(\cE_{\hat\rH},\varphi)$ if there is a holomorphic $\hat\rH$-subbundle $\cE_{\hat\rH}\subset\cE_\rH$ such that $\varphi\in H^0(\cE_{\hat\rH}[\hat\fm]\otimes K)$. 
\end{itemize}
\end{definition}

The following is an immediate consequence of the nonabelian Hodge correspondence. 

\begin{proposition}
 A reductive representation $\rho:\pi_1\Sigma\to\rG^\R$ factors through a reductive subgroup $\hat\rG^\R\subset\rG^\R$ if and only if the associated $\rG^\R$-Higgs bundle $(\cE_\rH,\varphi)$ reduces to a $\hat\rG^\R$-Higgs bundle. In particular, $\rho$ factors through a compact subgroup if and only if the Higgs field $\varphi$ is identically zero.
\end{proposition}

The \emph{centralizer of a representation} $\rho:\pi_1\Sigma\to\rG^\R$ is the reductive subgroup  
\[Z_{\rG^\R}(\rho)=\{g\in\rG^\R~|~g\cdot\rho(\gamma)\cdot g^{-1}=\rho(\gamma) \text{ for all $\gamma\in\pi_1\Sigma$}\}.\]
The double centralizer $Z_{\rG^\R}(Z_{\rG^\R}(\rho))\subset\rG^\R$ is reductive, and by construction, $\rho$ factors through $Z_{\rG^\R}(Z_{\rG^\R}(\rho))$. 
\begin{proposition}
    Let $\rG$ be a complex reductive Lie group and $\rho:\pi_1\Sigma\to\rG$ be a reductive representation. Then the centralizer $Z_\rG(\rho)$ of $\rho$ is naturally a subgroup of the automorphism group of the associated $\rG$-Higgs bundle. 
\end{proposition}
\begin{proof}
    Set $\hat\rG=Z_\rG(Z_\rG(\rho))$ and write $\rho=\iota\circ\hat\rho$, where $\hat\rho:\pi_1\Sigma\to\hat\rG$. The flat bundle $E_\rho$ is given by $E_{\hat\rho}[\rG].$ Thus, the associated $\rG$-Higgs bundle $(\cE_\rG,\varphi)$ reduces to a $\hat\rG$-Higgs bundle
    \[(\cE_\rG,\varphi)=(\cE_{\hat\rG}[\rG],\varphi).\]
    Any element $g\in Z_\rG(\rho)$ defines a constant gauge transformation $g$ of the flat bundle $E_\rho=E_{\hat\rho}[\rG].$ 
    Since $\rG$ and $\hat\rG$ are complex, this defines a gauge transformation of the resulting $\rG$-Higgs bundle. 
   But the constant gauge transformation $g$ acts trivially on $(\cE_{\hat\rG}[\rG],\varphi)$ and hence defines an element of $\Aut(\cE_{\hat\rG}[\rG],\varphi)$. 
\end{proof}

\begin{proposition}
Let $(\cE_\rH,\varphi)$ be a $\rG^\R$-Higgs bundle and $(\cE_\rH[\rG],\varphi)$ be the $\rG$-Higgs bundle obtained by extension of structure group. If the second hypercohomology group $\HH^2(C^\bullet(\cE_\rH,\varphi))$ from \eqref{EQ deformation complex DEF} vanishes, then we have an isomorphism 
\[\HH^0(C^\bullet(\cE_\rH,\varphi))\cong\HH^0(C^\bullet(\cE_\rH[\rG],\varphi)).\]
In particular, we have an isomorphism of the Lie algebras $\aut(\cE_\rH,\varphi)\cong\aut(\cE_\rH[\rG],\varphi)$.
\end{proposition}
\begin{proof}
Serre duality for the complex $C^\bullet(\cE_\rH[\rG],\varphi)$ yields an isomorphism  
\[\HH^0(C^\bullet(\cE_\rH[\rG],\varphi))\cong \HH^0(C^\bullet(\cE_\rH,\varphi))\oplus \HH^2(C^\bullet(\cE_\rH,\varphi))^*;\]    
see \cite[Corollary 3.16]{HiggsPairsSTABILITY}. So $\HH^2(C^\bullet(\cE_\rH,\varphi))=0$ implies $\HH^0(C^\bullet(\cE_\rH[\rG],\varphi))\cong \HH^0(C^\bullet(\cE_\rH,\varphi)).$
\end{proof}

We are now set up to prove Theorem A from the Introduction.

\begin{theorem}\label{thm character variety components}
Let $\rG$ be a complex simple Lie group with Lie algebra $\fg$ and $\{f,h,e\}\subset\fg$ be a magical $\fsl_2$-triple with canonical real form $\rG^\R\subset\rG$. Let $\Sigma$ be a closed orientable surface of genus $g\geq2$ and $\cX(\rG^\R)$ be the $\rG^\R$-character variety. Then, there exists a nonempty open and closed subset
    \[\cP_e(\rG^\R)\subset\cX(\rG^\R),\]
 such that every $\rho\in\cP_e(\rG^\R)$ has a compact centralizer and does not factor through a compact subgroup.
Moreover, the components $\cP_e(\rG^\R)$ contain the Fuchsian locus defined by $\{f,h,e\}$
 \[\iota_e(\Fuch(\rS^\R))\subset\cP_e(\rG^\R),\]
 where $\iota_e:\rS^\R\hookrightarrow\rG^\R$ is the subgroup associated to the $\fsl_2\R$-subalgebra defined by $\{f,h,e\}$.
\end{theorem}
\begin{remark}
    The components $\cP_e(\rG^\R)\subset\cX(\rG^\R)$ are obtained by applying the nonabelian Hodge correspondence to the components defined by the Cayley map $\Psi_e$ from Theorem \ref{thm Cayley map open closed}. For the magical $\fsl_2$-triples from Case (1) of Theorem \ref{thm: classification weighted dynkin}, the components $\cP_e(\rG^\R)$ are the spaces of Hitchin representations and the above theorem was proven by Hitchin in \cite{liegroupsteichmuller}. For the magical triples from Case (2) of Theorem \ref{thm: classification weighted dynkin}, the components $\cP_e(\rG^\R)$ are the spaces of maximal representations and most aspects of the above theorem were proven in \cite{BGRmaximalToledo}. For Case (3), the statement was proven in \cite{so(pq)BCGGO}.
\end{remark}
Since the center of a proper parabolic $\rP^\R\subset\rG^\R$ is not compact, the following is immediate.
\begin{corollary}\label{cor:norepfactors properparab}
    If $\rho$ is any representation in $\cP_e(\rG^\R)$, then there is no proper parabolic subgroup $\rP^\R\subset\rG^\R$ such that $\rho$ factors through $\rP^\R$.
\end{corollary}
\begin{proof}[Proof of Theorem \ref{thm character variety components}]
    By Theorem \ref{thm Cayley map open closed}, the image of the Cayley map $\Psi_e$ defines nonempty connected components of the moduli space $\cM(\rG^\R)$. Applying the nonabelian Hodge correspondence to these components defines an nonempty, open and closed subset $\cP_e(\rG^\R)$ of the $\rG^\R$-character variety $\cX(\rG^\R).$
    Since the Higgs field in the image of the Cayley map is never zero, the associated representations never factor through compact subgroups. 

    By construction of the Cayley map, when $\cE_\rC$ is the trivial $\rC$-bundle and all sections $\tilde\psi_{m_c}$ and $q_{j}$ are zero, the resulting Higgs bundle reduces to the uniformizing $\rS^\R$-Higgs bundle for the Riemann surface $X$. Applying the nonabelian Hodge correspondence to this point defines a point in the Fuchsian locus $\iota_e(\Fuch(\rS^\R))$. Actually, $\iota_e(\Fuch(\rS^\R))$ corresponds, under the nonabelian Hodge correspondence, to 
    \[\Psi_e\left(\left\{((\cE_{\rC},0),q_2,0,\ldots,0)\,|\,\cE_{\rC}\text{ trivial, }q_2\in H^0(K^2)\right\}\right).\]
     Thus, $\cP_e(\rG^\R)$ contains the Fuchsian locus defined by the magical $\fsl_2$-triple. 

    Finally we show that the centralizer is compact. Let $\rho:\pi_1\Sigma\to\rG^\R$ be a representation in such a component and let $Z_{\rG^\R}(\rho)\subset\rG^\R$ be its centralizer. 
    Considering the induced complex representation $\rho:\pi_1\Sigma\to\rG^\R\subset\rG$, we have 
    \[Z_{\rG^\R}(\rho)=Z_\rG(\rho)\cap\rG^\R.\]
   It suffices to show that the Lie algebra $\fz_{\fg^\R}(\rho)\subset\fg^\R$ is contained in $\fh^\R$.  By Proposition  \ref{prop automorphism} and Proposition \ref{prop injective Cayley}, the automorphism group $\Aut(\cE_\rH,\varphi)$ is identified with a closed subgroup of $\rC$, and hence $\aut(\cE_\rH,\varphi)\subset\fc$. Thus, 
    \[\fz_{\rG}(\rho)\subset\aut(\cE_\rH[\rG],\varphi)=\aut(\cE_\rH,\varphi)\subset\fc.\]
    Since $\fg^\R\cap\fc=\fc^\R\subset\fh^\R$, we conclude that the centralizer $Z_{\rG^\R}(\rho)$ of $\rho$ is compact.
\end{proof}

Points in the domain of the Cayley map \eqref{eq:Cayleymoduli} are given by 
\[((\cE_\rC,\tilde\psi_{m_c}),q_1,\ldots,q_{\rkfge})\in \cM_{K^{m_c+1}}(\tilrG^\R)\times \prod_{j=1}^{\rkfge} H^0(K^{l_j+1}).\]
When $\tilde\psi_{m_c}=0,$ the associated Higgs bundle reduces to a $\rG(e)^\R*\rC^\R$-Higgs bundle, where $\rG(e)^\R\subset\rG^\R$ is the connected group with Lie algebra $\fg(e)^\R$ and $\rC^\R$ is the compact real form of $\rC$. 
Moreover, by construction of the Cayley map the Higgs field of the associated Higgs bundle is in the image of the Cayley map for the magical $\fsl_2$-triple in $\fg(e)$ from Case (1) of Theorem \ref{thm: classification weighted dynkin}. Hence, the associated representations $\rho:\pi_1\Sigma\to\rG^\R$ are of the form 
$\rho=\rho_{Hit}*\rho_{\rC^\R}$, where $\rho_{Hit}:\pi_1\Sigma\to\rG(e)^\R$ is a Hitchin representation into $\rG(e)^\R$ and $\rho_{\rC^\R}:\pi_1\Sigma\to\rC^\R$ is any representation into the compact group $\rC^\R$. In particular, we have the following proposition. 
\begin{proposition}\label{prop hitchin rep in Pe}
 Each of the sets $\cP_e(\rG^\R)$ contains all representations of the form 
 \[\rho=\rho_{Hit}*\rho_{\rC^\R}:\pi_1\Sigma\longrightarrow\rG(e)^\R*\rC^\R\subset\rG^\R,\]
 where $\rho_{Hit}:\pi_1\Sigma\to\rG(e)^\R$ is any $\rG(e)^\R$-Hitchin representation and $\rho_{\rC^\R}:\pi_1\Sigma\to\rC^\R$ is any $\rC^\R$-representation.
\end{proposition}

\subsection{Positive Anosov representations}\label{subsec: Positive Anosov representations}
Anosov representations were introduced by Labourie in \cite{AnosovFlowsLabourie} and have many interesting geometric and dynamic properties, generalizing convex cocompact representations into rank one Lie groups. 
Important examples of Anosov representations include Fuchsian representations, quasi-Fuchsian representations, Hitchin representations into split real groups and maximal representations into Lie groups of Hermitian type. We will briefly recall the important points for our applications and refer the reader to  \cite{AnosovFlowsLabourie,guichard_wienhard_2012,AnosovAndProperGGKW,KLPDynamicsProperCocompact} for more details. 

Let $\rG^\R$ be a real semisimple Lie group, $\rP^\R\subset\rG^\R$ be a proper parabolic subgroup and $\rL^\R\subset\rG^\R$ be a Levi subgroup of $\rP^\R$. If $\rP^\R_{opp}$ is the opposite parabolic of $\rP^\R$, then $\rL^\R=\rP^\R\cap\rP^\R_{opp}$ and the homogeneous space $\rG^\R/\rL^\R$ is realized as the unique open $\rG^\R$-orbit in $\rG^\R/\rP^\R\times \rG^\R/\rP^\R_{opp}$. 
The pairs of elements $(x,y)\in \rG^\R/\rP^\R\times \rG^\R/\rP^\R_{opp}$ which lie in this open orbit are called \emph{transverse}. 

\begin{definition}\label{DEF: Anosov rep}
    Let $\Sigma$ be a closed orientable surface of genus $g\geq 2$. Let $\partial_\infty\pi_1\Sigma$ be the Gromov boundary of the fundamental group $\pi_1\Sigma$. Topologically $\partial_\infty\pi_1\Sigma\cong\R\mathbb{P}^1$. A representation $\rho:\pi_1\Sigma\to\rG^\R$ \emph{is $\rP^\R$-Anosov} if there exists a unique continuous boundary map $\xi_\rho:\partial_\infty\pi_1\Sigma\to\rG^\R/\rP^\R$
which satisfies the following properties:
\begin{itemize}
    \item Equivariance: $\xi(\gamma\cdot x)=\rho(\gamma)\cdot\xi(x)$ for all $\gamma\in\pi_1\Sigma$ and all $x\in\partial_\infty\pi_1S$.
    \item Transversality: for all distinct $x,y\in\partial_\infty\pi_1S$ the generalized flags $\xi(x)$ and $\xi(y)$ are transverse.
    \item Dynamics preserving: see \cite{AnosovFlowsLabourie,guichard_wienhard_2012,AnosovAndProperGGKW,KLPDynamicsProperCocompact} for the precise notion. 
\end{itemize}
The map $\xi_\rho$ will be called the {\em $\rP^\R$-Anosov boundary curve}.
\end{definition}

An important property of Anosov representations is that they are stable, that is, they define an \emph{open} set of the character variety \cite{AnosovFlowsLabourie}. However, in general, the set of Anosov representations is \emph{not closed}. For example, the set of Anosov representations in the $\rPSL_2\C$-character variety is the open set of quasi-Fuchsian representations, which is not closed. 
On the other hand, the set of Hitchin representations in split real groups and the set of maximal representations in Hermitian Lie groups do define sets of Anosov representations which are both open and closed in the character variety. 
For both of these cases, the representations satisfy an additional positivity property \cite{AnosovFlowsLabourie,fock_goncharov_2006,MaxRepsAnosov}. 
These notions have been unified into the notion of $\Theta$-positive Anosov representations by Guichard--Wienhard \cite{PosRepsGWPROCEEDINGS,GuichardWienhardPosFull}, which we now briefly recall. 

Let $\rP^\R\subset\rG^\R$ be a parabolic subgroup, $\rL^\R\subset\rP^\R$ be a Levi subgroup and $\rU^\R\subset\rP^\R$ be the unipotent radical. The Lie algebra $\fp^\R$ decomposes $\Ad_{\rL^\R}$-invariantly as $\fp^\R=\fl^\R\oplus\fu^\R.$ Moreover, the nilpotent Lie algebra $\fu^\R$ decomposes into irreducible $\rL^\R$-representations
\[\fu^\R=\bigoplus_{\beta\in\fz(\fl^\R)^*}\fu_\beta.\]
The parabolic subgroup $\rP^\R$ is determined by fixing a restricted root system $\Delta$ of a maximal $\R$-split torus of $\rG^\R$ and then choosing a subset $\Theta\subset\Delta$ of simple roots. To each simple root $\beta_j\in\Theta,$ there is a corresponding irreducible $\rL^\R$-representation $\fu_{\beta_j}.$

\begin{definition}\cite[Definition 4.2]{PosRepsGWPROCEEDINGS}\label{DEF: Positive Structure}
    A pair $(\rG^\R,\rP_\Theta^\R)$ admits \emph{a $\Theta$-positive structure} if, for all $\beta_j\in\Theta,$ the $\rL_\Theta^\R$-representation space $\fu_{\beta_j}$ has an $(\rL^\R_{\Theta})_0$-invariant acute convex cone $c_{\beta_j}^\Theta$, where $(\rL^\R_\Theta)_0$ denotes the identity component of $\rL^\R_\Theta$. 
\end{definition}

The set of pairs $(\rG^\R,\rP_\Theta^\R)$ which admit a positive structure were classified in \cite[Theorem 4.3]{PosRepsGWPROCEEDINGS}, and we now relate this classification with the classification of magical $\fsl_2$-triples given in Theorem \ref{thm: classification weighted dynkin}. 
Fix a magical $\fsl_2$-triple $\{f,h,e\}\subset\fg$ and let $\fg=\fh\oplus\fm$ be the complexified Cartan decomposition defined by the involution $\sigma_e$ from \eqref{eq magical involution}. Fix an involution $\tau_e:\fg\to\fg$ which commutes with $\sigma_e$. 
Recall that $\tau_e$ defines the canonical real form $\fg^\R$ associated to the magical $\fsl_2$-triple.
Recall also from \S\ref{subsec: Real nilpotents and the Sekiguchi correspondence} that $\{f,h,e\}$ is a normal $\fsl_2$-triple and its Cayley transform $\gamma^{-1}(\{f,h,e\})=\{\hat f,\hat h,\hat e\}$ is a Cayley triple (see \eqref{eq inverse cayley transform}) which is a magical $\fsl_2\R$-triple of $\fg^\R$. 
In particular, the nonzero nilpotent $\hat e$ belongs to $\fg^\R$ and hence it determines a parabolic subgroup $\rP_{\hat e}^\R\subset\rG^\R$ of the canonical real form. Comparing the two classification yields the following theorem. 

\begin{theorem}\label{thm pos and mag}
    A pair $(\rG^\R,\rP^\R_\Theta)$ admits a $\Theta$-positive structure if and only if there is a magical $\fsl_2\R$-triple $\{\hat f,\hat h,\hat e\}\subset\fg^\R$ such that $(\rG^\R,\rP^\R_\Theta)=(\rG^\R,\rP_{\hat e}^\R)$. In particular, there are four such families
    \begin{enumerate}
        \item $\rG^\R$-split and $\rP^\R_\Theta$ is the Borel subgroup.
        \item $\rG^\R$ is a Hermitian group of tube type and $\rP^\R_\Theta$ is the maximal parabolic associated the Shilov boundary.
        \item $\rG^\R$ is locally isomorphic to $\rSO_{p,q}$ and $\rP^\R_\Theta$ stabilizes an isotropic flag of the form  
        \[\R\subset\R^2\subset\cdots\subset\R^{p-1}\subset\R^{q+1}\subset\cdots\subset\R^{p+q-1}\subset\R^{p+q}.\]
        \item $\rG^\R$ is a quaternionic real form of $\rE_6$, $\rE_7$, $\rE_8$ or $\rF_4$, so that its restricted root system is that of $\rF_4$, and $\Theta=\{\alpha_1,\alpha_2\}$, where 
       \begin{center}
            \begin{tikzpicture}[scale=.4]
    \draw (-1,0) node[anchor=east]  {$\rF_4:$};
  \foreach \x in {0,...,3}
    \draw[xshift=\x cm,thick] (\x cm,0) circle (.25cm);
    \draw[thick] (0.25 cm,0) -- +(1.5 cm,0);
    \draw[thick] (2.25 cm, .1 cm) -- +(1.5 cm,0);
    \draw[thick] (2.8 cm, 0) -- +(.3cm, .3cm);
    \draw[thick] (2.8 cm, 0) -- +(.3cm, -.3cm);
    \draw[thick] (2.25 cm, -.1 cm) -- +(1.5 cm,0);
    \draw[thick] (4.25 cm, 0) -- +(1.5cm,0);
    \node at (0,0) [below = 1 mm ] {{\scriptsize $\alpha_1$}};
    \node at (2,0) [below = 1 mm ] {{\scriptsize $\alpha_2$}};
    \node at (4,0) [below = 1 mm ] {{\scriptsize $\alpha_3$}};
    \node at (6,0) [below = 1 mm ] {{\scriptsize $\alpha_4$}};
        \node at (7,0)  {.};
  \end{tikzpicture}
        \end{center}
    \end{enumerate}
\end{theorem}
\begin{proof}
    By \cite[Theorem 4.3]{PosRepsGWPROCEEDINGS}, the set of pairs $(\rG^\R,\rP_\Theta)$ which admit a $\Theta$-positive structure are given by the above list. 
    The correspondence with magical $\fsl_2$-triples follows from Theorem \ref{thm: classification weighted dynkin} and Proposition \ref{prop canonical real form}.
\end{proof}

\begin{remark}\label{rem descriptions of cones}
    The cones in Cases (1) and (2) are the only relevant cones. For Case (1), we have $\fu_{\beta_j}\cong\R$ for all $j$ and the cone is $\R^+\subset\R.$ For Case (2), the cones are related to the causal structure on the Shilov boundary (see \cite{SojiCausalStruct}). For example, when $\rG^\R=\rSU_{n,n},$ the positive cone is the set of positive definite $(n\times n)$-matrices inside the set of all $(n\times n)$-matrices and for $\rG^\R=\rSO_{2,N-2}$ the cone is the light cone in $\R^{1,N-3}$. For Case (3), there are $p-2$ cones isomorphic to $\R^+\subset\R$ and one isomorphic to the cone in Case (2) for $\rSO_{2,q-p+1}.$ 
    For Case (4), there are two invariant cones. One is $\R^+\subset\R$ corresponding to the simple root space $\tilde\alpha$ from \eqref{eq G0 invariant g2 decomp}, and the other is isomorphic to the cone in Case (2) for $\rSp_{6}\R$, $\rSU_{3,3}$, $\rSO^*_{12}$ and $\rE_7^{-25}$, for $\rG^\R$ given by the quaternionic real form of $\rF_4,$ $\rE_6,$ $\rE_7$ and $\rE_8$ respectively.  
\end{remark}

For pairs $(\rG^\R,\rP^\R_\Theta)$ which admit a $\Theta$-positive structure, there is a distinguished semigroup $\rU_{\Theta,+}^\R\subset\rU_\Theta^\R$ of the unipotent radical \cite[Theorem 4.5]{PosRepsGWPROCEEDINGS}, which allows one to define a notion of positively ordered triples in  $\rG^\R/\rP^\R_\Theta$ as follows. Since the group $\rG^\R$ acts transitively on the space of transverse points in $\rG^\R/\rP^\R$, any two points $x,y\in\rG^\R/\rP^\R_{\Theta}$ can be mapped to the points $(x_+,x_-)$ associated to $\rP^\R_\Theta$ and $\rP_{\Theta,opp}^\R$ respectively.   
\begin{definition}\cite[Definition 4.6]{PosRepsGWPROCEEDINGS}
    \label{def: Positive triples in G/P}
 Let $x_+,x_-\in\rG^\R/\rP^\R_\Theta$ be the points associated to $\rP^\R_\Theta$ and $\rP_{\Theta,opp}^\R$ respectively. 
A point $x_0$ which is transverse to $x_+$ is the image of $x_-$ under a unique element $u_0\in\rU^\R_\Theta$. The triple $(x_+,x_0,x_-)$ is {\em positive} if $u_0\in \rU^\R_{\Theta,+}.$
 \end{definition} 
 With respect to the orientation on $\partial_\infty\Gamma$, we say that a triple of pairwise distinct points $(a,b,c)$ is a {\em positive triple} if the points appear in this order.
\begin{definition}
\cite[Definition 5.3]{PosRepsGWPROCEEDINGS}\label{DEF: Positive rep}
   Suppose the pair $(\rG^\R,\rP^\R_\Theta)$ admits a $\Theta$-positive structure. Then a $\rP^\R_\Theta$-Anosov representation $\rho:\pi_1\Sigma\to\rG^\R$ is \emph{$\Theta$-positive} if the Anosov boundary curve $\xi:\partial_\infty\pi_1\Sigma\to \rG^\R/\rP^\R_\Theta$ sends positively ordered triples in $\partial_\infty\pi_1S$ to positive triples in $\G^\R/\rP^\R_\Theta.$
\end{definition}
\begin{remark}As mentioned in the introduction, Guichard--Wienhard conjecture that the set $\Theta$-positive Anosov representations is an open and closed subset of $\cX(\rG^\R)$. This conjecture aims to characterize connected components of the character variety consisting entirely of discrete and faithful representations as precisely those arising from positive Anosov representations. Such components are now commonly referred to as \emph{higher rank Teichm\"uller spaces} (cf.~\cite{PosRepsGWPROCEEDINGS}). 
\end{remark}


The construction of the positive semigroup $\rU_{\Theta,+}^\R\subset\rU^\R_\Theta$ is defined by exponentiating certain combinations of elements in the cones $c_{\beta_j}\subset\fu_{\beta_j}$. 
Namely, there is a certain Weyl group $\cW_\Theta$, and if $w_\Theta=\sigma_{i_1} \cdots \sigma_{i_l}$ is an expression for the longest word in $\cW_\Theta$, it defines the map 
\[F_{\sigma_{i_1} \cdots \sigma_{i_l}}:c_{\beta_{i_1}}^0\times \cdots \times c_{\beta_{i_l}}^0\longrightarrow \rU^\R_\Theta~;~ F_{\sigma_{i_1} \cdots \sigma_{i_l}}(v_{i_1},\ldots,v_{i_l})=\exp(v_{i_1})\cdots\exp(v_{i_l}),\]
where $c_{\beta_{i_j}}^0$ is the interior of $c_{\beta_{i_l}}.$ 
By \cite[Theorem 4.5]{PosRepsGWPROCEEDINGS}, the semigroup $\rU^\R_{\Theta,+}\subset\rU^\R_\Theta$ is given by 
\[\rU^\R_{\Theta,+}=F_{\sigma_{i_1} \cdots \sigma_{i_l}}(c_{\beta_{i_1}}^0\times \cdots \times c_{\beta_{i_l}}^0).\]
 
Recall from Proposition \ref{prop subalge g(e)}, that if $\{f,h,e\}\subset\fg$ is a magical $\fsl_2$-triple and $\fc\subset\fg$ is its centralizer, then we denoted the semisimple part of the centralizer of $\fc$ by $\fg(e)\subset\fg.$ For magical triples we showed that $\fg(e)$ is simple and $\{f,h,e\}\subset\fg(e)$ is a principal $\fsl_2$-triple in $\fg(e).$ The next result relates the Weyl group $\cW_\Theta$ with the Weyl group of $\fg(e)$, for each one of positive families from Theorem \ref{thm pos and mag}. 
\begin{proposition}
   Let $\{f,h,e\}\subset\fg$ be a magical $\fsl_2$-triple with canonical real form $\rG^\R$ and let $\fg(e)\subset\fg$ be the semisimple part of the double centralizer of $\{f,h,e\}.$ Then the relevant Weyl group $\cW_\Theta$ used to define the semigroup $\rU^\R_{\Theta,+}$ is the Weyl group of $\fg(e).$ In particular, 
   \begin{enumerate}
        \item For Case (1) of Theorem \ref{thm pos and mag}, $\fg(e)=\fg$ and $\cW_\Theta$ is the Weyl group of $\fg.$
        \item For Case (2) of Theorem \ref{thm pos and mag}, $\fg(e)=\langle f,h,e\rangle$ and $\cW_\Theta$ is the Weyl group of $\fsl_2\C.$
        \item For Case (3) of Theorem \ref{thm pos and mag}, $\fg(e)\cong\fso_{2p-1}\C$ and $\cW_\Theta$ is the Weyl group of $\fso_{2p-1}\C.$
        \item For Case (4) of Theorem \ref{thm pos and mag}, $\fg(e)\cong Lie(\rG_2)$ and $\cW_\Theta$ is the Weyl group of $Lie(\rG_2).$
    \end{enumerate} 
\end{proposition}


Recall that the canonical real form $\tau_e:\fg\to\fg$ associated to a magical $\fsl_2$-triple $\{f,h,e\}$ preserves the subalgebra $\fg(e)\oplus\fc$ and the fixed point set defines a subalgebra 
\[\fg(e)^\R\oplus\fc^\R\subset\fg^\R.\] 
Here $\fg(e)^\R$ is the split real form of $\fg(e)$ and contains the Cayley transform $\{\hat f,\hat h,\hat e\}$ of $\{f,h,e\}$ and $\fc^\R$ is the compact real form of $\fc.$ This defines an embedding of the connected subgroup with Lie algebra $\fg(e)^\R$ 
\[\iota:\rG(e)^\R\longrightarrow\rG^\R.\]
Moreover, the intersection of the parabolic $\rP_\Theta=\rP_{\hat e}\subset\rG^\R$ defined by $\hat e$ with $\rG(e)^\R$ is the Borel subgroup $\rB_e^\R$ of $\rG(e)^\R.$ 
As a result, there are two important semigroups appearing: the semigroup $\rU^\R_{\Theta,+}\subset\rU^\R_\Theta$ coming from $\Theta$-positivity for $\{f,h,e\}\subset\fg$,  and the semigroup $\rU^\R_{e,+}\subset\rU^\R_{e}\subset\rB_e^\R$ coming from $\Theta$-positivity of $\{f,h,e\}\subset\fg(e)$ from Case (1) of Theorem \ref{thm pos and mag}.
\begin{proposition}\label{prop inclusions  preserve positive}
    Let $\{f,h,e\}\subset\fg$ be a magical $\fsl_2$-triple with canonical real form $\fg^\R$ and consider the parabolic $\rP_\Theta\subset\rG$ and the Borel subgroup $\rB_e^\R\subset\rG(e)^\R.$ Then the inclusion $\iota:\rB_e^\R\to\rP_\Theta^\R$ induces an inclusion of the positive semigroups
    \[\iota:\rU^\R_{e,+}\longrightarrow\rU^\R_{\Theta,+}.\]
\end{proposition}
\begin{proof}
    For Case (1) of Theorem \ref{thm pos and mag}, there is nothing to prove since $\fg(e)=\fg.$ For Case (2) of Theorem \ref{thm pos and mag}, $\fg(e)=\{f,h,e\}$ and the semigroup is just the exponential of the positive cone $c_{\beta_1}^0$. In this case the Cayley transform $\hat e$ of $e$ is contained in the cone, and hence $\exp(t\hat e)$ is contained in $c_{\beta_1}^0$ for $t>0$. For Case (3) of Theorem \ref{thm pos and mag} the statement was proven in \cite{CollierSOnn+1components} for $\rG^\R=\rSO_{p,p+1}$ and the proof for $\rSO_{p,q}$ is identical; see \cite[\S 7.2]{so(pq)BCGGO}.

    Finally we focus on Case (4) of Theorem \ref{thm pos and mag}. Note that there are two simple roots $\alpha_3,\alpha_4\notin \Theta$, and the $\rL^\R_\Theta$-invariant decomposition 
    $\fu^\R_{\alpha_3}\oplus\fu_{\alpha_4}^\R$ is a real version of the decomposition $\fg_2=\fg_2^b\oplus\fg_{\tilde\alpha}$ in \eqref{eq G0 invariant g2 decomp}. Recall from Remark \ref{rem descriptions of cones}, that the two cones $c_{\alpha_3}\subset\fu_{\alpha_3}$ and $c_{\alpha_4}\subset\fu_{\alpha_4}$ are described as follows: $c_{\alpha_4}\subset\fu_{\alpha_4}$ is $\R^+\subset\R$ and $c_{\alpha_3}\subset\fu_{\alpha_3}$ is the cone from Case (2) for the Lie algebras $\fsp_6\R,$ $\fsu_{3,3},$ $\fso_{12}^*$ and $\fe_7^{-25}$, with $\fg^\R$ equal to the quaternionic real forms of $\ff_4,$ $\fe_6,$ $\fe_7$ and $\fe_8$ respectively. 

    We claim that the Cayley transform $\hat e$ of the magical nilpotent $e$ is contained in $c_{\alpha_3}^0\times c_{\alpha_4}^0.$ 
    First note, that the projections $\hat e_{\alpha_3}$ and $\hat e_{\alpha_4}$ of $\hat e$ onto each factor $\fu_{\alpha_3}^\R\oplus\fu_{\alpha_4}^\R$ are nonzero since the parabolic $\fp_\Theta^\R=\fp_{\hat e}^\R$ is determined by $\hat e.$ Since the projection of $\hat e$ onto $\fu^\R_{\alpha_4}$ is nonzero, we conclude that it is in the cone $c^0_{\alpha_4}\subset\fu_{\alpha_4}^\R$. 
    Recall from Remark \ref{rem fb magical from case 2}, that $\{f_b,[f_b,e_b],e_b\}$ is a magical $\fsl_2$-triple from Case (2). 
    Since the Cayley transform of $e_b$ is contained in the cone from Case (2), the projection of $\hat e$ onto $\fu_{\alpha_3}$ is contained in the cone $c_{\alpha_3}^0\subset\fu_{\alpha_3}.$ 
  Now, the Weyl group $\cW_\Theta$ is the Weyl group of $\fg(e)^\R$, thus that of $Lie(\rG_2)$, and $\fg(e)^\R$ is the split real form of $\rG_2$. Moreover, the  projections $\hat e_{\alpha_3}$ and $\hat e_{\alpha_4}$ generate the nilpotent part of the Borel subalgebra $\fb_e^\R\subset\fg(e)^\R$. Hence, the inclusion $\iota:\rB_e^\R\to\rP_\Theta^\R$ induces an inclusion $\iota:\rU^\R_{e,+}\to\rU^\R_{\Theta,+}.$
 \end{proof}
As in \cite[Theorem 7.13]{CollierSOnn+1components}, we can now prove that, for a magical $\fsl_2$-triple $\{f,h,e\}\subset\fg$ with canonical real form $\rG^\R$,  the set of representations in $\cP_e(\rG^\R)$ described by Proposition \ref{prop hitchin rep in Pe} are $\Theta$-positive Anosov representations. Using openness of $\Theta$-positive Anosov representations, we conclude from this that the union of connected components $\cP_e(\rG^\R)$ contains an open set of $\Theta$-positive Anosov representations.
\begin{theorem}\label{thm positive reps in Pe comps}
  Let $\rG$ be a simple complex Lie group and with Lie algebra $\fg.$ Let $\{f,h,e\}\subset\fg$ be a magical $\fsl_2$-triple with canonical real form $\rG^\R\subset\rG.$ Then the set of representations $\rho_{Hit}*\rho_{\rC^\R}$  from Proposition \ref{prop hitchin rep in Pe} are $\Theta$-positive Anosov representations. 
  In particular, each of the sets $\cP_e(\rG^\R)\subset\cX(\rG^\R)$ from Theorem \ref{thm character variety components} contains a nonempty open set of $\Theta$-positive Anosov representations. 
\end{theorem}
\begin{proof}
Consider a $\rG(e)^\R$-Hitchin representation $\rho_{Hit}:\pi_1\Sigma\to\rG(e)^\R$. Since $\rho_{Hit}$ is a $\Theta$-positive Anosov representation for Case (1) of Theorem \ref{thm pos and mag}, the Anosov boundary curve 
\[\xi_{\rho_{Hit}}:\partial_\infty\pi_1\Sigma\longrightarrow\rG(e)^\R/\rB_e^\R\]
sends positive triples in $\partial_\infty\pi_1\Sigma$ to positive triples of transverse points in $\rG(e)^\R/\rB_e^\R$. 
The inclusion $\iota:\rG(e)^\R\to\rG^\R$ induces a representation $\iota\circ\rho_{Hit}$ and an Anosov boundary curve
\[\iota\circ\xi_{\rho_{Hit}}:\partial_\infty\pi_1\Sigma\longrightarrow\rG(e)^\R/\rB_e^\R\hookrightarrow \rG^\R/\rP_\Theta^\R.\]
By Proposition \ref{prop inclusions  preserve positive}, $\iota\circ\xi_{\rho_{Hit}}$ also sends positive triples in $\partial_\infty\pi_1\Sigma$ to positive triples of transverse points in $\rG(e)^\R/\rB_e^\R$, and hence $\iota\circ\rho_{Hit}$ is a $\Theta$-positive Anosov representation. 

The centralizer of $\iota\circ\rho_{Hit}$ is $\rC^\R$, so compact. Since multiplication by an element in the compact part of the centralizer does not change the boundary curve and does not affect the  Anosov property, the boundary curve $\iota\circ\xi_{\rho_{Hit}}$ is also the Anosov boundary curve for the representation $\rho=(\iota\circ\rho_{Hit})*\rho_{\rC^\R}$, where $\rho_{\rC^\R}:\pi_1\Sigma\to\rC^\R$ is any representation into the compact group $\rC^\R$.  
Therefore, all representations from Proposition \ref{prop hitchin rep in Pe} are $\Theta$-positive Anosov representations. Since the set of $\Theta$-positive Anosov representations is open, each of the spaces $\cP_e(\rG^\R)$ contain an open set of $\Theta$-positive Anosov representations. 
\end{proof}
\begin{remark}
     By Corollary \ref{cor:norepfactors properparab}, none of the representations in $\cP_e(\rG^\R)$ factors through a proper parabolic subgroup of $\rG^\R$. This fact should be important in proving that in fact every connected component of $\cP_e(\rG^\R)$ which contains the $\Theta$-positive Anosov representations described in Theorem \ref{thm positive reps in Pe comps} consists entirely of $\Theta$-positive Anosov representations. There are known examples of components in $\cP_e(\rG^\R)$ which do not contain the locus described in Theorem \ref{thm positive reps in Pe comps}, namely for the group $\rSO_{p,p+1}$ \cite{CollierSOnn+1components}. However, each of these components lie in a component of $\cP_e(\rSO_{p,p+2})$ which does contain representations in the locus of Theorem \ref{thm positive reps in Pe comps}. In fact, one expects that all $\Theta$-positive Anosov representations do not factor through proper parabolic subgroups. This gives further evidence that the space of $\Theta$-positive Anosov representations is exactly described by the space $\cP_e(\rG^\R)$, and thus that the higher rank Teichm\"uller spaces coincide precisely with the spaces $\cP_e(\rG^\R)$.  
\end{remark}

\newpage

\section{Diagrams and Tables}\label{sec diagrams and tables}

\subsection{Tables}\mbox{}

\begin{table}[h]
    \centering
    \caption{Table of magical triples for inner real forms of exceptional Lie algebras}
    \begin{tabular}
        {|c|c|c|c|c|c|}\hline
            real form&table in \cite{ExceptionalNilpotentsInner}&row(s)&Columns 4 \& 5&$\fc^\R$&weighted Dynkin diagram\\     
            \hline
            $\fg_2^2$&VI&$5$&$0$&$0$&Theorem \ref{thm: classification weighted dynkin} Case (1)\\
            \hline
            $\ff_4^4$&VII&$19$&$3$&$\fso_3$&Theorem \ref{thm: classification weighted dynkin} Case (4)\\
            \hline
            $\ff_4^4$&VII&$26$&$0$&$0$&Theorem \ref{thm: classification weighted dynkin} Case (1)\\
            \hline
            $\ff_4^{-20}$&VIII&---&---&---&---\\
            \hline
            $\fe_6^2$&IX&$23$&$8$&$\fsu_3$&Theorem \ref{thm: classification weighted dynkin} Case (4)\\
            \hline
            $\fe_6^{-14}$&X&---&---&---&---\\
            \hline
            $\fe_7^7$&XI&$93, 94$&$0$&$0$&Theorem \ref{thm: classification weighted dynkin} Case (1)\\
            \hline
            $\fe_7^{-5}$&XII&$22$&$21$&$\fsp_6$&Theorem \ref{thm: classification weighted dynkin} Case (4)\\
            \hline
            $\fe_7^{-25}$&XIII&$6,7$&$52$&$\ff_{4}^{-52}$&Theorem \ref{thm: classification weighted dynkin} Case (2)\\
            \hline
            $\fe_8^8$&XIV&$115$&$0$&$0$&Theorem \ref{thm: classification weighted dynkin} Case (1)\\
            \hline
            $\fe_8^{-24}$&XV&$21$&$52$&$\ff_{4}^{-52}$&Theorem \ref{thm: classification weighted dynkin} Case (4)\\
            \hline
        \end{tabular}
        \label{Table of exceptional inner real forms}
\end{table}

\begin{table}[h]
    \centering
    \caption{Table of noncompact real forms of classical simple Lie algebras}
   \resizebox{\columnwidth}{!}{
    \begin{tabular}
            {|c|c|c|c|}\hline
            $\fg$&$\fg^\R$&Description&$\dim\fm-\dim\fh$\\\hline
            $\fsl_n\C$ & $\fsl_n\R$& traceless $(n \times n)$ $\R$-matrices&$n-1$\\\hline
            $\fsl_{p+q}\C$ &$\fsu_{p,q}$& traceless $(p+q)\times (p+q)$ $\C$-matrices  which are skew-adjoint&$1-(q-p)^2$\\ 
            &&  w.r.t.~a nondegenerate signature $(p,q)$ Hermitian form&\\\hline
             $\fsl_{2m}\C$&$\fsu_{2m}^*$& $m\times m$ $\mathbb{H}$-matrices with purely imaginary trace&$-2m-1$\\\hline
            $\fso_{p+q}\C$ &$\fso_{p,q}$ &$(p+q)\times (p+q)$ $\R$-matrices which are skew-adjoint&$\frac{1}{2}(p+q-(q-p)^2)$\\ &&w.r.t.~a nondegenerate signature $(p,q)$ symmetric form&\\\hline
            $\fso_{2m}\C$ &$\fso_{2m}^*$& $(m\times m)$ $\mathbb H$-matrices which are skew-adjoint&$-m$\\&& w.r.t.~a nondegenerate skew-Hermitian form&\\\hline
            $\fsp_{2m}\C$&$\fsp_{2m}(\R)$ &$(2m\times 2m)$ $\R$-matrices which are skew-adjoint&$m$\\ &&w.r.t.~a nondegenerate skew-symmetric form&\\\hline
            $\fsp_{2p+2q}\C$&$\fsp_{2p,2q}$ & $(m\times m)$ $\mathbb{H}$-matrices which are skew-adjoint&$-p-q-2(q-p)^2$\\&& w.r.t.~a nondegenerate signature $(p,q)$ Hermitian form&\\\hline
        \end{tabular}}
        \label{Table of classical real forms}
\end{table}
\newpage
\subsection{Weighted root poset for magical nilpotents in $\rE_6$, $\rE_7$, $\rE_8$ and $\rF_4$}\label{subsec root posets exceptional}
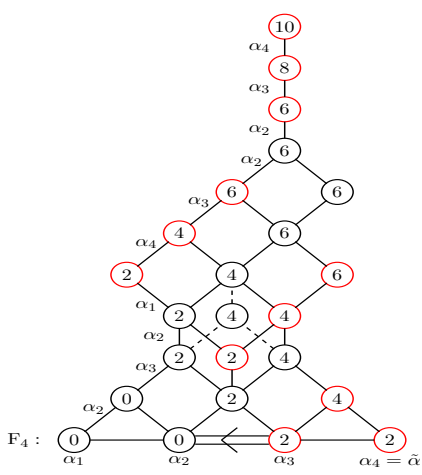
\begin{center}

    \scalebox{.7}[.55]{\begin{tikzpicture}[scale=1]
        \draw[white] (0,-1) circle (.3cm);
    \draw (-.5,0) node[anchor=east]  {$\rF_4:$};
    \node at (0,-.5) {$\alpha_1$};
    \node at (2,-.5) {$\alpha_2$};
    \node at (4,-.5) {$\alpha_3$};
    \node at (6,-.5) {$\alpha_4=\tilde\alpha$};
    \draw[thick] (0 cm,0) circle (.3cm);
    \draw[thick] (2 cm,0) circle (.3cm);
    \draw[thick,red] (4 cm,0) circle (.3cm);
    \draw[thick,red] (6 cm,0) circle (.3cm);
    \draw[thick] (0.3 cm,0) -- +(1.4 cm,0) ;
    \draw[thick] (2.3 cm, .1 cm) -- +(1.4 cm,0);
    \draw[thick] (2.8 cm, 0) -- +(.3cm, .3cm);
    \draw[thick] (2.8 cm, 0) -- +(.3cm, -.3cm);
    \draw[thick] (2.3 cm, -.1 cm) -- +(1.4 cm,0);
    \draw[thick] (4.3 cm, 0) -- +(1.4cm,0);
    \node at (0,0)  {$0$};
    \node at (2,0)  {$0$};
    \node at (4,0)  {$2$};
    \node at (6,0)  {$2$};
    \node at (1,1)  {$0$};
        \draw[thick] (1 cm,1) circle (.3cm);
    \node at (3,1)  {$2$};
    \draw[thick] (3 cm,1) circle (.3cm);
    \node at (5,1) {$4$};
    \draw[thick,red] (5 cm,1) circle (.3cm);

   \draw[thick] (.2,.2) --node[above]{$\alpha_2\ \ $} +(.6,.6);
    \draw[thick] (1.8,.2) -- +(-.6,.6);
    \draw[thick] (2.2,.2) -- +(.6,.6);
    \draw[thick] (3.8,.2) -- +(-.6,.6);
    \draw[thick] (4.2,.2) -- +(.6,.6);
    \draw[thick] (5.8,.2) -- +(-.6,.6);

    \node at (2,2)  {$2$};
    \node at (3,2)  {$2$};
    \node at (4,2)  {$4$};

    \draw[thick] (2 cm,2) circle (.3cm);
    \draw[thick,red] (3 cm,2) circle (.3cm);
    \draw[thick] (4 cm,2) circle (.3cm);
    \draw[thick] (1.2 cm,1.2 cm) --node[above]{$\alpha_3\ \ $} +(.6,.6);
    \draw[thick] (2.8 cm,1.2 cm) -- +(-.6 ,.6);
    \draw[thick] (3.2 cm,1.2 cm) -- +(.6,.6);
    \draw[thick] (3 cm,1.3 cm) -- +(0,.4);
    \draw[thick] (4.8 cm,1.2 cm) -- +(-.6 ,.6);
    \node at (2,3)  {$2$};
    \draw[thick] (2 cm,3 cm) circle (.3cm);
    \node at (3,3)  {$4$};
    \draw[thick] (3 cm,3 cm) circle (.3cm);
    \node at (4,3)  {$4$};
    \draw[thick,red] (4 cm,3 cm) circle (.3cm);

    \draw[thick] (2 cm,2.3 cm) -- node[left]{$\alpha_2 \ $}+(0,.4);
    \draw[dashed,thick] (2.2 cm,2.2 cm) -- +(.6 cm,.6);
    \draw[thick] (2.8 cm,2.2 cm) -- +(-.6 cm,.6);
    \draw[thick] (3.2 cm,2.2 cm) -- +(.6 cm,.6);
    \draw[dashed,thick] (3.8 cm,2.2 cm) -- +(-.6 cm,.6);
    \draw[thick] (4 cm,2.3 cm) -- +(0,.4);

    \node at (1,4)  {$2$};
    \draw[thick,red] (1 cm,4) circle (.3cm);
    \node at (3,4)  {$4$};
    \draw[thick] (3 cm,4) circle (.3cm);
    \node at (5,4)  {$6$};
    \draw[thick,red] (5 cm,4) circle (.3cm);

    \draw[thick] (1.8 cm,3.2 cm) --node[below]{$\alpha_1\ \ $} +(-.6,.6);
    \draw[thick] (2.2 cm,3.2 cm) -- +(.6,.6);
    \draw[dashed,thick] (3,3.3 cm) -- +(0,.4);
    \draw[thick] (3.8 cm,3.2 cm) -- +(-.6,.6);
    \draw[thick] (4.2 cm,3.2 cm) -- +(.6,.6);
    \node at (2,5)  {$4$};
    \draw[thick,red] (2 ,5) circle (.3cm);
    \node at (4,5)  {$6$};
    \draw[thick] (4 ,5) circle (.3cm);

    \draw[thick] (1.2 cm,4.2) --node[above]{$\alpha_4\ \ $} +(.6,.6);
    \draw[thick] (3.2 cm,4.2) -- +(.6,.6);
    \draw[thick] (2.8 cm,4.2) -- +(-.6,.6);
    \draw[thick] (4.8 cm,4.2) -- +(-.6,.6);
    \node at (3,6)  {$6$};
    \draw[thick,red] (3,6) circle (.3cm);
    \node at (5,6)  {$6$};
    \draw[thick] (5,6) circle (.3cm);

    \draw[thick] (2.2 cm,5.2 cm) --node[above]{$\alpha_3\ \ $} +(.6,.6);
    \draw[thick] (4.2 cm,5.2 cm) -- +(.6,.6);
    \draw[thick] (3.8 cm,5.2 cm) -- +(-.6,.6);
    \node at (4,7)  {$6$};
    \draw[thick] (4 cm,7) circle (.3cm);
    \draw[thick] (3.2 cm,6.2 cm) --node[above]{$\alpha_2\ \ $} +(.6,.6);
    \draw[thick] (4.8 cm,6.2 cm) -- +(-.6,.6);
    \node at (4,8)  {$6$};
    \draw[thick,red] (4 cm,8) circle (.3cm);
    \draw[thick] (4,7.3 cm) --node[left]{$\alpha_2\ $} +(0,.4);
    \node at (4,9)  {$8$};
    \draw[thick,red] (4 cm,9) circle (.3cm);
    \draw[thick] (4,8.3 cm) --node[left]{$\alpha_3 \ $} +(0,.4);
    \node at (4,10)  {$10$};
    \draw[thick,red] (4 cm,10) circle (.3cm);
    \draw[thick] (4,9.3 cm) --node[left]{$\alpha_4 \ $} +(0,.4);
  \end{tikzpicture}
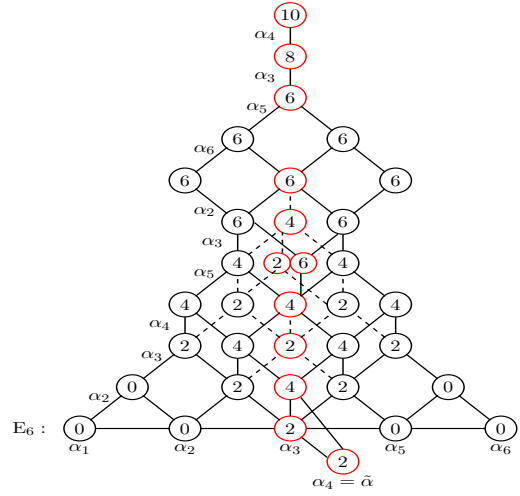
\hspace{2cm}
  \begin{tikzpicture}[scale=1]
        \draw (-.5,0) node[anchor=east]  {$\rE_6:$};
        \node at (0,-.5)  {$\alpha_1$};
    \node at (2,-.5)  {$\alpha_2$};
    \node at (4,-.5)  {$\alpha_3$};
    \node at (6,-.5)  {$\alpha_5$};
    \node at (8,-.5)  {$\alpha_6$};
    \node at (5,-1.3)  {$\alpha_4=\tilde\alpha$};
    \foreach \y in {0,...,3}
    \draw[thick,xshift=\y cm] (\y cm,0) ++(.3 cm, 0) -- +(14 mm,0);
    \draw[thick] (4.2 cm, -.2 cm) -- +(.5,-.5 cm);
    \node at (0,0)  {$0$};
    \draw[thick] (0 cm,0cm) circle (.3cm);
    \node at (2,0)  {$0$};
    \draw[thick] (2 cm,0cm) circle (.3cm);
    \node at (4,0)  {$2$};
    \draw[thick,red] (4 cm,0cm) circle (.3cm);
    \node at (6,0)  {$0$};
    \draw[thick] (6 cm,0cm) circle (.3cm);
    \node at (8,0)  {$0$};
    \draw[thick] (8 cm,0cm) circle (.3cm);
    \node at (5,-.8)  {$2$};
 \draw[thick,red] (5 cm, -.8 cm) circle (3 mm);

    \draw[thick] (1 cm,1 cm) circle (.3cm);
    \node at (1,1)  {$0$};
    \draw[thick] (3 cm,1 cm) circle (.3cm);
    \node at (3,1)  {$2$};
    \draw[thick,red] (4 cm,1 cm) circle (.3cm);
    \node at (4,1)  {$4$};
    \draw[thick] (5 cm,1 cm) circle (.3cm);
    \node at (5,1)  {$2$};
    \draw[thick] (7 cm,1 cm) circle (.3cm);
    \node at (7,1)  {$0$};

    \draw[thick] (0.2,.2) --node[above]{$\alpha_2\ \ $} +(.6,.6);
    \draw[thick] (1.8, .2 cm) -- +(-.6,.6);
    \draw[thick] (2.2, .2 cm) -- +(.6,.6);
    \draw[thick] (3.8, .2 cm) -- +(-.6,.6);
    \draw[thick] (4, .3 cm) -- +(0,.4);
    \draw[thick] (4.2, .2 cm) -- +(.6,.6);
    \draw[thick] (5.8, .2 cm) -- +(-.6,.6);
    \draw[thick] (6.2, .2 cm) -- +(.6,.6);
    \draw[thick] (7.8, .2 cm) -- +(-.6,.6);
    \draw[thick] (5,-.5) -- +(-.8,1.25);

    \node at (2,2)  {$2$};
    \draw[thick] (2 cm,2) circle (.3cm);
    \node at (3,2)  {$4$};
    \draw[thick] (3 cm,2) circle (.3cm);
    \node at (4,2)  {$2$};
    \draw[thick,red] (4 cm,2) circle (.3cm);
    \node at (5,2)  {$4$};
    \draw[thick] (5 cm,2) circle (.3cm);
    \node at (6,2)  {$2$};
    \draw[thick] (6 cm,2) circle (.3cm);

    \draw[thick] (1.2,1.2) --node[above]{$\alpha_3\ \ $} +(.6,.6);
    \draw[thick] (2.8,1.2) -- +(-.6,.6);
    \draw[thick] (3 cm,1.3) -- +(0,.4);
    \draw[dashed,thick] (3.2,1.2) -- +(.6,.6);
    \draw[thick] (3.8,1.2) -- +(-.6,.6);
    \draw[thick] (4.2,1.2) -- +(.6,.6);
    \draw[dashed,thick] (4.8,1.2) -- +(-.6,.6);
    \draw[thick] (5,1.3) -- +(0,.4);
    \draw[thick] (5.2,1.2) -- +(.6,.6);
    \draw[thick] (6.8,1.2) -- +(-.6,.6);
    \node at (2,3)  {$4$};
    \draw[thick] (2 cm,3 cm) circle (.3cm);
    \node at (3,3)  {$2$};
    \draw[thick] (3 cm,3 cm) circle (.3cm);
    \node at (4,3)  {$4$};
    \draw[thick,red] (4 cm,3 cm) circle (.3cm);
    \node at (5,3)  {$2$};
    \draw[thick] (5 cm,3 cm) circle (.3cm);
    \node at (6,3)  {$4$};
    \draw[thick] (6 cm,3 cm) circle (.3cm);

    \draw[thick] (2 ,2.3) --node[left]{$\alpha_4\ $} +(0,.4);
    \draw[dashed,thick] (2.2 ,2.2) -- +(.6,.6);
    \draw[thick] (2.8,2.2) -- +(-.6,.6);
    \draw[thick] (3.2,2.2) -- +(.6,.6);
    \draw[dashed,thick] (3.8,2.2) -- +(-.6,.6);
    \draw[dashed,thick] (4 ,2.3) -- +(0 ,.4);
    \draw[dashed,thick] (4.2,2.2) -- +(.6,.6);
    \draw[thick] (4.8,2.2) -- +(-.6,.6);
    \draw[thick] (5.2,2.2) -- +(.6,.6);
    \draw[dashed,thick] (5.8,2.2) -- +(-.6,.6);
    \draw[thick] (6 ,2.3) -- +(0,.4);

    \node at (3,4)  {$4$};
    \draw[thick] (3 cm,4 cm) circle (.3cm);
    \node at (3.75,4)  {$2$};
    \draw[thick,red] (3.75 cm,4 cm) circle (.25cm);
    \node at (4.25,4)  {$6$};
    \draw[thick,red] (4.25 cm,4 cm) circle (.25cm);
    \node at (5,4)  {$4$};
    \draw[thick] (5 cm,4 cm) circle (.3cm);
    \draw[thick] (2.2,3.2) --node[above]{$\alpha_5\ \ $} +(.6,.6);
    \draw[dashed,thick] (3,3.3) -- +(0,.4);
    \draw[dashed,thick] (3.2,3.3) -- +(.55,.55);
    \draw[thick] (3.8,3.2) -- +(-.6,.6);
    \draw[thick] (4.2,3.2) -- +(0,.55);
    \draw[thick] (4.2,3.2) -- +(.6,.6);
    \draw[dashed,thick] (4.7,3) -- +(-.8,.8);
    \draw[dashed,thick] (5,3.3) -- +(0,.4);
    \draw[thick] (5.8,3.2) -- +(-.6,.6);

    \node at (3,5)  {$6$};
    \draw[thick] (3 cm,5 cm) circle  (.3cm);
    \node at (4,5)  {$4$};
    \draw[thick,red] (4 cm,5 cm) circle (.3cm);
    \node at (5,5)  {$6$};
    \draw[thick] (5 cm,5 cm) circle (.3cm);

    \draw[thick] (3,4.3) --node[left]{$\alpha_3 \ $} +(0,.4);
    \draw[dashed,thick] (3.2,4.2) -- +(.6,.6);
    \draw[dashed,thick] (3.85,4.25) -- +(0,.5);
    \draw[thick] (4.1,4.2) -- +(-.78,.78);
    \draw[thick] (4.4,4.2) -- +(.5,.5);
    \draw[dashed,thick] (4.8,4.2) -- +(-.6,.6);
    \draw[thick] (5,4.3) -- +(0,.4);

    \node at (2,6)  {$6$};
    \draw[thick] (2 cm,6 cm) circle (.3cm);
    \node at (4,6)  {$6$};
    \draw[thick,red] (4 cm,6 cm) circle (.3cm);
    \node at (6,6)  {$6$};
    \draw[thick] (6 cm,6 cm) circle (.3cm);
    
    \draw[thick] (2.8 cm,5.2 cm) --node[below]{$\alpha_2\ \ $} +(-.6,.6);
    \draw[thick] (3.2 cm,5.2 cm) -- +(.6,.6);
    \draw[dashed,thick] (4 cm,5.3 cm) -- +(0,.4);
    \draw[thick] (4.8 cm,5.2 cm) -- +(-.6,.6);
    \draw[thick] (5.2 cm,5.2 cm) -- +(.6,.6);
    \node at (3,7)  {$6$};
      \draw[thick] (3 cm,7) circle (.3cm);
    \node at (5,7)  {$6$};
    \draw[thick] (5 cm,7) circle (.3cm);

    \draw[thick] (2.2 cm,6.2 cm) --node[above]{$\alpha_6\ \ $} +(.6,.6);
    \draw[thick] (3.8 cm,6.2 cm) -- +(-.6,.6);
    \draw[thick] (4.2 cm,6.2 cm) -- +(.6,.6);
    \draw[thick] (5.8 cm,6.2 cm) -- +(-.6,.6);
    \node at (4,8)  {$6$};
    \draw[thick,red] (4 cm,8) circle (.3cm);

    \draw[thick] (3.2 cm,7.2 cm) --node[above]{$\alpha_5\ \ $} +(.6,.6);
    \draw[thick] (4.8 cm,7.2)  -- +(-.6,.6);
    \node at (4,9)  {$8$};
    \draw[thick,red] (4 cm,9) circle (.3cm);

    \draw[thick] (4 cm,8.3 cm) --node[left]{$\alpha_3\ $} +(0,.4);
    \node at (4,10)  {$10$};
    \draw[thick,red] (4 cm,10) circle (.3cm);

    \draw[thick] (4 cm,9.3 cm) --node[left]{$\alpha_4 \ $} +(0,.4);
      \end{tikzpicture}}
\end{center}
\begin{center}
    \scalebox{.7}[.5]
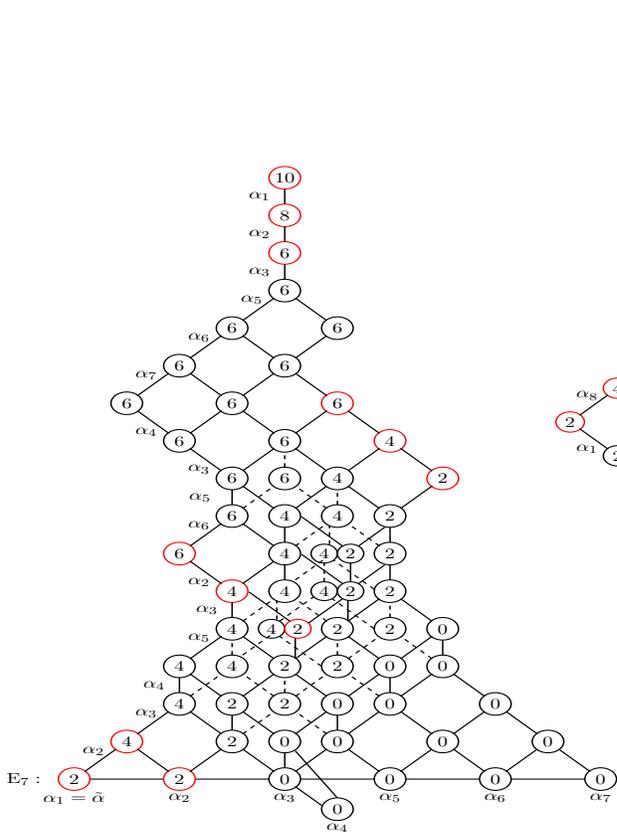
{
\begin{tikzpicture}[scale=1]
        \draw[white] (0,-6) circle (.3cm);
        \draw (-.5,0) node[anchor=east]  {$\rE_7:$};
        \node at (0,-.5)  {$\alpha_1=\tilde\alpha$};
    \node at (2,-.5)  {$\alpha_2$};
    \node at (4,-.5)  {$\alpha_3$};
    \node at (6,-.5)  {$\alpha_5$};
    \node at (8,-.5)  {$\alpha_6$};
    \node at (10,-.5)  {$\alpha_7$};
    \node at (5,-1.3)  {$\alpha_4$};
    \foreach \y in {0,...,4}
    \draw[thick,xshift=\y cm] (\y cm,0) ++(.3 cm, 0) -- +(14 mm,0);
      \draw[thick] (5 cm, -.8 cm) circle (3 mm);
    \draw[thick] (4.2 cm, -.2 cm) -- +(.5,-.5 cm);
    \node at (0,0)  {$2$};
     \draw[thick,red] (0,0)  circle (.3cm);
    \node at (2,0)  {$2$};
     \draw[thick,red] (2,0)  circle (.3cm);
    \node at (4,0)  {$0$};
     \draw[thick] (4,0)  circle (.3cm);
    \node at (6,0)  {$0$};
     \draw[thick] (6,0)  circle (.3cm);
    \node at (8,0)  {$0$};
     \draw[thick] (8,0)  circle (.3cm);
    \node at (10,0)  {$0$};
     \draw[thick] (10,0)  circle (.3cm);
    \node at (5,-.8)  {$0$};

    \node at (1,1)  {$4$};
    \draw[thick,red] (1 cm,1 cm) circle (.3cm);
    \node at (3,1)  {$2$};
    \draw[thick] (3 cm,1 cm) circle (.3cm);
    \node at (4,1)  {$0$};
    \draw[thick] (4 cm,1 cm) circle (.3cm);
    \node at (5,1)  {$0$};
    \draw[thick] (5 cm,1 cm) circle (.3cm);
    \node at (7,1)  {$0$};
    \draw[thick] (7 cm,1 cm) circle (.3cm);
    \node at (9,1)  {$0$};
    \draw[thick] (9 cm,1 cm) circle (.3cm);

    \draw[thick] (0.2,.2) -- node[above]{$\alpha_2\ \ $} +(.6,.6);
    \draw[thick] (1.8, .2 cm) -- +(-.6,.6);
    \draw[thick] (2.2, .2 cm) -- +(.6,.6);
    \draw[thick] (3.8, .2 cm) -- +(-.6,.6);
    \draw[thick] (4, .3 cm) -- +(0,.4);
    \draw[thick] (4.2, .2 cm) -- +(.6,.6);
    \draw[thick] (5.8, .2 cm) -- +(-.6,.6);
    \draw[thick] (6.2, .2 cm) -- +(.6,.6);
    \draw[thick] (7.8, .2 cm) -- +(-.6,.6);
    \draw[thick] (8.2, .2 cm) -- +(.6,.6);
    \draw[thick] (9.8, .2 cm) -- +(-.6,.6);
    \draw[thick] (5,-.5) -- +(-.8,1.25);

    \node at (2,2)  {$4$};
    \draw[thick] (2 cm,2 cm) circle (.3cm);
    \node at (3,2)  {$2$};
    \draw[thick] (3 cm,2 cm) circle (.3cm);
    \node at (4,2)  {$2$};
    \draw[thick] (4 cm,2 cm) circle (.3cm);
    \node at (5,2)  {$0$};
    \draw[thick] (5 cm,2 cm) circle (.3cm);
    \node at (6,2)  {$0$};
    \draw[thick] (6 cm,2 cm) circle (.3cm);
    \node at (8,2)  {$0$};
    \draw[thick] (8 cm,2 cm) circle (.3cm);

    \draw[thick] (1.2,1.2) --node[above]{$\alpha_3\ \ $} +(.6,.6);
    \draw[thick] (2.8,1.2) -- +(-.6,.6);
    \draw[thick] (3 cm,1.3) -- +(0,.4);
    \draw[dashed,thick] (3.2,1.2) -- +(.6,.6);
    \draw[thick] (3.8,1.2) -- +(-.6,.6);
    \draw[thick] (4.2,1.2) -- +(.6,.6);
    \draw[dashed,thick] (4.8,1.2) -- +(-.6,.6);
    \draw[thick] (5,1.3) -- +(0,.4);
    \draw[thick] (5.2,1.2) -- +(.6,.6);
    \draw[thick] (6.8,1.2) -- +(-.6,.6);
    \draw[thick] (7.2 cm,1.2 cm) -- +(.6,.6);
    \draw[thick] (8.8 cm,1.2 cm) -- +(-.6,.6);

    \node at (2,3)  {$4$};
     \draw[thick] (2 cm,3 cm) circle (.3cm);
    \node at (3,3)  {$4$};
    \draw[thick] (3 cm,3 cm) circle (.3cm);
    \node at (4,3)  {$2$};
    \draw[thick] (4 cm,3 cm) circle (.3cm);
    \node at (5,3)  {$2$};
    \draw[thick] (5 cm,3 cm) circle (.3cm);
    \node at (6,3)  {$0$};
    \draw[thick] (6 cm,3 cm) circle (.3cm);
    \node at (7,3)  {$0$};
    \draw[thick] (7 cm,3 cm) circle (.3cm);

    \draw[thick] (2 ,2.3) --node[left]{$\alpha_4 \ $} +(0,.4);
    \draw[dashed,thick] (2.2 ,2.2) -- +(.6,.6);
    \draw[thick] (2.8,2.2) -- +(-.6,.6);
    \draw[thick] (3.2,2.2) -- +(.6,.6);
    \draw[dashed,thick] (3.8,2.2) -- +(-.6,.6);
    \draw[dashed,thick] (4 ,2.3) -- +(0 ,.4);
    \draw[dashed,thick] (4.2,2.2) -- +(.6,.6);
    \draw[thick] (4.8,2.2) -- +(-.6,.6);
    \draw[thick] (5.2,2.2) -- +(.6,.6);
    \draw[dashed,thick] (5.8,2.2) -- +(-.6,.6);
    \draw[thick] (6 ,2.3) -- +(0,.4);
    \draw[thick] (6.2 cm,2.2 cm) -- +(.6,.6);
    \draw[thick] (7.8 cm,2.2 cm) -- +(-.6,.6);

    \node at (3,4)  {$4$};
    \draw[thick] (3 cm,4 cm) circle (.3cm);
    \node at (3.75,4)  {$4$};
    \draw[thick] (3.75 cm,4 cm) circle (.25cm);
    \node at (4.25,4)  {$2$};
    \draw[thick,red] (4.25 cm,4 cm) circle (.25cm);
    \node at (5,4)  {$2$};
    \draw[thick] (5 cm,4 cm) circle (.3cm);
    \node at (6,4)  {$2$};
    \draw[thick] (6 cm,4 cm) circle (.3cm);
    \node at (7,4)  {$0$};
    \draw[thick] (7 cm,4 cm) circle (.3cm);

    \draw[thick] (2.2,3.2) --node[above]{$\alpha_5\ \ $} +(.6,.6);
    \draw[dashed,thick] (3,3.3) -- +(0,.4);
    \draw[dashed,thick] (3.2,3.3) -- +(.55,.55);
    \draw[thick] (3.8,3.2) -- +(-.6,.6);
    \draw[thick] (4.2,3.2) -- +(0,.55);
    \draw[thick] (4.2,3.2) -- +(.6,.6);
    \draw[dashed,thick] (4.7,3) -- +(-.8,.8);
    \draw[dashed,thick] (5,3.3) -- +(0,.4);
    \draw[dashed,thick] (5.2,3.2) -- +(.6,.6);
    \draw[thick] (5.8,3.2) -- +(-.6,.6);
    \draw[thick] (6.2,3.2) -- +(.6,.6);
    \draw[dashed,thick] (6.8,3.2) -- +(-.6,.6);
    \draw[thick] (7,3.3) -- +(0,.4);

    \node at (3,5)  {$4$};
    \draw[thick,red] (3 cm,5 cm) circle (.3cm);
    \node at (4,5)  {$4$};
    \draw[thick] (4 cm,5 cm) circle (.3cm);
    \node at (4.75,5)  {$4$};
    \draw[thick] (4.75 cm,5 cm) circle (.25cm);
    \node at (5.25,5)  {$2$};
    \draw[thick] (5.25 cm,5 cm) circle (.25cm);
    \node at (6,5)  {$2$};
    \draw[thick] (6 cm,5 cm) circle (.3cm);

    \draw[thick] (3,4.3) --node[left]{$\alpha_3 \ $} +(0,.4);
    \draw[dashed,thick] (3.2,4.2) -- +(.6,.6);
    \draw[dashed,thick] (3.85,4.25) -- +(0,.5);
    \draw[dashed,thick] (3.9,4.2) -- +(.65,.65);
    \draw[thick] (4.1,4.2) -- +(-.78,.78);
    \draw[thick] (4.45,4.15) -- +(.65,.65);
    \draw[dashed,thick] (4.8,4.2) -- +(-.6,.6);
    \draw[thick] (5.2,4.2) -- +(0,.55);
    \draw[thick] (5.2,4.2) -- +(.6,.6);
    \draw[dashed,thick] (5.7,4) -- +(-.8,.8);;
    \draw[dashed,thick] (6,4.3) -- +(0,.4);
    \draw[thick] (6.8,4.2) -- +(-.6,.6);

    \node at (2,6)  {$6$};
    \draw[thick,red] (2 cm,6 cm) circle (.3cm);
    \node at (4,6)  {$4$};
    \draw[thick] (4 cm,6 cm) circle (.3cm);
    \node at (4.75,6)  {$4$};
    \draw[thick] (4.75 cm,6 cm) circle (.25cm);
    \node at (5.25,6)  {$2$};
    \draw[thick] (5.25 cm,6 cm) circle (.25cm);
    \node at (6,6)  {$2$};
    \draw[thick] (6 cm,6 cm) circle (.3cm);

    \draw[thick] (2.8,5.2) --node[below]{$\alpha_2\ \ $} +(-.6,.6);
    \draw[thick] (3.2,5.2) -- +(.6,.6);
    \draw[thick] (4,5.3) -- +(.0,.4);
    \draw[dashed,thick] (4.2,5.3) -- +(.55,.55);
    \draw[dashed,thick] (4.75,5.25) -- +(.0,.5);
    \draw[thick] (5.1,5.2) -- +(-.78,.78);
    \draw[thick] (5.25,5.25) -- +(.0,.5);
    \draw[thick] (5.4,5.2) -- +(.5,.5);
    \draw[dashed,thick] (5.7,5.05) -- +(-.78,.78);
    \draw[thick] (6,5.3) -- +(0,.4);

    \node at (3,7)  {$6$};
    \draw[thick] (3 cm,7 cm) circle (.3cm);
    \node at (4,7)  {$4$};
    \draw[thick] (4 cm,7 cm) circle (.3cm);
    \node at (5,7)  {$4$};
    \draw[thick] (5 cm,7 cm) circle (.3cm);
    \node at (6,7)  {$2$};
    \draw[thick] (6 cm,7 cm) circle (.3cm);

    \draw[thick] (2.2,6.2) --node[above]{$\alpha_6\ \ $} +(.6,.6);
    \draw[thick] (3.8,6.2) -- +(-.6,.6);
    \draw[thick] (4,6.3) -- +(0,.4);
    \draw[dashed,thick] (4.2,6.2) -- +(.6,.6);
    \draw[dashed,thick] (4.85,6.25) -- +(0,.5);
    \draw[thick] (5.1,6.2) -- +(-.78,.78);
    \draw[thick] (5.4,6.2) -- +(.5,.5);
    \draw[dashed,thick] (5.8,6.2) -- +(-.6,.6);
    \draw[thick] (6,6.3) -- +(0,.4);

    \node at (3,8)  {$6$};
    \draw[thick] (3,8) circle (.3cm) ;
    \draw[thick] (4,8) circle (.3cm) ;
    \node at (4,8)  {$6$};
    \node at (5,8)  {$4$};
    \draw[thick] (5,8) circle (.3cm) ;
    \node at (7,8)  {$2$};
    \draw[thick,red] (7,8) circle (.3cm) ; 

    \draw[thick] (3,7.3) -- node[left]{$\alpha_5\ \ $}+(0,.4);
    \draw[dashed,thick] (3.2,7.2) -- +(.6,.6);
    \draw[thick] (3.8,7.2) -- +(-.6,.6);
    \draw[thick] (4.2,7.2) -- +(.6,.6);
    \draw[dashed,thick] (4.8,7.2) -- +(-.6,.6);
    \draw[dashed,thick] (5,7.3) -- +(0,.4);
    \draw[thick] (5.8,7.2) -- +(-.6,.6);
    \draw[thick] (6.2,7.2) -- +(.6,.6);

    \node at (2,9)  {$6$};
    \draw[thick] (2,9) circle (.3cm) ;
    \node at (4,9)  {$6$};
    \draw[thick] (4,9) circle (.3cm) ;
    \node at (6,9)  {$4$};
    \draw[thick,red] (6,9) circle (.3cm) ;

    \draw[thick] (2.8 cm,8.2 cm) --node[below]{$\alpha_3\ \ $} +(-.6,.6);
    \draw[thick] (3.2 cm,8.2 cm) -- +(.6,.6);
    \draw[thick] (4.8 cm,8.2 cm) -- +(-.6,.6);
    \draw[thick] (5.2 cm,8.2 cm) -- +(.6,.6);
    \draw[thick] (6.8 cm,8.2 cm) -- +(-.6,.6);
    \draw[dashed,thick] (4 cm,8.3 cm) -- +(0,.4);
    \node at (1,10)  {$6$};
    \draw[thick] (1,10)  circle (.3cm);
    \node at (3,10)  {$6$};
    \draw[thick] (3,10)  circle (.3cm);
    \node at (5,10)  {$6$};
    \draw[thick,red] (5,10)  circle (.3cm);

    \draw[thick] (3.8 cm,9.2 cm) -- +(-.6,.6);
    \draw[thick] (4.2 cm,9.2 cm) -- +(.6,.6);
    \draw[thick] (5.8 cm,9.2 cm) -- +(-.6,.6);
    \draw[thick] (2.2 cm,9.2 cm) -- +(.6,.6);
    \draw[thick] (1.8 cm,9.2 cm) --node[below ]{$\alpha_4\ \ $} +(-.6,.6);
    \node at (2,11)  {$6$};
    \draw[thick] (2,11)  circle (.3cm);
    \node at (4,11)  {$6$};
    \draw[thick] (4,11)  circle (.3cm);

    \draw[thick] (1.2 cm,10.2 cm) --node[above]{$\alpha_7\ \ $} +(.6,.6);
    \draw[thick] (2.8 cm,10.2 cm) -- +(-.6,.6);
    \draw[thick] (3.2 cm,10.2 cm) -- +(.6,.6);
    \draw[thick] (4.8 cm,10.2 cm) -- +(-.6,.6);

    \node at (3,12)  {$6$};
    \draw[thick] (3,12)  circle (.3cm);
    \node at (5,12)  {$6$};
    \draw[thick] (5,12)  circle (.3cm);

    \draw[thick] (2.2 cm,11.2 cm) --node[above]{$\alpha_6\ \ $} +(.6,.6);
    \draw[thick] (3.8 cm,11.2 cm) -- +(-.6,.6);
    \draw[thick] (4.2 cm,11.2 cm) -- +(.6,.6);
     \node at (4,13)  {$6$};
    \draw[thick] (4,13)  circle (.3cm);

    \draw[thick] (3.2 cm,12.2 cm) --node[above]{$\alpha_5\ \ $} +(.6,.6);
    \draw[thick] (4.8 cm,12.2 cm) -- +(-.6,.6);

    \node at (4,14)  {$6$};
    \draw[thick,red] (4,14)  circle (.3cm);

    \draw[thick] (4 cm,13.3 cm) --node[left]{$\alpha_3 \ $} +(0,.4);

     \node at (4,15)  {$8$};
    \draw[thick,red] (4,15)  circle (.3cm);

    \draw[thick] (4 cm,14.3 cm) --node[left]{$\alpha_2 \ $} +(0,.4);

    \node at (4,16)  {$10$};
    \draw[thick,red] (4,16)  circle (.3cm);

    \draw[thick] (4 cm,15.3 cm) --node[left]{$\alpha_1 \ $} +(0,.4);
    \end{tikzpicture}
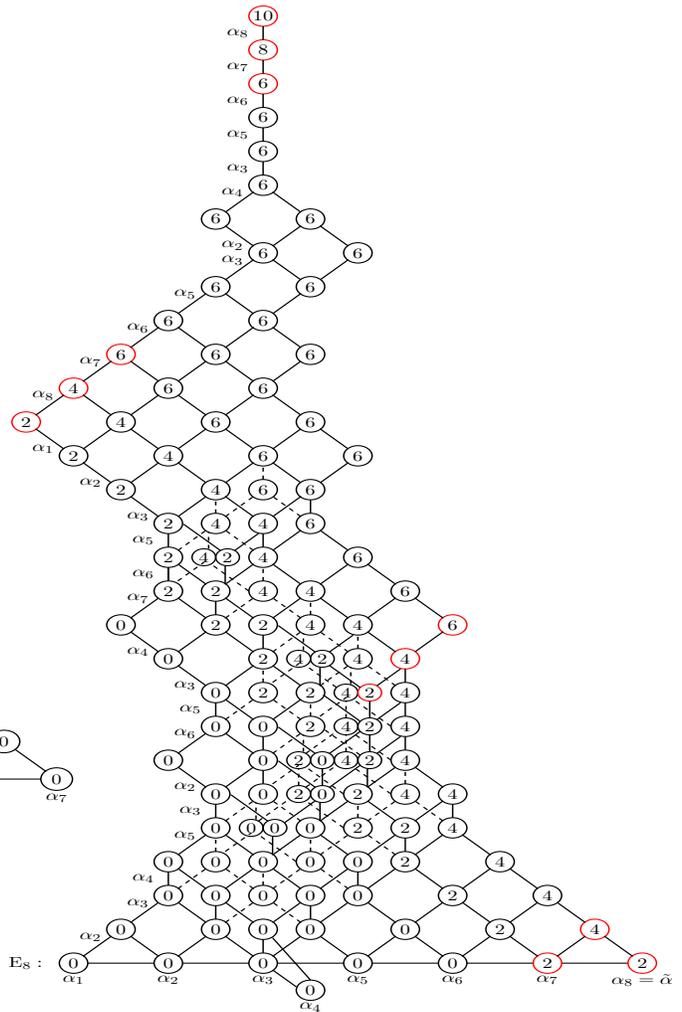
\hspace{-1.5cm}
    \begin{tikzpicture}[scale=.9]
        \draw (-.5,0) node[anchor=east]  {$\rE_8:$};
         \node at (0,-.5)  {$\alpha_1$};
    \node at (2,-.5)  {$\alpha_2$};
    \node at (4,-.5)  {$\alpha_3$};
    \node at (6,-.5)  {$\alpha_5$};
    \node at (8,-.5)  {$\alpha_6$};
    \node at (10,-.5)  {$\alpha_7$};
    \node at (12,-.5)  {$\alpha_8=\tilde\alpha$};
    \node at (5,-1.3)  {$\alpha_4$};
    \foreach \y in {0,...,5}
    \draw[thick,xshift=\y cm] (\y cm,0) ++(.3 cm, 0) -- +(14 mm,0);
    \draw[thick] (5 cm, -.8 cm) circle (3 mm);
    \draw[thick] (4.2 cm, -.2 cm) -- +(.5,-.5 cm);

    \node at (0,0)  {$0$};
    \draw[thick] (0,0)  circle (.3cm);
    \node at (2,0)  {$0$};
     \draw[thick] (2,0)  circle (.3cm);
    \node at (4,0)  {$0$};
     \draw[thick] (4,0)  circle (.3cm);
    \node at (6,0)  {$0$};
     \draw[thick] (6,0)  circle (.3cm);
    \node at (8,0)  {$0$};
     \draw[thick] (8,0)  circle (.3cm);
    \node at (10,0)  {$2$};
     \draw[thick,red] (10,0)  circle (.3cm);
    \node at (12,0)  {$2$};
     \draw[thick,red] (12,0)  circle (.3cm);
    \node at (5,-.8)  {$0$};
    \node at (1,1) {$0$}; 
    \draw[thick] (1,1)  circle (.3cm);
    \node at (3,1)  {$0$};
    \draw[thick] (3,1)  circle (.3cm);
    \node at (4,1)  {$0$};
    \draw[thick] (4,1)  circle (.3cm);
    \node at (5,1)  {$0$};
    \draw[thick] (5,1)  circle (.3cm);
    \node at (7,1)  {$0$};
    \draw[thick] (7,1)  circle (.3cm);
    \node at (9,1)  {$2$};
    \draw[thick] (9,1)  circle (.3cm);
    \node at (11,1)  {$4$};
    \draw[thick,red] (11,1)  circle (.3cm);

    \draw[thick] (.2,.2) --node[above]{$\alpha_2\ \ $} +(.6,.6);
    \draw[thick] (1.8,.2) -- +(-.6,.6);
    \draw[thick] (2.2,.2) -- +(.6,.6);
    \draw[thick] (3.8,.2) -- +(-.6,.6);
    \draw[thick] (4.2,.2) -- +(.6,.6);
    \draw[thick] (4,.3) -- +(0,.4);
    \draw[thick] (5.8,.2) -- +(-.6,.6);
    \draw[thick] (6.2,.2) -- +(.6,.6);
    \draw[thick] (7.8,.2) -- +(-.6,.6);
    \draw[thick] (8.2,.2) -- +(.6,.6);
    \draw[thick] (9.8,.2) -- +(-.6,.6);
    \draw[thick] (10.2,.2) -- +(.6,.6);
    \draw[thick] (11.8,.2) -- +(-.6,.6);
    \draw[thick] (5,-.5) -- +(-.8,1.25);
    \node at (2,2)  {$0$};
    \draw[thick] (2,2)  circle (.3cm);
    \node at (3,2)  {$0$};
    \draw[thick] (3,2)  circle (.3cm);
    \node at (4,2)  {$0$};
    \draw[thick] (4,2)  circle (.3cm);
    \node at (5,2)  {$0$};
    \draw[thick] (5,2)  circle (.3cm);
    \node at (6,2)  {$0$};
    \draw[thick] (6,2)  circle (.3cm);
    \node at (8,2)  {$2$};
    \draw[thick] (8,2)  circle (.3cm);
    \node at (10,2)  {$4$};
    \draw[thick] (10,2)  circle (.3cm);

    \draw[thick] (1.2,1.2) --node[above]{$\alpha_3\ \ $} +(.6,.6);
    \draw[thick] (2.8,1.2) -- +(-.6,.6);
    \draw[thick] (3,1.3) -- +(0,.4);
    \draw[dashed,thick] (3.2,1.2) -- +(.6,.6);
    \draw[thick] (3.8,1.2) -- +(-.6,.6);
    \draw[thick] (4.2,1.2) -- +(.6,.6);
    \draw[dashed,thick] (4.8,1.2) -- +(-.6,.6);
    \draw[thick] (5,1.3) -- +(0,.4);
    \draw[thick] (5.2,1.2) -- +(.6,.6);
    \draw[thick] (6.8,1.2) -- +(-.6,.6);
    \draw[thick] (7.2,1.2) -- +(.6,.6);
    \draw[thick] (8.8,1.2) -- +(-.6,.6);
    \draw[thick] (9.2,1.2) -- +(.6,.6);
    \draw[thick] (10.8,1.2) -- +(-.6,.6);

    \node at (2,3)  {$0$};
    \draw[thick] (2,3)  circle (.3cm);
    \node at (3,3)  {$0$};
    \draw[thick] (3,3)  circle (.3cm);
    \node at (4,3)  {$0$};
    \draw[thick] (4,3)  circle (.3cm);
    \node at (5,3)  {$0$};
    \draw[thick] (5,3)  circle (.3cm);
    \node at (6,3)  {$0$};
    \draw[thick] (6,3)  circle (.3cm);
    \node at (7,3)  {$2$};
    \draw[thick] (7,3)  circle (.3cm);
    \node at (9,3)  {$4$};
    \draw[thick] (9,3)  circle (.3cm);

    \draw[thick] (2,2.3) --node[left]{$\alpha_4 \ $} +(0,.4);
    \draw[dashed,thick] (2.2,2.2) -- +(.6,.6);
    \draw[thick] (3.2,2.2) -- +(.6,.6);
    \draw[thick] (2.8,2.2) -- +(-.6,.6);
    \draw[dashed,thick] (3.8,2.2) -- +(-.6,.6);
    \draw[dashed,thick] (4,2.3) -- +(0,.4);
    \draw[dashed,thick] (4.2,2.2) -- +(.6,.6);
    \draw[thick] (4.8,2.2) -- +(-.6,.6);
    \draw[thick] (5.2,2.2) -- +(.6,.6);
    \draw[dashed,thick] (5.8,2.2) -- +(-.6,.6);
    \draw[thick] (6,2.3) -- +(0,.4);
    \draw[thick] (6.2,2.2) -- +(.6,.6);
    \draw[thick] (7.8,2.2) -- +(-.6,.6);
    \draw[thick] (8.2,2.2) -- +(.6,.6);
    \draw[thick] (9.8,2.2) -- +(-.6,.6);

    \node at (3,4)  {$0$};
    \draw[thick] (3,4)  circle (.3cm);
    \node at (3.75,4)  {$0$};
    \draw[thick] (3.75,4)  circle (.25cm);
    \node at (4.25,4)  {$0$};
    \draw[thick] (4.25,4)  circle (.25cm);
    \node at (5,4)  {$0$};
    \draw[thick] (5,4)  circle (.3cm);
    \node at (6,4)  {$2$};
    \draw[thick] (6,4)  circle (.3cm);
    \node at (7,4)  {$2$};
    \draw[thick] (7,4)  circle (.3cm);
    \node at (8,4)  {$4$};
    \draw[thick] (8,4)  circle (.3cm);

    \draw[thick] (2.2,3.2) -- node[above]{$\alpha_5\ \ $} +(.6,.6);
    \draw[dashed,thick] (3,3.3) -- +(0,.4);
    \draw[dashed,thick] (3.2,3.3) -- +(.55,.55);
    \draw[thick] (3.8,3.2) -- +(-.6,.6);
    \draw[thick] (4.2,3.2) -- +(0,.55);
    \draw[thick] (4.2,3.2) -- +(.6,.6);
    \draw[dashed,thick] (4.7,3) -- +(-.8,.8);
    \draw[dashed,thick] (5,3.3) -- +(0,.4);
    \draw[dashed,thick] (5.2,3.2) -- +(.6,.6);
    \draw[thick] (5.8,3.2) -- +(-.6,.6);
    \draw[thick] (6.2,3.2) -- +(.6,.6);
    \draw[dashed,thick] (6.8,3.2) -- +(-.6,.6);
    \draw[thick] (7,3.3) -- +(0,.4);
    \draw[thick] (7.2,3.2) -- +(.6,.6);
    \draw[thick] (8.8,3.2) -- +(-.6,.6);

    \node at (3,5)  {$0$};
    \draw[thick] (3,5)  circle (.3cm);
    \node at (4,5)  {$0$};
    \draw[thick] (4,5)  circle (.3cm);
    \node at (4.75,5)  {$2$};
    \draw[thick] (4.75,5)  circle (.25cm);
    \node at (5.25,5)  {$0$};
    \draw[thick] (5.25,5)  circle (.25cm);
    \node at (6,5)  {$2$};
    \draw[thick] (6,5)  circle (.3cm);
    \node at (7,5)  {$4$};
    \draw[thick] (7,5)  circle (.3cm);
    \node at (8,5)  {$4$};
    \draw[thick] (8,5)  circle (.3cm);

    \draw[thick] (3,4.3) --node[left]{$\alpha_3 \ $} +(0,.4);
    \draw[dashed,thick] (3.2,4.2) -- +(.6,.6);
    \draw[dashed,thick] (3.85,4.25) -- +(0,.5);
    \draw[dashed,thick] (3.9,4.2) -- +(.65,.65);
    \draw[thick] (4.1,4.2) -- +(-.78,.78);
    \draw[thick] (4.45,4.15) -- +(.65,.65);
    \draw[dashed,thick] (4.8,4.2) -- +(-.6,.6);
    \draw[thick] (5.2,4.2) -- +(0,.55);
    \draw[thick] (5.2,4.2) -- +(.6,.6);
    \draw[dashed,thick] (5.7,4) -- +(-.8,.8);;
    \draw[dashed,thick] (6,4.3) -- +(0,.4);
    \draw[dashed,thick] (6.2,4.2) -- +(.6,.6);
    \draw[thick] (6.8,4.2) -- +(-.6,.6);
    \draw[thick] (7.2,4.2) -- +(.6,.6);
    \draw[dashed,thick] (7.8,4.2) -- +(-.6,.6);
    \draw[thick] (8,4.3) -- +(0,.4);

    \node at (2,6)  {$0$};
    \draw[thick] (2,6)  circle (.3cm);
    \node at (4,6)  {$0$};
    \draw[thick] (4,6)  circle (.3cm);
    \node at (4.75,6)  {$2$};
    \draw[thick] (4.75,6)  circle (.25cm);
    \node at (5.25,6)  {$0$};
    \draw[thick] (5.25,6)  circle (.25cm);
    \node at (5.75,6)  {$4$};
    \draw[thick] (5.75,6)  circle (.25cm);
    \node at (6.25,6)  {$2$};
    \draw[thick] (6.25,6)  circle (.25cm);
    \node at (7,6)  {$4$};
    \draw[thick] (7,6)  circle (.3cm);

    \draw[thick] (2.8,5.2) --node[below]{$\alpha_2\ \ $} +(-.6,.6);
    \draw[thick] (3.2,5.2) -- +(.6,.6);
    \draw[thick] (4,5.3) -- +(.0,.4);
    \draw[dashed,thick] (4.2,5.3) -- +(.55,.55);
    \draw[dashed,thick] (4.75,5.25) -- +(.0,.5);
    \draw[dashed,thick] (4.9,5.2) -- +(.65,.65);
    \draw[thick] (5.1,5.2) -- +(-.78,.78);
    \draw[thick] (5.25,5.25) -- +(.0,.5);
    \draw[thick] (5.4,5.2) -- +(.65,.65);
    \draw[dashed,thick] (5.7,5) -- +(-.8,.8);
    \draw[thick] (6.2,5.2) -- +(0,.55);
    \draw[thick] (6.2,5.2) -- +(.6,.6);
    \draw[dashed,thick] (6.7,5) -- +(-.8,.8);;
    \draw[dashed,thick] (7,5.3) -- +(0,.4);
    \draw[thick] (7.8,5.2) -- +(-.6,.6);

    \node at (3,7)  {$0$};
    \draw[thick] (3,7)  circle (.3cm);
    \node at (4,7)  {$0$};
    \draw[thick] (4,7)  circle (.3cm);
    \node at (5,7)  {$2$};
    \draw[thick] (5,7)  circle (.3cm);
    \node at (5.75,7)  {$4$};
    \draw[thick] (5.75,7)  circle (.25cm);
    \node at (6.25,7)  {$2$};
    \draw[thick] (6.25,7)  circle (.25cm);
    \node at (7,7)  {$4$};
    \draw[thick] (7,7)  circle (.3cm);

    \draw[thick] (2.2,6.2) --node[above]{$\alpha_6\ \ $} +(.6,.6);
    \draw[thick] (3.8,6.2) -- +(-.6,.6);
    \draw[thick] (4,6.3) -- +(0,.4);
    \draw[dashed,thick] (4.2,6.2) -- +(.6,.6);
    \draw[dashed,thick] (4.85,6.25) -- +(0,.5);
    \draw[dashed,thick] (4.9,6.2) -- +(.65,.65);
    \draw[thick] (5.1,6.2) -- +(-.78,.78);
    \draw[thick] (5.4,6.2) -- +(.65,.65);
    \draw[dashed,thick] (5.75,6.25) -- +(.0,.5);
    \draw[dashed,thick] (6.05,6.15) -- +(-.78,.78);
    \draw[thick] (6.25,6.25) -- +(.0,.5);
    \draw[thick] (6.4,6.2) -- +(.5,.5);
    \draw[dashed,thick] (6.7,6) -- +(-.8,.8);
    \draw[thick] (7,6.3) -- +(0,.4);

    \node at (3,8)  {$0$};
    \draw[thick] (3,8)  circle (.3cm);
    \node at (4,8)  {$2$};
    \draw[thick] (4,8)  circle (.3cm);
    \node at (5,8)  {$2$};
    \draw[thick] (5,8)  circle (.3cm);
    \node at (5.75,8)  {$4$};
    \draw[thick] (5.75,8)  circle (.25cm);
    \node at (6.25,8)  {$2$};
    \draw[thick,red] (6.25,8)  circle (.25cm);
    \node at (7,8)  {$4$};
    \draw[thick] (7,8)  circle (.3cm);

    \draw[thick] (3,7.3) --node[left]{$\alpha_5 \ $} +(0,.4);
    \draw[dashed,thick] (3.2,7.2) -- +(.6,.6);
    \draw[thick] (3.8,7.2) -- +(-.6,.6);
    \draw[thick] (4.2,7.2) -- +(.6,.6);
    \draw[dashed,thick] (4.8,7.2) -- +(-.6,.6);
    \draw[dashed,thick] (5,7.3) -- +(0,.4);
    \draw[dashed,thick] (5.2,7.3) -- +(.55,.55);
    \draw[dashed,thick] (5.75,7.25) -- +(.0,.5);
    \draw[thick] (6.05,7.15) -- +(-.78,.78);
    \draw[thick] (6.25,7.25) -- +(.0,.5);
    \draw[thick] (6.4,7.2) -- +(.5,.5);
    \draw[dashed,thick] (6.7,7) -- +(-.8,.8);
    \draw[thick] (7,7.3) -- +(0,.4);

    \node at (2,9)  {$0$};
    \draw[thick] (2,9)  circle (.3cm);
    \node at (4,9)  {$2$};
    \draw[thick] (4,9)  circle (.3cm);
    \node at (4.75,9)  {$4$};
    \draw[thick] (4.75,9)  circle (.25cm);
    \node at (5.25,9)  {$2$};
    \draw[thick] (5.25,9)  circle (.25cm);
    \node at (6,9)  {$4$};
    \draw[thick] (6,9)  circle (.3cm);
    \node at (7,9)  {$4$};
    \draw[thick,red] (7,9)  circle (.3cm);

    \draw[thick] (2.8,8.2) --node[below]{$\alpha_3\ \ $} +(-.6,.6);
    \draw[thick] (3.2,8.2) -- +(.6,.6);
    \draw[dashed,thick] (4,8.3) -- +(0,.4);
    \draw[dashed,thick] (4.2, 8.3) -- +(.55,.55);
    \draw[thick] (4.8,8.2) -- +(-.6,.6);
    \draw[thick] (5.2,8.2) -- +(0,.55);
    \draw[dashed,thick] (5.2,8.2) -- +(.6,.6);
    \draw[dashed,thick] (5.55,8.15) -- +(-.65,.65);
    \draw[dashed,thick] (5.85,8.25) -- +(0,.5);
    \draw[thick] (6.05,8.15) -- +(-.65,.65);
    \draw[thick] (6.4,8.2) -- +(.5,.5);
    \draw[dashed,thick] (6.8,8.2) -- +(-.6,.6);
    \draw[thick] (7,8.3) -- +(0,.4);

    \node at (1,10)  {$0$};
    \draw[thick] (1,10)  circle (.3cm);
    \node at (3,10)  {$2$};
    \draw[thick] (3,10)  circle (.3cm);
    \node at (4,10)  {$2$};
    \draw[thick] (4,10)  circle (.3cm);
    \node at (5,10)  {$4$};
    \draw[thick] (5,10)  circle (.3cm);
    \node at (6,10)  {$4$};
    \draw[thick] (6,10)  circle (.3cm);
    \node at (8,10)  {$6$};
    \draw[thick,red] (8,10)  circle (.3cm);

    \draw[thick] (1.8,9.2) --node[below]{$\alpha_4\ \ $} +(-.6,.6);
    \draw[thick] (2.2,9.2) -- +(.6,.6);
    \draw[thick] (3.8,9.2) -- +(-.6,.6);
    \draw[thick] (4,9.3) -- +(0,.4);
    \draw[dashed,thick] (4.2,9.2) -- +(.6,.6);
    \draw[dashed,thick] (4.85,9.25) -- +(0,.5);
    \draw[thick] (5.1,9.2) -- +(-.78,.78);
    \draw[thick] (5.4,9.2) -- +(.5,.5);
    \draw[dashed,thick] (5.8,9.2) -- +(-.6,.6);
    \draw[dashed,thick] (6,9.3) -- +(0,.4);
    \draw[thick] (6.8,9.2) -- +(-.6,.6);
    \draw[thick] (7.2,9.2) -- +(.6,.6);

    \node at (2,11)  {$2$};
    \draw[thick] (2,11)  circle (.3cm);
    \node at (3,11)  {$2$};
    \draw[thick] (3,11)  circle (.3cm);
    \node at (4,11)  {$4$};
    \draw[thick] (4,11)  circle (.3cm);
    \node at (5,11)  {$4$};
    \draw[thick] (5,11)  circle (.3cm);
    \node at (7,11)  {$6$};
    \draw[thick] (7,11)  circle (.3cm);

    \draw[thick] (1.2,10.2) --node[above]{$\alpha_7\ \ $} +(.6,.6);
    \draw[thick] (2.8,10.2) -- +(-.6,.6);
    \draw[thick] (3,10.3) -- +(0,.4);
    \draw[dashed,thick] (3.2,10.2) -- +(.6,.6);
    \draw[thick] (3.8,10.2) -- +(-.6,.6);
    \draw[thick] (4.2,10.2) -- +(.6,.6);
    \draw[dashed,thick] (4.8,10.2) -- +(-.6,.6);
    \draw[dashed,thick] (5,10.3) -- +(0,.4);
    \draw[thick] (5.8,10.2) -- +(-.6,.6);
    \draw[thick] (6.2,10.2) -- +(.6,.6);
    \draw[thick] (7.8,10.2) -- +(-.6,.6);

    \node at (2,12)  {$2$};
    \draw[thick] (2,12)  circle (.3cm);
    \node at (2.75,12)  {$4$};
    \draw[thick] (2.75,12)  circle (.25cm);
    \node at (3.25,12)  {$2$};
    \draw[thick] (3.25,12)  circle (.25cm);
    \node at (4,12)  {$4$};
    \draw[thick] (4,12)  circle (.3cm);
    \node at (6,12)  {$6$};
    \draw[thick] (6,12)  circle (.3cm);

    \draw[thick] (2,11.3) --node[left]{$\alpha_6 \ $} +(0,.4);
    \draw[dashed,thick] (2.2,11.3) -- +(.55,.55);
    \draw[thick] (2.8,11.2) -- +(-.6,.6);
    \draw[thick] (3.2,11.2) -- +(0,.55);
    \draw[thick] (3.2,11.2) -- +(.6,.6);
    \draw[dashed,thick] (3.7,11) -- +(-.8,.8);
    \draw[dashed,thick] (4,11.3) -- +(0,.4);
    \draw[thick] (4.8,11.2) -- +(-.6,.6);
    \draw[thick] (5.2,11.2) -- +(.6,.6);
    \draw[thick] (6.8,11.2) -- +(-.6,.6);

    \node at (2,13)  {$2$};
    \draw[thick] (2,13)  circle (.3cm);
    \node at (3,13)  {$4$};
    \draw[thick] (3,13)  circle (.3cm);
    \node at (4,13)  {$4$};
    \draw[thick] (4,13)  circle (.3cm);
    \node at (5,13)  {$6$};
    \draw[thick] (5,13)  circle (.3cm);

    \draw[thick] (2,12.3) --node[left]{$\alpha_5 \ $} +(0,.4);
    \draw[dashed,thick] (2.2,12.2) -- +(.6,.6);
    \draw[thick] (2,11.3) -- +(0,.4);
    \draw[dashed,thick] (2.85,12.25) -- +(0,.5);
    \draw[thick] (3.1,12.2) -- +(-.78,.78);
    \draw[thick] (3.4,12.2) -- +(.5,.5);
    \draw[dashed,thick] (3.8,12.2) -- +(-.6,.6);
    \draw[thick] (4,12.3) -- +(0,.4);
    \draw[thick] (4.2,12.2) -- +(.6,.6);
    \draw[thick] (5.8,12.2) -- +(-.6,.6);

    \node at (1,14)  {$2$};
    \draw[thick] (1,14)  circle (.3cm);
    \node at (3,14)  {$4$};
    \draw[thick] (3,14)  circle (.3cm);
    \node at (4,14)  {$6$};
    \draw[thick] (4,14)  circle (.3cm);
    \node at (5,14)  {$6$};
    \draw[thick] (5,14)  circle (.3cm);

    \draw[thick] (1.8,13.2) --node[below]{$\alpha_3\ \ $} +(-.6,.6);
    \draw[thick] (2.2,13.2) -- +(.6,.6);
    \draw[dashed,thick] (3,13.3) -- +(0,.4);
    \draw[dashed,thick] (3.2,13.2) -- +(.6,.6);
    \draw[thick] (3.8,13.2) -- +(-.6,.6);
    \draw[thick] (4.2,13.2) -- +(.6,.6);
    \draw[dashed,thick] (4.8,13.2) -- +(-.6,.6);
    \draw[thick] (5,13.3) -- +(0,.4);

    \node at (0,15)  {$2$};
    \draw[thick] (0,15)  circle (.3cm);
    \node at (2,15)  {$4$};
    \draw[thick] (2,15)  circle (.3cm);
    \node at (4,15)  {$6$};
    \draw[thick] (4,15)  circle (.3cm);
    \node at (6,15)  {$6$};
    \draw[thick] (6,15)  circle (.3cm);

    \draw[thick] (.8,14.2) --node[below]{$\alpha_2\ \ $} +(-.6,.6);
    \draw[thick] (1.2,14.2) -- +(.6,.6);
    \draw[thick] (2.8,14.2) -- +(-.6,.6);
    \draw[thick] (3.2,14.2) -- +(.6,.6);
    \draw[dashed,thick] (4,14.3) -- +(0,.4);
    \draw[thick] (4.8,14.2) -- +(-.6,.6);
    \draw[thick] (5.2,14.2) -- +(.6,.6);

    \node at (-1,16)  {$2$};
    \draw[thick,red] (-1,16)  circle (.3cm);
     \node at (1,16)  {$4$};
    \draw[thick] (1,16)  circle (.3cm);
    \node at (3,16)  {$6$};
    \draw[thick] (3,16)  circle (.3cm);
    \node at (5,16)  {$6$};
    \draw[thick] (5,16)  circle (.3cm);

    \draw[thick] (-.2,15.2) --node[below]{$\alpha_1\ \ $} +(-.6,.6);
    \draw[thick] (.2,15.2) -- +(.6,.6);
    \draw[thick] (1.8,15.2) -- +(-.6,.6);
    \draw[thick] (2.2,15.2) -- +(.6,.6);
    \draw[thick] (3.8,15.2) -- +(-.6,.6);
    \draw[thick] (4.2,15.2) -- +(.6,.6);
    \draw[thick] (5.8,15.2) -- +(-.6,.6);

    \node at (0,17)  {$4$};
    \draw[thick,red] (0,17)  circle (.3cm);
    \node at (2,17)  {$6$};
    \draw[thick] (2,17)  circle (.3cm);
    \node at (4,17)  {$6$};
    \draw[thick] (4,17)  circle (.3cm);

    \draw[thick] (-.8,16.2) --node[above]{$\alpha_8\ \ $} +(.6,.6);
    \draw[thick] (.8,16.2) -- +(-.6,.6);
    \draw[thick] (1.2,16.2) -- +(.6,.6);
    \draw[thick] (2.8,16.2) -- +(-.6,.6);
    \draw[thick] (3.2,16.2) -- +(.6,.6);
    \draw[thick] (4.8,16.2) -- +(-.6,.6);
    
    %
     \node at (1,18)  {$6$};
    \draw[thick,red] (1,18)  circle (.3cm);
    \node at (3,18)  {$6$};
    \draw[thick] (3,18)  circle (.3cm);
    \node at (5,18)  {$6$};
    \draw[thick] (5,18)  circle (.3cm);

    \draw[thick] (.2,17.2) --node[above]{$\alpha_7\ \ $} +(.6,.6);
    \draw[thick] (1.8,17.2) -- +(-.6,.6);
    \draw[thick] (2.2,17.2) -- +(.6,.6);
    \draw[thick] (3.8,17.2) -- +(-.6,.6);
    \draw[thick] (4.2,17.2) -- +(.6,.6);
    %
     \node at (2,19)  {$6$};
    \draw[thick] (2,19)  circle (.3cm);
    \node at (4,19)  {$6$};
    \draw[thick] (4,19)  circle (.3cm);

    \draw[thick] (1.2,18.2) -- node[above]{$\alpha_6\ \ $}+(.6,.6);
    \draw[thick] (2.8,18.2) -- +(-.6,.6);
    \draw[thick] (3.2,18.2) -- +(.6,.6);
    \draw[thick] (4.8,18.2) -- +(-.6,.6);
    \node at (3,20)  {$6$};
    \draw[thick] (3,20)  circle (.3cm);
    \node at (5,20)  {$6$};
    \draw[thick] (5,20)  circle (.3cm);

    \draw[thick] (2.2,19.2) --node[above]{$\alpha_5\ \ $} +(.6,.6);
    \draw[thick] (3.8,19.2) -- +(-.6,.6);
    \draw[thick] (4.2,19.2) -- +(.6,.6);
    \node at (4,21)  {$6$};
    \draw[thick] (4,21)  circle (.3cm);
    \node at (6,21)  {$6$};
    \draw[thick] (6,21)  circle (.3cm);

    \draw[thick] (3.2,20.2) --node[above]{$\alpha_3\ \ $} +(.6,.6);
    \draw[thick] (4.8,20.2) -- +(-.6,.6);
    \draw[thick] (5.2,20.2) -- +(.6,.6);
    \node at (3,22)  {$6$};
    \draw[thick] (3,22)  circle (.3cm);
    \node at (5,22)  {$6$};
    \draw[thick] (5,22)  circle (.3cm);

    \draw[thick] (3.8,21.2) --node[below]{$\alpha_2\ \ $} +(-.6,.6);
    \draw[thick] (4.2,21.2) -- +(.6,.6);
    \draw[thick] (5.8,21.2) -- +(-.6,.6);
    \node at (4,23)  {$6$};
    \draw[thick] (4,23)  circle (.3cm);

    \draw[thick] (3.2,22.2) --node[above]{$\alpha_4\ \ $} +(.6,.6);
    \draw[thick] (4.8,22.2) -- +(-.6,.6);
    \node at (4,24)  {$6$};
    \draw[thick] (4,24)  circle (.3cm);

    \draw[thick] (4,23.3) --node[left]{$\alpha_3 \ $} +(0,.4);

    \node at (4,25)  {$6$};
    \draw[thick] (4,25)  circle (.3cm);

    \draw[thick] (4,24.3) --node[left]{$\alpha_5 \ $} +(0,.4);
    \node at (4,26)  {$6$};
    \draw[thick,red] (4,26)  circle (.3cm);

    \draw[thick] (4,25.3) -- node[left]{$\alpha_6 \ $}+(0,.4);

    \node at (4,27)  {$8$};
    \draw[thick,red] (4,27)  circle (.3cm);

    \draw[thick] (4,26.3) --node[left]{$\alpha_7 \ $} +(0,.4);

    \node at (4,28)  {$10$};
    \draw[thick,red] (4,28)  circle (.3cm);
    \draw[thick] (4,27.3) --node[left]{$\alpha_8 \ $} +(0,.4);

    \end{tikzpicture}}
    \end{center}
    \newpage

\subsection{Notation}
\label{sec:notation}

The following is a non-exhaustive list of notation 
used throughout  the paper. We consider an $\fsl_2$-triple
$\{f,h,e\}$ in a finite-dimensional complex simple Lie algebra $\fg$
and a complex connected Lie group $\rG$ with Lie algebra $\fg$.

\begin{tabular}{p{.2\textwidth} p{.75\textwidth}}
  $\fs=\langle f,h,e\rangle$
  &subalgebra of $\fg$ generated by $\{f,h,e\}$; see Remark~\ref{remark compatible involutions}.\\
 
  $\fg=\bigoplus_{j=0}^MW_{j}$
  &decomposition of $\fg$ into irreducible $\fsl_2$-modules;  see
    \eqref{eq sl2 module decomp}.\\
  

  $\fg=\bigoplus_{j=-l}^l\fg_j$
  &$\ad_h$-weight space decomposition of $\fg$; see \eqref{eq adh
    weight decomp}.\\

  $V=V(e)$
  &centralizer of $e$ in $\fg$; see
   p.~\pageref{p:centralizer-e}.\\

  $V=\bigoplus_{j\geq0}V_j$
  &decomposition into highest weight spaces $V_j=W_j\cap\fg_j$ in $W_j$; see
    \eqref{eq decomp highest weight spaces}.\\

  $\fc = W_0 = V_0$
  &subalgebra which centralizes $\langle f,h,e\rangle\subset\fg$; see
    Remark~\ref{rem slodowy slice}.\\

  $Z_{2m_j}=W_{2m_j}\cap\fg_0$
  &weight zero subspace of the $\fsl_2$-module $W_{2m_j}$;
    see \eqref{eq:g0Sl2module}.\\

  $\fg(e)\subset\fg$
  &semisimple part of double centralizer of magical $\langle
    f,h,e\rangle$; see Proposition~\ref{prop subalge g(e)}.\\

  $r(e)=\rk(\fg(e))$
  &rank of $\fg(e)$; see \eqref{eq:def-re-fgss}.\\

  $l_1,\dots,l_{r(e)}$
  &exponents of $\fg(e)$; see Lemma~\ref{lem identification of Slodowy
    domain with Cayley group}.\\ 

  $\tilfg\subset\fg_0$
  &semisimple part of $\fg_0$; see \eqref{eq:def-re-fgss}.\\

  $\sigma_e : \fg\to\fg$
  &magical involution associated to $\langle
    f,h,e\rangle\subset\fg$; see \eqref{eq magical involution}.\\
  
  $\fg^\R\subset\fg$
  &canonical real form with $\sigma_e$ as Cartan involution; see Definition~\ref{def: canonical real form}.\\

  $\fg=\fh\oplus\fm$
  &complex Cartan decomposition defined by $\sigma_e$; see
    p.~\pageref{p:maximal-compact}.\\ 

  $\fg^\R=\fh^\R\oplus\fm^\R$
  &real Cartan decomposition defined by $\sigma_e$; see p.~\pageref{p:maximal-compact}.\\

  $\theta_e : \fg_0\to\fg_0$
  &Cayley involution associated to magical $\{f,h,e\}$; see \eqref{eq g0 involution}.\\

  $\fg_\cC^\R\subset\fg_0$
  &Cayley real form of $\fg_0$ with $\theta_e$ as Cartan involution; see Definition~\ref{def: Cayley real form}.\\

  $\tilfg^\R\subset \fg_\cC^\R$
  &semisimple part of $\fg_\cC^\R$; see Proposition~\ref{prop cayley real form class}.\\

  $\tilfg=\fc\oplus\tilfm$
  &Cartan decomposition given by restriction to $\tilfg$ of Cayley involution
    $\theta_e$; see p.~\pageref{p:cartan-tilfg}.\\

  $\rS\subset\rG$
  & connected subgroup with Lie algebra $\fs$; see p.~\pageref{p:def-S-group}.\\

  $\rC\subset \rG$
  &centralizer in $\rG$ of  $\langle f,h,e\rangle\subset\fg$ (with Lie
    algebra $\fc$); see
  Lemma~\ref{lemma:Ccentralizesg(e)}.\\

  $\rH\subset\rG$
  &fixed-point group for $\sigma_e$; see p.~\pageref{p:maximal-compact}.\\

  $\rG^\R\subset\rG$
  &canonical real group associated to magical $\{f,h,e\}$; see Definition~\ref{def Lie group canonical real form}.\\

   $\rH^\R=\rH\cap\rG^\R$
  &maximal compact subgroup of $\rG^\R$; see p.~\pageref{p:maximal-compact}.\\

  $\rG_\cC^\R$
  &Cayley group associated to magical $\{f,h,e\}$ and $\rG$; see Definition~\ref{def cayley group}.\\

  $\tilrG^\R\subset \rG_\cC^\R$
  &semisimple part of Cayley group; see Definition~\ref{def cayley group}.\\

  $(\cE_\rT,f)$
  &uniformizing Higgs bundle; see Definition~\ref{def:uniformizing-HB}.\\

  $(\cE_{\rH_1}\star\cE_{\rH_2})[\rH]$
  &star product $\rH$-bundle for commuting subgroups
    $\rH_1,\rH_2\subset\rH$; see \eqref{eq: star notation}.\\

  $\widehat\Psi_e$
  &Cayley map on configuration space; see \eqref{eq Cayley
    map config}.\\

  $\Psi_e$
  &Cayley map on moduli space; see \eqref{eq:Cayleymoduli}.\\

\end{tabular}

\newpage

\bibliographystyle{plain}
\bibliography{AnnalsResubmission}

\begin{thebibliography}{10}

\bibitem{CartanGaloisCohomAdams}
Jeffrey Adams and Olivier Ta\"{\i}bi.
\newblock Galois and {C}artan cohomology of real groups.
\newblock {\em Duke Math. J.}, 167(6):1057--1097, 2018.

\bibitem{KHCquiversvortices}
Luis {\'A}lvarez-C{\'o}nsul and Oscar Garc{\'{\i}}a-Prada.
\newblock Hitchin-{K}obayashi correspondence, quivers, and vortices.
\newblock {\em Comm. Math. Phys.}, 238(1-2):1--33, 2003.

\bibitem{so(pq)BCGGO}
Marta Aparicio-Arroyo, Steven Bradlow, Brian Collier, Oscar Garc\'{\i}a-Prada,
  Peter~B. Gothen, and Andr\'{e} Oliveira.
\newblock {$\mathrm{SO}(p,q)$}-{H}iggs bundles and higher {T}eichm\"{u}ller
  components.
\newblock {\em Invent. Math.}, 218(1):197--299, 2019.

\bibitem{Baraglia18}
David Baraglia.
\newblock Monodromy of the {$SL(n)$} and {$GL(n)$} {H}itchin fibrations.
\newblock {\em Math. Ann.}, 370(3-4):1681--1716, 2018.

\bibitem{baragliaschaposnikmonodromyrank2}
David Baraglia and Laura Schaposnik.
\newblock Monodromy of rank 2 twisted {H}itchin systems and real character
  varieties.
\newblock {\em Trans. Amer. Math. Soc.}, 370(8):5491--5534, 2018.

\bibitem{DavidLauraCayleyLanglands}
David Baraglia and Laura~P. Schaposnik.
\newblock Cayley and {L}anglands type correspondences for orthogonal {H}iggs
  bundles.
\newblock {\em Trans. Amer. Math. Soc.}, 371(10):7451--7492, 2019.

\bibitem{BGLPWPosClosed}
Jonas {Beyrer}, Olivier {Guichard}, Fran{\c{c}}ois {Labourie}, Beatrice
  {Pozzetti}, and Anna {Wienhard}.
\newblock Positivity, cross-ratios, and the collar lemma.
\newblock {\em (in preparation)}.

\bibitem{BeyrerPozzettiSOpq}
Jonas {Beyrer} and Beatrice {Pozzetti}.
\newblock {Positive surface group representations in $\mathsf{PO}(p,q)$}.
\newblock {\em arXiv:2106.14725 [math.GT]}, June 2021.
\newblock \url{https://arxiv.org/abs/2106.14725}.

\bibitem{BGRmaximalToledo}
Olivier Biquard, Oscar García-Prada, and Roberto Rubio.
\newblock Higgs bundles, the {T}oledo invariant and the {C}ayley
  correspondence.
\newblock {\em Journal of Topology}, 10(3):795--826, 2017.

\bibitem{biswas-ramanan}
I.~Biswas and S.~Ramanan.
\newblock An infinitesimal study of the moduli of {H}itchin pairs.
\newblock {\em J. London Math. Soc. (2)}, 49(2):219--231, 1994.

\bibitem{chains-2018}
Steven Bradlow, Oscar Garcia-Prada, Peter Gothen, and Jochen Heinloth.
\newblock Irreducibility of moduli of semistable chains and applications to
  {$\mathrm{U}(p,q)$}-{H}iggs bundles.
\newblock In {\em Geometry and Physics: Volume 2, A Festschrift in honour of
  Nigel Hitchin}, pages 455--470. Oxford University Press, 2018.

\bibitem{UpqHiggs}
Steven~B. Bradlow, Oscar Garc{\'{\i}}a-Prada, and Peter~B. Gothen.
\newblock Surface group representations and {${\rm U}(p,q)$}-{H}iggs bundles.
\newblock {\em J. Differential Geom.}, 64(1):111--170, 2003.

\bibitem{HermitianTypeHiggsBGG}
Steven~B. Bradlow, Oscar Garc{\'{\i}}a-Prada, and Peter~B. Gothen.
\newblock Maximal surface group representations in isometry groups of classical
  {H}ermitian symmetric spaces.
\newblock {\em Geom. Dedicata}, 122:185--213, 2006.

\bibitem{BGGHomotopyGroups}
Steven~B. Bradlow, Oscar Garc{\'{\i}}a-Prada, and Peter~B. Gothen.
\newblock Homotopy groups of moduli spaces of representations.
\newblock {\em Topology}, 47(4):203--224, 2008.

\bibitem{SO2n*connected}
Steven~B. Bradlow, Oscar Garc{\'{\i}}a-Prada, and Peter~B. Gothen.
\newblock Higgs bundles for the non-compact dual of the special orthogonal
  group.
\newblock {\em Geom. Dedicata}, 175:1--48, 2015.

\bibitem{MaxRepsAnosov}
Marc Burger, Alessandra Iozzi, Fran{\c{c}}ois Labourie, and Anna Wienhard.
\newblock Maximal representations of surface groups: symplectic {A}nosov
  structures.
\newblock {\em Pure Appl. Math. Q.}, 1(3, Special Issue: In memory of Armand
  Borel. Part 2):543--590, 2005.

\bibitem{CompteRenduBIW}
Marc Burger, Alessandra Iozzi, and Anna Wienhard.
\newblock Surface group representations with maximal {T}oledo invariant.
\newblock {\em C. R. Math. Acad. Sci. Paris}, 336(5):387--390, 2003.

\bibitem{BIWmaximalToledoAnnals}
Marc Burger, Alessandra Iozzi, and Anna Wienhard.
\newblock Surface group representations with maximal {T}oledo invariant.
\newblock {\em Ann. of Math. (2)}, 172(1):517--566, 2010.

\bibitem{CollierSOnn+1components}
Brian Collier.
\newblock $\mathsf{SO}(n,n+1)$-surface group representations and their {H}iggs
  bundles.
\newblock {\em Ann. Sci. \'Ec. Norm. Sup\'er. (4)}, 63(6):1561--1616, 2020.

\bibitem{ColSandGlobalSlodowy}
Brian Collier and Andrew Sanders.
\newblock ({G},{P})-{O}pers and global {S}lodowy slices.
\newblock {\em Advances in Mathematics}, 377:107490, 2021.

\bibitem{CollMcGovNilpotents}
David~H. Collingwood and William~M. McGovern.
\newblock {\em Nilpotent orbits in semisimple {L}ie algebras}.
\newblock Van Nostrand Reinhold Mathematics Series. Van Nostrand Reinhold Co.,
  New York, 1993.

\bibitem{canonicalmetrics}
Kevin Corlette.
\newblock Flat {$G$}-bundles with canonical metrics.
\newblock {\em J. Differential Geom.}, 28(3):361--382, 1988.

\bibitem{DokovicKostConj}
Dragomir~\v{Z}. Dokovi\'{c}.
\newblock Proof of a conjecture of {K}ostant.
\newblock {\em Trans. Amer. Math. Soc.}, 302(2):577--585, 1987.

\bibitem{ExceptionalNilpotentsInner}
Dragomir~\v{Z}. Dokovi\'{c}.
\newblock Classification of nilpotent elements in simple exceptional real {L}ie
  algebras of inner type and description of their centralizers.
\newblock {\em J. Algebra}, 112(2):503--524, 1988.

\bibitem{ExceptionalNilpotentsOuter}
Dragomir~\v{Z}. Dokovi\'{c}.
\newblock Classification of nilpotent elements in simple real {L}ie algebras
  {$E_{6(6)}$} and {$E_{6(-26)}$} and description of their centralizers.
\newblock {\em J. Algebra}, 116(1):196--207, 1988.

\bibitem{DonagiGaitsgory}
R.~Y. Donagi and D.~Gaitsgory.
\newblock The gerbe of {H}iggs bundles.
\newblock {\em Transform. Groups}, 7(2):109--153, 2002.

\bibitem{DonagiGrassmannians}
Ron~Y. Donagi.
\newblock On the geometry of {G}rassmannians.
\newblock {\em Duke Math. J.}, 44(4):795--837, 1977.

\bibitem{harmoicmetric}
Simon Donaldson.
\newblock Twisted harmonic maps and the self-duality equations.
\newblock {\em Proc. London Math. Soc. (3)}, 55(1):127--131, 1987.

\bibitem{fanHiggsmodulianalytic}
Yue Fan.
\newblock Construction of the moduli space of {H}iggs bundles using analytic
  methods.
\newblock {\em arXiv:2004.07182 [math.DG]}, May 2020.
\newblock \url{https://arxiv.org/abs/2004.07182}.

\bibitem{Fischer-CAG-book}
Gerd Fischer.
\newblock {\em Complex analytic geometry}.
\newblock Lecture Notes in Mathematics, Vol. 538. Springer-Verlag, Berlin-New
  York, 1976.

\bibitem{fock_goncharov_2006}
Vladimir Fock and Alexander Goncharov.
\newblock Moduli spaces of local systems and higher {T}eichm\"uller theory.
\newblock {\em Publ. Math. Inst. Hautes \'Etudes Sci.}, 103:1--211, 2006.

\bibitem{nonmaxSp4}
O.~Garc{\'{\i}}a-Prada and I.~Mundet~i Riera.
\newblock Representations of the fundamental group of a closed oriented surface
  in {${\rm Sp}(4,{\Bbb R})$}.
\newblock {\em Topology}, 43(4):831--855, 2004.

\bibitem{HiggsPairsSTABILITY}
Oscar Garc{\'{\i}}a-Prada, Peter Gothen, and Ignasi Mundet~i Riera.
\newblock The {H}itchin-{K}obayashi correspondence, {H}iggs pairs and surface
  group representations.
\newblock {\em arXiv:0909.4487 [math.DG]}, September 2009.
\newblock \url{https://arxiv.org/abs/0909.4487}.

\bibitem{HiggsbundlesSP2nR}
Oscar Garc{\'{\i}}a-Prada, Peter Gothen, and Ignasi Mundet~i Riera.
\newblock Higgs bundles and surface group representaions in the real symplectic
  group.
\newblock {\em Journal of Topology}, 6(1):64--118, 2013.

\bibitem{Oliveira_GarciaPrada_2016}
Oscar Garc{\'{\i}}a-Prada and Andr{\'e} Oliveira.
\newblock Connectedness of {H}iggs bundle moduli for complex reductive {L}ie
  groups.
\newblock {\em Asian Journal of Mathematics}, 21(5):791--810, 2017.

\bibitem{AndreOscarSUstar}
Oscar Garc\'{\i}a-Prada and Andr\'{e}~G. Oliveira.
\newblock Higgs bundles for the non-compact dual of the unitary group.
\newblock {\em Illinois J. Math.}, 55(3):1155--1181 (2013), 2011.

\bibitem{Sp(2p2q)modulispaceconnected}
Oscar Garc{\'{\i}}a-Prada and Andr{\'e}~G. Oliveira.
\newblock Connectedness of the moduli of {${\rm Sp}(2p,2q)$}-{H}iggs bundles.
\newblock {\em Q. J. Math.}, 65(3):931--956, 2014.

\bibitem{GarciaPradaPeonRamanan-HKRsection}
Oscar Garc\'{\i}a-Prada, Ana Pe\'{o}n-Nieto, and S.~Ramanan.
\newblock Higgs bundles for real groups and the {H}itchin-{K}ostant-{R}allis
  section.
\newblock {\em Trans. Amer. Math. Soc.}, 370(4):2907--2953, 2018.

\bibitem{TopologicalComponents}
William~M. Goldman.
\newblock Topological components of spaces of representations.
\newblock {\em Invent. Math.}, 93(3):557--607, 1988.

\bibitem{sp4GothenConnComp}
Peter~B. Gothen.
\newblock Components of spaces of representations and stable triples.
\newblock {\em Topology}, 40(4):823--850, 2001.

\bibitem{AndreQuadraticPairs}
Peter~B. Gothen and Andr{\'e}~G. Oliveira.
\newblock Rank two quadratic pairs and surface group representations.
\newblock {\em Geom. Dedicata}, 161:335--375, 2012.

\bibitem{EGAIV-tome2}
A.~Grothendieck.
\newblock \'{E}l\'{e}ments de g\'{e}om\'{e}trie alg\'{e}brique. {IV}. \'{E}tude
  locale des sch\'{e}mas et des morphismes de sch\'{e}mas. {II}.
\newblock {\em Inst. Hautes \'{E}tudes Sci. Publ. Math.}, 24:231, 1965.

\bibitem{AnosovAndProperGGKW}
Fran\c{c}ois Gu\'{e}ritaud, Olivier Guichard, Fanny Kassel, and Anna Wienhard.
\newblock Anosov representations and proper actions.
\newblock {\em Geom. Topol.}, 21(1):485--584, 2017.

\bibitem{TopInvariantsAnosov}
O.~Guichard and A.~Wienhard.
\newblock Topological invariants of {A}nosov representations.
\newblock {\em Journal of Topology}, 3(3):578--642, Jan 2010.

\bibitem{GLWPosRepsarXiv}
Olivier {Guichard}, Fran{\c{c}}ois {Labourie}, and Anna {Wienhard}.
\newblock {Positivity and representations of surface groups}.
\newblock {\em arXiv:2106.14584 [math.DG]}, June 2021.
\newblock \url{https://arxiv.org/abs/2106.14584}.

\bibitem{guichard_wienhard_2012}
Olivier Guichard and Anna Wienhard.
\newblock Anosov representations: domains of discontinuity and applications.
\newblock {\em Invent. Math.}, 190(2):357--438, 2012.

\bibitem{PosRepsGWPROCEEDINGS}
Olivier Guichard and Anna Wienhard.
\newblock Positivity and higher {T}eichm{\"u}ller theory.
\newblock {\em Proceedings of the 7th {E}uropean {C}ongress of {M}athematics},
  2016.

\bibitem{GuichardWienhardPosFull}
Olivier {Guichard} and Anna {Wienhard}.
\newblock {Generalizing {L}usztig's total positivity}.
\newblock {\em arXiv:2208.10114 [math.DG]}, August 2022.
\newblock \url{https://arxiv.org/abs/2208.10114}.

\bibitem{HarisE7}
Stephen~J. Haris.
\newblock Some irreducible representations of exceptional algebraic groups.
\newblock {\em Amer. J. Math.}, 93:75--106, 1971.

\bibitem{Hartshorne-AlgebraicGeometry}
Robin Hartshorne.
\newblock {\em Algebraic Geometry}.
\newblock Springer-Verlag, New York-Heidelberg, 1977.
\newblock Graduate Texts in Mathematics, No. 52.

\bibitem{selfduality}
Nigel Hitchin.
\newblock The self-duality equations on a {R}iemann surface.
\newblock {\em Proc. London Math. Soc. (3)}, 55(1):59--126, 1987.

\bibitem{IntSystemFibration}
Nigel Hitchin.
\newblock Stable bundles and integrable systems.
\newblock {\em Duke Math. J.}, 54(1):91--114, 1987.

\bibitem{liegroupsteichmuller}
Nigel Hitchin.
\newblock Lie groups and {T}eichm\"uller space.
\newblock {\em Topology}, 31(3):449--473, 1992.

\bibitem{Igusa}
Jun-ichi Igusa.
\newblock A classification of spinors up to dimension twelve.
\newblock {\em Amer. J. Math.}, 92:997--1028, 1970.

\bibitem{SojiCausalStruct}
Soji Kaneyuki.
\newblock On the causal structures of the \v{S}ilov boundaries of symmetric
  bounded domains.
\newblock In {\em Prospects in complex geometry ({K}atata and {K}yoto, 1989)},
  volume 1468 of {\em Lecture Notes in Math.}, pages 127--159. Springer,
  Berlin, 1991.

\bibitem{KLPDynamicsProperCocompact}
Michael Kapovich, Bernhard Leeb, and Joan Porti.
\newblock Dynamics on flag manifolds: domains of proper discontinuity and
  cocompactness.
\newblock {\em Geom. Topol.}, 22(1):157--234, 2018.

\bibitem{knappbeyondintro}
Anthony~W. Knapp.
\newblock {\em Lie Groups -- Beyond an Introduction}, volume 140 of {\em
  Progress in Mathematics}.
\newblock Birkh\"auser Boston Inc., Boston, MA, second edition, 2002.

\bibitem{DiffGeomCompVectBun}
Shoshichi Kobayashi.
\newblock {\em Differential geometry of complex vector bundles}, volume~15 of
  {\em Publications of the Mathematical Society of Japan}.
\newblock Princeton University Press, Princeton, NJ; Iwanami Shoten, Tokyo,
  1987.
\newblock Kan{\^o} Memorial Lectures, 5.

\bibitem{KostantRallis}
B.~Kostant and S.~Rallis.
\newblock Orbits and representations associated with symmetric spaces.
\newblock {\em Amer. J. Math.}, 93:753--809, 1971.

\bibitem{ptds}
Bertram Kostant.
\newblock The principal three-dimensional subgroup and the {B}etti numbers of a
  complex simple {L}ie group.
\newblock {\em Amer. J. Math.}, 81:973--1032, 1959.

\bibitem{AnosovFlowsLabourie}
Fran{\c{c}}ois Labourie.
\newblock Anosov flows, surface groups and curves in projective space.
\newblock {\em Invent. Math.}, 165(1):51--114, 2006.

\bibitem{LandsbergManivelorbits}
J.~M. Landsberg and L.~Manivel.
\newblock The projective geometry of {F}reudenthal's magic square.
\newblock {\em J. Algebra}, 239(2):477--512, 2001.

\bibitem{JunLiConnectedness}
Jun Li.
\newblock The space of surface group representations.
\newblock {\em Manuscripta Math.}, 78(3):223--243, 1993.

\bibitem{LusztigTotPos}
G.~Lusztig.
\newblock Total positivity in reductive groups.
\newblock In {\em Lie theory and geometry}, volume 123 of {\em Progr. Math.},
  pages 531--568. Birkh\"{a}user Boston, Boston, MA, 1994.

\bibitem{NitsureHiggs}
Nitin Nitsure.
\newblock Moduli space of semistable pairs on a curve.
\newblock {\em Proc. London Math. Soc. (3)}, 62(2):275--300, 1991.

\bibitem{AndrePGLnR}
Andr\'e~Gama Oliveira.
\newblock Representations of surface groups in the projective general linear
  group.
\newblock {\em Internat. J. Math.}, 22(2):223--279, 2011.

\bibitem{BeatriceBourbaki}
Maria~Beatrice {Pozzetti}.
\newblock {Higher rank Teichm{\"u}ller theories}.
\newblock {\em S{\'e}minaire N. Bourbaki}, Mars 2019.

\bibitem{ramanathan_1975}
Annamalai Ramanathan.
\newblock Stable principal bundles on a compact {R}iemann surface.
\newblock {\em Mathematische Annalen}, 213(2):129--152, 1975.

\bibitem{schmitt_2005}
Alexander Schmitt.
\newblock Moduli for decorated tuples of sheaves and representation spaces for
  quivers.
\newblock {\em Proceedings Mathematical Sciences}, 115(1):15--49, 2005.

\bibitem{Sekiguchi}
Jir\={o} Sekiguchi.
\newblock Remarks on real nilpotent orbits of a symmetric pair.
\newblock {\em J. Math. Soc. Japan}, 39(1):127--138, 1987.

\bibitem{KatzMiddleInvCyclicHiggs}
Carlos Simpson.
\newblock Katz's middle convolution algorithm.
\newblock {\em Pure Appl. Math. Q.}, 5(2, Special Issue: In honor of Friedrich
  Hirzebruch. Part 1):781--852, 2009.

\bibitem{SimpsonVHS}
Carlos~T. Simpson.
\newblock Constructing variations of {H}odge structure using {Y}ang-{M}ills
  theory and applications to uniformization.
\newblock {\em J. Amer. Math. Soc.}, 1(4):867--918, 1988.

\bibitem{SimpsonModuli1}
Carlos~T. Simpson.
\newblock Moduli of representations of the fundamental group of a smooth
  projective variety. {I}.
\newblock {\em Inst. Hautes \'Etudes Sci. Publ. Math.}, 79:47--129, 1994.

\bibitem{SimpsonModuli2}
Carlos~T. Simpson.
\newblock Moduli of representations of the fundamental group of a smooth
  projective variety. {II}.
\newblock {\em Inst. Hautes \'Etudes Sci. Publ. Math.}, 80:5--79 (1995), 1994.

\bibitem{Slodowy-Simplesingularities-book}
Peter Slodowy.
\newblock {\em Simple singularities and simple algebraic groups}, volume 815 of
  {\em Lecture Notes in Mathematics}.
\newblock Springer, Berlin, 1980.

\bibitem{AnnaICM}
Anna Wienhard.
\newblock An invitation to higher {T}eichm\"{u}ller theory.
\newblock In {\em Proceedings of the {I}nternational {C}ongress of
  {M}athematicians---{R}io de {J}aneiro 2018. {V}ol. {II}. {I}nvited lectures},
  pages 1013--1039. World Sci. Publ., Hackensack, NJ, 2018.

\end{thebibliography}

\end{document}